\documentclass[a4paper,11pt]{article}
\usepackage{makeidx}
\usepackage[mathscr]{eucal}
\usepackage[dvips]{color}
\usepackage{amssymb, amsfonts, amsmath, amsthm, graphics} 

\newtheorem{proposition}{Proposition}[section]
\newtheorem{theorem}{Theorem}[section]
\newtheorem{lemma}[proposition]{Lemma}
\newtheorem{remark}{Remark}[section]
\newtheorem{corollary}[theorem]{Corollary}

\numberwithin{equation}{section}

\title{Blowup and Scattering problems for the Nonlinear Schr\"odinger equations}
\author{Takafumi Akahori and Hayato Nawa}
\date{}
\begin{document}
\maketitle

\begin{abstract}
We consider $L^{2}$-supercritical and $H^{1}$-subcritical focusing nonlinear Schr\"odinger equations. We introduce a subset $PW$ of $H^{1}(\mathbb{R}^{d})$ for $d\ge 1$, and investigate behavior of the solutions with initial data in this set. For this end, we divide $PW$ into two disjoint components $PW_{+}$ and $PW_{-}$. Then, it turns out that any solution starting from a datum in $PW_{+}$ behaves asymptotically free, and solution starting from a datum in $PW_{-}$ blows up or grows up, from which we find that the ground state has two unstable directions. 
We also investigate some properties of generic global and blowup solutions.
\end{abstract}

\section{Introduction}
In this paper, we consider the Cauchy problem for the nonlinear Schr\"odinger equation 
\begin{equation}\label{08/05/13/8:50}
2i\displaystyle{\frac{\partial \psi}{\partial t}}(x,t) +\Delta \psi(x,t)+|\psi(x,t)|^{p-1}\psi(x,t)=0,  \quad (x,t) \in \mathbb{R}^{d}\times \mathbb{R}, 
\end{equation}
where $i:=\sqrt{-1}$, $\psi$ is a complex-valued function on $\mathbb{R}^{d}\times \mathbb{R}$, $\Delta$ is the Laplace operator on $\mathbb{R}^{d}$ and $p$ satisfies the so-called $L^{2}$-supercritical and $H^{1}$-subcritical condition 
\begin{equation}\label{09/05/13/15:03}
2+\frac{4}{d} <p+1 <2^{*}:=\left\{ \begin{array}{ccl} 
\infty &\mbox{if}& d=1,2,
\\[6pt]
\displaystyle{\frac{2d}{d-2}} &\mbox{if}& d\ge 3. 
\end{array} \right.
\end{equation}
We associate this equation with the initial datum from the usual Sobolev space 
$H^{1}(\mathbb{R}^{d})$: 
\begin{equation}\label{09/12/16/15:39}
\psi(\cdot, 0)=\psi_{0} \in H^{1}(\mathbb{R}^{d}).
\end{equation}
We summarize the basic properties of this Cauchy problem (\ref{08/05/13/8:50}) and (\ref{09/12/16/15:39}) (see, e.g., \cite{Cazenave, Ginibre-Velo1, Kato1, Kato2, Kato1995, Sulem-Sulem}). The unique local existence of solutions is well known: for any $\psi_{0} \in H^{1}(\mathbb{R}^{d})$, there exists a unique solution $\psi$ in $C(I_{\max};H^{1}(\mathbb{R}^{d}))$ for some interval $I_{\max}=(-T_{\max}^{-}, T_{\max}^{+}) \subset \mathbb{R}$: maximal existence interval including $0$;  $T_{\max}^{+}$ ($-T_{\max}^{-}$) is the maximal existence time for the future (the past). If $I_{\max} \subsetneqq \mathbb{R}$, then we have 
\begin{equation}\label{10/01/27/11:26}
\lim_{t \to *T_{\max}^{*}}\left\| \nabla \psi(t) \right\|_{L^{2}}
=\infty \quad  (\mbox{blowup}),
\end{equation}
provided that $T_{\max}^{*}<\infty$, where $*$ stands for $+$ or $-$. Besides, the solution $\psi$ satisfies the following conservation laws of the mass $\mathcal{M}$, the Hamiltonian $\mathcal{H}$ and the momentum $\mathcal{P}$ in this order: for all $t \in I_{\max}$, 
\begin{align}
\label{08/05/13/8:59}
\mathcal{M}(\psi(t))
&:=\|\psi(t)\|_{L^{2}}^{2}=\mathcal{M}(\psi_{0}),
\\[6pt]
\label{08/05/13/9:03}
\mathcal{H}(\psi(t))
&:=\|\nabla \psi(t) \|_{L^{2}}^{2}- \frac{2}{p+1}\|\psi(t)\|_{L^{p+1}}^{p+1}=\mathcal{H}(\psi_{0}),
\\[6pt]
\label{08/10/20/4:37}
\mathcal{P}(\psi(t))
&:=\Im \int_{\mathbb{R}^{d}} \nabla \psi(x,t)\overline{\psi(x,t)}\,dx
=\mathcal{P}(\psi_{0}). 
\end{align}
If, in addition, $\psi_{0}\in L^{2}(\mathbb{R}^{d},|x|^{2}dx)$, then the corresponding solution $\psi$ also belongs to $C(I_{\max};L^{2}(\mathbb{R}^{d},|x|^{2}dx))$ and  satisfies the so-called virial identity (see \cite{Ginibre-Velo1}):
\begin{equation}\label{08/12/03/16:55}
\begin{split}
\int_{\mathbb{R}^{d}}|x|^{2}\left| \psi(x,t) \right|^{2}\,dx
&=
\int_{\mathbb{R}^{d}}|x|^{2}\left| \psi_{0}(x) \right|^{2}\,dx
+2t\Im{\int_{\mathbb{R}^{d}}x\cdot \nabla \psi_{0}(x)\overline{\psi_{0}(x)}\,dx}
\\
&\qquad +2\int_{0}^{t}\int_{0}^{t'}\mathcal{K}(\psi(t''))\,dt''dt'
\qquad 
\mbox{for all $t \in I_{\max}$}, 
\end{split}
\end{equation}
where $\mathcal{K}$ is a functional defined by 
\begin{equation}\label{10/02/13/11:38}
\mathcal{K}(f)
:=\|\nabla f\|_{L^{2}}^{2}-\frac{d(p-1)}{2(p+1)}\|f\|_{L^{p+1}}^{p+1},
\quad 
f \in H^{1}(\mathbb{R}^{d}).
\end{equation}
It is worth while noting that 
\begin{equation}\label{08/05/13/9:16}
\mathcal{K}(f)=\mathcal{H}(f)-\frac{d}{2(p+1)}\left\{ p-\left(1+\frac{4}{d}\right) \right\}\|f\|_{L^{p+1}}^{p+1},
\quad 
f\in H^{1}(\mathbb{R}^{d}), 
\end{equation}
so that: for any $f \in H^{1}(\mathbb{R}^{d})\setminus \{0\}$, we have 
\begin{eqnarray}
\label{08/05/13/9:11}
\mathcal{K}(f) > \mathcal{H}(f) & \mbox{if} & p<1+\frac{4}{d} ,
\\[6pt]
\label{08/05/13/11:12}
\mathcal{K}(f) = \mathcal{H}(f) & \mbox{if} & p=1+\frac{4}{d},
\\[6pt]
\label{08/05/13/11:13}
\mathcal{K}(f) < \mathcal{H}(f) & \mbox{if} & p>1+\frac{4}{d}.
\end{eqnarray}
The virial identity (\ref{08/12/03/16:55}) tells us the behavior of the ``variance'' of solution, from which we expect to obtain a kind of propagation or concentration estimates. However, we can not use (\ref{08/12/03/16:55}) as it is, since we do not require the weight condition $\psi_{0}\in L^{2}(\mathbb{R}^{d},|x|^{2}dx)$. We will work in the pure energy space $H^{1}(\mathbb{R}^{d})$, introducing a generalized version of the virial identity (see (\ref{08/03/29/19:05}) in the appended section \ref{08/10/19/14:57}) to discuss propagation or concentration  of a ``solution'' in Section \ref{09/05/06/9:13} and Section \ref{08/07/03/0:40}. 
\par
In some physical literature, nonlinear Schr\"odinger equations, abbreviated to NLSs, (or Gross-Pitaevskii equation) arise as NLSs with linear potentials: 
\begin{equation}\label{09/05/17/15:43}
2i\displaystyle{\frac{\partial \psi}{\partial t}}(x,t) +\Delta \psi(x,t) +|\psi(x,t)|^{p-1}\psi(x,t)=V(x)\psi(x,t), \quad (x,t) \in \mathbb{R}^{d}\times \mathbb{R}, 
\end{equation}
where $V$ is a real potential on $\mathbb{R}^{d}$. The case $p=d=3$ corresponds to the conventional $\phi^{4}$-model for Bose gasses with negative scattering length (see, e.g., \cite{Leggett}). Typical examples of the potential $V$ are: $V(x)=c_{h}|x|^{2}$ (harmonic potentials with $c_{h}>0$); and $V(x)=g\cdot x$  (the Stark potentials or uniform gravitational fields along $g \in \mathbb{R}^{d}\setminus \{0\}$). However, it is known that these potentials can be ``removed'' by appropriate space-time transformations: For $V=g\cdot x$, the Avrov-Herbst formula works well (see \cite{Carles-Yoshihisa, Cycon-Froese-Kirsch}); For $V=c_{h}|x|^{2}$, the space-time transformations in \cite{Carles, Niederer} do. Therefore, the study of NLSs without potentials may reveal the essential structure of the solutions to NLSs with potentials (\ref{09/05/17/15:43}).  
\par 
Our equation (\ref{08/05/13/8:50}) has several kinds of solutions: standing waves, blowup solutions (see (\ref{10/01/27/11:26}) above), and global-in-time solutions which asymptotically behave like  free solutions in the distant future/ distant past. Here, the standing waves are nontrivial solutions to our equation (\ref{08/05/13/8:50}) of the form 
\begin{equation}\label{10/03/24/16:08}
\psi(x,t)=e^{\frac{i}{2}\omega t}Q(x), \qquad \omega>0,\quad Q \in H^{1}(\mathbb{R}^{d})\setminus \{0\}.
\end{equation}
Thus, $Q$ solves the following semilinear elliptic equation (nonlinear scalar field equation):
\begin{equation}\label{08/05/13/11:22}
\Delta Q- \omega Q+|Q|^{p-1}Q=0, 
\qquad 
\omega>0,
\quad
Q \in H^{1}(\mathbb{R}^{d})\setminus \{0\}.
\end{equation} 
Here, we remark that every solution $Q$ to the equation (\ref{08/05/13/11:22}) satisfies $\mathcal{K}(Q)=0$. Indeed; since any solution $Q$ to (\ref{08/05/13/11:22}) belongs to the space $H^{1}(\mathbb{R}^{d})\cap L^{2}(\mathbb{R}^{d},|x|^{2}dx)$, the standing wave $\psi=e^{\frac{i}{2}\omega t}Q$ enjoys the virial identity (\ref{08/12/03/16:55}), which immediately leads us to $\mathcal{K}(Q)=0$. 
\par  
The standing waves are one of the interesting objects in the study of NLSs for both mathematics and physics: standing waves are considered to be the states of Bose-Einstein condensations.  In this paper, we are interested in precise instability mechanism of what we call the ground state (the least action solution to    (\ref{08/05/13/11:22}), see also (\ref{09/12/16/20:17}) below). 
 For this end, we employ the classical ``potential well'' theory traced back to Sattinger \cite{Stattinger}. To define our potential well $PW$, we need to know some variational properties of the ground state. We shall give the precise definition of $PW$ in (\ref{09/05/18/22:03}) below. Anyway, our $PW$ is divided into $PW_{+}$ and $PW_{-}$ according to the sign of the functional $\mathcal{K}$, i.e.,  $PW_{+}=PW\cap [\mathcal{K}>0]$, $PW_{-}=PW\cap [\mathcal{K}<0]$ (see (\ref{10/02/13/14:44}), (\ref{10/02/13/14:45})), and the ground state belongs to $\overline{PW_{+}}\cap \overline{PW_{-}}$. 
 We shall show that any solution starting from $PW_{+}$ exists globally in time and asymptotically behaves like a free solution in the distant future and past (see Theorem \ref{08/05/26/11:53}); in contrast, $PW_{-}$ gives rise to  ``singular'' solutions (see Theorem \ref{08/06/12/9:48}). Thus, the ground state shows at least two type of instability, since it belongs to $\overline{PW_{+}}\cap \overline{PW_{-}}$. 
\par 
In order to define our potential well $PW$, we shall investigate some properties of the ground states here. 
There are many literature concerning the elliptic equation (\ref{08/05/13/11:22}) (see, e.g.,  \cite{Berestycki-Lions, Gidas-Ni-Nirenberg, Kwong, Strauss}). We know that if $d\ge 2$, there are infinitely many solutions (bound states) $Q_{\omega}^{n}$ ($n=1,2,\ldots$) such that 
\begin{equation}\label{09/12/16/20:16}
\mathcal{S}_{\omega}(Q_{\omega}^{n}):=\frac{1}{2}\left\| \nabla Q_{\omega}^{n} \right\|_{L^{2}}^{2}+
\frac{\omega}{2} \left\| Q_{\omega}^{n} \right\|_{L^{2}}^{2}
-\frac{1}{p+1}\left\| Q_{\omega}^{n} \right\|_{L^{p+1}}^{p+1}\to \infty \quad 
(n\to \infty),
\end{equation}
where the functional $\mathcal{S}_{\omega}$ is called the action for (\ref{08/05/13/11:22}) (see, e.g.,  \cite{Berestycki-Lions, Strauss}). The ground state $Q_{\omega}$ is the least action solution to (\ref{08/05/13/11:22}), more strongly, it satisfies that  
\begin{equation}\label{09/12/16/20:17}
\mathcal{S}_{\omega}(Q_{\omega})=\inf \left\{ \mathcal{S}_{\omega}(Q) \ | \ Q \in H^{1}(\mathbb{R}^{d})\setminus \{0\}, \ 
\mathcal{K}(Q)= 0
 \right\}.
\end{equation}
In the $L^{2}$-supercritical case ($p>1+\frac{4}{d}$), it turns out that the ground state solves the following variational problems (see Proposition \ref{08/05/13/15:13} below):
\begin{align}
\label{08/07/02/23:23}
N_{1}&:=\inf{ \left\{ \|f\|_{\widetilde{H}^{1}}^{2}
\bigm| f \in H^{1}(\mathbb{R}^{d})\setminus \{0\},\ \mathcal{K}(f)\le 0 \right\}},\\[6pt]
\label{08/07/02/23:24}
N_{2}&:=\inf{\left\{ \mathcal{N}_{2}(f) \bigm| f \in H^{1}(\mathbb{R}^{d})\setminus \{0\}, \ \mathcal{K}(f)\le 0
\right\} }, 
\\[6pt]
\label{08/07/02/23:25}
N_{3}&:=\inf{\Big\{ \mathcal{I}(f) \bigm|  f \in H^{1}(\mathbb{R}^{d})\setminus \{0\} \Big\}}, 
\end{align}
where 
\begin{align}
\label{10/02/19/15:27}
\|f\|_{\widetilde{H}^{1}}^{2}
&:=
\frac{d(p-1)-4}{d(p-1)}\|\nabla f \|_{L^{2}}^{2}+\|f\|_{L^{2}}^{2},
\\[6pt]
\label{10/02/19/15:28}
\mathcal{N}_{2}(f)
&:=
\left\| f \right\|_{L^{2}}^{p+1-\frac{d}{2}(p-1)}
\left\| \nabla f \right\|_{L^{2}}^{\frac{d}{2}(p-1)-2},
\\[6pt]
\label{10/02/19/15:29}
\mathcal{I}(f)
&:=
\frac{\|f\|_{L^{2}}^{p+1-\frac{d}{2}(p-1)}\|\nabla f\|_{L^{2}}^{\frac{d}{2}(p-1)}}{\|f\|_{L^{p+1}}^{p+1}}
.
\end{align}
We remark that $N_{3}$ is positive and gives the best constant of the Gagliardo-Nirenberg inequality, i.e., 
\begin{equation}\label{08/05/13/15:45}
\|f\|_{L^{p+1}}^{p+1}\le \frac{1}{N_{3}}\|f\|_{L^{2}}^{p+1-\frac{d}{2}(p-1)}\|\nabla f \|_{L^{2}}^{\frac{d}{2}(p-1)}
\quad 
\mbox{for all $f \in H^{1}(\mathbb{R}^{d})$}
\end{equation}
(for the details, see \cite{Weinstein} and the exposition by Tao \cite{Tao-book1}).
A salient point of our approach to these variational problems (\ref{08/07/02/23:23})--(\ref{08/07/02/23:25}) is that we consider them at the same time, which enables us to obtain some identities automatically (see (\ref{08/05/13/15:06}) and (\ref{08/09/24/15:13}) below). The first variational problem (\ref{08/07/02/23:23}) is introduced to make our variational problems easy to solve. Indeed, all minimizing sequences  for (\ref{08/07/02/23:23}) are bounded in $H^{1}(\mathbb{R}^{d})$; In contrast,  minimizing sequences in the other problems (\ref{08/07/02/23:24}) and (\ref{08/07/02/23:25}) are not necessarily bounded in $H^{1}(\mathbb{R}^{d})$, this fact comes from the invariance of the functionals $\mathcal{K}$, $\mathcal{N}_{2}$ and $\mathcal{I}$ under the scaling 
\begin{equation}\label{09/12/21/10:33}
Q\to \lambda^{\frac{2}{p-1}}Q(\lambda \, \cdot), \quad \lambda>0.
\end{equation} 
The variational values $N_{1}$, $N_{2}$ and $N_{3}$ are closely related each other, as stated in the following: 
\begin{proposition}\label{08/06/16/15:24}
Assume that $d\ge 1$ and $2+\frac{4}{d}<p+1<2^{*}$. Then, we have the following relations:
\begin{equation}\label{08/05/13/11:33}
N_{1}^{\frac{p-1}{2}}=\left( \frac{2}{d}\right)^{\frac{p-1}{2}}
\left\{ \frac{d(p-1)}{(d+2)-(d-2)p}\right\}^{\frac{1}{4}\left\{ (d+2)-(d-2)p \right\}}N_{2}
\end{equation}
and
\begin{equation}\label{08/05/13/11:25}
N_{3}=\frac{d(p-1)}{2(p+1)} N_{2}.
\end{equation}
\end{proposition}

The next proposition tells us that the variational problems (\ref{08/07/02/23:23}), (\ref{08/07/02/23:24}) and (\ref{08/07/02/23:25}) have the same minimizer. 
\begin{proposition}
\label{08/05/13/15:13}
Assume that $d\ge 1$ and $2+\frac{4}{d}<p+1<2^{*}$. Then, there exists a positive function $Q$ in the Schwartz space $\mathcal{S}(\mathbb{R}^{d})$, which is unique modulo translation and phase shift, such that 
\begin{align}
\label{08/05/13/11:23}
N_{1}&=\|Q\|_{\widetilde{H}^{1}}^{2},
\\[6pt]
\label{08/05/26/11:41}
N_{2}&=\mathcal{N}_{2}(Q),
\\[6pt]
\label{08/05/13/11:27}
N_{3}&=\mathcal{I}(Q),
\\[6pt]
\label{08/05/13/15:06}
\mathcal{K}(Q)&=0,
\\[6pt]
\label{08/09/24/15:13}
\left\| Q \right\|_{L^{2}}^{2}
&=
\frac{d+2-(d-2)p}{d(p-1)}
\left\| \nabla Q\right\|_{L^{2}}^{2}
=
\frac{d+2-(d-2)p}{2(p+1)}
\left\| Q \right\|_{L^{p+1}}^{p+1},
\end{align}
and $Q$ solves the equation (\ref{08/05/13/11:22}) with $\omega =1$. 
\end{proposition} 
\begin{remark}\label{08/09/27/22:26}
It is known that the function $Q$ found in Proposition \ref{08/05/13/15:13} has the following properties:
\\
{\rm (i)} There exists $y \in \mathbb{R}^{d}$ such that $Q(x-y)$ is radially symmetric (see \cite{Gidas-Ni-Nirenberg}).  
\\
{\rm (ii)} $Q$ is the ground state of the equation (\ref{08/05/13/11:22}) with $\omega=1$. Moreover, for all $\omega>0$, the rescaled function $\omega^{-\frac{1}{p-1}}Q(\omega^{-\frac{1}{2}}\cdot )$ becomes the ground state of (\ref{08/05/13/11:22}), and denoting this function by $Q_{\omega}$, we easily verify that  
\begin{equation}\label{10/03/27/10:23}
N_{2}=\mathcal{N}_{2}(Q_{\omega}), 
\quad 
N_{3}=\mathcal{I}(Q_{\omega}),
\quad 
\mathcal{K}(Q_{\omega})=0
\qquad 
\mbox{for all $\omega >0$.}
\end{equation}
\end{remark}

Now, using the variational value $N_{2}$, we define a ``potential well'' $PW$ by\begin{equation}\label{09/05/18/22:03}
PW=\left\{ f \in H^{1}\setminus \{0\} \bigm| \mathcal{H}(f) <\mathcal{B}(f) \right\},
\end{equation}
where 
\begin{equation}\label{09/05/18/22:02}
\mathcal{B}(f)=\frac{d(p-1)-4}{d(p-1)}\left( \frac{N_{2}}{\|f\|_{L^{2}}^{p+1-\frac{d}{2}(p-1)}}\right)^{\frac{4}{d(p-1)-4}}
.
\end{equation}
 
We divide $PW$ into two components according to the sign of $\mathcal{K}$:
\begin{equation}\label{10/02/13/14:44}
PW_{+}=\left\{ f \in PW \bigm| \mathcal{K}(f) >0 \right\},
\end{equation}
\begin{equation}\label{10/02/13/14:45}
PW_{-}=\left\{ f \in PW \bigm| \mathcal{K}(f) <0 \right\}. 
\end{equation}
It is worth while noting the following facts:
\begin{enumerate}
\item[1.] $PW_{+}$ and $PW_{-}$ are unbounded open sets in $H^{1}(\mathbb{R}^{d})$:  Indeed, one can easily verify this fact by considering the scaled functions  $f_{\lambda}(x):=\lambda^{\frac{2}{p-1}}f(\lambda x)$ for $f \in H^{1}(\mathbb{R}^{d})$ and $\lambda>0$. 
\item[2.]  $PW=PW_{+}\cup PW_{-}$ and $PW_{+}\cap PW_{-}=\emptyset$ (see Lemma \ref{09/06/21/15:50})
\item[3.] 
$PW_{+}$ and $PW_{-}$ are invariant under the flow defined by the equation (\ref{08/05/13/8:50}) (see Proposition \ref{09/06/21/19:28} and Proposition \ref{08/05/26/10:57}). 
\end{enumerate}
\begin{enumerate}
\item[4.] The ground state $Q_{\omega}$ belongs to  $\overline{PW_{+}}\cap \overline{PW_{-}}$ and $Q_{\omega} \not \in PW_{+}\cup PW_{-}$ for all $\omega>0$, where $\overline{PW_{+}}$ and $\overline{PW_{-}}$ are the closures of $PW_{+}$ and $PW_{-}$ in the $H^{1}$-topology, respectively (see Theorem \ref{08/10/20/5:21}). Moreover, the orbit under the action $((0,\infty) \ltimes \mathbb{R}^{d})\times S^{1}$
\begin{equation}\label{10/03/31/12:23}
\left\{ 
\lambda^{\frac{2}{p-1}} e^{i\theta}Q_{\omega}(\lambda (\cdot-a)) \ 
\left| \ \lambda >0, 
\
a \in \mathbb{R}^{d},
\ \theta \in S^{1}  \right.
\right\}
\end{equation}
is contained in $\overline{PW_{+}} \cap \overline{PW_{-}}$. 
\end{enumerate}
Here, the last fact above is the key to show the instability of the ground state. We will prove these facts in Section \ref{08/07/02/23:51}, among other properties of these sets $PW$, $PW_{+}$ and $PW_{-}$.
\par 
For the later convenience, we prepare another expression of $PW_{+}$: Besides the functional $\mathcal{N}_{2}$ and the variational value $N_{2}$, we define a functional $\widetilde{\mathcal{N}}_{2}$ and a number $\widetilde{N}_{2}$ by   
\begin{equation}\label{09/04/29/14:30}
\widetilde{\mathcal{N}}_{2}(f)
:=
\left\| f \right\|_{L^{2}}^{p+1-\frac{d}{2}(p-1)}\sqrt{\mathcal{H}(f)}^{\frac{d}{2}(p-1)-2}, \quad 
f \in H^{1}(\mathbb{R}^{d}) \ \mbox{with}\  \mathcal{H}(f)\ge 0, 
\end{equation}
and 
\begin{equation}\label{09/04/29/14:31}
\widetilde{N}_{2}:=\sqrt{\frac{d(p-1)-4}{d(p-1)}}^{\frac{d}{2}(p-1)-2}N_{2}.
\end{equation}
Then, it follows from (\ref{08/05/26/11:41}) and (\ref{08/09/24/15:13})  in Proposition \ref{08/05/13/15:13} that  
\begin{equation}\label{09/07/14/8:24}
\widetilde{N}_{2}=\widetilde{\mathcal{N}}_{2}(Q).
\end{equation} 
We can easily verify that: if $f\in H^{1}(\mathbb{R}^{d})$ with $\mathcal{H}(f)\ge 0$, then   
\begin{equation}\label{08/06/15/14:38}
f \in PW
\quad \mbox{if and only if} \quad  
\widetilde{\mathcal{N}}_{2}(f)
<
\widetilde{N}_{2}.
\end{equation}
Therefore, we have another expression of $PW_{+}$: 
\begin{equation}\label{09/12/16/17:07}
PW_{+}=\left\{ f \in H^{1}(\mathbb{R}^{d}) \ |\  \mathcal{K}(f)>0, \ 
\ \widetilde{\mathcal{N}}_{2}(f)
<
\widetilde{N}_{2}
\right\}.
\end{equation}

Here, in order to consider the wave operators, we introduce a set $\Omega$ which is a subset of $PW_{+}$ (see Remark \ref{08/11/16/22:46}, (ii) below): 
\begin{equation}\label{10/06/13/11:42}
\Omega:=\left\{f \in H^{1}(\mathbb{R}^{d})\setminus \{0\} \ | \  
\mathcal{N}_{2}(f)<\widetilde{N}_{2} 
 \right\}.
\end{equation}

\begin{remark}\label{08/11/16/22:46} 
{\rm (i)} When $p=1+\frac{4}{d}$, by (\ref{09/04/29/14:30}) and (\ref{09/04/29/14:31}), the condition $\widetilde{\mathcal{N}}_{2}(f)< \widetilde{N}_{2}=\widetilde{\mathcal{N}}_{2}(Q)$ can be reduced to 
 \[
\left\| f \right\|_{L^{2}}< \left\| Q\right\|_{L^{2}},
\]
since  
\[
\lim_{p\downarrow 1+\frac{4}{d}}\sqrt{\frac{d(p-1)-4}{d(p-1)}}^{\frac{d}{2}(p-1)-2} = 1.
\]
Hence, $PW_{+}$ formally becomes in this $L^{2}$-critical case   
\begin{equation}
\begin{split}
\label{10/02/16/17:03}
PW_{+}
&=  
\left\{ f \in H^{1}(\mathbb{R}^{d}) \bigm| 
\mathcal{H}(f)> 0, \ \left\| f \right\|_{L^{2}}<
\left\| Q \right\|_{L^{2}} \right\}
\\[6pt]
&=
\left\{ f \in H^{1}(\mathbb{R}^{d})\setminus \{0\} \bigm| 
\, \left\| f \right\|_{L^{2}}<
\left\| Q \right\|_{L^{2}} \right\},
\end{split}
\end{equation}
where we have used the facts that $\mathcal{K}=\mathcal{H}$ (see (\ref{08/05/13/11:12})) and 
 that $\left\| f \right\|_{L^{2}}<
\left\| Q \right\|_{L^{2}}$ implies that $\mathcal{H}(f)>0$ (see \cite{Weinstein, Nawa5}). 
On the other hand, $PW_{-}$ formally becomes 
\begin{equation}\label{10/02/16/17:04}
PW_{-}=\left\{ f \in H^{1}(\mathbb{R}^{d}) \ | \ 
\mathcal{H}(f)< 0
\right\},
\end{equation}
since $\mathcal{K}=\mathcal{H}$ (\ref{08/05/13/11:12}) again.
It is well known that the solutions of (\ref{08/05/13/8:50}) with $p=1+\frac{4}{d}$ with initial data from the set (\ref{10/02/16/17:03}) exist globally in time (see \cite{Weinstein}), and the ones with data from (\ref{10/02/16/17:04}) blow up or grow up (see \cite{Nawa8}). Thus, we may say that our potential wells $PW_{+}$ and $PW_{-}$ in (\ref{10/02/13/14:44}) and (\ref{10/02/13/14:45}) are natural extensions of those in (\ref{10/02/16/17:03}) and (\ref{10/02/16/17:04}) to the case of $2+\frac{4}{d}<p+1<2^{*}$.
\\
{\rm (ii)} Since $\widetilde{N}_{2}<N_{2}$, we find by the definition of $N_{2}$ (see (\ref{08/07/02/23:24})) that $\Omega \subset PW_{+}$.
\end{remark}

Now, we are in a position to state our main results. When symbols with $\pm$ appear in the following theorems and propositions, we always take both upper signs or both lower signs in the double signs. 
\par 
The first theorem below is concerned with the behavior of the solutions with initial data from $PW_{+}$. 
 
\begin{theorem}[Global existence and scattering]
\label{08/05/26/11:53}
Assume that $d\ge 1$, $2+\frac{4}{d}<p+1 <2^{*}$ and $\psi_{0} \in PW_{+}$. Then, the corresponding solution $\psi$ to the equation (\ref{08/05/13/8:50}) exists globally in time and has the following properties: 
\\[6pt]
{\rm (i)} $\psi$ stays in $PW_{+}$ for all time, and satisfies that   
\begin{equation}\label{08/12/16/10:09}
\inf_{t\in \mathbb{R}}\mathcal{K}(\psi(t))\ge  
\left( 1-\frac{\widetilde{N}_{2}(\psi_{0})}{\widetilde{N}_{2}}\right)
\mathcal{H}( \psi_{0})>0. 
\end{equation}
{\rm (ii)} $\psi$ belongs to $L^{\infty}(\mathbb{R};H^{1}(\mathbb{R}))$. In particular,  
\begin{equation}\label{08/09/03/17:03}
\sup_{t\in \mathbb{R}}\|\nabla \psi(t)\|_{L^{2}}^{2}
\le \frac{d(p-1)}{d(p-1)-4} \mathcal{H}(\psi_{0}).
\end{equation}
Furthermore, 
\\
{\rm (iii)} There exist unique $\phi_{+} \in \Omega$ and $\phi_{-}\in \Omega$ such that 
\begin{equation}\label{10/01/26/20:32}
\lim_{t\to \pm \infty}\left\|\psi(t)-e^{\frac{i}{2}t\Delta}\phi_{\pm} \right\|_{H^{1}}=
\lim_{t\to \pm \infty}\left\|e^{-\frac{i}{2}t\Delta}\psi(t)-\phi_{\pm} \right\|_{H^{1}}
=0.
\end{equation}
This formula defines the operators $W_{\pm}^{*} \colon \psi_{0} \ \mapsto \ \phi_{\pm}=\lim_{t\to \pm \infty}e^{-\frac{i}{2}t\Delta}\psi(t)$. 
 These operators become homeomorphisms from $PW_{+}$ to $\Omega$, so that we can define the scattering operator $S:=W_{+}^{*}W_{-}$ from $\Omega$ into itself, where $W_{-}:=(W_{-}^{*})^{-1}$:
\begin{equation}\label{10/03/31/12:00}
\begin{array}{ccll}
S=W_{+}^{*}W_{-} : &  \Omega & \to &  \Omega
\\
{}& \text{\rotatebox{90}{$\in$}} & {} & 
\text{\rotatebox{90}{$\in$}}
\\
{}& \phi_{-} &\mapsto &  \phi_{+} \qquad .
\end{array}
\end{equation}
\end{theorem}
\vspace{12pt}
\begin{remark}\label{10/01/27/18:15}
{\rm (i)} Theorem \ref{08/05/26/11:53} is an extension of the result by Duyckaerts, Holmer and Roudenko \cite{D-H-R}. See {\bf Notes and Comments} below for the details. 
\\[6pt]
{\rm (ii)} $\phi_{+}$ and $\phi_{-}$ found in Theorem \ref{08/05/26/11:53} are 
 called the asymptotic states at $+\infty$ and $-\infty$, respectively.
\\[6pt]
{\rm (iii)} To prove the surjectivity of $W_{\pm}^{*}$, we need to construct so-called wave operators $W_{\pm}$ (see Section \ref{09/05/30/21:25}). In fact,  the operator $W_{+}^{*}$ ($W_{-}^{*}$) is the inverse of the wave operator $W_{+}$ ($W_{-}$). According to the terminology of spectrum scattering theory for linear Schr\"odinger equation, we might say that $\Omega$ is a set of scattering states. 
\end{remark}
\indent 
In contrast to the case of $PW_{+}$, the solutions with initial data from $PW_{-}$ become singular: 
\begin{theorem}[Blowup or growup]
\label{08/06/12/9:48}
Assume that $d\ge 1$, $2+\frac{4}{d}< p+1< 2^{*}$ and $\psi_{0} \in PW_{-}$. 
Then, the corresponding solution $\psi$ to the equation (\ref{08/05/13/8:50}) satisfies the followings:
\\[6pt]
{\rm (i)} $\psi$ stays $PW_{-}$ as long as it exists  and satisfies that  
\begin{equation}\label{09/12/23/22:48}
\mathcal{K}(\psi(t))<- \left( \mathcal{B}(\psi_{0})-\mathcal{H}(\psi_{0})\right)<0
\quad 
\mbox{for all $t \in I_{\max}$}.
\end{equation}
{\rm (ii)} $\psi$ blows up in a finite time or grows up, that is, 
\begin{equation}\label{09/05/14/13:32}
\sup_{t\in [0, T_{\max}^{+})}\left\| \nabla \psi(t) \right\|_{L^{2}}=
\sup_{t\in (-T_{\max}^{-}, 0]}\left\| \nabla \psi(t) \right\|_{L^{2}}
=\infty .
\end{equation}
In particular, if $T_{\max}^{\pm}= \infty$, then we have    
\begin{equation}\label{09/05/13/15:58}
\limsup_{t \to \pm \infty} \int_{|x|>R} |\nabla \psi(x,t)|^{2}\,dx =\infty\quad \mbox{for all $R>0$}.
\end{equation}
\end{theorem}

\begin{remark}
{\rm (i)} We do not know whether a solution growing up at infinity exists.
\\[6pt]
{\rm (ii)} We know (see \cite{Glassey}) that if $\psi_{0} \in H^{1}(\mathbb{R}^{d})\cap L^{2}(\mathbb{R}^{d},|x|^{2}dx)$, then $T_{\max}^{\pm}<\infty$ and the corresponding solution $\psi$ satisfies that  
\[
\lim_{t\to \pm T_{\max}^{\pm}}\left\| \nabla \psi(t)\right\|_{L^{2}}=\infty.
\]
For the case $\psi_{0}\not\in L^{2}(\mathbb{R}^{d},|x|^{2}dx)$, see Theorem \ref{08/04/21/9:28} below (see also \cite{Ogawa-Tsutsumi}).
\end{remark}
Combining Theorems \ref{08/05/26/11:53} and \ref{08/06/12/9:48}, we can show the instability of the ground states: Precisely; 
\begin{theorem}[Instability of ground state]
\label{08/10/20/5:21}
Let $Q_{\omega}$ be the ground state of the equation (\ref{08/05/13/11:22}) for  $\omega>0$. Then, $Q_{\omega}$ has two unstable directions in the sense that $Q_{\omega} \in \overline{PW_{+}}\cap \overline{PW_{-}}$. In particular, for any $\varepsilon>0$, there exist $f_{+} \in PW_{+}$ and $f_{-}\in PW_{-}$ such that
\[
\left\| Q_{\omega}-f_{\pm}\right\|_{H^{1}}\le \varepsilon.
\]
\end{theorem}
\begin{remark} 
{\rm (i)} 
An example of $f_{\pm}$ is $\left(1\mp \frac{\varepsilon}{\left\| Q_{\omega}\right\|_{H^{1}}}\right)Q_{\omega}$, where both upper or both lower signs should be chosen in the double signs. 
\\
{\rm (ii)}
The ground state $Q_{\omega}$ also has an (orbitally) stable ``direction''. Indeed, if we start with $e^{ivx}Q_{\omega}$ for ``small'' $v \in \mathbb{R}^{d}$, then $e^{i(vx-\frac{1}{2}v^{2}t)}e^{\frac{i}{2}\omega t}Q_{\omega}(x-vt)$ solves the equation (\ref{08/05/13/8:50}) and stays in a neighborhood of the orbit of $Q_{\omega}$ under the action of $\mathbb{R}^{d} \times S^{1}$ 
\[
\left\{ 
e^{i\theta}Q_{\omega}(\cdot-a) \ 
\left| \ a \in \mathbb{R}^{d},\ \theta \in S^{1}
 \right.
\right\}.
\]  
\end{remark}

We state further properties of solutions.

\begin{theorem}\label{09/05/18/10:43}
Assume that $d\ge 1$ and $2+\frac{4}{d}<p+1<2^{*}$. Let $\psi$ be a global solution to the equation (\ref{08/05/13/8:50}) with an initial datum in $H^{1}(\mathbb{R}^{d})$. Then, the following five conditions are equivalent: 
\\[6pt]
{\rm (i)} Decay of $L^{p+1}$-norm:  
\begin{equation}\label{10/04/22/12:19}
\lim_{t\to \infty}\left\| \psi(t) \right\|_{L^{p+1}}=0
.
\end{equation}
{\rm (ii)} Decay of $L^{q}$-norm:  
\begin{equation}\label{10/05/31/8:54}
\lim_{t\to \infty}\left\| \psi(t) \right\|_{L^{q}}=0
\quad 
\mbox{for all $q \in (2,2^{*})$}.
\end{equation}
\noindent 
{\rm (iii)} Boundedness of the Strichartz norms:  
\begin{equation}\label{10/04/22/12:24}
\left\|(1-\Delta)^{\frac{1}{2}} 
\psi 
\right\|_{L^{r}([0,\infty);L^{q})}< 
\infty
\quad 
\mbox{for all admissible pair $(q,r)$}.
\end{equation}
\noindent 
{\rm (iv)} Boundedness of $X$-norm (see (\ref{08/09/01/10:02}) for the definition of the space $X$): 
\begin{equation}\label{10/04/22/12:20}
\left\| \psi \right\|_{X([0,\infty))}<\infty,
\quad 
\left\| \psi \right\|_{L^{\infty}([0,\infty);H^{1})}<\infty.
\end{equation}
{\rm (v)} Existence of an asymptotic state at $+\infty$: There exists $\phi_{+} \in H^{1}(\mathbb{R}^{d})$ such that 
\[
\lim_{t\to \infty}\left\|\psi(t)-e^{\frac{i}{2}t\Delta}\phi_{+} \right\|_{H^{1}}=0.
\]
A similar result holds for the negative time case. 
\end{theorem}

As a corollary of Theorem \ref{08/05/26/11:53}, with the help of Theorem \ref{09/05/18/10:43}, we obtain the following result: 
\begin{corollary}\label{09/12/23/22:22}
$PW_{+}\cup \{0\}$ is connected in the $L^{q}(\mathbb{R}^{d})$-topology for all $q\in (2, 2^{*})$.
\end{corollary}

If, for example, $PW_{+}\cup\{0\}$ contains a neighborhood of $0$ in the $L^{q}$-topology, then we can say that $PW_{+}$ is connected in the $L^{q}$-topology, without adding $\{0\}$. However, for any $\varepsilon>0$, there is a function $f_{\varepsilon} \in H^{1}(\mathbb{R}^{d})$ such that $f_{\varepsilon} \notin PW_{+}$ and $\|f_{\varepsilon}\|_{L^{p+1}}=\varepsilon$, so that we need to add $\{0\}$ in Corollary \ref{09/12/23/22:22} above. On the other hand, we can verify that  $PW_{+}\cup \{0\}$ contains a sufficiently small ball in $H^{1}(\mathbb{R}^{d})$ (clearly, $PW_{+}$ does not contain $0$). However, we do not know that $PW_{+}$ is connected in the $H^{1}$-topology.
\\
\par 
Next, we consider singular solutions. The following theorem tells us that solutions with radially symmetric data from $PW_{-}$ blow up in a finite time: 
\begin{theorem}[Existence of blowup solution]
\label{08/04/21/9:28}
Assume that $d\ge 2$, $2+\frac{4}{d}< p+1 <2^{*}$, and  $p\le 5$ if $d=2$. Let $\psi_{0}$ be a radially symmetric function in $PW_{-}$ and let $\psi$ be the corresponding solution to the equation (\ref{08/05/13/8:50}).  Then, we have

\begin{equation}\label{10/02/22/20:33}
T_{\max}^{\pm}<\infty
\quad 
\mbox{and}
\quad  
\lim_{t \to \pm T_{\max}^{\pm}}\|\nabla \psi(t)\|_{L^{2}}=\infty.
\end{equation}
Furthermore, we have the followings:
\\
{\rm (i)} 
For all $m >0$,  there exists a constant $R_{m}>0$ such that  
\begin{equation}\label{08/06/12/8:54}
\int_{|x|>R}|\psi(x,t)|^{2}\,dx < m  \quad \mbox{for all $R\ge R_{m}$ and $t \in I_{\max}$}
.
\end{equation}
\vspace{6pt}
{\rm (ii)} For all sufficiently large $R>0$, we have  
\begin{equation}\label{08/06/10/10:21}
\int_{0}^{T_{\max}^{+}}(T_{\max}^{+}-t)\left( \int_{|x|>R} |\nabla \psi(x,t)|^{2}dx \right)dt <\infty ,
\end{equation}
\vspace{6pt}
\begin{equation}\label{08/06/10/10:22}
\int_{0}^{T_{\max}^{+}}(T_{\max}^{+}-t) \left( \int_{|x|>R} |\psi(x,t)|^{p+1}dx \right)dt <\infty.
\end{equation}
\vspace{6pt}
{\rm (iii)} For all sufficiently large $R>0$, we have 
\begin{equation}\label{08/06/12/6:56}
\int_{0}^{T_{\max}^{+}}(T_{\max}^{+}-t)\|\psi(t)\|_{L^{\infty}(|x|>R)}^{4}\,dt < \infty ,
\end{equation}
\vspace{6pt}
\begin{equation}\label{08/06/12/6:57}
\int_{0}^{T_{\max}^{+}}(T_{\max}^{+}-t)\|\psi(t)\|_{L^{p+1}(|x|>R)}^{\frac{4(p+1)}{p-1}}\,dt < \infty ,
\end{equation}
\vspace{6pt}
\begin{equation}\label{08/06/12/6:58}
\liminf_{t \to T_{\max}^{+}} (T_{\max}^{+}-t)\|\psi(t)\|_{L^{\infty}(|x|>R)}^{2}=0 ,
\end{equation}
\vspace{6pt}
\begin{equation}\label{08/06/12/6:59}
\liminf_{t \to T_{\max}^{+}} (T_{\max}^{+}-t)^{\frac{p-1}{2}}\|\psi(t)\|_{L^{p+1}(|x|>R)}^{p+1} =0 .
\end{equation}
For the negative time case, the corresponding results to {\rm (ii)} and {\rm (iii)} hold valid. 
\end{theorem} 

In Theorem \ref{08/04/21/9:28}, we require the condition $p\le 5$ if $d=2$, as well as  Ogawa and Tsutsumi \cite{Ogawa-Tsutsumi}. A difficulty in the case $p>5$ with $d=2$ might come from a kind of quantum forces such that the stronger the nonlinear effect becomes, the stronger the dispersion effect does. 
\\
\par 
We do not know a lot of things about the asymptotic behavior of such singular solutions as found in Theorem \ref{08/06/12/9:48}. What we can say is the following (for simplicity, we state the forward time case only):

\begin{proposition}[Asymptotic profiles of singular solutions]
\label{08/11/02/22:52}
Assume that $d\ge 1$ and $2+\frac{4}{d}<p+1<2^{*}$. Let $\psi$ be a solution to the equation (\ref{08/05/13/8:50}) such that 
\begin{equation}\label{10/03/31/16:55}
\limsup_{t\to T_{\max}^{+}}\|\nabla \psi(t)\|_{L^{2}}=
\limsup_{t\to T_{\max}^{+}}\|\psi(t)\|_{L^{p+1}}
=\infty,
\end{equation}
and let $\{t_{n}\}_{n\in \mathbb{N}}$ be a sequence in $[0,T_{\max}^{+})$ such that 
\begin{equation}\label{10/01/25/19:06}
\lim_{n\to \infty}t_{n}= T_{\max}^{+}, 
\qquad 
\left\| \psi(t_{n})\right\|_{L^{p+1}}
=\sup_{t\in [0,t_{n})}
\left\| \psi(t) \right\|_{L^{p+1}}.
\end{equation}
For this sequence $\{t_{n}\}$, we put  
\begin{equation}\label{10/03/31/16:59}
\lambda_{n}=
\left\| \psi(t_{n})\right\|_{L^{p+1}}^{-\frac{(p-1)(p+1)}{d+2-(d-2)p}},
\end{equation}
and consider the scaled functions   
\begin{equation}\label{10/03/31/17:00}
\psi_{n}(x,t):=\lambda_{n}^{\frac{2}{p-1}}\overline{\psi(\lambda_{n}x,t_{n}-\lambda_{n}^{2}t)}, \quad 
t \in \biggm( -\frac{T_{\max}^{+}-t_{n}}{\lambda_{n}^{2}}, 
\frac{t_{n}}{\lambda_{n}^{2}}
\biggm],
\quad 
n \in \mathbb{N}.
\end{equation}
Suppose that 
\begin{equation}\label{10/01/26/15:36}
\left\| \psi \right\|_{L^{\infty}([0,T_{\max}^{+});L^{\frac{d}{2}(p-1)})}<\infty
. 
\end{equation}
Then, there exists a subsequence of $\{\psi_{n}\}$ (still denoted by the same symbol) with the following properties: There exist
\\
{\rm (i)} a nontrivial function $\psi_{\infty} \in L^{\infty}([0,\infty);\dot{H}^{1}(\mathbb{R}^{d})\cap L^{p+1}(\mathbb{R}^{d}))$ solving the equation  (\ref{08/05/13/8:50}) in the $\mathscr{D}'([0,\infty);\dot{H}^{-1}(\mathbb{R}^{d})+L^{\frac{p+1}{p}}(\mathbb{R}^{d}))$-sense and
\\
{\rm (ii)} a sequence $\{\gamma_{n}\}$ in $\mathbb{R}^{d}$
\\
such that, putting $\widetilde{\psi}_{n}(x,t):=\psi_{n}(x+\gamma_{n},t)$, we have, for any $T>0$,   
\begin{align*}
&\lim_{n\to \infty}\widetilde{\psi}_{n}= \psi_{\infty} \quad \mbox{strongly in $L^{\infty}([0,T];L_{loc}^{q}(\mathbb{R}^{d}))$} 
\quad 
\mbox{for all $q \in [1,  2^{*})$}, 
\\[6pt]
&\lim_{n\to \infty}\nabla \widetilde{\psi}_{n} = \nabla \psi_{\infty} \quad \mbox{weakly* in $L^{\infty}([0,T];L_{loc}^{2}(\mathbb{R}^{d}))$}.
\end{align*}
Furthermore, for any $\varepsilon>0$, there exists $R>0$ such that, 
for $y_{n}=\lambda_{n}\gamma_{n}$,  
\begin{equation}\label{10/03/10/18:47}
\lim_{n\to \infty}\int_{|x-y_{n}|\le \lambda_{n}R}|\psi(x,t_{n})|^{\frac{d}{2}(p-1)} dx \ge (1-\varepsilon)\| \psi_{\infty} (0) \|_{L^{\frac{d}{2}(p-1)}}^{\frac{d}{2}(p-1)}.
\end{equation}
\end{proposition}

\begin{remark}\label{10/04/21/21:40}
The $L^{\frac{d(p-1)}{2}}(\mathbb{R}^{d})$-norm is invariant under the scaling leaving the equation (\ref{08/05/13/8:50}) invariant: Precisely, when $\psi$ is a solution to (\ref{08/05/13/8:50}), putting  
\begin{equation}\label{10/04/22/11:13}
\psi_{\lambda}(x,t):=
\lambda^{\frac{d(p-1)}{2}}\psi(\lambda x, \lambda^{2}t)
,\quad 
\lambda>0, 
\end{equation}
we see that $\psi_{\lambda}$ solves (\ref{08/05/13/8:50}) and satisfies that 
\begin{equation}\label{10/04/21/12:42}
\left\| \psi_{\lambda}(t) \right\|_{L^{\frac{d}{2}(p-1)}}
=
\left\| \psi(\lambda t) \right\|_{L^{\frac{d}{2}(p-1)}}
\quad 
\mbox{for all $t \in (-\frac{T_{\max}^{-}}{\lambda}, \ \frac{T_{\max}^{+}}{\lambda})$}
.
\end{equation}
\end{remark}

Proposition \ref{08/11/02/22:52} tells us that the $L^{\frac{d(p-1)}{2}}(\mathbb{R}^{d})$-norm of a singular solution concentrates at some point under the assumption (\ref{10/01/26/15:36}). In general, it is difficult to check whether the assumption (\ref{10/01/26/15:36}) holds (cf. Merle and Rapha\"el \cite{Merle-Raphael}). 
\par 
Without the assumption (\ref{10/01/26/15:36}), we have the following:

\begin{proposition}\label{10/01/26/14:49}
Under the same assumptions except (\ref{10/01/26/15:36}), definitions and notation in Proposition \ref{08/11/02/22:52}, we define the ``renormalized'' functions $\Phi_{n}^{RN}$ by 
\begin{equation}\label{10/03/31/17:07}
\Phi_{n}^{RN}(x,t)=\psi_{n}(x,t)-e^{\frac{i}{2}t\Delta}\psi_{n}(x,0),
\quad n\in \mathbb{N}.
\end{equation}
Then, for any $T>0$, $\{\Phi_{n}^{RN}\}_{n\in \mathbb{N}}$ is a uniformly bounded  sequence in $C([0,T];H^{1}(\mathbb{R}^{d}))$, and satisfies the following alternatives {\rm (i)} and {\rm (ii)}: 
\\
{\rm (i)} If   
\begin{equation}\label{10/01/25/18:38}
\lim_{n\to \infty}\sup_{t\in [0,T]}\left\| \Phi_{n}^{RN}(t)\right\|_{L^{\frac{d}{2}(p-1)}}=0,
\end{equation}
then 
\begin{equation}\label{08/11/20/17:56}
\lim_{n\to \infty}\sup_{t \in \left[t_{n}-\lambda_{n}^{2}T,
, t_{n} \right]}
\left\|\psi(t) -e^{\frac{i}{2}(t-t_{n})\Delta}\psi(t_{n}) \right\|_{L^{\frac{d}{2}(p-1)}}=0.
\end{equation}
\vspace{6pt}
{\rm (ii)} If     
\begin{equation}\label{10/01/25/18:39}
\limsup_{n\to \infty}\sup_{t\in [0,T]}\left\| \Phi_{n}^{RN}(t)\right\|_{L^{\frac{d}{2}(p-1)}}>0,
\end{equation}
then there exists a subsequence of $\{\Phi_{n}^{RN}\}$ (still denoted by the same symbol) with the following properties: There exist a nontrivial function 
 $\Phi \in L^{\infty}([0,\infty);H^{1}(\mathbb{R}^{d}))$ and a sequence $\{y_{n}\}$ in $\mathbb{R}^{d}$ such that, putting  
$\widetilde{\Phi}^{RN}_{n}(x,t)=\Phi^{RN}_{n}(x+y_{n},t)$, we have
\begin{equation}\label{10/03/15/19:51}
\lim_{n\to \infty}
\widetilde{\Phi}^{RN}_{n}= \Phi \quad 
\mbox{weakly* in $L^{\infty}([0,T];H^{1}(\mathbb{R}^{d}))$} .
\end{equation}
Here, $\Phi$ solves the following equation 
\begin{equation}\label{10/03/15/19:52}
\displaystyle{2i\frac{\partial \Phi}{\partial t}} +\Delta \Phi =-F,
\end{equation}
where $F$ is the nontrivial function in $L^{\infty}([0,\infty); L^{\frac{p+1}{p}}(\mathbb{R}^{d}))$ given by  
\begin{equation}\label{10/03/15/19:53}
\lim_{n\to \infty}|\psi_{n}|^{p-1}\psi_{n}= F \quad \mbox{weakly* in $\ L^{\infty}([0,T];L^{\frac{p+1}{p}}(\mathbb{R}^{d}))$}.
\end{equation}
Furthermore, for any $\varepsilon>0$, there exists $R>0$ such that 
\begin{equation}\label{08/11/20/17:57}
\lim_{n\to \infty} \int_{|x-y_{n}|\le \lambda_{n}R}
\left| \psi(x,t_{n}-\lambda_{n}^{2}T) - e^{-\frac{i}{2}\lambda_{n}^{2}T \Delta}\psi(x, t_{n})\right|^{\frac{d}{2}(p-1)}\!\! dx 
\ge (1-\varepsilon)\left\| \Phi(T) \right\|_{L^{\frac{d}{2}(p-1)}}^{\frac{d}{2}(p-1)}.
\end{equation}
\end{proposition}

If the case (i) of Proposition \ref{10/01/26/14:49} occurs, then we may say that the dynamics of the solution is composed of the free evolution and the dilation (\ref{10/04/22/11:13}). On the other hand, in (ii),  concentration of $L^{\frac{d(p-1)}{2}}$-mass occurs further. We remark that the left-hand side of (\ref{08/11/20/17:57}) is finite.
\\
\par 
Here, we discuss some relations between the previous works and our results: 
\\
{\bf Notes and Comments}. 
\begin{enumerate}
\item 
Our analysis in $PW_{+}$ is inspired by the previous work by Duyckaerts, Holmer and Roudenko \cite{D-H-R, Holmer-Roudenko} (also Kenig and Merle \cite{Kenig-Merle}). They considered a typical nonlinear Schr\"odinger equation, the equation (\ref{08/05/13/8:50}) with $d=p=3$, and proved, in \cite{D-H-R}, that: if $\psi_{0} \in H^{1}(\mathbb{R}^{3})$ satisfies that  
\begin{equation}\label{10/01/31/13:08}
\mathcal{M}(\psi_{0})\mathcal{H}(\psi_{0})
< \mathcal{M}(Q)\mathcal{H}(Q) ,
\quad 
\left\| \psi_{0} \right\|_{L^{2}} \left\| \nabla \psi_{0} \right\|_{L^{2}} 
< \left\| Q \right\|_{L^{2}} \left\| \nabla Q \right\|_{L^{2}},
\end{equation}
then the corresponding solution exists globally in time and has asymptotic states at $\pm \infty$, where $Q$ denotes the ground state of the equation (\ref{08/05/13/11:22}) with $\omega=1$. 
 In our terminology, we see that the condition (\ref{10/01/31/13:08}) is equivalent to that  $\psi_{0} \in PW_{+}\cup\{0\}$ via the variational problem for $N_{2}$ (see (\ref{08/07/02/23:24})). In this paper, we intensively study 
 the scattering problem on $PW_{+}$, so that we have Theorem \ref{08/05/26/11:53}, which is an extension of the result by Duycaerts et al \cite{D-H-R, Holmer-Roudenko} to all spatial dimensions $d\ge 1$ and $L^{2}$-supercritical and $H^{1}$-subcritical powers $2+\frac{4}{d}<p+1< 2^{*}$. Furthermore,  we establish the so-called asymptotic completeness: the wave operators $W_{\pm}$ exists on $\Omega$ and they are homeomorphisms from $\Omega$ to $PW_{+}$. \\
For the nonlinear Klein-Gordon equation, the corresponding result to Theorem \ref{08/05/26/11:53} was obtained by Ibrahim, Masmoudi and Nakanishi \cite{IMN}.
 \item 
In order to prove $PW_{+}$ being a set of scattering states, we basically 
 employ the argument of Kenig and Merle \cite{Kenig-Merle}. In their argument, the Bahouri-Gerard type compactness \cite{Bahouri-Gerard} plays an important role: Duyckaerts et al \cite{D-H-R} also used such a compactness (the profile decomposition due to Keraani \cite{Keraani}). However, we employ the classical compactness device due to Brezis and Lieb \cite{Brezis-Lieb} instead of that due to Bahouri and Gerard; The Brezis-Lieb type compactness device is also used to prove the existence of the ground states (see Section \ref{08/05/13/15:57}), and to investigate the blowup solutions (see the proof of Proposition \ref{10/01/26/14:49} in Section \ref{08/07/03/0:40}).  As long as we consider the $H^{1}$-solutions, this classical compactness devise seems to be enough for our analysis. 
\item
The decomposition scheme in Sections \ref{09/05/05/10:03} (also the profile decomposition due to Bahouri-Gerard \cite{Bahouri-Gerard}, Keraani \cite{Keraani}) seems to be a kind of perturbation methods employed in quantum physics like a ``Born type approximation scheme''.
 
\item
In the course of the proof of Theorem \ref{08/05/26/11:53}, we encounter a ``fake soliton'' (critical element in the terminology of Kenig-Merle \cite{Kenig-Merle}). Then, we make a slightly different approach from Duyckaerts et al \cite{D-H-R} to trace its motion (for details, see Section \ref{09/05/06/9:13}). Moreover, our choice of function spaces is different from theirs \cite{D-H-R, Holmer-Roudenko} (see Section \ref{08/10/07/9:01}). We are choosing our function spaces so that the generalized inhomogeneous Strichartz estimates due to Foshci \cite{Foschi} work well there. 
\item 
In \cite{Holmer-Roudenko} (also in \cite{Holmer-Roudenko2}), Holmer and Roudenko also considered the equation (\ref{08/05/13/8:50}) with $d=p=3$ and 
proved that if $\psi_{0} \in H^{1}(\mathbb{R}^{3})$ satisfies that  
\begin{align}
\label{10/04/05/15:37}
&\mbox{$\psi_{0}$ is radially symmetric},
\\[6pt]
\label{10/01/31/13:09}
&\mathcal{M}(\psi_{0})\mathcal{H}(\psi_{0})
< \mathcal{M}(Q)\mathcal{H}(Q) ,
\quad 
\left\|\psi_{0} \right\|_{L^{2}} \left\| \nabla \psi_{0} \right\|_{L^{2}} 
> \left\| Q \right\|_{L^{2}} \left\| \nabla Q \right\|_{L^{2}}
,
\end{align}
then the corresponding solution blows up in a finite time. 
In our terminology, we see that the condition (\ref{10/01/31/13:09}) is equivalent to $\psi_{0} \in PW_{-}$ via the variational problem for $N_{2}$. Hence, Theorem \ref{08/06/12/9:48}, together with Theorem \ref{08/04/21/9:28}, is an extension of their result, in particular, to all spatial dimensions $d\ge 1$ and powers $2+\frac{4}{d}<p+1 < 2^{*}$. 
\item 
Our $PW_{+}$ and $PW_{-}$ are naturally introduced by the potential well $PW$ by appealing to the variational structure of the ground states. We note again that the functional $\mathcal{K}$ divides the $PW$ into $PW_{+}$ and $PW_{-}$. 
Our potential well $PW$ seems new. One may find a similarity between $PW$ and the set of initial data given in Theorem 4.1 in Begout \cite{Begout}.  
However, the relevance is not clear.
\item
Stubbe \cite{Stubbe} already introduced the condition (\ref{10/01/31/13:08}) and proved the global existence of the solutions with initial data satisfying it. 
  He also conjectured that the condition is sharp in the sense that there exists an initial datum such that it does not satisfy the condition (\ref{10/01/31/13:08}) and leads to a solution blowing up in a finite time. 
Our result concerning $PW_{-}$ gives an affirmative answer to his conjecture.   \end{enumerate}
\par 
This paper is organized as follows. In Section \ref{08/07/02/23:51}, we discuss  properties of the potential well $PW$. In Section \ref{08/08/05/14:29}, we introduce function spaces in which Strichartz type estimates work well. We also give a small date theory and a long time perturbation theory. Theorem \ref{09/05/18/10:43} is proved here (see Section \ref{09/05/30/15:49}). In Section \ref{08/10/03/15:11}, we give the proofs of Theorem \ref{08/05/26/11:53} and Corollary \ref{09/12/23/22:22}. Section \ref{08/07/03/0:40} is devoted to the proofs of Theorem \ref{08/06/12/9:48}, Theorem \ref{08/04/21/9:28}, and Propositions \ref{08/11/02/22:52} and \ref{10/01/26/14:49}. In Appendices  \ref{08/10/19/14:57}, \ref{08/10/03/15:12}, \ref{08/10/07/9:00}, \ref{09/03/06/16:43} and \ref{08/10/07/9:02} are devoted to preliminaries and auxiliary results. Finally, in Section \ref{08/05/13/15:57}, we give the proofs of Propositions \ref{08/06/16/15:24} and \ref{08/05/13/15:13}. 
\\
\\
{\bf Notation}. We summarize the notation used in this paper. 
\par 
We keep the letters $d$ and $p$ to denote the spatial dimension and the power of nonlinearity of the equation (\ref{08/05/13/8:50}), respectively. 
\par 
$\mathbb{N}$ denotes the set of natural numbers, i.e., $\mathbb{N}=\{1,2,3,\ldots\}$. 
\par 
$I_{\max}$ denotes the maximal existence interval of the considering solution, which has the form 
\[
I_{\max}=(-T_{\max}^{-}, T_{\max}^{+}),
\]
where $T_{\max}^{+}>0$ is the maximal existence time for the future, and $T_{\max}^{-}>0$ is the one for the past.
\par  
Functionals concerned with conservation laws for the equation  (\ref{08/05/13/8:50}) are: the mass
\[
\mathcal{M}(f):=\left\| f \right\|_{L^{2}}^{2}
\qquad 
\mbox{(see (\ref{08/05/13/8:59}))}, 
\]
the Hamiltonian 
\[
\mathcal{H}(f):=\left\| \nabla f \right\|_{L^{2}}^{2}-\frac{2}{p+1}\left\| f \right\|_{L^{p+1}}^{p+1}
\qquad 
\mbox{(see (\ref{08/05/13/9:03}))},
\]
and the momentum 
\[
\mathcal{P}(f)=:\Im \int_{\mathbb{R}^{d}}\nabla f(x) \overline{f(x)}\,dx
\qquad 
\mbox{(see (\ref{08/10/20/4:37}))}.
\]
We also use the functional   
\[
\mathcal{K}(f):=\left\| \nabla f \right\|_{L^{2}}^{2}-\frac{d(p-1)}{2(p+1)}\left\| f \right\|_{L^{p+1}}^{p+1}
\qquad 
\mbox{(see (\ref{10/02/13/11:38}))}.
\]
The symbol $\mathcal{K}$ might stand for ``Kamiltonian'' (?).
\par 
 Functionals concerned with variational problems are the followings: 
\[
\|f\|_{\widetilde{H}^{1}}^{2}:=\frac{d(p-1)-4}{d(p-1)}\|\nabla f \|_{L^{2}}^{2}+\|f\|_{L^{2}}^{2}
\qquad 
\mbox{(see (\ref{10/02/19/15:27}))},
\]
\[
\mathcal{N}_{2}(f):=\left\| f \right\|_{L^{2}}^{p+1-\frac{d}{2}(p-1)}
\left\| \nabla f \right\|_{L^{2}}^{\frac{d}{2}(p-1)-2}
\qquad 
\mbox{(see (\ref{10/02/19/15:28}))},
\]
\[
\mathcal{I}(f):=\frac{\|f\|_{L^{2}}^{p+1-\frac{d}{2}(p-1)}\|\nabla f\|_{L^{2}}^{\frac{d}{2}(p-1)}}{\|f\|_{L^{p+1}}^{p+1}}
\qquad 
\mbox{(see (\ref{10/02/19/15:29})))}.
\]
Variational values concerned with these functionals are  
\[
N_{1}:=\inf{ \left\{ \|f\|_{\widetilde{H}^{1}}^{2}
\bigm| f \in H^{1}(\mathbb{R}^{d})\setminus \{0\},\ \mathcal{K}(f)\le 0 
\right\} }
\qquad 
\mbox{(see (\ref{08/07/02/23:23}))},
\]
\[
N_{2}:=\inf{\left\{ \mathcal{N}_{2}(f) \bigm| f \in H^{1}(\mathbb{R}^{d})\setminus \{0\}, \ \mathcal{K}(f)\le 0
\right\}}
\qquad 
\mbox{(see (\ref{08/07/02/23:24}))},
\]
\[
N_{3}:=\inf{\Big\{ \mathcal{I}(f) \bigm|  f \in H^{1}(\mathbb{R}^{d})\setminus 
\{0\} \Big\}}
\qquad 
\mbox{(see (\ref{08/07/02/23:25}))}.
\]
\indent 
We define our ``potential well'' by 
\[
PW:=\left\{ f \in H^{1}\setminus \{0\} \bigm| \mathcal{H}(f) <\mathcal{B}(f) \right\}
\qquad 
\mbox{(see (\ref{09/05/18/22:03}))},
\]
where 
\[
\mathcal{B}(f):=\frac{d(p-1)-4}{d(p-1)}\left( \frac{N_{2}}{\|f\|_{L^{2}}^{p+1-\frac{d}{2}(p-1)}}\right)^{\frac{4}{d(p-1)-4}}
\qquad 
\mbox{(see (\ref{09/05/18/22:02}))}.
\]
We put 
\[
\varepsilon_{0}:=\mathcal{B}(\psi_{0})-\mathcal{H}(\psi_{0}).
\]
We divide the set $PW$ into two disjoint components:
\[
PW_{+}:=\left\{ f \in PW \bigm| \mathcal{K}(f) >0 \right\}
\qquad 
\mbox{(see (\ref{10/02/13/14:44}))},
\]
\[
PW_{-}:=\left\{ f \in PW \bigm| \mathcal{K}(f) <0 \right\}
\qquad 
\mbox{(see (\ref{10/02/13/14:45}))}.
\]
We can rewrite the set $PW_{+}$ in the form 
\[
PW_{+}=\left\{ f \in H^{1}(\mathbb{R}^{d})\, \biggm| \, \mathcal{K}(f)>0,
\ 
\widetilde{\mathcal{N}}_{2}(f)<\widetilde{N}_{2}\right\},
\]
where   
\[
\widetilde{\mathcal{N}}_{2}(f):=\left\| f \right\|_{L^{2}}^{p+1-\frac{d}{2}(p-1)}
\sqrt{\mathcal{H}(f)}^{\frac{d}{2}(p-1)-2}
\qquad 
\mbox{(see (\ref{09/04/29/14:30}))}
\]
and 
\[
\widetilde{N}_{2}:=\sqrt{\frac{d(p-1)-4}{d(p-1)}}^{\frac{d}{2}(p-1)-2}N_{2}
\qquad
\mbox{(see (\ref{09/04/29/14:31}))}.
\]
We need a subset $\Omega$ of $PW_{+}$ below to consider the wave operators:
\[
\Omega:=\left\{ f \in H^{1}(\mathbb{R}^{d})\setminus \{0\} \, \biggm| \, \mathcal{N}_{2}(f)<\widetilde{N}_{2}\right\}
\qquad 
\mbox{(see (\ref{10/06/13/11:42}))}
.
\]
\indent  
$s_{p}$ stands for the critical regularity of our equation (\ref{08/05/13/8:50}), i.e.,   
\[
s_{p}:=\frac{d}{2}-\frac{2}{p-1}
\qquad 
\mbox{(see (\ref{09/05/13/15:00}))}.
\]
\indent 
We will fix a number $q_{1} \in (p+1,2^{*})$ in Sections \ref{08/08/05/14:29} and  \ref{08/10/03/15:11}.  The number $r_{0}$ is chosen for the pair $(q_{1},r_{0})$ being admissible, i.e., 
\[
\frac{1}{r_{0}}
:=
\frac{d}{2}\left( \frac{1}{2}-\frac{1}{q_{1}}\right).
\]
Furthermore, $r_{1}$ and $\widetilde{r}_{1}$ are defined respectively by  
\[
\frac{1}{r_{1}}
:=
\frac{d}{2}\left( \frac{1}{2}-\frac{1}{q_{1}}-\frac{s_{p}}{d}\right),
\]
\[
\frac{1}{\widetilde{r}_{1}}
:=
\frac{d}{2}\left( \frac{1}{2}-\frac{1}{q_{1}}+\frac{s_{p}}{d} \right).
\]
A pair $(q_{2},r_{2})$ is defined by 
\[
\frac{1}{q_{2}}:=
\frac{1}{p-1}\left( 1-\frac{2}{q_{1}}\right),
\quad 
\frac{1}{r_{2}}
:=
\frac{d}{2}\left( \frac{1}{2}-\frac{1}{q_{2}}-\frac{s_{p}}{d}\right).
\]
\indent 
For an interval $I$, we define the Strichartz type space $X(I)$ by 
\[
X(I):=L^{r_{1}}(I;L^{q_{1}}(\mathbb{R}^{d}))\cap L^{r_{2}}(I; L^{q_{2}}(\mathbb{R}^{d}))
\qquad
\mbox{(see (\ref{08/09/01/10:02}))}.
\]
Besides, we define the usual Strichartz space $S(I)$ by 
\[
S(I):=L^{\infty}(I; L^{2}(\mathbb{R}^{d})) \cap L^{r_{0}}(I; L^{q_{1}}(\mathbb{R}^{d}))
\qquad 
\mbox{(see (\ref{10/02/19/22:40}))}.
\]
\indent 
In order to show that the solutions starting from $PW_{+}$ have asymptotic states, we will introduce a set   
\[
PW_{+}(\delta):=
\left\{ 
f \in PW_{+} \, \big| \,  \widetilde{\mathcal{N}}_{2}(f)
< \delta 
\right\},
\quad \delta>0
\qquad
\mbox{(see (\ref{10/02/19/23:05}))}.
\]
We will also consider a variational value  
\[
\begin{split}
\widetilde{N}_{c}&:=
\sup{\left\{ \delta >0 \bigm| 
\mbox{$\forall \psi_{0} \in PW_{+}(\delta)$, $\|\psi\|_{X(\mathbb{R})}< \infty$} \right\}}
\\
&=\inf{\left\{ \delta >0 \bigm| \mbox{ $\exists \psi_{0} \in PW_{+}(\delta)$, $\|\psi\|_{X(\mathbb{R})}= \infty$} \right\}}
\qquad 
\mbox{(see (\ref{08/09/02/18:06}))},
\end{split}
\]
where $\psi$ denotes the solution to (\ref{08/05/13/8:50}) with $\psi(0)=\psi_{0}$.
\par  
If $(A_{0},\|\cdot \|_{A_{0}})$ and $(A_{1},\|\cdot \|_{A_{1}})$ are ``compatible''  normed vector spaces (see \cite{Bergh}, p.24.), then $\|\cdot \|_{A_{0}\cap A_{1}}$ denotes the norm of their intersection $X\cap Y$, i.e.,   
\[
\left\| f \right\|_{A_{0}\cap A_{1}}:=\max\{ \left\| f \right\|_{A_{0}}, \left\| f \right\|_{A_{1}}\}\quad \mbox{for all $f \in A_{0}\cap A_{1}$}.
\]
The symbol $(\cdot, \cdot)$ denotes the inner product of $L^{2}(\mathbb{R}^{d})=L^{2}(\mathbb{R}^{d};\mathbb{C})$, i.e., 
\[
(f,g):=\int_{\mathbb{R}^{d}}f(x)\overline{g(x)}\,dx,  
\quad f,g \in L^{2}(\mathbb{R}^{d}).
\]
\indent 
$C_{c}^{\infty}(\mathbb{R}^{d})$ denotes the set of infinitely differentiable functions from $\mathbb{R}^{d} \to \mathbb{C}$ with compact supports. 
\par 
Using the Fourier transformation $\mathcal{F}$, we define differential operators  $|\nabla|^{s}$, $(-\Delta)^{\frac{s}{2}}$ and $(1-\Delta)^{\frac{s}{2}}$, for $s \in \mathbb{R}$, by   
\[
|\nabla|^{s}f=(-\Delta)^{\frac{s}{2}}f
:=
\mathcal{F}^{-1}[|\xi|^{s}\mathcal{F}[f]],
\qquad 
(1-\Delta)^{\frac{s}{2}}f:=\mathcal{F}^{-1}[(1+|\xi|^{2})^{\frac{s}{2}}\mathcal{F}[f]].
\]

\section{Potential well $\boldsymbol{PW}$}\label{08/07/02/23:51}
In this section, we shall discuss fundamental properties of the sets $PW$, $PW_{-}$ and $PW_{+}$. In particular, we will prove that these sets are invariant under the flow defined by the equation (\ref{08/05/13/8:50}) (see Propositions \ref{09/06/21/18:53}, \ref{09/06/21/19:28} and \ref{08/05/26/10:57}). Moreover, we  prove  Theorem \ref{08/10/20/5:21} here.
\par 
We begin with the following fact:
\begin{lemma}
\label{09/06/21/15:50}
The set $PW$ does not contain any function $f$ with $\mathcal{K}(f)=0$, i.e.,   \begin{equation}\label{10/02/16/20:47}
\left\{ f \in H^{1}(\mathbb{R}^{d}) \bigm| \mathcal{K}(f)=0 \right\}
\cap PW =\emptyset ,
\end{equation}
so that 
\begin{equation}\label{10/02/16/20:48}
PW=PW_{+}\cup PW_{-}, \qquad PW_{+}\cap PW_{-}=\emptyset.
\end{equation}
\end{lemma}
\begin{proof}[Proof of Lemma \ref{09/06/21/15:50}]
Let $f$ be a function in $H^{1}(\mathbb{R}^{d})$ with $\mathcal{K}(f)=0$. Then,  the condition $\mathcal{K}(f)=0$ leads to that 
\begin{equation}\label{09/12/15/8:28}
\mathcal{H}(f)=\frac{d(p-1)-4}{d(p-1)}\left\| \nabla f \right\|_{L^{2}}^{2}.
\end{equation}
Using (\ref{09/12/15/8:28}) and the definitions of $N_{2}$ (see (\ref{08/07/02/23:24})) and $\widetilde{N}_{2}$  (see (\ref{09/04/29/14:31})), we obtain that 
\begin{equation}\label{10/01/06/18:01}
\begin{split}
\widetilde{\mathcal{N}}_{2}(f)&=
\sqrt{ \frac{d(p-1)-4}{d(p-1)}
}^{\frac{d}{2}(p-1)-2}
\mathcal{N}_{2}(f) 
\\
& \ge 
\sqrt{ \frac{d(p-1)-4}{d(p-1)} }^{\frac{d}{2}(p-1)-2}N_{2}=\widetilde{N}_{2}.
\end{split}
\end{equation}
Hence, it follows from (\ref{08/06/15/14:38}) that $f \not \in PW$. 
\par 
We immediately obtain (\ref{10/02/16/20:48}) from (\ref{10/02/16/20:47}) and the definitions of $PW_{+}$ and $PW_{-}$ (see (\ref{10/02/13/14:44}) and (\ref{10/02/13/14:45})).   
\end{proof}
In the next lemma, we consider a path constructed from the ground state:  
\begin{lemma}\label{08/06/14/23:37}
Let $Q_{\omega}$ be the ground state of the equation (\ref{08/05/13/11:22}) for  $w>0$. We consider a path $\Gamma_{\omega} \colon [0,\infty) \to H^{1}(\mathbb{R}^{d})$ given by $\Gamma_{\omega}(s)=sQ_{\omega}$ for $s\ge 0$. Then, $\Gamma_{\omega}$ is continuous and satisfies that 
\begin{align}\label{10/02/16/20:56}
& \Gamma_{\omega}(s)\in PW_{+} 
\quad 
\mbox{for all $s \in (0,1)$},
\\[6pt]
\label{10/02/16/21:03}
&\Gamma_{\omega}(1)=Q_{\omega} \notin PW=PW_{+}\cup PW_{-},
\\[6pt]
\label{10/02/16/20:58}
&\Gamma_{\omega}(s) \in \left\{f \in H^{1}(\mathbb{R}^{d}) \bigm| \mathcal{H}(f) > 0\right\} \cap PW_{-}
\quad 
\mbox{for all  $s \in \left(1, \ \left\{ \frac{d(p-1)}{4}\right\}^{\frac{1}{p-1}}\right)$},
\\[6pt]
\label{10/02/16/20:57}
&\Gamma_{\omega}(s) 
 \in \left\{f \in H^{1}(\mathbb{R}^{d})\setminus \{0\} \bigm| \mathcal{H}(f) \le 0\right\} \subset PW_{-}
\quad
\mbox{for all $s \in \left[ \left\{ \frac{d(p-1)}{4}\right\}^{\frac{1}{p-1}}, \ \infty\right)$}.
\end{align}
In particular, $PW_{-}\neq \emptyset$ and $PW_{+} \neq \emptyset$.
\end{lemma}
\begin{proof}[Proof of Lemma \ref{08/06/14/23:37}] The continuity of $\Gamma_{\omega}\colon (0,\infty)\to H^{1}(\mathbb{R}^{d})$ is obvious from its definition. We shall prove the properties (\ref{10/02/16/20:56})--(\ref{10/02/16/20:57}). As stated in (\ref{10/03/27/10:23}) (see also (\ref{09/12/16/20:17})), the ground state $Q_{\omega}$ satisfies that $\mathcal{K}(Q_{\omega})=0$, which immediately yields that  
\begin{equation}\label{10/03/27/10:02}
\left\| \nabla Q_{\omega} \right\|_{L^{2}}^{2}
=
\frac{d(p-1)}{2(p+1)}\left\| Q_{\omega} \right\|_{L^{p+1}}^{p+1}.
\end{equation}
Using (\ref{10/03/27/10:02}), we obtain that  
\begin{equation}\label{09/09/14/0:33}
\mathcal{H}(\Gamma_{\omega}(s))=s^{2}\left( 1-\frac{4}{d(p-1)}s^{p-1} \right)\|\nabla Q_{\omega} \|_{L^{2}}^{2}.
\end{equation}
This formula gives us that 
\begin{align}
\label{10/03/15/21:47}
&\mathcal{H}(\Gamma_{\omega}(s))>0 
\quad 
\mbox{for all  $s \in \left(0, \ \left\{ \frac{d(p-1)}{4}\right\}^{\frac{1}{p-1}}\right)$},
\\[6pt]
\label{10/03/15/21:35}
&\mathcal{H}(\Gamma_{\omega}(s)) \le 0
\quad 
\mbox{for all $s \in \left[ \left\{ \frac{d(p-1)}{4}\right\}^{\frac{1}{p-1}}, \ \infty\right)$}.
\end{align}
Similarly, we can verify that  
\begin{align}\label{10/03/15/21:20}
&\mathcal{K}(\Gamma_{\omega}(s))>0 \quad \mbox{for all $s \in (0,1)$},
\\[6pt]
\label{10/03/15/21:26}
&\mathcal{K}(\Gamma_{\omega}(1))=\mathcal{K}(Q_{\omega})=0, 
\\[6pt]
\label{10/03/15/21:27}
&\mathcal{K}(\Gamma_{\omega}(s))<0 \quad \mbox{for all $s\in (1,\infty)$}. 
\end{align}
Now, it follows from (\ref{09/09/14/0:33}) and (\ref{10/03/27/10:23}) that 
\begin{equation}\label{10/03/15/21:11}
\widetilde{\mathcal{N}}_{2}(\Gamma_{\omega}(s))=s^{p-1}\left( 1-\frac{4}{d(p-1)}s^{p-1} \right)^{\frac{d}{4}(p-1)-1}
\hspace{-7pt}
N_{2}
\qquad 
\mbox{for all $s \in \left(0,\  \left\{ \frac{d(p-1)}{4}\right\}^{\frac{1}{p-1}}\right)$}.
\end{equation}
This relation (\ref{10/03/15/21:11}) shows that 
\begin{equation}\label{10/03/15/21:22}
\widetilde{\mathcal{N}}_{2}(\Gamma_{\omega}(s))< \widetilde{\mathcal{N}}_{2}(\Gamma_{\omega}(1))=\widetilde{N}_{2} 
\qquad 
\mbox{for all $s \in \left(0,\  \left\{ \frac{d(p-1)}{4}\right\}^{\frac{1}{p-1}}\right)\setminus\{1\}$}.
\end{equation} 
Hence, (\ref{10/03/15/21:47}) and (\ref{10/03/15/21:22}), with the help of (\ref{08/06/15/14:38}), lead us to the conclusion that 
\begin{equation}\label{10/03/15/22:13}
\Gamma_{\omega}(s) \in PW
\qquad 
\mbox{for all $s \in \left(0,\  \left\{ \frac{d(p-1)}{4}\right\}^{\frac{1}{p-1}}\right) \setminus \{1\}$}
.
\end{equation}
Then, (\ref{10/02/16/20:56}) follows from (\ref{10/03/15/21:20}). 
Moreover, (\ref{10/02/16/20:58}) follows from  (\ref{10/03/15/21:47}) and (\ref{10/03/15/21:27}). The second claim (\ref{10/02/16/21:03}) is a direct consequence of (\ref{10/03/15/21:26}) and Lemma \ref{09/06/21/15:50}. 
\par 
It remains to prove (\ref{10/02/16/20:57}). Since $\mathcal{B}(f)\ge 0$ and $\mathcal{K}(f)<\mathcal{H}(f)$ for all $f \in H^{1}(\mathbb{R}^{d})\setminus \{0\}$, we have a relation
\begin{equation}\label{10/03/15/21:34}
\left\{f \in H^{1}(\mathbb{R}^{d})\setminus \{0\} \bigm| \mathcal{H}(f) \le 0\right\} \subset PW_{-}.
\end{equation}
Then, (\ref{10/02/16/20:57}) immediately follows from (\ref{10/03/15/21:35}) and this relation (\ref{10/03/15/21:34}).
\end{proof}

Now, we are in a position to prove Theorem \ref{08/10/20/5:21}. 

\begin{proof}[Proof of Theorem \ref{08/10/20/5:21}]
We consider the path $\Gamma_{\omega}$ given in Lemma \ref{08/06/14/23:37}. 
Then, (\ref{10/02/16/20:56}) and the continuity of $\Gamma_{\omega}$ yield that\begin{equation}\label{10/03/28/9:20}
\Gamma_{\omega}(s) \in PW_{+} 
\quad 
\mbox{for all $s \in (0,1)$}, 
\qquad 
\lim_{s\uparrow 1}\Gamma_{\omega}(s) =\Gamma_{\omega}(1)=Q_{\omega}
\quad 
\mbox{strongly in $H^{1}(\mathbb{R}^{d})$}.
\end{equation} 
Hence, we have $Q_{\omega} \in \overline{PW_{+}}$. Similarly, (\ref{10/02/16/20:58}) and the continuity of $\Gamma_{\omega}$ show that $Q_{\omega} \in \overline{PW_{-}}$. 
\end{proof}

We find from the following lemma that $PW_{+}$ has a ``foliate structure''.

\begin{lemma}\label{08/09/01/23:45}
For all $\eta \in (0,\widetilde{N}_{2})$ and $\alpha>0$, there exists $f \in PW_{+}$ such that 
 $\widetilde{\mathcal{N}}_{2}(f)=\eta$
and 
 $\left\| f \right\|_{L^{2}}=\alpha$. 
\end{lemma}

\begin{proof}[Proof of Lemma \ref{08/09/01/23:45}]
We construct a desired function from the continuous path $\Gamma_{\omega}\colon [0,\infty) \to H^{1}(\mathbb{R}^{d})$ ($\omega>0$) given in Lemma \ref{08/06/14/23:37}. Let us remind you that    
\begin{align}
\label{10/03/27/15:09}
&\Gamma_{\omega}(s) \in PW_{+} 
\quad 
\mbox{for all $s \in (0,1)$},
\\[6pt]
&
\label{10/03/27/15:19}
\widetilde{\mathcal{N}}_{2}(\Gamma_{\omega}(s))=s^{p-1}\left( 1-\frac{4}{d(p-1)}s^{p-1} \right)^{\frac{d}{4}(p-1)-1}
N_{2}
\quad 
\mbox{for all $s \in (0, 1)$},
\\[6pt]
&
\label{10/03/27/15:37}
\widetilde{\mathcal{N}}_{2}(\Gamma_{\omega}(0))=0,
\qquad 
\widetilde{\mathcal{N}}_{2}(\Gamma_{\omega}(1))=\widetilde{N}_{2}
\end{align}
(see (\ref{10/03/15/21:11}) for (\ref{10/03/27/15:19})). Using (\ref{10/03/27/15:19}), we find  that $\widetilde{\mathcal{N}}_{2}(\Gamma_{\omega}(s))$ is continuous and monotone increasing with respect to $s$ on $[0,1]$. Hence, the intermediate value theorem, together with (\ref{10/03/27/15:37}), shows that for any $\eta \in (0, \widetilde{N}_{2})$, there exists $s_{\eta} \in (0,1)$ such that $\widetilde{\mathcal{N}}_{2}(\Gamma_{\omega}(s_{\eta}))=\eta$. Moreover, (\ref{10/03/27/15:09}) gives us that $\Gamma_{\omega}(s_{\eta}) \in PW_{+}$. We put $f=\Gamma_{\omega}(s_{\eta})$ and consider the scaled function $f_{\lambda}:=\lambda^{\frac{2}{p-1}}f(\lambda \cdot)$ for $\lambda>0$. Then, it is easy to see that  
\begin{equation}
\label{10/03/27/16:10}
f_{\lambda} \in PW_{+},
\quad 
\widetilde{\mathcal{N}}_{2}(f_{\lambda})=\widetilde{\mathcal{N}}_{2}(f)=
\eta,
\quad 
\left\| f_{\lambda} \right\|_{L^{2}}=\lambda^{-\frac{d(p-1)-4}{2(p-1)}}\left\| f\right\|_{L^{2}}
\quad 
\mbox{for all $\lambda>0$}.
\end{equation}
Hence, for any $\alpha>0$, $f_{\lambda}$ with $\lambda =\left( \frac{ \left\| f \right\|_{L^{2}}}{\alpha}\right)^{\frac{2(p-1)}{d(p-1)-4}}$ is what we want.
\end{proof}

Finally, we give the invariance results of the sets $PW$, $PW_{+}$ and $PW_{-}$  under the flow defined by the equation (\ref{08/05/13/8:50}): 

\begin{lemma}[Invariance of $PW$]\label{09/06/21/18:53}
Let $\psi_{0} \in PW$ and $\psi$ be the corresponding solution to 
the equation (\ref{08/05/13/8:50}). Then, we have that 
\[
\psi(t) \in PW 
\quad 
\mbox{for all $t \in I_{\max}$}.
\]
\end{lemma}
\begin{proof}[Proof of Lemma \ref{09/06/21/18:53}]
This lemma immediately follows from the mass and energy conservation laws (\ref{08/05/13/8:59}) and (\ref{08/05/13/9:03}). 
\end{proof}

\begin{proposition}[Invariance of $PW_{+}$]
\label{09/06/21/19:28}
Let $\psi_{0} \in PW_{+}$ and $\psi$ be the corresponding solution to the equation (\ref{08/05/13/8:50}). Then, $\psi$ exists globally in time and satisfies the followings:
\begin{equation}\label{10/01/06/21:29}
 \psi(t) \in PW_{+}
\quad  
\mbox{for all  $t \in \mathbb{R}$},
\end{equation}
\begin{equation}\label{09/06/21/21:58}
\left\| \nabla \psi (t) \right\|_{L^{2}}^{2}
< \frac{d(p-1)}{d(p-1)-4}\mathcal{H}(\psi_{0}) 
\quad 
\mbox{for all  $t \in \mathbb{R}$},  
\end{equation}
\begin{equation}\label{09/06/21/19:30}
\mathcal{K}(\psi(t)) > 
\left( 1-\frac{\widetilde{N}_{2}(\psi_{0})}{\widetilde{N}_{2}}\right)
\mathcal{H}(\psi_{0})
\quad 
\mbox{for all  $t \in \mathbb{R}$}.
\end{equation}
\end{proposition}

\begin{proof}[Proof of Proposition \ref{09/06/21/19:28}] 
We first prove the invariance of $PW_{+}$ under the flow defined by (\ref{08/05/13/8:50}). With the help of Lemma \ref{09/06/21/18:53}, it suffices to show that 
\begin{equation}\label{09/08/15/18:45}
\mathcal{K}(\psi(t))>0  
\quad 
\mbox{for all $t \in I_{\max}$}.
\end{equation}
Supposing the contrary that (\ref{09/08/15/18:45}) fails, we can take $t_{0}\in I_{\max}$ such that 
\begin{equation}\label{10/03/27/16:44}
\mathcal{K}(\psi(t_{0}))= 0.
\end{equation}
Then, (\ref{10/03/27/16:44}) and the energy conservation law (\ref{08/05/13/9:03}) yield that 
\begin{equation}\label{10/03/27/17:02}
\begin{split}
0&=\mathcal{K}(\psi(t_{0}))
\\[6pt]
&=\mathcal{H}(\psi(t_{0}))
- \frac{2}{p+1}\left\{ \frac{d}{4}(p-1)-1 \right\}\|\psi(t_{0})\|_{L^{p+1}}^{p+1}
\\[6pt]
&=
\mathcal{H}(\psi_{0}) - \frac{2}{p+1}\left\{ \frac{d}{4}(p-1)-1 \right\}
\frac{2(p+1)}{d(p-1)} \|\nabla \psi(t_{0})\|_{L^{2}}^{2}.
\end{split}
\end{equation}
Since $\psi_{0}\in PW_{+}\subset PW$, we have $\mathcal{H}(\psi_{0})<\mathcal{B}(\psi_{0})$. This inequality and (\ref{10/03/27/17:02}) lead us to that  
\begin{equation}\label{10/03/27/17:07}
0 < \mathcal{B}(\psi_{0})-\frac{d(p-1)-4}{d(p-1)} \|\nabla \psi(t_{0})\|_{L^{2}}^{2},
\end{equation}
which is  equivalent to   
\begin{equation}\label{10/03/27/17:18}
\|\nabla \psi(t_{0})\|_{L^{2}}^{2}
< 
\frac{d(p-1)}{d(p-1)-4}\mathcal{B}(\psi_{0})
=
\left( \frac{ N_{2}}{\|\psi(t_{0})\|_{L^{2}}^{p+1-\frac{d}{2}(p-1)}}\right)^{\frac{4}{d(p-1)-4}}
.
\end{equation}
Dividing the both sides of (\ref{10/03/27/17:18}) by $\|\nabla \psi(t_{0})\|_{L^{2}}^{2}$, 
we obtain that 
\begin{equation}\label{10/03/27/17:24}
1< \left( \frac{N_{2}}{\mathcal{N}_{2}(\psi(t_{0}))}\right)^{\frac{4}{d(p-1)-4}}.
\end{equation}
On the other hand, (\ref{10/03/27/16:44}), together with the definition of $N_{2}$, leads us to that $\mathcal{N}_{2}(\psi(t_{0}))\ge N_{2}$, so that  
\begin{equation}\label{10/03/27/17:26}
\left( \frac{N_{2}}{\mathcal{N}_{2}(\psi(t_{0}))}\right)^{\frac{4}{d(p-1)-4}}
\le 1.
\end{equation}
This inequality (\ref{10/03/27/17:26}) contradicts (\ref{10/03/27/17:24}). Thus, (\ref{09/08/15/18:45}) must hold.
\par 
Once we obtain (\ref{09/08/15/18:45}), we can easily obtain the following uniform bound:   
\begin{equation}\label{10/01/06/23:07}
\left\| \nabla \psi (t) \right\|_{L^{2}}^{2}
< \frac{d(p-1)}{d(p-1)-4}\mathcal{H}(\psi_{0}) 
\quad 
\mbox{for all $t \in I_{\max}$}.
\end{equation}
Indeed, it follows from the energy conservation law (\ref{08/05/13/9:03}) and (\ref{09/08/15/18:45}) that 
\begin{equation}\label{10/03/28/10:15}
\begin{split}
\mathcal{H}(\psi_{0})&=\mathcal{H}(\psi(t))=
\left\| \nabla \psi(t)\right\|_{L^{2}}^{2}-
\frac{2}{p+1}\left\| \psi(t) \right\|_{L^{p+1}}^{p+1}
\\[6pt]
&>\frac{d(p-1)-4}{d(p-1)}\left\| \nabla \psi(t)\right\|_{L^{2}}^{2}
\quad 
\mbox{for all $t \in I_{\max}$}.
\end{split}
\end{equation}

The estimate (\ref{10/01/06/23:07}), together with the sufficient condition for the blowup (\ref{10/01/27/11:26}), leads us to that $I_{\max}=\mathbb{R}$. Hence,  (\ref{10/01/06/21:29}) and (\ref{09/06/21/21:58}) follow from (\ref{09/08/15/18:45}) and (\ref{10/01/06/23:07}), respectively.
\par   
It remains to prove (\ref{09/06/21/19:30}). The Gagliardo-Nirenberg inequality (\ref{08/05/13/15:45}) gives us that 
\begin{equation}\label{10/03/28/10:21}
\begin{split}
\mathcal{K}(\psi(t))
&
=\left\| \nabla \psi(t) \right\|_{L^{2}}^{2}
-\frac{d(p-1)}{2(p+1)}\left\| \psi(t) \right\|_{L^{p+1}}^{p+1}
\\[6pt]
&\ge \left\| \nabla \psi(t) \right\|_{L^{2}}^{2}
-\frac{d(p-1)}{2(p+1)}\frac{1}{N_{3}}\mathcal{N}_{2}(\psi(t))
\left\| \nabla \psi(t)\right\|_{L^{2}}^{2}.
\end{split}
\end{equation}
Moreover, this inequality (\ref{10/03/28/10:21}) and the relation $N_{3}=\frac{d(p-1)}{2(p+1)}N_{2}$ (see (\ref{08/05/13/11:25})) yield that  
\begin{equation}\label{08/12/16/18:09}
\begin{split}
\mathcal{K}(\psi(t))
&
\ge \left\| \nabla \psi(t) \right\|_{L^{2}}^{2}
-\frac{\mathcal{N}_{2}(\psi(t))}{N_{2}}
\left\| \nabla \psi(t) \right\|_{L^{2}}^{2}
.
\end{split}
\end{equation}
Here, it follows from (\ref{09/06/21/21:58}) that   
\begin{equation}
\label{09/08/15/19:08}
\begin{split}
\mathcal{N}_{2}(\psi(t))&<
\sqrt{\frac{d(p-1)}{d(p-1)-4}}^{\frac{d}{2}(p-1)-2}\hspace{-1cm}
\left\|\psi_{0} \right\|_{L^{2}}^{p+1-\frac{d}{2}(p-1)}
\sqrt{\mathcal{H}(\psi_{0})}^{\frac{d}{2}(p-1)-2}
\\
&= \sqrt{\frac{d(p-1)}{d(p-1)-4}}^{\frac{d}{2}(p-1)-2}\widetilde{\mathcal{N}}_{2}(\psi_{0}).
\end{split}
\end{equation} 
Hence, combining (\ref{08/12/16/18:09}), (\ref{09/08/15/19:08}) and $\mathcal{H}(\psi(t))>\mathcal{K}(\psi(t))>0$, we obtain that\begin{equation}\label{10/03/10:27}
\begin{split}
\mathcal{K}(\psi(t))&> 
\left\| \nabla \psi(t) \right\|_{L^{2}}^{2}
-
\sqrt{\frac{d(p-1)}{d(p-1)-4}}^{\frac{d}{2}(p-1)-2}
\frac{\mathcal{\widetilde{\mathcal{N}}}_{2}(\psi_{0})}{N_{2}}
\left\| \nabla \psi(t) \right\|_{L^{2}}^{2}
\\[6pt]
&
=
\left( 1-\frac{\widetilde{\mathcal{N}}_{2}(\psi_{0})}{\widetilde{N}_{2}}\right)
\left\| \nabla \psi(t)\right\|_{L^{2}}^{2}
\\[6pt]
&\ge 
\left( 1-\frac{\widetilde{\mathcal{N}}_{2}(\psi_{0})}{\widetilde{N}_{2}}\right) \mathcal{H}(\psi(t))
=
\left( 1-\frac{\widetilde{\mathcal{N}}_{2}(\psi_{0})}{\widetilde{N}_{2}}\right) \mathcal{H}(\psi_{0})
.
\end{split}
\end{equation}
\end{proof}

\begin{proposition}[Invariance of $PW_{-}$]\label{08/05/26/10:57}
Let $\psi_{0} \in PW_{-}$ and $\psi$ be the corresponding solution to the equation (\ref{08/05/13/8:50}). Then, we have 
\begin{equation}\label{09/06/21/18:52}
\psi(t) \in PW{-},
\quad 
\mathcal{K}(\psi(t)) < -\varepsilon_{0} 
\qquad 
\mbox{for all $t \in I_{\max}$},
\end{equation}
where $\varepsilon_{0}=\mathcal{B}(\psi_{0})-\mathcal{H}(\psi_{0})>0$. 
\end{proposition}
\begin{proof}[Proof of Proposition \ref{08/05/26/10:57}] 
By the virtue of Lemma \ref{09/06/21/18:53}, it suffices to show that 
\begin{equation}\label{10/03/28/10:48}
\mathcal{K}(\psi(t)) < -\varepsilon_{0} 
\qquad 
\mbox{for all $t \in I_{\max}$}.
\end{equation}
We first prove that 
\begin{equation}\label{10/03/28/11:09}
\mathcal{K}(\psi_{0})<-\varepsilon_{0}.
\end{equation}
Supposing the contrary that $-\varepsilon_{0}\le \mathcal{K}(\psi_{0})$, 
 we have 
\begin{equation}\label{10/03/28/10:54}
\begin{split}
0&\le \mathcal{K}(\psi_{0})+\varepsilon_{0}
\\[6pt]
&=\mathcal{H}(\psi_{0})+\varepsilon_{0}
- \frac{2}{p+1}\left\{ \frac{d}{4}(p-1)-1 \right\}\|\psi_{0}\|_{L^{p+1}}^{p+1}
\\[6pt]
&=
\mathcal{B}(\psi_{0})
- 
\frac{2}{p+1}\left\{ \frac{d}{4}(p-1)-1 \right\}\|\psi_{0}\|_{L^{p+1}}^{p+1}
.
\end{split}
\end{equation}
Moreover, it follows from (\ref{10/03/28/10:54}) and $\mathcal{K}(\psi_{0})<0$ that 
\begin{equation}\label{10/03/28/10:55}
\begin{split}
0&
< \mathcal{B}(\psi_{0}) - \frac{2}{p+1}\left\{ \frac{d}{4}(p-1)-1 \right\}
\frac{2(p+1)}{d(p-1)}  \|\nabla \psi_{0}\|_{L^{2}}^{2}
\\[6pt]
&=
\mathcal{B}(\psi_{0})-\frac{d(p-1)-4}{d(p-1)} \|\nabla \psi_{0}\|_{L^{2}}^{2}
,
\end{split}
\end{equation}
so that   
\begin{equation}\label{10/03/28/11:02}
\|\nabla \psi_{0}\|_{L^{2}}^{2}
< \frac{d(p-1)}{d(p-1)-4}\mathcal{B}(\psi_{0})
= \left( \frac{N_{2}}{\|\psi_{0}\|_{L^{2}}^{p+1-\frac{d}{2}(p-1)}}\right)^{\frac{4}{d(p-1)-4}}
.
\end{equation}
Dividing the both sides of (\ref{10/03/28/11:02}) by $\|\nabla \psi_{0}\|_{L^{2}}^{2}$, we obtain that 
\begin{equation}\label{10/03/28/11:05}
1< \left( \frac{ N_{2}}{\mathcal{N}_{2}(\psi_{0})}\right)^{\frac{4}{d(p-1)-4}}
.
\end{equation}
On the other hand, since $\mathcal{K}(\psi_{0})<0$, the definition of $N_{2}$ (see (\ref{08/07/02/23:24})) implies that 
\begin{equation}\label{10/03/28/11:06}
\left( \frac{ N_{2}}{\mathcal{N}_{2}(\psi_{0})}\right)^{\frac{4}{d(p-1)-4}}
\le 1 ,
\end{equation}
which contradicts (\ref{10/03/28/11:05}). Hence, we have proved (\ref{10/03/28/11:09}). 
\par 
Next, we prove (\ref{10/03/28/10:48}). Suppose the contrary that (\ref{10/03/28/10:48}) fails. Then, it follows from (\ref{10/03/28/11:09}) and $\psi \in C(I_{\max};H^{1}(\mathbb{R}^{d}))$  
 that there exists $t_{1}\in I_{\max} \setminus \{0\}$ such that $-\varepsilon_{0}= \mathcal{K}(\psi(t_{1}))$. This relation and the energy conservation law (\ref{08/05/13/9:03}) lead us to that  
\begin{equation}\label{10/03/28/11:17}
\begin{split}
0&= \mathcal{K}(\psi(t_{1}))+\varepsilon_{0}
\\[6pt]
&=\mathcal{H}(\psi_{0})+\varepsilon_{0}
- \frac{2}{p+1}\left\{ \frac{d}{4}(p-1)-1 \right\}\|\psi(t_{1})\|_{L^{p+1}}^{p+1}
\\[6pt]
&= \mathcal{B}(\psi_{0}) - \frac{2}{p+1}\left\{ \frac{d}{4}(p-1)-1 \right\}
\frac{2(p+1)}{d(p-1)}  \left( \|\nabla \psi(t_{1})\|_{L^{2}}^{2}+\varepsilon_{0}\right)
\\[6pt]
&< \mathcal{B}(\psi_{0})-\frac{d(p-1)-4}{d(p-1)} \|\nabla \psi(t_{1})\|_{L^{2}}^{2}
.
\end{split}
\end{equation}
Hence, we have that  
\begin{equation}\label{10/03/28/11:20}
\begin{split}
\|\nabla \psi(t_{1})\|_{L^{2}}^{2}
&
<
\frac{d(p-1)}{d(p-1)-4}\mathcal{B}(\psi_{0})
=\left( \frac{N_{2}}{\|\psi_{0}\|_{L^{2}}^{p+1-\frac{d}{2}(p-1)}}\right)^{\frac{4}{d(p-1)-4}}
\\[6pt]
&=
\left( \frac{N_{2}}{\|\psi(t_{1})\|_{L^{2}}^{p+1-\frac{d}{2}(p-1)}}\right)^{\frac{4}{d(p-1)-4}}
,
\end{split}
\end{equation}
where we have used the mass conservation law (\ref{08/05/13/8:59}) to derive the last equality.  Then, an argument similar to the above yields a contradiction: Thus, we completed the proof.   
\end{proof}

\section{Strichartz type estimate and scattering}
\label{08/08/05/14:29}

In this section, we introduce a certain space-time function space in addition to the usual Strichartz spaces, which enable us to control long-time behavior of solutions. Using this function space, we prepare two important propositions: Proposition \ref{08/08/22/20:59} in Section \ref{09/05/30/15:49} (small data theory) and Proposition \ref{08/08/05/14:30} in Section \ref{09/05/30/21:16} (long time perturbation theory). The former is used to avoid the vanishing and the latter to avoid the dichotomy in our concentration compactness like argument in Section \ref{09/05/05/10:03}. In the end of this section, we show the existence of the wave operators on $PW_{+}$.

\subsection{Auxiliary function space $X$}\label{08/10/07/9:01}
In order to prove the scattering result ((\ref{10/01/26/20:32}) in Theorem \ref{08/05/26/11:53}), we need to handle the inhomogeneous term of the integral equation associated with (\ref{08/05/13/8:50}) in a suitable function space. Therefore, we will prepare a function space $X(I)$, $I\subset \mathbb{R}$, in which Strichartz type estimate works well. 
\par 
Our equation (\ref{08/05/13/8:50}) is invariant under the scaling 
\begin{equation}\label{10/03/29/10:21}
\psi(x,t) \mapsto \psi_{\lambda}(x,t):= \lambda^{\frac{2}{p-1}} \psi(\lambda x, \lambda^{2}t),
\end{equation}
which determines a critical regularity 
\begin{equation}\label{09/05/13/15:00}
s_{p}:=\frac{d}{2}-\frac{2}{p-1}.
\end{equation}
The condition (\ref{09/05/13/15:03}) implies that $0<s_{p}<1$.
\par 
Throughout this paper, we fix a number $q_{1}$ with $p+1<q_{1}<2^{*}$. Then, we define indices $r_{0}$, $r_{1}$ and $\widetilde{r}_{1}$ by 
\begin{align}
\label{09/09/23/18:00}
&\frac{1}{r_{0}}:=\frac{d}{2}\left( \frac{1}{2}-\frac{1}{q_{1}} \right), 
\\[6pt]
\label{08/09/04/10:16}
&\frac{1}{r_{1}}:=\frac{d}{2}\left( \frac{1}{2}-\frac{1}{q_{1}}-\frac{s_{p}}{d}\right),
\\[6pt]
\label{09/09/23/18:02}
&\frac{1}{\widetilde{r}_{1}}:=\frac{d}{2}\left( \frac{1}{2}-\frac{1}{q_{1}}+\frac{s_{p}}{d}\right).
\end{align} 
Here, the pair $(q_{1},r_{0})$ is admissible. Besides these indices, we define a pair $(q_{2},r_{2})$ by   
\begin{equation}\label{09/12/05/16:15}
\frac{p-1}{q_{2}}=1-\frac{2}{q_{1}},
\qquad 
\frac{1}{r_{2}}:=\frac{d}{2}\left( \frac{1}{2}-\frac{1}{q_{2}}-\frac{s_{p}}{d} \right).
\end{equation} 
\begin{figure}[ht]
\begin{center}
\input{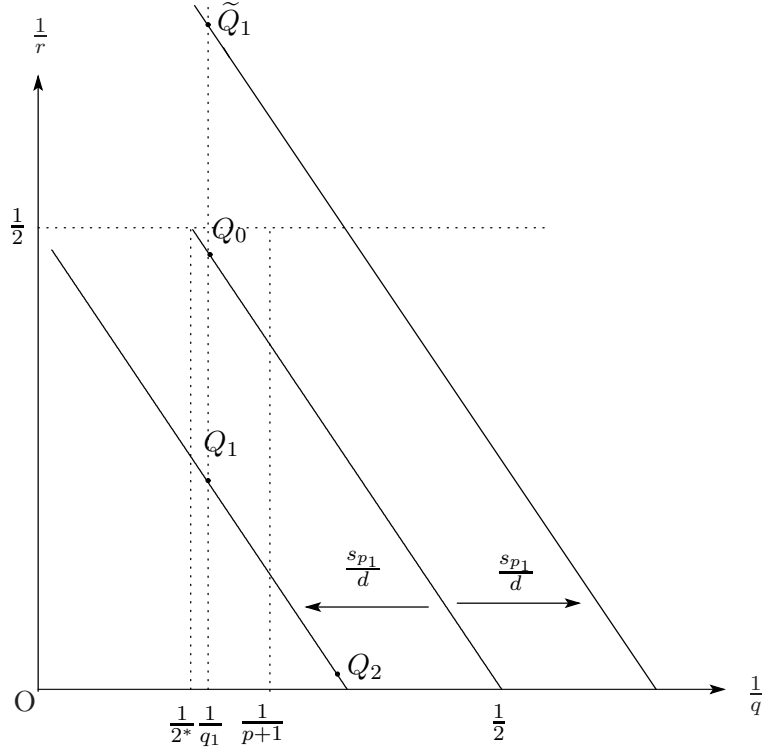}
\end{center}
\caption{
Strichartz type estimates: 
$Q_{0}: (\frac{1}{q_{1}},\frac{1}{r_{0}})$,
$Q_{1}: (\frac{1}{q_{1}},\frac{1}{r_{1}})$, 
$Q_{2}: (\frac{1}{q_{2}},\frac{1}{r_{2}})$,
$\widetilde{Q}_{1}: (\frac{1}{q_{1}},\frac{1}{\widetilde{r}_{1}})$.
}
\end{figure}
\vspace{12pt}
\\
It is worth while noting that the Sobolev embedding and the Strichartz estimate 
 lead us to the following estimate: For any pair $(q,r)$ satisfying
\begin{equation}\label{10/03/29/10:52} 
\frac{d}{2}(p-1) \le q < 2^{*},
\qquad 
\frac{1}{r}=\frac{d}{2}\left( \frac{1}{2}-\frac{1}{q}-\frac{s_{p}}{d}\right),
\end{equation}
we have 
\begin{equation}\label{08/10/25/23:16}
\left\|e^{\frac{i}{2}t\Delta}f \right\|_{L^{r}(I;L^{q})} 
\lesssim \left\|(-\Delta)^{\frac{s_{p}}{2}}f\right\|_{L^{2}}
\quad 
\mbox{for all $f \in \dot{H}^{s_{p}}(\mathbb{R}^{d})$ and interval $I$}, 
\end{equation}
where the implicit constant depends only on $d$, $p$ and $q$. 
The pairs $(q_{1},r_{1})$ and $(q_{2},r_{2})$ satisfy the condition (\ref{10/03/29/10:52}), so that the estimate (\ref{08/10/25/23:16}) is valid for these pairs. 
\par 
Now, for any interval $I$, we put 
\begin{align}
\label{08/09/01/10:02}
X(I)&= L^{r_{1}}(I;L^{q_{1}})\cap L^{r_{2}}(I; L^{q_{2}}),
\\[6pt]
\label{10/02/19/22:40}
S(I)&=L^{\infty}(I; L^{2}) \cap L^{r_{0}}(I; L^{q_{1}}).
\end{align}
We find that Strichartz type estimates work well in the space $X(I)$: 
\begin{lemma}\label{08/07/31/12:02}
Assume that  $d\ge 1$ and $2+\frac{4}{d}< p+1 <2^{*}$. Let $t_{0} \in \mathbb{R}$ and $I$ be an interval whose closure contains $t_{0}$. Then, we have 
\begin{align}
\label{10/03/29/11:06}
\left\| \int_{t_{0}}^{t}e^{i(t-t')\Delta} v(t')\,dt' \right\|_{X(I)}&\lesssim 
\left\|v  \right\|_{L^{\widetilde{r}_{1}'}(I;L^{q_{1}'})},
\\[6pt]
\label{10/03/29/11:07}
\left\| \int_{t_{0}}^{t}e^{i(t-t')\Delta} \left( v_{1}v_{2} \right)(t')\,dt' \right\|_{X(I)}
&\lesssim 
\left\| v_{1} \right\|_{L^{r_{1}}(I;L^{q_{1}})}
\left\| v_{2} \right\|_{L^{\frac{r_{2}}{p-1}}(I;L^{\frac{q_{2}}{p-1}})}
,
\end{align}
where the implicit constants depend only on $d$, $p$ and $q_{1}$. 
\end{lemma}
The estimate (\ref{10/03/29/11:06}) in Lemma \ref{08/07/31/12:02} is due to Foschi (see \cite{Foschi}, Theorem 1.4). The estimate (\ref{10/03/29/11:07}) is an immediate consequence of (\ref{10/03/29/11:06}) and the H\"older inequality.
\\
\par 
The following lemma is frequently used in the next section (Section \ref{08/10/03/15:11}); It is fundamental and easily obtained from the H\"older inequality and the chain rule:    
\begin{lemma}
\label{08/08/18/15:41}
Assume that $d \ge 1$ and $2+\frac{4}{d}< p+1 < 2^{*}$. Let $t_{0} \in \mathbb{R}$ and $I$ be an interval whose closure contains $t_{0}$. Then, we have
\begin{align}
\label{08/08/18/15:43}
&\left\||v|^{p-1}v  \right\|_{L^{r_{0}'}(I;
L^{q_{1}'})}
\le 
\left\|v \right\|_{L^{r_{0}}(I;L^{q_{1}})}
\left\| v \right\|_{L^{r_{2}}(I;L^{q_{2}})}^{p-1},
\\[6pt]
\label{10/03/29/11:14}
&\left\|\nabla \left( |v|^{p-1}v \right) \right\|_{L^{r_{0}'}(I;
L^{q_{1}'})}
\lesssim 
\left\|\nabla v \right\|_{L^{r_{0}}(I;L^{q_{1}})}
\left\| v \right\|_{L^{r_{2}}(I;L^{q_{2}})}^{p-1}
, 
\end{align}
where the implicit constant depends only on $d$, $p$ and $q_{1}$.
\end{lemma}

We also need the following interpolation estimate in the next section (Section \ref{09/05/05/10:03}): 
\begin{lemma}\label{09/04/29/15:57}
For $j \in \{1,2\}$, there exist a constant $\theta_{j} \in (0,1)$ such that 
\[
\left\|e^{\frac{i}{2}t\Delta}f  \right\|_{L^{r_{j}}(I;L^{q_{j}})}
\lesssim
 \left\| e^{\frac{i}{2}t\Delta}f  \right\|_{L^{\infty}(I;L^{\frac{d}{2}(p-1)})}^{1-\theta_{j}}
\left\| (-\Delta)^{\frac{s_{p}}{2}} f \right\|_{L^{2}}^{\theta_{j}}
\quad 
\mbox{for all $f \in \dot{H}^{s_{p}}(\mathbb{R}^{d})$}, 
\]
where the implicit constant depends only on $d$, $p$ and $q_{1}$. 
\end{lemma}
\begin{proof}[Proof of Lemma \ref{09/04/29/15:57}]
Fix a pair $(q,r)$ satisfying (\ref{10/03/29/10:52}) and $q_{1}<q <2^{*}$. Applying the H\"older inequality first and (\ref{08/10/25/23:16}) afterword, 
 we obtain that   
\begin{equation}\label{10/03/29/11:26}
\begin{split}
\left\|e^{\frac{i}{2}t\Delta}f  \right\|_{L^{r_{j}}(I;L^{q_{j}})}
&\le  \left\| e^{\frac{i}{2}t\Delta}f  \right\|_{L^{\infty}(I;L^{\frac{d}{2}(p-1)})}^{1-\theta_{j}}
\left\| e^{\frac{i}{2}t\Delta}f \right\|_{L^{r}(I;L^{q})}^{\theta_{j}}
\\[6pt]
&\lesssim  
\left\| e^{\frac{i}{2}t\Delta}f  \right\|_{L^{\infty}(I;L^{\frac{d}{2}(p-1)})}^{1-\theta_{j}}
\left\| (-\Delta)^{\frac{s_{p}}{2}} f \right\|_{L^{2}}^{\theta_{j}},
\quad 
j=1,2
\end{split}
\end{equation}
where 
\[
\theta_{j}:=\frac{q}{q_{j}}\frac{2q_{j}-d(p-1)}{2q-d(p-1)},
\]
and the implicit constant depends only on $d$, $p$ and $q_{1}$ (we may ignore the dependence of $q$). 
\end{proof}

At the end of this subsection, we record a decay estimate for the free solution:
\begin{lemma}\label{10/04/08/9:20}
Assume that $d\ge 1$. Then, we have 
\begin{equation}\label{10/04/08/9:22}
\lim_{t\to \pm \infty}
\left\| e^{\frac{i}{2}t\Delta}f \right\|_{L^{q}}=0
\quad 
\mbox{for all $q \in (2,2^{*})$ and $f \in H^{1}(\mathbb{R}^{d})$}.
\end{equation}
\end{lemma}
\begin{proof}[Proof of Lemma \ref{10/04/08/9:22}]
This lemma is easily follows from the $L^{p}$-$L^{q}$-estimate for the free solution and the density of the Schwartz space $\mathcal{S}(\mathbb{R}^{d})$ in $H^{1}(\mathbb{R}^{d})$.
\end{proof}

\subsection{Sufficient conditions for scattering}
\label{09/05/30/15:49}
We shall give two sufficient conditions for solutions to have asymptotic states in the energy space $H^{1}(\mathbb{R}^{d})$. One of them is the small data theory (see Proposition \ref{08/08/22/20:59} and Remark \ref{10/03/29/12:24}). We also give the proof of Proposition \ref{09/05/18/10:43} here. 
\par  
We begin with the following proposition: 
\begin{proposition}[Scattering in the energy space]
\label{08/08/18/16:51}
Assume that $d\ge 1$ and $2+\frac{4}{d}<p+1<2^{*}$. Let $\psi$ be a solution to the equation (\ref{08/05/13/8:50}).
\\
{\rm (i)} Suppose that $\psi$ exists on $[0,\infty)$ and  satisfies that 
\begin{equation}\label{10/03/29/12:27}
\left\| \psi \right\|_{X([0,\infty))}< \infty,
\qquad 
\left\|\psi \right\|_{L^{\infty}([0,\infty);H^{1})}< \infty.
\end{equation}
Then, there exists a unique asymptotic state $\phi_{+} \in H^{1}(\mathbb{R}^{d})$ such that 
\begin{equation}\label{10/03/29/12:28}
\lim_{t \to \infty}\left\| \psi(t)-e^{\frac{i}{2}t\Delta}\phi_{+}\right\|_{H^{1}}
=0.
\end{equation}
{\rm (ii)} Suppose that $\psi$ exists on $(-\infty, 0]$ and  satisfies that 
\begin{equation}\label{10/03/29/12:29}
\left\| \psi \right\|_{X((-\infty,0])}< \infty,
\qquad
\left\|\psi \right\|_{L^{\infty}((-\infty,0];H^{1})}< \infty.
\end{equation} 
Then, there exists a unique asymptotic state $\phi_{-} \in H^{1}(\mathbb{R}^{d})$ such that 
\begin{equation}\label{10/03/29/12:30}
\lim_{t \to -\infty}\left\| \psi(t)-e^{\frac{i}{2}t\Delta}\phi_{-}\right\|_{H^{1}}
=0.
\end{equation}
\end{proposition}

For the proof of Proposition \ref{08/08/18/16:51}, we need the following lemma. 
\begin{lemma}
\label{08/08/18/15:53}
Assume that $d\ge 1$ and $2+\frac{4}{d}<p+1<2^{*}$. Let $I$ be an interval and  $\psi$ be a solution to the equation (\ref{08/05/13/8:50}) on $I$. Suppose that \begin{equation}\label{10/03/29/18:06} 
\|\psi\|_{X(I)}<\infty, 
\quad 
\|\psi\|_{L^{\infty}(I;H^{1})}<\infty.
\end{equation}
Then, we have that 
\[
\left\|(1-\Delta)^{\frac{1}{2}}\psi \right\|_{S(I)}<\infty
.
\]
\end{lemma}
\begin{proof}[Proof of Lemma \ref{08/08/18/15:53}] 
For the desired result,  it suffices to show that 
\begin{equation}\label{10/03/29/18:09}
\left\|(1-\Delta)^{\frac{1}{2}} \psi \right\|_{L^{r}(I;L^{q})}< \infty
\end{equation}
for all pair $(q,r)$ with $q_{1}<q<2^{*}$ and $\frac{1}{r}=\frac{d}{2}\left( \frac{1}{2}-\frac{1}{q}\right)$ ($(q,r)$ is an admissible pair).
Indeed, the H\"older inequality, together with (\ref{10/03/29/18:06}) and (\ref{10/03/29/18:09}), gives us that 
\begin{equation}\label{10/03/29/18:45}
\left\|(1-\Delta)^{\frac{1}{2}} \psi \right\|_{L^{r_{0}}(I;L^{q_{1}})}
\le 
\left\| \psi \right\|_{L^{\infty}(I;H^{1})}^{1-\frac{q(q_{1}-2)}{q_{1}(q-2)}}
\left\|(1-\Delta)^{\frac{1}{2}} \psi \right\|_{L^{r}(I;L^{q})}^{\frac{q(q_{1}-2)}{q_{1}(q-2)}}
<\infty.
\end{equation}
We shall prove (\ref{10/03/29/18:09}). Let $J$ be a subinterval of $I$ with the property that 
\begin{equation}
\left\| (1-\Delta)^{\frac{1}{2}}\psi \right\|_{L^{r}(J;L^{q})}<\infty. 
\end{equation}
Then, applying the Strichartz estimate to the integral equation associated with $(\ref{08/05/13/8:50})$, we obtain that 
\begin{equation}\label{09/09/24/18:21}
\left\| (1-\Delta)^{\frac{1}{2}}\psi \right\|_{L^{r}(J;L^{q})}
\lesssim 
\left\| \psi(t_{0}) \right\|_{H^{1}}
+
\left\|(1-\Delta)^{\frac{1}{2}} \left( |\psi|^{p-1}\psi \right) \right\|_{L^{r_{0}'}(J;L^{q_{1}'})},
\end{equation}
where the implicit constant depends only on $d$, $p$, $q_{1}$ and $q$. Moreover, it follows from Lemma \ref{08/08/18/15:41} and the H\"older inequality that    \begin{equation}\label{10/03/29/17:55}
\begin{split}
\left\| (1-\Delta)^{\frac{1}{2}}\psi \right\|_{L^{r}(J;L^{q})}
&\lesssim 
\left\| \psi(t_{0}) \right\|_{H^{1}}
+
\left\| (1-\Delta)^{\frac{1}{2}} \psi \right\|_{L^{r_{0}}(J;L^{q_{1}})}\left\| \psi \right\|_{L^{r_{2}}(J;L^{q_{2}})}^{p-1}
\\[6pt]
&\le 
\left\|\psi \right\|_{L^{\infty}(I;H^{1})}
+
\left\|\psi \right\|_{L^{\infty}(J;H^{1})}^{1-\theta}
\left\| (1-\Delta)^{\frac{1}{2}} \psi \right\|_{L^{r}(J;L^{q})}^{\theta}
\left\| \psi \right\|_{X(J)}^{p-1},
\end{split}
\end{equation}
where 
\[
\theta :=\frac{q(q_{1}-2)}{q_{1}(q-2)} \in (0,1)
\]
and the implicit constant depends only on $d$, $p$, $q_{1}$ and $q$. 
This estimate (\ref{10/03/29/17:55}) and the Young inequality show that 
\begin{equation}\label{10/03/29/18:24}
\left\| (1-\Delta)^{\frac{1}{2}}\psi \right\|_{L^{r}(J;L^{q})}
\lesssim 
\left\|\psi \right\|_{L^{\infty}(I;H^{1})}
+
\left\|\psi \right\|_{L^{\infty}(I;H^{1})}
\left\| \psi \right\|_{X(I)}^{\frac{p-1}{1-\theta}},
\end{equation}
where the implicit constant depends only on $d$, $p$, $q_{1}$ and $q$.
Since the right-hand side of (\ref{10/03/29/18:24}) is independent of $J$, the condition (\ref{10/03/29/18:06}) leads us to (\ref{10/03/29/18:09}).  
\end{proof} 

Now, we give the proof of Proposition \ref{08/08/18/16:51}. 
\begin{proof}[Proof of Proposition \ref{08/08/18/16:51}] 
Since the proofs of (i) and (ii) are very similar, we consider (i) only. A starting point is the following formula derived from the integral equation associated to (\ref{08/05/13/8:50}):\begin{equation}\label{10/03/29/21:57}
e^{-\frac{i}{2}t\Delta }\psi(t)-e^{-\frac{i}{2}s \Delta}\psi(s)
=
\frac{i}{2}\int_{s}^{t}e^{-\frac{i}{2}t'\Delta}\left(|\psi|^{p-1}\psi \right)(t')dt'
\quad 
\mbox{for all $s,t \in [0,\infty)$}.
\end{equation}
Applying the Strichartz estimate and Lemma \ref{08/08/18/15:41} to this formula,  we obtain that 
\begin{equation}\label{09/12/14/14:32}
\begin{split}
\left\| e^{-\frac{i}{2}t\Delta }\psi(t)-e^{-\frac{i}{2}s\Delta}\psi(s) \right\|_{H^{1}}
&\le 
\sup_{s'\in [s,t]}
\left\| \int_{s}^{s'}e^{-\frac{i}{2}t'\Delta}\left(|\psi|^{p-1}\psi \right)(t')dt' \right\|_{H^{1}}
\\[6pt]
& \lesssim 
\left\| (1-\Delta)^{\frac{1}{2}}\left( |\psi|^{p-1}\psi \right) 
 \right\|_{L^{r_{0}'}([s,t];L^{q_{1}'})}
\\[6pt]
&\lesssim 
\left\|(1-\Delta)^{\frac{1}{2}} \psi \right\|_{S([0,\infty))}
\left\|\psi \right\|_{X([s,t])}^{p-1},
\end{split}
\end{equation}
where the implicit constant depends only on $d$, $p$ and $q_{1}$. Then, it follows from Lemma \ref{08/08/18/15:53} and the condition (\ref{10/03/29/12:27})  
that the right-hand side of (\ref{09/12/14/14:32}) vanishes as $s,t \to \infty$, so that the completeness of $H^{1}(\mathbb{R}^{d})$ leads us to that there exists $\phi_{+} \in H^{1}(\mathbb{R}^{d})$ such that 
\begin{equation}\label{09/12/14/14:43}
\lim_{t\to \infty}\left\| \psi(t)- e^{\frac{i}{2}t\Delta }\phi_{+} \right\|_{H^{1}}
=
\lim_{t\to \infty}
\left\| e^{-\frac{i}{2}t\Delta }\psi(t)-\phi_{+} \right\|_{H^{1}}
= 0.
\end{equation}
To complete the proof, we shall show the uniqueness of a function $\phi_{+}$ satisfying (\ref{09/12/14/14:43}). Let $\phi_{+}$ and $\phi_{+}'$ be functions in $H^{1}(\mathbb{R}^{d})$ satisfying (\ref{09/12/14/14:43}). Then, we easily verify that   
\begin{equation}\label{10/03/29/22:44}
\left\| \phi_{+} -\phi_{+}' \right\|_{H^{1}}
\le 
\left\| e^{-\frac{i}{2}t\Delta }\psi(t)- \phi_{+} \right\|_{H^{1}}
+
\left\| e^{-\frac{i}{2}t\Delta }\psi(t)- \phi_{+}' \right\|_{H^{1}}
\quad 
\mbox{for all $t>0$}.
\end{equation}
Hence, taking $t\to \infty$, we have by (\ref{09/12/14/14:43}) that $\phi_{+}=\phi_{+}'$ in $H^{1}(\mathbb{R}^{d})$. 
\end{proof}

The following proposition gives us another sufficient condition for the 
 boundedness of $X(I)$ and $S(I)$-norms. 
\begin{proposition}[Small data theory]
\label{08/08/22/20:59}
Assume that $d\ge 1$ and $2+\frac{4}{d}<p+1<2^{*}$. Let $t_{0} \in \mathbb{R}$ and $I$ be an interval whose closure contains $t_{0}$. Then, there exists a positive constant $\delta$, depending only on $d$, $p$ and $q_{1}$,  with the following property: for any $\psi_{0} \in H^{1}(\mathbb{R}^{d})$ satisfying that 
\begin{equation}
\left\| e^{\frac{i}{2}(t-t_{0})\Delta}\psi_{0}\right\|_{X(I)}\le \delta, 
\end{equation}
there exists a unique solution $\psi \in C(I;H^{1}(\mathbb{R}^{d}))$ to the equation (\ref{08/05/13/8:50}) with  $\psi(t_{0})=\psi_{0}$ such that    
\begin{equation}\label{10/03/30/9:44}
\left\| \psi \right\|_{X(I)}< 2
\left\| e^{\frac{i}{2}(t-t_{0})\Delta} \psi_{0} \right\|_{X(I)},
\quad
\left\| (1-\Delta)^{\frac{1}{2}} \psi \right\|_{S(I)}\lesssim 
 \left\| \psi_{0}\right\|_{H^{1}},
\end{equation}
where the implicit constant depends only on $d$, $p$ and $q_{1}$. 
\end{proposition}

\begin{remark}\label{10/03/29/12:24}
Proposition \ref{08/08/22/20:59}, with the help of (\ref{08/10/25/23:16}) and 
 Proposition \ref{08/08/18/16:51}, shows a small data scattering, i.e., there exists $\varepsilon>0$ such that for any $\psi_{0} \in H^{1}(\mathbb{R}^{d})$ with $\left\| \psi_{0} \right\|_{H^{1}}<\varepsilon$, the corresponding solution has an asymptotic state in $H^{1}(\mathbb{R}^{d})$. For: It follows from (\ref{08/10/25/23:16}) that there exists $\varepsilon>0$ such that if $\| \psi_{0} \|_{H^{1}}<\varepsilon$, then $\left\| e^{\frac{i}{2}(t-t_{0})\Delta}\psi_{0}\right\|_{X(\mathbb{R})}<\delta$ for $\delta$ found in Proposition \ref{08/08/22/20:59}. This fact, together with Proposition \ref{08/08/18/16:51}, yields the desired result.
\end{remark}

\begin{proof}[Proof of Proposition \ref{08/08/22/20:59}]
We prove this lemma by the standard contraction mapping principle. 
\par 
It follows from the Strichartz estimate that there exists a constant $C_{0}>0$, depending only on $d$, $p$ and $q_{1}$, such that 
\begin{equation}\label{09/10/08/17:18}
\left\| e^{\frac{i}{2}t\Delta} f \right\|_{S(I)}
\le 
C_{0}\left\| f \right\|_{L^{2}}
\quad 
\mbox{for all $f \in L^{2}(\mathbb{R}^{d})$}.
\end{equation}
Using this constant, we define a set $Y(I)$ and a metric $\rho$ there by 
\begin{equation}\label{10/03/30/10:09}
Y(I):=\left\{ u \in C(I;H^{1}(\mathbb{R}^{d})) 
\left| 
\begin{array}{c}
\left\| u \right\|_{X(I)}
\le 
2\left\| e^{\frac{i}{2}(t-t_{0})\Delta}\psi_{0} \right\|_{X(I)}, 
\\[6pt] 
\left\| (1-\Delta)^{\frac{1}{2}} u  \right\|_{S(I)}\le 2C_{0}\left\| \psi_{0}\right\|_{H^{1}}
\end{array} 
\right.
\right\},
\end{equation}
and 
\begin{equation}
\rho(u,v)=\left\| u-v \right\|_{X(I)\cap S(I)},
\quad 
u,v \in Y(I).
\end{equation}
We can verify that $(Y(I),\rho)$ is a complete metric space. 
Moreover, we define  a map $\mathcal{T}$ on this space by  
\begin{equation}
\mathcal{T}(u):=e^{\frac{i}{2}(t-t_{0})\Delta}\psi_{0}+\frac{i}{2}\int_{t_{0}}^{t}e^{\frac{i}{2}(t-t')\Delta}|u(t')|^{p-1}u(t')\,dt',
\quad 
u \in Y(I).
\end{equation}
Now, let $\delta>0$ be a constant to be specified later, and suppose that  
\begin{equation}\label{09/10/08/17:17}
\left\|e^{\frac{i}{2}(t-t_{0})\Delta}\psi_{0} \right\|_{X(I)} 
< 
\delta,
\end{equation}
so that 
\begin{equation}\label{09/12/15/11:23}
\left\|u  \right\|_{X(I)}
\le 
2\delta 
\quad 
\mbox{for all $u \in Y(I)$}. 
\end{equation}
We shall show that $\mathcal{T}$ maps $Y(I)$ into itself for sufficiently small $\delta>0$. Take any $u \in Y(I)$. Then,  we have $\mathcal{T}(u) \in C(I;H^{1}(\mathbb{R}^{d}))$ (see \cite{Kato2}). Moreover, Lemma \ref{08/07/31/12:02} and (\ref{09/12/15/11:23}) yield that 
\begin{equation}\label{08/09/01/11:29}
\begin{split}
\left\|\mathcal{T}(u) \right\|_{X(I)}
&\le   
\left\|e^{\frac{i}{2}(t-t_{0})\Delta} \psi_{0} \right\|_{X(I)}
+ C_{1} \left\|u  \right\|_{X(I)}^{p}
\\[6pt]
&\le 
\left\|e^{\frac{i}{2}(t-t_{0})\Delta} \psi_{0} \right\|_{X(I)}
+2^{p}C_{1}\delta^{p-1}
\left\|e^{\frac{i}{2}(t-t_{0})\Delta} \psi_{0} \right\|_{X(I)}
\end{split}
\end{equation}
for some constant $C_{1}>0$ depending only on $d$, $p$ and $q_{1}$. Hence, if we take $\delta$ so small that 
\begin{equation}\label{10/03/30/10:49}
2^{p}C_{1}\delta^{p-1}\le 1,
\end{equation}
then 
\begin{equation}\label{10/03/30/11:05}
\left\|\mathcal{T}(u) \right\|_{X(I)}\le 2\left\|e^{\frac{i}{2}(t-t_{0})\Delta} \psi_{0} \right\|_{X(I)}.
\end{equation}
On the other hand, it follows from the Strichartz estimate and Lemma \ref{08/08/18/15:41} that 
\begin{equation}\label{10/03/30/10:54}
\begin{split}
\left\|(1-\Delta)^{\frac{1}{2}}\mathcal{T}(u) \right\|_{S(I)}
&\le C_{0}\left\| \psi_{0} \right\|_{H^{1}}   
+
C_{2} \left\| (1-\Delta)^{\frac{1}{2}} u \right\|_{S(I)}
\left\|u  \right\|_{X(I)}^{p-1}
\\[6pt]
&\le 
C_{0}\left\| \psi_{0} \right\|_{H^{1}}
+ 
2^{p}C_{2}  C_{0} \delta^{p-1} \left\| \psi_{0} \right\|_{H^{1}}
\end{split}
\end{equation}
for some constant  $C_{2}>0$ depending only on $d$, $p$ and $q_{1}$. If $\delta$ satisfies that 
\begin{equation}\label{10/03/30/10:58}
2^{p}C_{2} \delta^{p-1} \le 1, 
\end{equation}
then we have from (\ref{10/03/30/10:54}) that 
\begin{equation}\label{10/03/30/10:59}
\left\| (1-\Delta)^{\frac{1}{2}} \mathcal{T}(u) \right\|_{S(I)}
\le 
2C_{0}\left\|\psi_{0} \right\|_{H^{1}}.
\end{equation}
Thus, if we take $\delta$ satisfying (\ref{10/03/30/10:49}) and (\ref{10/03/30/10:58}), then $\mathcal{T}$ becomes a self-map on $Y(I)$.  
\par 
Next, we shall show that $\mathcal{T}$ is a contraction map on $(Y(I),\rho)$ for sufficiently small $\delta$. Lemma \ref{08/07/31/12:02}, together with (\ref{09/09/27/21:10}) and (\ref{09/12/15/11:23}), gives us that  
\begin{equation}\label{10/03/30/11:28}
\begin{split}
\left\|\mathcal{T}(u) -\mathcal{T}(v)\right\|_{X(I)}
& \le  
\left\| \int_{t_{0}}^{t} e^{\frac{i}{2}(t-t')\Delta} 
\left( |u|^{p-1}u -|v|^{p-1}v \right)(t')dt' \right\|_{X(I)}
\\[6pt]
&\le   
C_{3} \left( \left\|u  \right\|_{X(I)}^{p-1}+ \left\| v \right\|_{X(I)}^{p-1}\right)\left\| u-v \right\|_{X(I)}
\\[6pt]
&\le 
2^{p}C_{3} \delta^{p-1}\left\| u-v \right\|_{X(I)} 
\quad 
\mbox{for all $u,v \in Y(I)$},
\end{split}
\end{equation}
where $C_{3}$ is some positive constant depending only on $d$, $p$ and $q_{1}$.  Similarly, we have by Lemma \ref{08/08/18/15:41} that   
\begin{equation}\label{10/03/30/11:39}
\left\|\mathcal{T}(u) -\mathcal{T}(v)\right\|_{S(I)}
\le 2^{p}C_{4} \delta ^{p-1} \left\| u-v \right\|_{S(I)}
\end{equation}
for some constant $C_{4}>0$ depending only on $d$, $p$ and $q_{1}$. 
Hence, we find that $\mathcal{T}$ becomes a contraction map in $(Y(I),\rho)$ for  sufficiently small $\delta$. Then, the contraction mapping principle shows that there exists a solution $\psi \in Y(I)$ to the equation (\ref{08/05/13/8:50}) with $\psi(t_{0})=\psi_{0}$, which, 
 together with the uniqueness of solutions in $C(I;H^{1}(\mathbb{R}^{d}))$ (see \cite{Kato1995}), completes the proof. 
\end{proof}

At the end of this subsection, we give the proof of Proposition \ref{09/05/18/10:43}. 
\begin{proof}[Proof of Theorem \ref{09/05/18/10:43}]
Let $\psi$ be a global solution to (\ref{08/05/13/8:50}) with an initial datum $\psi_{0} \in H^{1}(\mathbb{R}^{d})$ at $t=0$.
\\
{\it We shall prove that (i) implies (ii)}: Suppose that (i) holds. 
 We first show the uniform boundedness:
\begin{equation}\label{10/04/29/15:46}
\sup_{t\in [0,\infty)}\left\| \psi(t)\right\|_{H^{1}}<\infty.
\end{equation}
The condition (i) yields that 
\begin{equation}\label{10/04/29/15:37}
\sup_{t\ge T} \left\| \psi(t)\right\|_{L^{p+1}}\le 1
\quad 
\mbox{for some $T>0$}.
\end{equation}
Combining this with the energy conservation law (\ref{08/05/13/9:03}), we obtain that  
\begin{equation}\label{10/04/29/15:32}
\left\| \nabla \psi(t) \right\|_{L^{2}}^{2}
=
\mathcal{H}(\psi(t))-\left\| \psi(t)\right\|_{L^{p+1}}^{p+1}
\le 
\mathcal{H}(\psi_{0})+1 
\quad 
\mbox{for all $t\ge T$}.
\end{equation}
Hence, this estimate and the mass conservation law (\ref{08/05/13/8:59}) give 
 (\ref{10/04/29/15:46}). 
\par 
Now, we shall prove (ii). We employ the H\"older inequality and the Sobolev embedding to obtain that 
\begin{equation}\label{10/05/31/8:56}
\left\| \psi(t) \right\|_{L^{q}}
\lesssim 
\left\| \psi(t)\right\|_{L^{p+1}}^{1-\theta}
\left\| \psi(t) \right\|_{H^{1}}^{\theta}
\quad 
\mbox{for all $t \in [0,\infty)$ and $q \in (p+1,2^{*})$},
\end{equation}
where $\theta$ is some constant in $(0,1)$, and the implicit constant is independent of $t$. 
Hence, the condition (i), together with (\ref{10/04/29/15:46}), shows that 
\begin{equation}\label{10/04/29/15:48}
\lim_{t\to \infty}\left\| \psi(t) \right\|_{L^{q}}=0
\quad 
\mbox{for all $q \in (p+1,2^{*})$}.
\end{equation}
On the other hand, we have by the H\"older inequality and the mass conservation law (\ref{08/05/13/8:59}) that 
\begin{equation}\label{10/04/29/15:57}
\lim_{t\to \infty}
\left\| \psi(t)\right\|_{L^{q}}
\le 
\left\| \psi_{0} \right\|_{L^{2}}^{1-\frac{(q-2)(p+1)}{q(p-1)}}
\lim_{t\to \infty}
\left\| \psi(t) \right\|_{L^{p+1}}^{\frac{(q-2)(p+1)}{q(p-1)}}
=0
\quad
\mbox{for all $q \in (2,p+1)$}.
\end{equation}
\\
\noindent 
{\it We shall prove (ii) implies (iii)}: We introduce a number $r_{2,0}$ such that $(q_{2}, r_{2,0})$ is admissible, i.e.,
\begin{equation}\label{10/02/24/16:35}
\frac{1}{r_{2,0}}=\frac{d}{2}\left( \frac{1}{2}-\frac{1}{q_{2}}\right),
\end{equation}
where $q_{2}$ is the number defined in (\ref{09/12/05/16:15}). 
 Then, we have that 
\begin{equation}\label{10/02/24/16:37}
2< r_{2,0} < r_{2}<\infty
\end{equation}
for $r_{2}$ defined in (\ref{09/12/05/16:15}). Moreover, we define a number $q_{e}$ by $q_{e}=2^{*}$ if $d\ge 3$ and $q_{e}=
 2q_{1}(>q_{2})$ if $d=1,2$, and a space $\bar{S}(I)$ by  
\begin{equation}\label{10/02/24/17:02}
\bar{S}(I) = \bigcap_{{(q,r):admissible }\atop {2\le q\le q_{e}}}L^{r}(I;L^{q})
\quad 
\mbox{ for an interval $I$}.
\end{equation}
Then, it follows from the Strichartz estimate and Lemma \ref{08/08/18/15:41} 
 that 
\begin{equation}\label{10/02/24/16:42}
\begin{split}
&\left\|(1-\Delta)^{\frac{1}{2}}\psi \right\|_{\bar{S}([t_{0},t_{1}))}
\\[6pt]
&\lesssim 
\left\| \psi(t_{0}) \right\|_{H^{1}}
+
\left\| (1-\Delta)^{\frac{1}{2}}|\psi|^{p-1}\psi \right\|_{L^{r_{0}'}([t_{0},t_{1});L^{q_{1}'})}
\\[6pt]
&\le 
\left\|\psi \right\|_{L^{\infty}([0,\infty);H^{1})}
+
\left\|(1-\Delta)^{\frac{1}{2}}\psi  \right\|_{L^{r_{0}}([t_{0},t_{1});L^{q_{1}})}
\left\|\psi  \right\|_{L^{r_{2}}([t_{0},t_{1});L^{q_{2}})}^{p-1}
\\[6pt]
&\le 
\left\|\psi \right\|_{L^{\infty}([0,\infty);H^{1})}
+
\sup_{t\ge t_{0}}\left\| \psi(t)\right\|_{L^{q_{2}}}^{\frac{p-1}{r_{2}}(r_{2}-r_{2,0})}
\left\|(1-\Delta)^{\frac{1}{2}}\psi  \right\|_{L^{r_{0}}([t_{0},t_{1});L^{q_{1}})}
\left\|\psi  \right\|_{L^{r_{2,0}}([t_{0},t_{1});L^{q_{2}})}^{\frac{p-1}{r_{2}}r_{2,0}}
\\[6pt]
&\le 
\left\|\psi \right\|_{L^{\infty}([0,\infty);H^{1})}
+
\sup_{t\ge t_{0}}\left\| \psi(t)\right\|_{L^{q_{2}}}^{\frac{p-1}{r_{2}}(r_{2}-r_{2,0})}
\left\|(1-\Delta)^{\frac{1}{2}} \psi  \right\|_{\bar{S}([t_{0},t_{1});L^{q_{1}})}^{1+\frac{p-1}{r_{2}}r_{2,0}},
\end{split}
\end{equation}
where the implicit constant is independent of $t_{0}$ and $t_{1}$. 
Note here that since the condition (ii) includes (i), we have that 
 $\|\psi\|_{L^{\infty}([0,\infty);H^{1})}<\infty$ (see (\ref{10/04/29/15:46}) above). Hence, the estimate (\ref{10/02/24/16:42}), together with the condition (ii), shows that
\begin{equation}\label{10/02/24/17:15}
\left\| (1-\Delta)^{\frac{1}{2}}\psi \right\|_{\bar{S}([t_{0},+\infty))}< \infty\quad 
\mbox{for all sufficiently large $t_{0}>0$},
\end{equation}
so that (iii) holds. 
\\
\\
{\it We shall prove that (iii) implies (iv)}: The estimate (\ref{08/10/25/23:16}) immediately 
 gives us the desired result.
\\
\\
{\it We shall prove that (iv) implies (v)}: This follows from Proposition \ref{08/08/18/16:51}. 
\\
\\
{\it We shall prove that (v) implies (i)}: We define $r$ by  
\begin{equation}\label{10/03/30/11:57}
\frac{1}{r}=\frac{d}{2} \left( \frac{1}{2}-\frac{1}{p+1}\right),
\end{equation}
so that $(p+1, r)$ is an admissible pair. Then, the Strichartz estimate
yields that 
\begin{equation}\label{10/03/30/11:59}
\left\|e^{\frac{i}{2}t\Delta}\phi_{+} \right\|_{L^{r}(\mathbb{R};L^{p+1})} 
\lesssim 
\left\| \phi_{+} \right\|_{L^{2}},
\end{equation}
where the implicit constant depends only on $d$ and $p$. Hence, 
 we have that  
\begin{equation}\label{08/11/16/15:23}
\liminf_{t\to \infty} 
\left\| e^{\frac{i}{2}t\Delta}\phi_{+} \right\|_{L^{p+1}}
=0.
\end{equation}
Moreover, it follows from the existence of an asymptotic state $\phi_{+}$ and (\ref{08/11/16/15:23})  that 
\begin{equation}\label{10/03/30/12:13}
\liminf_{t\to \infty}\left\| \psi (t) \right\|_{L^{p+1}}
\le 
\lim_{t\to \infty} \left\| \psi(t) -e^{\frac{i}{2}t\Delta}\phi_{+} \right\|_{L^{p+1}}
+ 
\liminf_{t\to \infty}\left\| e^{\frac{i}{2}t\Delta}\phi_{+}\right\|_{L^{p+1}}
=0.
\end{equation}
Using the existence of an asymptotic state again, we also obtain that
\begin{equation}\label{09/04/30/11:54}
\lim_{t\to \infty}\left\| \nabla \psi(t) \right\|_{L^{2}}
=
\lim_{t\to \infty}\left\| \nabla e^{-\frac{i}{2}t\Delta} \psi(t) 
\right\|_{L^{2}}
=\left\| \nabla \phi_{+}\right\|_{L^{2}}.
\end{equation}
This formula (\ref{09/04/30/11:54}) and the energy conservation law (\ref{08/05/13/9:03}) lead us to that    
\begin{equation}\label{10/03/30/12:09}
\lim_{t\to \infty}\left\| \psi(t) \right\|_{L^{p+1}}^{p+1}
=
-\frac{p+1}{2} \mathcal{H}(\psi(0))
+\frac{p+1}{2} \left\|\nabla \phi_{+} \right\|_{L^{2}}^{2}.
\end{equation}
In particular, $\displaystyle{\lim_{t\to \infty}\left\| \psi(t) \right\|_{L^{p+1}}^{p+1}}$ exists. Hence, the desired result (i) follows from (\ref{10/03/30/12:13}). 
\end{proof}

\subsection{Long time perturbation theory}
\label{09/05/30/21:16}
We will employ the so-called ``Born type approximation'' to prove that the solutions starting from $PW_{+}$ have asymptotic states (see Section \ref{08/10/03/15:11}). The following proposition plays a crucial role there. 

\begin{proposition}[Long time perturbation theory]
\label{08/08/05/14:30}
Assume that $d\ge 1$ and $2+\frac{4}{d}<p+1<2^{*}$. Then, for any $A>1$, there exists $\varepsilon>0$, depending only on $A$, $d$, $p$ and $q_{1}$, such that 
 the following holds: Let $I$ be an interval, and $u$ be a function in $C(I; H^{1}(\mathbb{R}^{d})) $ such that 
\begin{align}
\label{10/04/02/18:14}
&\left\| u \right\|_{X(I)}\le A,
\\[6pt]
\label{10/04/02/18:21}
&\left\| 2i\partial_{t}u+\Delta u+|u|^{p-1}u \right\|_{L^{\widetilde{r}_{1}'}(I;L^{q_{1}'})}
\le \varepsilon.
\end{align}
If $\psi\in C(\mathbb{R};H^{1})$ is a global solution to the equation (\ref{08/05/13/8:50}) and satisfies that 
\begin{equation}
\label{09/09/15/18:03}
\left\| e^{\frac{i}{2}(t-t_{1})\Delta} (\psi(t_{1})-u(t_{1}))\right\|_{X(I)} \le\varepsilon 
\quad 
\mbox{for some $t_{1} \in I$},
\end{equation}
then we have   
\begin{equation}\label{10/04/06/14:41}
\left\| \psi \right\|_{X(I)} \lesssim 1,
\end{equation}
where the implicit constant depends only on $d$, $p$, $q_{1}$ and $A$. 
\end{proposition}
\begin{remark}
We can find from the proof below that Proposition \ref{08/08/05/14:30} still holds valid if we replace $\psi$ with a function $\widetilde{\psi} \in C(\mathbb{R};H^{1}(\mathbb{R}^{d}))$ satisfying that\begin{equation}\label{10/04/02/18:26}
\left\| 
2i\partial_{t}\widetilde{\psi}+\Delta \widetilde{\psi}
+\big|\widetilde{\psi}\big|^{p-1}\widetilde{\psi} 
\right\|_{L^{\widetilde{r}_{1}'}(I;L^{q_{1}'})}\le \varepsilon.
\end{equation}
\end{remark}
\begin{proof}[Proof of Proposition \ref{08/08/05/14:30}]
Let $u$ be a function in $C(I;H^{1}(\mathbb{R}^{d})$ satisfying (\ref{10/04/02/18:14}) and (\ref{10/04/02/18:21}) for $\varepsilon>0$ to be chosen later. Moreover, let $\psi$ be a global solution to the equation (\ref{08/05/13/8:50}) satisfying (\ref{09/09/15/18:03}). 
\par 
For simplicity, we suppose that $I=[t_{1},\infty)$ in what follows. The other cases are treated in a way similar to this case. 
\par 
We have from the condition (\ref{10/04/02/18:14}) that: for any $\delta >0$,  there exist disjoint intervals $I_{1},\ldots, I_{N}$, with a form $I_{j}=[t_{j},t_{j+1})$ ($t_{N+1}=\infty$), such that 
\begin{equation}\label{10/04/02/13:33}
I=\bigcup_{j=1}^{N}I_{j},
\end{equation}
and 
\begin{equation}\label{10/04/03/13:13}
\left\| u \right\|_{X(I_{j})}\le \delta 
\quad 
\mbox{for all $j=1,\ldots, N$},
\end{equation}
where $N$ is some number depending only on $\delta$, $A$, $d$, $p$ and $q_{1}$. \par   
 We put 
\begin{equation}\label{10/04/06/18:15}
w:=\psi-u,
\qquad 
e:=2i\partial_{t}u+\Delta u+|u|^{p-1}u,
\end{equation}
Then, $w$ satisfies that  
\begin{align}
\label{08/08/05/15:12}
&2i\partial_{t}w + \Delta w + |w+u|^{p-1}(w+u)-|u|^{p-1}u+e=0,
\\[6pt]
\label{08/08/06/2:37}
&\left\|e^{\frac{i}{2}(t-t_{1})\Delta} w(t_{1}) \right\|_{X(I)}
\le \varepsilon.
\end{align}
We consider the integral equations associated with (\ref{08/08/05/15:12}):
\begin{equation}\label{08/08/05/15:52}
w(t)=e^{\frac{i}{2}(t-t_{j})\Delta}w(t_{j})
 +\frac{i}{2}\int_{t_{j}}^{t}e^{\frac{i}{2}(t-t')\Delta}W(t')\,dt'
,
\quad 
j=1,\ldots, N,
\end{equation}
where 
\begin{equation}\label{10/04/03/14:00}
W:= |w+u|^{p-1}(w+u)-|u|^{p-1}u-e.
\end{equation}
It follows from the inhomogeneous Strichartz estimate (\ref{10/03/29/11:06}), 
the elementary inequality (\ref{09/09/27/21:10}) and the H\"older inequality 
that
\begin{equation}\label{08/08/16:30}
\begin{split}
&\left\|\int_{t_{j}}^{t}e^{\frac{i}{2}(t-t')\Delta}W(t')\,dt'  
\right\|_{X([t_{j}, t_{j}'))} 
\\[6pt]
&\le 
\mathscr{C}
\left\{
\left\| w \right\|_{L^{r_{1}}([t_{j}, t_{j}');L^{q_{1}})}
\left\|u \right\|_{L^{r_{2}}([t_{j}, t_{j}');L^{q_{2}})}^{p-1}
+ 
\left\| w \right\|_{L^{r_{1}}([t_{j}, t_{j}');L_{x}^{q_{1}})}
\left\|w \right\|_{L^{r_{2}}([t_{j}, t_{j}');L^{q_{2}})}^{p-1}
\right.
\\[6pt]
&\qquad \qquad \qquad 
+
\left.
\left\| e\right\|_{L^{\widetilde{r}_{1}'}([t_{j}, t_{j}');L^{q_{1}'})}
\right\}
\qquad 
\mbox{for all $j=1,\ldots, N$ and $t_{j}' \in I_{j}$}
,
\end{split}
\end{equation}
where $\mathscr{C}$ is some constant depending only on $d$, $p$ and $q_{1}$. 
Hence, combining this estimate with (\ref{10/04/02/18:21}) and (\ref{10/04/03/13:13}), we obtain that 
\begin{equation}
\label{08/08/05/17:32}
\begin{split}
\left\| w \right\|_{X([t_{j}, t_{j}'))} 
 \le 
\left\|e^{\frac{i}{2}(t-t_{j})\Delta} w(t_{j}) \right\|_{X(I_{j})}
+
\mathscr{C}
\left\{
\delta^{p-1} \left\| w \right\|_{X([t_{j}, t_{j}'))}
+ \left\|w \right\|_{X([t_{j}, t_{j}'))}^{p}
+\varepsilon
\right\}
&
\\[6pt]
\mbox{for all $j=1,\ldots, N$ and $t_{j}' \in I_{j}$}.
&
\end{split}
\end{equation}
Now, we fix a constant $\delta$ such that 
\begin{equation}
\label{08/08/05/17:39}
\delta < \left( \frac{1}{4(1+2\mathscr{C})}\right)^{\frac{1}{p-1}},
\end{equation}
so that the number $N$ is also fixed. Then, we shall show that 
\begin{align}
\label{08/08/06/16:52}
\left\| e^{\frac{i}{2}(t-t_{j})\Delta}w(t_{j}) \right\|_{X(I)}
&\le 
(1 + 2^{j}\mathscr{C}) \varepsilon
\qquad 
\mbox{for all $j=1,\ldots, N$}, 
\\[6pt]
\label{09/02/13/20:16}
\left\| w \right\|_{X(I_{j})}
&\le 
\left( 1 +2^{j+1} \mathscr{C} \right)\varepsilon 
\qquad
\mbox{for all $j=1,\ldots, N$},
\end{align}
provided that  
\begin{equation}\label{08/08/06/16:48}
\varepsilon  
< 
\left( \frac{1}{4(1+2^{N+1} \mathscr{C})^{p}} \right)^{\frac{1}{p-1}}.
\end{equation}

We first consider the case $j=1$. The estimate (\ref{08/08/06/16:52}) for $j=1$ obviously follows from (\ref{08/08/06/2:37}). We put 
\begin{equation}\label{10/04/03/19:15}
\bar{t}_{1}
=
\sup \left\{ t_{1}\le t \le t_{2} \bigm| \left\|
w \right\|_{X([t_{1},t])}\le (1+4\mathscr{C}) \varepsilon \right\}.
\end{equation}
Then, the estimate (\ref{08/08/05/17:32}), together with (\ref{08/08/06/2:37}), shows that $\bar{t}_{1}>t_{1}$. Supposing the contrary that $\bar{t}_{1}<t_{2}$, we have from the continuity of $w$ that 
\begin{equation}\label{10/04/03/16:39} 
\left\| w \right\|_{X([t_{1},\bar{t}_{1}])} 
=(1+4\mathscr{C})\varepsilon.
\end{equation} 
However, (\ref{08/08/05/17:32}), together with (\ref{08/08/06/16:52}) with $j=1$ and (\ref{10/04/03/16:39}), leads us to that 
\begin{equation}\label{10/04/03/16:41}
\begin{split}
\left\| w \right\|_{X([t_{1},\bar{t}_{1}])} 
&\le 
\left(1+2\mathscr{C}\right) \varepsilon
+ \mathscr{C} 
\left\{  
\delta^{p-1}(1+4\mathscr{C})\varepsilon 
+
(1+4\mathscr{C})^{p}\varepsilon^{p}
+
\varepsilon
\right\}
\\[6pt]
&
<\varepsilon+4\mathscr{C}\varepsilon.
\end{split}
\end{equation}
This contradicts (\ref{10/04/03/16:39}). Thus, it must hold that $\bar{t}_{1}=t_{2}$ and 
 therefore (\ref{09/02/13/20:16}) holds for $j=1$.
\par 
We shall prove the general case $j=n$ ($2\le n\le N$), provided that both (\ref{08/08/06/16:52}) and (\ref{09/02/13/20:16}) hold  for all $j=1, \ldots, n-1$. 
\par 
Multiplying the integral equation (\ref{08/08/05/15:52}) with $j=n$ and $t=t_{n}$ by $e^{\frac{i}{2}(t-t_{n})\Delta}$, we obtain that      
\begin{equation}\label{08/08/05/18:09}
e^{\frac{i}{2}(t-t_{n})\Delta}w(t_{n})
=
e^{\frac{i}{2}(t-t_{n-1})\Delta}w(t_{n-1})
+\frac{i}{2}\int_{t_{n-1}}^{t_{n}}e^{\frac{i}{2}(t-t')\Delta}\chi_{I_{n-1}}(t')
W (t')\,dt',
\end{equation}
where $\chi_{I_{n-1}}$ is the characteristic function on $I_{n-1}$. 
 Then, the same estimate as (\ref{08/08/05/17:32}) yields that 
\begin{equation}\label{08/08/05/18:14}
\begin{split}
\left\| e^{\frac{i}{2}(t-t_{n})\Delta}w(t_{n}) \right\|_{X(I)}
&\le 
\left\|e^{\frac{i}{2}(t-t_{n-1})\Delta}w(t_{n-1})  \right\|_{X(I)}
\\[6pt]
&\qquad 
+
\mathscr{C}
\left\{
\delta^{p-1} \left\| w \right\|_{X(I_{n-1})}
+\left\| w \right\|_{X(I_{n-1})}^{p}
+\varepsilon
\right\}
\end{split}
\end{equation}
for the same constant $\mathscr{C}$ found in (\ref{08/08/05/17:32}). This estimate, together with  the inductive hypothesis, gives us that  
\begin{equation}\label{10/04/03/20:15}
\begin{split}
\left\| e^{\frac{i}{2}(t-t_{n})\Delta}w(t_{n}) \right\|_{X(I)}
&< 
(1 +2^{n-1}\mathscr{C}) \varepsilon 
+
\mathscr{C}
\left\{ 
2^{n-3}\varepsilon + \frac{1}{4}\varepsilon + \varepsilon
\right\}
\\[6pt]
&\le 
(1 + 2^{n}\mathscr{C})\varepsilon,
\end{split}
\end{equation}
which shows that (\ref{08/08/06/16:52}) holds for $j =n$. 
\par 
Next, we consider (\ref{09/02/13/20:16}) for $j=n$. As well as the case $j=1$, we put  
\begin{equation}\label{10/04/04/16:42}
\bar{t}_{n}
=
\sup \left\{ t_{n}\le t \le t_{n+1} \bigm| \left\|
w \right\|_{X([t_{n},t])}\le (1+2^{n+1}\mathscr{C}) \varepsilon \right\}.
\end{equation}
The estimate (\ref{08/08/05/17:32}), together with (\ref{10/04/03/20:15}), shows that 
$t_{n}<\bar{t}_{n}$. We suppose the contrary that $\bar{t}_{n}<t_{n+1}$, 
 so that   
\begin{equation}\label{10/04/03/20:34}
\left\| w \right\|_{X([t_{n},\bar{t}_{n}])}
=
\left( 1 +2^{n+1} \mathscr{C} \right) \varepsilon .
\end{equation} 
Then, combining (\ref{08/08/05/17:32}) with (\ref{10/04/03/20:15}) and (\ref{10/04/03/20:34}), we obtain that 
\begin{equation}\label{10/04/03/20:37}
\left\| w \right\|_{X([t_{n},\bar{t}_{n}])} 
< \left( 1+ 2^{n} \mathscr{C} \right) \varepsilon
+\mathscr{C} 
\left\{ 
2^{n-2} \varepsilon 
+ 
\frac{1}{4}\varepsilon
+
\varepsilon 
\right\}
 \le 
\left( 1+ 2^{n+1}\mathscr{C} \right) \varepsilon,
\end{equation}
which contradicts (\ref{10/04/03/20:34}). Hence, we have proved 
 (\ref{08/08/06/16:52}) and (\ref{09/02/13/20:16}). 
\par 
Now, it follows from (\ref{09/02/13/20:16}) and (\ref{08/08/06/16:48}) 
that 
\begin{equation}\label{10/04/03/21:43}
\left\| w \right\|_{L^{r_{j}}(I;L^{q_{j}})}^{r_{j}}
= 
\sum_{n=1}^{N}\left\| w \right\|_{L^{r_{j}}(I_{n};L^{q_{j}})}^{r_{j}}
\\[6pt]
\le 
\sum_{n=1}^{N}\left( 1 +2^{N+1} \mathscr{C}\right)^{r_{j}} \varepsilon^{r_{j}}
\le 
\frac{N}{4^{r_{j}}}
\le N^{r_{j}}
\quad 
\mbox{for $j=1,2$}.
\end{equation}
Hence, we have from (\ref{10/04/02/18:14}) and (\ref{10/04/03/21:43}) that 
\begin{equation}\label{10/04/03/21:50}
\left\|\psi \right\|_{X(I)}
\le \left\|u \right\|_{X(I)}
+
\left\|w \right\|_{X(I)}
\le A+N,
\end{equation}
which completes the proof. 
\end{proof}

\subsection{Wave operators}
\label{09/05/30/21:25}
The following proposition tells us that the wave operators are well-defined on $\Omega$.
\begin{proposition}[Existence of wave operators]
\label{09/01/12/16:36}
Assume $d \ge 1$ and $2+\frac{4}{d}\le p+1 <2^{*}$. 
\\
{\rm (i)} For any $\phi_{+} \in \Omega$, there exists a unique $\psi_{0} \in PW_{+}$ such that the corresponding solution $\psi$ to the equation (\ref{08/05/13/8:50}) with $\psi(0)=\psi_{0}$ exists globally in time and satisfies the followings: \begin{align}
\label{09/10/08/18:42}
&\psi \in X([0,+\infty)), 
\\[6pt]
\label{09/10/08/18:43}
&\lim_{t\to +\infty}\left\|\psi(t)-e^{\frac{i}{2}t\Delta}\phi_{+} \right\|_{H^{1}}=0,  
\\[6pt]
\label{09/10/08/18:44}
&\mathcal{H}(\psi(t))=\left\|\nabla \phi_{+ } \right\|_{L^{2}}^{2}
\quad 
\mbox{for all $t \in \mathbb{R}$}.
\end{align}
Furthermore, if $\left\| \phi_{+} \right\|_{H^{1}}$ is sufficiently small, then we have 
\begin{equation}\label{09/10/08/18:45}
\left\| \psi \right\|_{X(\mathbb{R})}\lesssim \left\| \phi_{+} \right\|_{H^{1}},\end{equation}
where the implicit constant depends only on $d$, $p$ and $q_{1}$.
\\
The map defined by $\phi_{+} \mapsto \psi_{0}$ is continuous from $\Omega$ into  $PW_{+}$ in the $H^{1}(\mathbb{R}^{d})$-topology. 
\\[6pt]
{\rm (ii)} For any $\phi_{-} \in \Omega$, there exists a unique $\psi_{0} \in PW_{+}$ such that the corresponding solution $\psi$ to the equation (\ref{08/05/13/8:50}) with $\psi(0)=\psi_{0}$ exists globally in time and satisfies that 
\begin{align}
\label{09/10/08/18:46}
&\psi \in X((-\infty,0]), 
\\[6pt]
\label{09/10/08/18:47}
&\lim_{t\to -\infty}\left\|\psi(t)-e^{\frac{i}{2}t\Delta}\phi_{-} \right\|_{H^{1}}=0, 
\\[6pt]
\label{09/10/08/18:48}
&\mathcal{H}(\psi(t))=\left\|\nabla \phi_{-} \right\|_{L^{2}}^{2}
\quad 
\mbox{for all $t \in \mathbb{R}$}.
\end{align}
Furthermore, if $\left\| \phi_{-} \right\|_{H^{1}}$ is sufficiently small, then we have 
\begin{equation}\label{09/10/08/18:49}
\left\| \psi \right\|_{X(\mathbb{R})}\lesssim \left\| \phi_{-} \right\|_{H^{1}},\end{equation}
where the implicit constant depends only on $d$, $p$ and $q_{1}$.
\\
The map defined by $\phi_{-} \mapsto \psi_{0}$ is continuous from $\Omega$ into $PW_{+}$ in the $H^{1}(\mathbb{R}^{d})$-topology. 
\end{proposition}

\begin{proof}[Proof of Proposition \ref{09/01/12/16:36}]  Since the proofs of (i) and (ii) are very similar, we prove (i) only. We look for a solution to the following integral equation in $X([0,\infty))$ for all $\phi_{+} \in \Omega$:  
\begin{equation}\label{09/01/12/16:43}
\psi(t)=e^{\frac{i}{2}t\Delta}\phi_{+}-\frac{i}{2}\int_{t}^{+\infty}e^{\frac{i}{2}(t-t')\Delta}\left\{ |\psi(t')|^{p-1}\psi(t')\right\}dt'.
\end{equation}
We first note that the estimate (\ref{08/10/25/23:16}) shows that: For any $\delta>0$, there exists 
 $T_{\delta}>0$ such that 
\begin{equation}\label{09/10/08/18:32}
\left\| e^{\frac{i}{2}t\Delta} \phi_{+} \right\|_{X([T_{\delta},+\infty))}
<
\delta.
\end{equation}
Moreover, it follows from the Strichartz estimate that: There exists $C_{0}>0$, depending only on $d$, $p$ and $q_{1}$, such that 
\begin{equation}\label{09/10/08/18:33}
\left\|(1-\Delta)^{\frac{1}{2}} e^{\frac{i}{2}t\Delta}\phi_{+} \right\|_{S(\mathbb{R})}
\le C_{0}\left\| \phi_{+} \right\|_{H^{1}}.
\end{equation}
Using these constants $T_{\delta}$ and $C_{0}$, we define a set $Y_{\delta}[\phi_{+}]$ by 
\begin{equation}\label{10/04/06/23:34}
Y_{\delta}[\phi_{+}]:=
\left\{ u \in C(I_{\delta};H^{1}(\mathbb{R}^{d})) \left| 
\begin{array}{l} 
\left\| u  \right\|_{X(I_{\delta})}\le 2\left\| e^{\frac{i}{2}t\Delta}\phi_{+} 
\right\|_{X(I_{\delta})},
\\ 
\left\|(1-\Delta)^{\frac{1}{2}}u \right\|_{S(I_{\delta})}\le 2C_{0}\left\|\phi_{+} \right\|_{H^{1}}
\end{array}
\right. 
\right\},
\end{equation}
where $I_{\delta}=[T_{\delta},\infty)$. We can verify that $Y_{\delta}[\phi_{+}]$ becomes a complete metric space with a metric $\rho$ defined by 
\begin{equation}\label{10/04/06/23:37}
\rho(u,v):=\max
\left\{ 
\left\| u-v \right\|_{X(I_{\delta})},
\  
\left\| u-v \right\|_{S(I_{\delta})}
\right\}
\quad 
\mbox{for all $u,v \in Y_{\delta}[\phi_{+}]$}.
\end{equation}
Then, an argument similar to the proof of Proposition \ref{08/08/22/20:59} (small data theory) leads us to that: If $\delta>0$ is sufficiently small, then there exists a unique solution $\psi \in Y_{\delta}[\phi_{+}]$ to the integral equation (\ref{09/01/12/16:43}). In particular, we find that such a solution exists globally in time and satisfies (\ref{09/10/08/18:49}), provided that $\|\phi_{+}\|_{H^{1}}$ is sufficiently small. 
\par 
We shall show that the function $\psi$ obtained above becomes a solution to the equation (\ref{08/05/13/8:50}). Multiplying the both sides of (\ref{09/01/12/16:43}) by the free operator, we have   
\begin{equation}\label{09/10/17/16:33}
e^{\frac{i}{2}(t-T)\Delta}\psi(T)=e^{\frac{i}{2}t\Delta}\phi_{+}
-\frac{i}{2}\int_{T}^{+\infty}e^{\frac{i}{2}(t-t')\Delta}\left\{ |\psi(t')|^{p-1}\psi(t') \right\}dt'
\quad
\mbox{for all $t,T \in I_{\delta}$}.
\end{equation}
Subtraction (\ref{09/10/17/16:33}) from (\ref{09/01/12/16:43}) yields further that  
\begin{equation}\label{09/10/17/16:35}
\psi(t)=e^{\frac{i}{2}(t-T)\Delta}\psi(T)+\frac{i}{2}\int_{T}^{t}e^{\frac{i}{2}(t-t')\Delta}\left\{ |\psi(t')|^{p-1}\psi(t') \right\}dt'
\quad 
\mbox{for all $t, T \in I_{\delta}$}.
\end{equation} 
Hence, we find by this formula (\ref{09/10/17/16:35}) that $\psi$ is a solution to the equation (\ref{08/05/13/8:50}) on $I_{\delta}$. 
\par 
Next, we shall extend the solution $\psi$ to the whole interval $\mathbb{R}$. 
 In view of Proposition \ref{09/06/21/19:28}, it suffices to show that $\psi(t) \in PW_{+}$ for some $t \in I_{\delta}$. 
Applying the usual inhomogeneous Strichartz estimate to (\ref{09/01/12/16:43}), and then applying Lemma \ref{08/08/18/15:41}, 
 we obtain that    
\begin{equation}\label{09/01/13/9:54}
\lim_{t\to +\infty}\left\| \psi(t)-e^{\frac{i}{2}t\Delta}  \phi_{+} \right\|_{H^{1}}
\lesssim 
\lim_{t\to +\infty}\left\|(1-\Delta)^{\frac{1}{2}} \psi \right\|_{S([t,\infty))}\left\| \psi \right\|_{X([t,\infty))}^{p-1}
=0.
\end{equation}
Using this estimate (\ref{09/01/13/9:54}) and Lemma \ref{10/04/08/9:20}, we conclude that 
\begin{equation}\label{10/04/07/22:54}
\begin{split}
\lim_{t\to +\infty}\mathcal{K}(\psi(t))
&=
\lim_{t\to +\infty}
\mathcal{K}(e^{\frac{i}{2}t\Delta}\phi_{+})
\\[6pt]
&=\lim_{t\to +\infty}
\left\{ 
\mathcal{K}(\phi_{+})
-
\frac{d(p-1)}{2(p+1)}\left\| e^{\frac{i}{2}t\Delta}\phi_{+}\right\|_{L^{p+1}}^{p+1}
+
\frac{d(p-1)}{2(p+1)}\left\| \phi_{+} \right\|_{L^{p+1}}^{p+1}
\right\}
\\[6pt]
&=\mathcal{K}(\phi_{+})
+
\frac{d(p-1)}{2(p+1)}\left\| \phi_{+} \right\|_{L^{p+1}}^{p+1}
>0
.
\end{split}
\end{equation}
Hence, it holds that  
\begin{equation}\label{09/10/17/17:06}
0< \mathcal{K}(\psi(t)) <\mathcal{H}(\psi(t))
\quad 
\mbox{for all sufficiently large $t >T_{\delta}$}.
\end{equation}
Moreover, using (\ref{09/01/13/9:54}) and Lemma \ref{10/04/08/9:20} again, and employing the condition $\phi_{+}\in \Omega$, we obtain that 
\begin{equation}\label{10/04/08/10:59}
\begin{split}
&\lim_{t\to +\infty}\widetilde{\mathcal{N}}_{2}(\psi(t))
=
\lim_{t\to +\infty}
\widetilde{\mathcal{N}}_{2}(e^{\frac{i}{2}t\Delta}\phi_{+})
\\[6pt]
&=
\lim_{t\to +\infty}
\left\|\phi_{+} \right\|_{L^{2}}^{p+1-\frac{d}{2}(p-1)}
\sqrt{
\left\| \nabla \phi_{+} \right\|_{L^{2}}^{2}-\frac{2}{p+1}\left\| e^{\frac{i}{2}t\Delta}\phi_{+}\right\|_{L^{p+1}}^{p+1}
}^{\frac{d}{2}(p-1)-2}
\\[6pt]
&=
\left\|\phi_{+} \right\|_{L^{2}}^{p+1-\frac{d}{2}(p-1)}
\left\| \nabla \phi_{+} \right\|_{L^{2}}^{\frac{d}{2}(p-1)-2}
=\mathcal{N}_{2}(\phi_{+})<\widetilde{N}_{2}.
\end{split}
\end{equation}
This estimate, together with (\ref{08/06/15/14:38}), shows that 
\begin{equation}\label{09/10/17/17:07}
\psi(t) \in PW_{+}
\quad 
\mbox{for all sufficiently large $t >T_{\delta}$.} 
\end{equation}
Thus, we have proved the global existence of $\psi$. Since $\psi \in X([T,+\infty))$ for sufficiently large $T>0$, we also have $\psi \in X([0,+\infty))$. 
\par 
Now, we put $\psi_{0}=\psi(0)$. Then, $\psi_{0}$ is what we want. 
Indeed, the desired properties (\ref{09/10/08/18:42}) and (\ref{09/10/08/18:43}) have been obtained. Moreover, (\ref{09/10/08/18:43}) and Lemma \ref{10/04/08/9:20} immediately give (\ref{09/10/08/18:44}).
\par 
We shall prove the uniqueness here. Let $\psi_{0}$ and $\widetilde{\psi}_{0}$ be functions in $PW_{+}$ such that the solutions $\psi$ and $\widetilde{\psi}$ to the equation (\ref{08/05/13/8:50}) with $\psi(0)=\psi_{0}$ and $\widetilde{\psi}(0)=\widetilde{\psi}_{0}$ satisfy that 
\begin{equation}\label{10/04/10/22:19}
\lim_{t\to \infty}\left\| \psi(t)-e^{\frac{i}{2}t\Delta}\phi_{+}\right\|_{H^{1}}=
\lim_{t\to \infty}\left\|\widetilde{\psi}(t)-e^{\frac{i}{2}t\Delta}\phi_{+}\right\|_{H^{1}}
=0.
\end{equation}
Using (\ref{10/04/10/22:19}), we find  that 
\begin{equation}\label{10/04/10/22:33}
\lim_{t\to \infty}
\left\| \psi(t)-\widetilde{\psi}(t)\right\|_{H^{1}}
=
0.
\end{equation}
Then, supposing the contrary that $\psi_{0}\neq \widetilde{\psi}_{0}$, we have by the standard uniqueness result (see \cite{Kato1995}) that $\psi(t)\neq \widetilde{\psi}(t)$ for all $t \in \mathbb{R}$, which contradicts (\ref{10/04/10/22:33}): Thus, $\psi_{0}= \widetilde{\psi}_{0}$.
\par 
Finally, we prove the continuity of the map defined by $\phi_{+}\in \Omega \mapsto \psi_{0}\in PW_{+}$. We use $W_{+}$ to denote this map. Let $\phi_{+} \in \Omega$, and let 
 $\{ \phi_{+,n}\}_{n\in \mathbb{N}}$ be a sequence in $\Omega$ satisfying that 
\begin{equation}\label{10/04/08/13:56}
\lim_{n\to\infty}
\left\|
\phi_{+,n}- \phi_{+}
\right\|_{H^{1}}
=0.
\end{equation}
The estimate (\ref{08/10/25/23:16}), together with (\ref{10/04/08/13:56}), shows that  
\begin{equation}\label{10/04/08/22:24}
\lim_{n\to \infty}\left\| e^{\frac{i}{2}t\Delta} \phi_{+,n} 
-
e^{\frac{i}{2}t\Delta} \phi_{+}
\right\|_{X(\mathbb{R})}
=0.
\end{equation}
Let $\psi$ and $\psi_{n}$ be the solutions to (\ref{08/05/13/8:50}) with $\psi(0)=W_{+}\phi_{+}$ and $\psi_{n}(0)=W_{+}\phi_{+,n}$, respectively. Then, $\psi$ and $\psi_{n}$ exist globally in time (see Proposition \ref{09/06/21/19:28}) and have the following properties as proved above:   
\begin{align}
\label{10/04/09/12:34}
&\psi(t), \psi_{n}(t) \in PW_{+} 
\quad
\mbox{for all $t \in \mathbb{R}$},
\\[6pt]
\label{10/04/14/17:38}
&\psi, \psi_{n} \in X([0,\infty))
, 
\\[6pt]
&
\label{10/04/14/16:58}
\lim_{t\to +\infty}
\left\| \psi(t)-e^{\frac{i}{2}t\Delta}\phi_{+}\right\|_{H^{1}}
=
\lim_{t\to +\infty}
\left\| \psi_{n}(t)-e^{\frac{i}{2}t\Delta}\phi_{+,n}\right\|_{H^{1}}
=0,
\\[6pt]
&
\label{10/04/14/18:04}
\mathcal{H}(\psi(t))=\left\| \nabla \phi_{+} \right\|_{L^{2}}^{2},
\quad 
\mathcal{H}(\psi_{n}(t))=\left\| \nabla \phi_{n,+} \right\|_{L^{2}}^{2}
\quad 
\mbox{for all $t \in \mathbb{R}$}.
\end{align}
Moreover, we find by (\ref{10/04/14/16:58}) that 
$\psi$ and $\psi_{n}$ satisfy that  
\begin{align}
\label{10/04/14/17:28}
&\psi(t)
=
e^{\frac{i}{2}t\Delta}\phi_{+}
-\frac{i}{2}\int_{t}^{+\infty}
\!\!
e^{\frac{i}{2}(t-t')\Delta}\left\{ |\psi(t')|^{p-1}\psi(t')\right\}dt'
,
\\[6pt]
\label{10/04/09/11:20}
&\psi_{n}(t)
=
e^{\frac{i}{2}t\Delta}\phi_{+,n}
-\frac{i}{2}\int_{t}^{+\infty}
\!\!
e^{\frac{i}{2}(t-t')\Delta}\left\{ |\psi_{n}(t')|^{p-1}\psi_{n}(t')\right\}dt'
.
\end{align}
Note here that it follows from  (\ref{10/04/08/13:56}) and (\ref{10/04/08/22:24}) that there exists a number $n_{0} \in \mathbb{N}$ such that 
\begin{align}
\label{09/09/13/22:06}
&\left\|e^{\frac{i}{2}t\Delta} \phi_{+,n} \right\|_{X(I)} 
\le 2\left\| e^{\frac{i}{2}t\Delta}\phi_{+}\right\|_{X(I)}
\quad 
\mbox{for all interval $I$ and $n\ge n_{0}$},
\\[6pt] 
\label{10/04/08/22:41}
&\left\|\phi_{+,n} \right\|_{H^{1}}
\le 
2 \left\| \phi_{+} \right\|_{H^{1}}
\quad 
\mbox{for all $n\ge n_{0}$}.
\end{align}
We find by (\ref{09/09/13/22:06}) that: For any $\delta>0$, 
 there exists  $T_{\delta}>0$, independent of $n$, such that 
\begin{equation}
\label{09/09/13/22:14}
\left\| e^{\frac{i}{2}t\Delta}\phi_{+} \right\|_{X([T_{\delta},\infty))}
<\delta,
\qquad
\left\|e^{\frac{i}{2}}t\Delta \phi_{+,n} \right\|_{X([T_{\delta},\infty))} \le \delta
\quad 
\mbox{for all $n\ge n_{0}$}.
\end{equation}
Using (\ref{09/09/13/22:14}) and the formulas (\ref{10/04/14/17:28}) and (\ref{10/04/09/11:20}), we also find that: There exists a constant $\delta_{0}>0$ with the following property: for any $\delta \in (0,\delta_{0}]$, there exists $T_{\delta}>0$ such 
 that 
\begin{equation}\label{10/04/14/17:53}
\left\| \psi \right\|_{X([T_{\delta},\infty))}
<2\delta,
\qquad
\left\|\psi_{n} \right\|_{X([T_{\delta},\infty))} \le 2\delta
\quad 
\mbox{for all $n\ge n_{0}$}.
\end{equation}

Now, we consider an estimate for the difference $\psi_{n}-\psi$. The formulas (\ref{10/04/14/17:28}) and (\ref{10/04/09/11:20}) shows that 
\begin{equation}\label{10/05/19/12:05}
\psi_{n}(t)-\psi(t)
=
e^{\frac{i}{2}t\Delta}\left( \phi_{+,n}-\phi_{+} \right)
-\frac{i}{2}\int_{t}^{+\infty}e^{\frac{i}{2}(t-t')\Delta}
\left\{ |\psi_{n}|^{p-1}\psi_{n}-|\psi|^{p-1}\psi \right\}
(t')\,dt'.
\end{equation}
Applying the Strichartz estimate to (\ref{10/05/19/12:05}), and using (\ref{09/09/27/21:10}) and (\ref{10/04/14/17:53}), we obtain that 
\begin{equation}\label{10/04/09/11:18}
\begin{split}
&\left\| \psi_{n}-\psi \right\|_{S([T_{\delta},\infty))}
\\[6pt]
&\lesssim 
\left\| \phi_{+,n}-\phi_{+} \right\|_{L^{2}}
+ \left( \left\| \psi_{n} \right\|_{X([T_{\delta},\infty))}^{p-1}+
\left\| \psi \right\|_{X([T_{\delta},\infty))}^{p-1} \right) 
\left\| \psi_{n}-\psi \right\|_{S([T_{\delta},\infty))}
\\[6pt]
&
\le \left\| \phi_{+,n}-\phi_{+} \right\|_{L^{2}}
+ 2^{p}\delta^{p-1}
\left\| \psi_{n}-\psi \right\|_{S([T_{\delta},\infty))}
\quad 
\mbox{for all $n\ge n_{0}$},
\end{split}
\end{equation}
where the implicit constant depends only on $d$, $p$ and $q_{1}$. 
This estimate, together with (\ref{10/04/08/13:56}), shows that   
\begin{equation}\label{10/05/20/6:11}
\lim_{n\to \infty}
\left\| \psi_{n}-\psi \right\|_{S([T_{\delta},\infty))}
=0 
\quad 
\mbox{for all sufficiently small $\delta>0$}.
\end{equation} 
Similarly, we can obtain that 
\begin{equation}\label{10/07/27/14:14}
\lim_{n\to \infty}
\left\| \psi_{n}-\psi \right\|_{X([T_{\delta},\infty))}
=0 
\quad 
\mbox{for all sufficiently small $\delta>0$}.
\end{equation}
Moreover, considering the integral equations of $\partial_{j}\psi$ and $\partial_{j}\psi_{n}$ for $1\le j\le d$, we obtain by the Strichartz estimate that 
\begin{equation}\label{10/05/20/5:55}
\begin{split}
\left\| \partial_{j}\psi_{n}-\partial_{j} \psi \right\|_{S([T_{\delta}, \infty))}
&\lesssim 
\left\| \partial_{j}\phi_{+,n}-\partial_{j}\phi_{+}
\right\|_{L^{2}}
+
\left\| 
|\psi_{n}|^{p-1}(\partial_{j}\psi_{n}-\partial_{j}\psi)
\right\|_{L^{r_{0}'}([T_{\delta},\infty);L^{q_{1}'})}
\\[6pt]
&\quad + 
\left\| 
\left( |\psi_{n}|^{p-1}-|\psi|^{p-1} \right) \partial_{j}\psi 
\right\|_{L^{r_{0}'}([T_{\delta},\infty);L^{q_{1}'})}
\\[6pt]
&\quad +
\left\| 
|\psi_{n}|^{p-3}\psi_{n}^{2}\left(\partial_{j}\overline{\psi_{n}}
-\partial_{j}\overline{\psi}\right)
\right\|_{L^{r_{0}'}([T_{\delta},\infty);L^{q_{1}'})}
\\[6pt]
&\quad +
\left\| 
\left( 
|\psi_{n}|^{p-3}\psi_{n}^{2}
-|\psi|^{p-3}\psi^{2}
\right) \partial_{j}\overline{\psi}
\right\|_{L^{r_{0}'}([T_{\delta},\infty);L^{q_{1}'})}
\\[6pt]
&\lesssim 
\left\| \partial_{j}\phi_{+,n}-\partial_{j}\phi_{+}
\right\|_{L^{2}}
+
\left\| \psi_{n} \right\|_{X([T_{\delta},\infty)}^{p-1}
\left\| \partial_{j}\psi_{n}-\partial_{j}\psi \right\|
_{S([T_{\delta},\infty))}
\\[6pt]
&\quad + 
\left\| |\psi_{n}|^{p-1}-|\psi|^{p-1} 
\right\|_{L^{\frac{r_{2}}{p-1}}([T_{\delta},\infty);L^{\frac{q_{2}}{p-1}})} 
\left\| \partial_{j}\psi \right\|_{S([T_{\delta},\infty))}
\\[6pt]
&\quad 
+
\left\| \psi_{n} \right\|_{X([T_{\delta},\infty))}^{p-1}
\left\| \partial_{j}\psi_{n}
-\partial_{j}\psi\right\|_{S([T_{\delta},\infty))}
\\[6pt]
&\quad +
\left\| 
|\psi_{n}|^{p-3}\psi_{n}^{2}
-|\psi|^{p-3}\psi^{2}
\right\|_{L^{\frac{r_{2}}{p-1}}([T_{\delta},\infty);L^{\frac{q_{2}}{p-1}})}
\left\| \partial_{j}\psi 
\right\|_{S([T_{\delta},\infty))}, 
\end{split}
\end{equation}
where the implicit constant depends only on $d$, $p$ and $q_{1}$. 
Note here that, Lemma \ref{08/08/18/15:53}, together with (\ref{10/04/14/17:38}), gives us that $\|(1-\Delta)^{\frac{1}{2}}\psi\|_{S([0,\infty))}<\infty$. We also note that (\ref{10/07/27/14:14}) implies that 
\begin{equation}\label{10/05/20/7:01}
\begin{split}
&\lim_{n\to \infty}
\left\| |\psi_{n}|^{p-1}-|\psi|^{p-1} 
\right\|_{L^{\frac{r_{2}}{p-1}}([T_{\delta},\infty);L^{\frac{q_{2}}{p-1}})} 
\\[6pt]
&=
\lim_{n\to \infty}
\left\| 
|\psi_{n}|^{p-3}\psi_{n}^{2}
-|\psi|^{p-3}\psi^{2}
\right\|_{L^{\frac{r_{2}}{p-1}}([T_{\delta},\infty);L^{\frac{q_{2}}{p-1}})}
=0
\quad 
\mbox{for all sufficiently small $\delta>0$}.
\end{split}
\end{equation}
Hence, we have by (\ref{10/05/20/5:55})  that  
\begin{equation}\label{10/05/20/6:25}
\lim_{n\to \infty}\left\|\partial_{j} \psi_{n}-\partial_{j} \psi \right\|_{S([T_{\delta},\infty))}
=0
\quad 
\mbox{for all sufficiently small $\delta>0$}.
\end{equation} 
Thus, we obtain from (\ref{10/05/20/6:11}) and (\ref{10/05/20/6:25}) that 
\begin{equation}\label{10/04/09/12:36}
\lim_{n\to \infty}
\left\| \psi_{n}-\psi \right\|_{L^{\infty}([T_{\delta},\infty);H^{1})}
=0 
\quad 
\mbox{for all sufficiently small  $\delta>0$},
\end{equation} 
so that it follows from the continuous dependence of solutions on initial data   that 
\begin{equation}\label{10/04/10/16:04}
\lim_{n\to \infty}\left\| W_{+}\phi_{+,n}-W_{+}\phi_{+} \right\|_{H^{1}}
= 
\lim_{n\to \infty}
\left\| \psi_{n}(0)-\psi(0) \right\|_{H^{1}}
=0,
\end{equation}
which completes the proof. 
\end{proof}

\section{Analysis on $\boldsymbol{PW_{+}}$}
\label{08/10/03/15:11}
Our aim here is to prove Theorem \ref{08/05/26/11:53}. Obviously, Proposition \ref{09/06/21/19:28} provides (\ref{08/12/16/10:09}) and (\ref{08/09/03/17:03}). Therefore, it remains to prove the asymptotic completeness in $PW_{+}$.  
\par  
In order to prove the existence of asymptotic states for a solution $\psi$, it suffices to show that $\|\psi\|_{X(\mathbb{R})}<\infty$ by virtue of Proposition \ref{08/08/18/16:51} and Proposition \ref{09/06/21/19:28}. To this end, we introduce a subset of $PW_{+}$: 
\begin{equation}\label{10/02/19/23:05}
PW_{+}(\delta):=
\left\{ 
f \in PW_{+} \Biggm| \widetilde{\mathcal{N}}_{2}(f):=\left\| f \right\|_{L^{2}}^{p+1-\frac{d}{2}(p-1)}\sqrt{\mathcal{H}(f)}^{\frac{d}{2}(p-1)-2}
< \delta 
\right\},
\quad 
\delta>0. 
\end{equation}
Moreover, we define a number $N_{c}$ by  
\begin{equation}\label{08/09/02/18:06}
\begin{split}
\widetilde{N}_{c}
&:=
\sup{
\left\{ \delta >0 \bigm| 
\|\psi\|_{X(\mathbb{R})}< \infty
\quad 
\mbox{for all $\psi_{0} \in PW_{+}(\delta)$} 
\right\}
}
\\[6pt]
&=
\inf{\left\{ \delta >0 \bigm| 
\|\psi\|_{X(\mathbb{R})}= \infty
\quad 
\mbox{for some $\psi_{0} \in PW_{+}(\delta)$}
\right\}},
\end{split}
\end{equation}
where $\psi$ denotes the solution to (\ref{08/05/13/8:50}) with $\psi(0)=\psi_{0}$. 
\par 
Since we have
\begin{equation}\label{09/05/05/8:52}
PW_{+}=PW_{+}(\widetilde{N}_{2}),
\end{equation}
our task is to prove  $\widetilde{N}_{c}=\widetilde{N}_{2}$. 
\par 
It is worth while  noting here that $\widetilde{N}_{c}>0$: For, it follows from  (\ref{08/10/25/23:16}), the interpolation estimate and Proposition \ref{09/06/21/19:28} that \begin{equation}\label{08/07/22/17:36}
\begin{split}
\left\|e^{\frac{i}{2}t\Delta}\psi_{0} \right\|_{X(\mathbb{R})}^{p-1}
&\lesssim 
\|(-\Delta)^{\frac{s_{p}}{2}} \psi_{0}\|_{L^{2}}^{p-1}
\\[6pt]
&\le 
\|\psi_{0}\|_{L^{2}}^{p+1-\frac{d}{2}(p-1)}
\|\nabla \psi_{0}\|_{L^{2}}^{\frac{d}{2}(p-1)-2}
\\
& < \|\psi_{0} \|_{L^{2}}^{p+1-\frac{d}{2}(p-1)}
    \sqrt{\frac{d(p-1)}{d(p-1)-4} \mathcal{H}(\psi_{0})}^{\frac{d}{2}(p-1)-2}\\
&=\sqrt{\frac{d(p-1)}{d(p-1)-4}}^{\frac{d}{2}(p-1)-2}\widetilde{\mathcal{N}}_{2}(\psi_{0})
\\
&< 
\sqrt{\frac{d(p-1)}{d(p-1)-4}}^{\frac{d}{2}(p-1)-2} 
\hspace{-2pt}
\delta
\qquad 
\mbox{for all $\delta>0$ and $\psi_{0} \in PW_{+}(\delta)$},
\end{split}
\end{equation}
where the implicit constant depends only on $d$, $p$ and $q_{1}$. 
 This estimate shows that there exists $\delta_{0}>0$, depending only on $d$, $p$ and  $q_{1}$, such that 
\begin{equation}\label{10/04/04/23:01}
\left\| e^{\frac{i}{2}t\Delta} \psi_{0} \right\|_{X(\mathbb{R})}
<
\delta_{Pr.\ref{08/08/22/20:59}}
\quad 
\mbox{for all $\psi_{0} \in PW_{+}(\delta_{0})$},
\end{equation}
where $\delta_{Pr.\ref{08/08/22/20:59}}$ is a constant found in Proposition \ref{08/08/22/20:59}. Hence, we have by the small data theory (Proposition \ref{08/08/22/20:59}) that $\widetilde{N}_{2}\ge \delta_{0}>0$.

\subsection{One soliton vs. Virial identity}\label{09/05/06/9:13}
In this section, we present our strategy to prove $\widetilde{N}_{c}=\widetilde{N}_{2}$. We suppose the contrary that $\widetilde{N}_{c}<\widetilde{N}_{2}$. In this undesired situation, we can find a one-soliton-like solution to our equation (\ref{08/05/13/8:50}) in $PW_{+}$ (see Proposition \ref{08/10/16/21:12} below). Then, the soliton-like behavior contradicts the one described by the generalized virial identity (Lemma \ref{08/10/23/18:45}), so that we conclude that $\widetilde{N}_{c}=\widetilde{N}_{2}$. At the end of this Section \ref{09/05/06/9:13}, we actually show this, provided that such a soliton-like solution exists. 
\par 
The construction of a soliton-like solution is rather long. We divide it into two parts; In Section \ref{09/05/05/10:03}, one finds a candidate for the soliton, and  in Section \ref{08/09/25/17:01}, one sees that the candidate actually behaves like a one-soliton-like solution.  
\par 
We here briefly explain how to find a candidate for the one-soliton-like solution. If $\widetilde{N}_{c}<\widetilde{N}_{2}$, then we can take a sequence $\{\psi_{n}\}$ of solutions to (\ref{08/05/13/8:50}) with the property that 
\begin{equation}\label{10/04/05/17:37}
\psi_{n}(t) \in PW_{+} 
\  \mbox{for all $t\in \mathbb{R}$},
\quad 
\left\| \psi_{n} \right\|_{X(\mathbb{R})}=\infty,
\quad 
\lim_{n\to \infty}\widetilde{\mathcal{N}}_{2}(\psi_{n}(0))=\widetilde{N}_{c}
.
\end{equation}
 We consider the integral equation for $\psi_{n}$: 
\begin{equation}\label{10/04/05/17:35}
\psi_{n}(t)=e^{\frac{i}{2}t\Delta}f_{n}+\frac{i}{2}\int_{0}^{t}
e^{\frac{i}{2}(t-t')\Delta}\left\{ |\psi_{n}(t')|^{p-1}\psi_{n}(t')\right\}dt',
\end{equation}
where we put $f_{n}=\psi_{n}(0)$. We first observe that the linear part of this integral equation possibly behaves like as follows\footnote{$e^{-\eta_{n}^{l}\cdot \nabla}$ denotes the space-translation by $-\eta_{n}$. We may expect the number of summands $f^{l}$ is finite.}:  
\begin{equation}\label{09/07/15/20:35}
e^{\frac{i}{2}t\Delta}f_{n}(x) \sim \sum_{l\ge 1}e^{\frac{i}{2}(t-\tau_{n}^{l})\Delta}e^{-\eta_{n}^{l}\cdot \nabla}f^{l}(x)
\end{equation}
for some nontrivial functions $f^{l}\in PW_{+}$,  $\tau_{n}^{l} \in \mathbb{R}$ and $\eta_{n}^{l} \in \mathbb{R}^{d}$. Of course, this is not a good approximation to $\psi_{n}$. So, putting $\displaystyle{\tau_{\infty}^{l}=\lim_{n\to \infty}\tau_{n}^{l}}$ (possibly $\tau_{\infty}^{l}=\pm \infty$), we solve our equation (\ref{08/05/13/8:50}) with the initial datum $e^{-\frac{i}{2}\tau_{\infty}^{l}\Delta}f^{l}$ at $t=-\tau_{\infty}^{l}$:
\begin{equation}\label{09/09/14/15:32}
\psi^{l}(t)=
e^{\frac{i}{2}(t+\tau_{\infty}^{l})\Delta}e^{-\frac{i}{2}\tau_{\infty}^{l}\Delta}
 f^{l}
+\frac{i}{2}\int_{-\tau_{\infty}^{l}}^{t}e^{\frac{i}{2}(t-t')\Delta}\left\{ |\psi^{l}(t')|^{p-1}\psi^{l}(t')\right\}dt'.
\end{equation}
Here, in case of $\tau_{\infty}^{l}= \pm \infty$, we are regarding this as the final value problem:
\begin{equation}\label{09/09/14/15:33}
e^{-\frac{i}{2}t\Delta}\psi^{l}(t)=
f^{l}
+\frac{i}{2}\int_{\mp \infty}^{t}e^{-\frac{i}{2}t'\Delta}\left\{ |\psi^{l}(t')|^{p-1}\psi^{l}(t')\right\}dt'.
\end{equation}
Then, instead of (\ref{09/07/15/20:35}), we consider the superposition of these solutions with the space-time translations: 
\begin{equation}\label{09/09/14/17:16}
\psi_{n}^{app}(x,t) :=
\sum_{l\ge 1} (e^{-\tau_{n}^{l} \frac{\partial }{\partial t} -\eta_{n}^{l}\cdot \nabla} \psi^{l})(x,t)
=
\sum_{l\ge 1} \psi^{l}(x-\eta_{n}^{l},t-\tau_{n}^{l})
.
\end{equation}
By Lemma \ref{08/08/23/16:32} below, we will see that this formal object $\psi_{n}^{app}$ is an ``almost''  solution to our equation (\ref{08/05/13/8:50}) with the initial datum $\displaystyle{\sum_{l\ge 1}e^{-\frac{i}{2}\tau_{n}^{l}\Delta}e^{-\eta_{n}^{l}\cdot \nabla}f^{l}}$, and supposed to be a good approximation to $\psi_{n}$. In other words, a kind of superposition principle holds valid in an asymptotic sense as $n\to \infty$. By virtue of the long time perturbation theory (Proposition \ref{08/08/05/14:30}), the sum in $\psi_{n}^{app}$ consists of a finite number of solutions. Actually, as a consequence of the minimizing property of the sequence $\{\psi_{n}\}$ (\ref{10/04/05/17:37}), the summand is just one: put $\Psi:=\psi^{1}$. Then, it turns out that $\Psi$ is a one-soliton-like solution which we are looking for. In fact, we can prove:

\begin{proposition}[One-soliton-like solution in $PW_{+}$]
\label{08/10/16/21:12}
Suppose that $\widetilde{N}_{c}< \widetilde{N}_{2}$. Then, there exists a global solution $\Psi \in C(\mathbb{R}; H^{1}(\mathbb{R}^{d}))$ to the equation (\ref{08/05/13/8:50}) with the following properties: $\{\Psi(t)\}_{t\in \mathbb{R}}$ is;  
\\
{\rm (i)} a minimizer such that  
\begin{equation}\label{08/10/20/2:27}
\left\| \Psi \right\|_{X(\mathbb{R})}=\infty,
\qquad 
\widetilde{\mathcal{N}}_{2}(\Psi(t))=\widetilde{N}_{c}, 
\quad 
\Psi(t) \in PW_{+}
\quad 
\mbox{for all $t \in \mathbb{R}$}, 
\end{equation} 
{\rm (ii)} uniformly bounded in $H^{1}(\mathbb{R}^{d})$ with  
\begin{equation}\label{08/10/18/22:17}
\left\| \Psi(t)\right\|_{L^{2}}=\left\| \Psi(0) \right\|_{L^{2}}=1 
\quad 
\mbox{for all $t \in \mathbb{R}$}, 
\end{equation}
and 
\begin{equation}\label{10/04/05/18:07} 
\sup_{t \in \mathbb{R}} \left\|\nabla \Psi (t)\right\|_{L^{2}}\le 
N_{c}^{\frac{1}{\frac{d}{2}(p-1)-2}},
\end{equation}
{\rm (iii)} a family of functions with zero momentum, i.e.,    
\begin{equation}\label{08/08/18/17:42}
\Im \int_{\mathbb{R}^{d}}\overline{\Psi}(x,t)\nabla \Psi(x,t)\,dx =0
\quad 
\mbox{for all $t \in \mathbb{R}$},  
\end{equation}
{\rm (iv)} tight in $H^{1}(\mathbb{R}^{d})$ in the following sense: For any $\varepsilon>0$, there exists $R_{\varepsilon}>0$ and a continuous path $\gamma_{\varepsilon} \in C([0,\infty);\mathbb{R}^{d})$ with $\gamma_{\varepsilon}(0)=0$ such that 
\begin{equation}\label{08/10/18/22:46}
\int_{|x-\gamma_{\varepsilon}(t)|<R_{\varepsilon}}\left| \Psi(x,t)\right|^{2}
 \,dx > 1-\varepsilon 
\quad 
\mbox{for all $t \in [0,\infty)$},
\end{equation}
and 
\begin{equation}\label{08/11/01/15:41}
\int_{|x-\gamma_{\varepsilon}(t)|<R_{\varepsilon}}\left| \nabla \Psi(x,t)\right|^{2}
 \,dx > \left\| \nabla \Psi(t) \right\|_{L^{2}}^{2}-\varepsilon
\quad
\mbox{for all $t \in [0,\infty)$}. 
\end{equation}
\end{proposition}
We will give the proof of Proposition \ref{08/10/16/21:12} in Sections \ref{09/05/05/10:03} and \ref{08/09/25/17:01} later, as mentioned before. The properties (i), (ii) and (iii) are easy matters. The most important part is the tightness (iv). To prove the tightness, we introduce a concentrate function for $\Psi$ (see (\ref{09/07/28/19:58}) below) and consider a sequence  $\{\Psi(\cdot+t_{n})\}_{n\in \mathbb{N}}$ for an appropriate sequence $\{t_{n}\}$ with $t_{n} \to +\infty$. Analogous arguments with  finding $\Psi$ work on $\{\Psi(\cdot +t_{n})\}_{n\in \mathbb{N}}$ as well, so that we can show the tightness. Once we get the tightness, Lemma \ref{08/10/03/9:46} immediately gives us the continuous path $\gamma_{\varepsilon}$ described in (iv) of Proposition \ref{08/10/16/21:12}. In order to prove $\widetilde{N}_{c}=\widetilde{N}_{2}$, we need to know more subtle     behavior of the path however. For a sufficiently long time, we can take $\gamma_{\varepsilon}$ as the almost center of mass, say $\gamma_{\varepsilon}^{ac}$:

\begin{lemma}[Almost center of mass]
\label{08/09/04/12:16}
Let $\Psi$ be a global solution to the equation (\ref{08/05/13/8:50}) satisfying the properties (\ref{08/10/20/2:27}), (\ref{08/10/18/22:17}), (\ref{10/04/05/18:07}), (\ref{08/08/18/17:42}), and (\ref{08/10/18/22:46}), (\ref{08/11/01/15:41}). Let $R_{\varepsilon}$ be a radius found in {\rm (iv)} of Proposition \ref{08/10/16/21:12} for $\varepsilon>0$. We define an ``almost center of mass'' by
\begin{equation}\label{08/10/19/16:06}
\gamma_{\varepsilon,R}^{ac}(t)
:=(\vec{w}_{20R}, |\Psi(t)|^{2})
\quad  
\mbox{for all $\varepsilon \in (0,\frac{1}{100})$ and $R>R_{\varepsilon}$},
\end{equation}
where $\vec{w}_{R}$ is the function defined by (\ref{10/02/27/22:18}). Then, we have: 
\begin{equation}\label{09/09/26/17:21}
\gamma_{\varepsilon,R}^{ac} \in C^{1}([0,\infty) ;\mathbb{R}^{d}),
\end{equation}
and there exists a constant $\alpha>0$, depending only on $d$ and $p$, such 
that 
\begin{align}
\label{08/09/04/16:51}
&\left|\gamma_{\varepsilon,R}^{ac}(t)\right| \le 
20R
\quad 
\mbox{for all $t \in \left[0,\alpha \frac{R}{ \sqrt{\varepsilon}} \right]$
},
\\[6pt]
\label{08/10/19/16:25}
&\int_{|x-\gamma_{\varepsilon,R}^{ac}(t)|\le 4R}\left| \Psi(x,t)\right|^{2}+\left| \nabla \Psi(x,t)\right|^{2}\,dx \ge \left\| \Psi(t) \right\|_{H^{1}}^{2}-\varepsilon 
\quad 
\mbox{for all 
$t \in 
\left[
0,
\alpha \frac{R}{ \sqrt{\varepsilon}} 
\right]$
}.
\end{align} 
\end{lemma}
\begin{remark}\label{10/04/11/22:59}
In the proof below, we find that the following estimate holds (see (\ref{08/10/19/18:10})):  
\begin{equation*}\label{08/10/19/16:24}
\left| \frac{d\gamma_{\varepsilon,R}^{ac}}{dt}(t)\right|
\lesssim \sqrt{\varepsilon } 
\quad 
\mbox{for all 
$t \in \left[0, 
\alpha \frac{ R}{\sqrt{  \varepsilon } }
\right]$}
,
\end{equation*}
where the implicit constant depends only on $d$ and $q$.
\end{remark}

\begin{proof}[Proof of Lemma \ref{08/09/04/12:16}]
We have from Proposition B.1 in \cite{Nawa8} that  
\begin{equation}\label{08/10/19/16:54}
\gamma_{\varepsilon,R}^{ac}(t)
=(\vec{w}_{20R},|\Psi(0)|^{2})
+
\left( 
2\Im \int_{0}^{t}
\!\!
\int_{\mathbb{R}^{d}}
\nabla \vec{w}_{20R}^{j}(x)\!\cdot \!\nabla \Psi(x,s)
\overline{\Psi(x,s)}\,dx\,ds
\right)_{j=1,\ldots, d}.
\end{equation}
This formula, with the help of (\ref{08/10/18/22:17}), (\ref{10/04/05/18:07}) and (\ref{08/10/23/18:35}), immediately shows (\ref{09/09/26/17:21}): $\gamma_{\varepsilon,R}^{ac}\in C^{1}([0,\infty);\mathbb{R}^{d})$. 
\par 
Next, we shall prove the properties (\ref{08/09/04/16:51}) and (\ref{08/10/19/16:25}). Let $\gamma_{\varepsilon}$ be a path found in Proposition \ref{08/10/16/21:12}, and $t_{\varepsilon}$ be the first time such that the size of $\gamma_{\varepsilon}$ reaches $10R$, i.e.,  
\begin{equation}\label{10/04/05/18:42}
t_{\varepsilon} := \inf\left\{ t\ge 0 \bigm| \left| \gamma_{\varepsilon}(t) \right|=10R \right\}.
\end{equation}
Since $\gamma_{\varepsilon}\in C([0,\infty);\mathbb{R}^{d})$ with $\gamma_{\varepsilon}(0)=0$, we have that $t_{\varepsilon}>0$ and  
\begin{equation}\label{08/10/19/22:02}
\left| \gamma_{\varepsilon}(t) \right| \le 10R
 \quad 
\mbox{for all $t \in [0,t_{\varepsilon}]$}.
\end{equation} 
We claim that 
\begin{equation}\label{08/10/19/23:22}
|\gamma_{\varepsilon,R}^{ac}(t) -\gamma_{\varepsilon}(t)|< 2R
\quad 
\mbox{for all $t \in [0,t_{\varepsilon}]$}.
\end{equation} 
It follows from the property (\ref{08/10/18/22:17}) that 
\begin{equation}\label{08/10/19/22:46}
\begin{split}
\left| \gamma_{\varepsilon,R}^{ac}(t) -\gamma_{\varepsilon}(t) \right| 
&=
\left| (\vec{w}_{20R}, |\Psi(t)|^{2} )- \gamma_{\varepsilon}(t)\|\Psi(t)\|_{L^{2}}^{2} \right| 
\\[6pt]
&=
\left| 
\int_{\mathbb{R}^{d}}
\left\{ 
\vec{w}_{20R}(x) - \gamma_{\varepsilon}(t)
\right\}\left| \Psi(x,t)\right|^{2}\,dx
\right|
\\[6pt]
&\le 
\int_{|x-\gamma_{\varepsilon}(t)|\le R}
\left| 
x - \gamma_{\varepsilon}(t)
\right| 
\left| \Psi(x,t)\right|^{2}\,dx 
\\[6pt]
& \qquad \qquad 
+
\int_{|x-\gamma_{\varepsilon}(t)|\ge R}
\left| 
\vec{w}_{20R} - \gamma_{\varepsilon}(t)
\right|
\left| \Psi(x,t) \right|^{2}\,dx.
\end{split}
\end{equation}
Moreover, applying (\ref{08/10/19/22:02}) and (\ref{08/10/23/18:42}) to the second term on the right-hand side above, we obtain that 
\begin{equation}\label{10/04/11/14:39}
\left| \gamma_{\varepsilon, R}^{ac}(t) 
-\gamma_{\varepsilon}(t) 
\right|
\le 
R \|\Psi(t)\|_{L^{2}}^{2} 
+ 
50R \int_{ |x-\gamma_{\varepsilon}(t)|\ge R}
 \left| \Psi(x,t) \right|^{2}
\,dx
\quad
\mbox{for all $t \in [0,t_{\varepsilon}]$}.
\end{equation}
Hence, this inequality (\ref{10/04/11/14:39}), together with (\ref{08/10/18/22:17}) and the tightness (\ref{08/10/18/22:46}), yields that    
\begin{equation}\label{10/04/11/14:43}
\left| \gamma_{\varepsilon,R}^{ac}(t) -\gamma_{\varepsilon}(t) \right|
\le R+50R \varepsilon < 2R
\quad 
\mbox{for all $\varepsilon < \frac{1}{100}$ and $t \in [0,t_{\varepsilon}]$}.
\end{equation}

Now, we have by (\ref{08/10/19/22:02}) and (\ref{08/10/19/23:22}) that 
\begin{equation}\label{09/09/26/18:10}
\left| \gamma_{\varepsilon,R}^{ac}(t)\right| \le 12R
\quad 
\mbox{for all $t \in [0,t_{\varepsilon}]$}.
\end{equation} 
Moreover, (\ref{08/10/19/23:22}) also gives us that 
\begin{equation}\label{10/04/11/15:16}
B_{R}(\gamma_{\varepsilon}(t)) 
\subset 
B_{4R}(\gamma_{\varepsilon}^{ac}(t))
\quad 
\mbox{for all $t \in [0,t_{\varepsilon}]$},
\end{equation}
so that the tightness of $\{\Psi(t)\}$ in $H^{1}(\mathbb{R}^{d})$ (see (\ref{08/10/18/22:46}) and (\ref{08/11/01/15:41})) gives us that  
\begin{equation}\label{08/10/19/22:50}
\int_{|x-\gamma_{\varepsilon,R}^{ac}(t)|\le 4R} \left| \Psi(x,t) \right|^{2}
+
\left| \nabla \Psi(t) \right|^{2} \,dx 
\ge 
\left\| \Psi(t)\right\|_{H^{1}}^{2}-\varepsilon
\quad 
\mbox{for all $t \in [0,t_{\varepsilon}]$}.
\end{equation}
Therefore, for the desired results (\ref{08/09/04/16:51}) and (\ref{08/10/19/16:25}), it suffices to show that there exists a constant $\alpha>0$, depending only on $d$ and $p$, such that  
\begin{equation}\label{09/09/26/18:21}
\alpha \frac{R}{\sqrt{\varepsilon}}\le t_{\varepsilon}.
\end{equation}
To this end, we prove that 
\begin{equation}\label{08/10/19/18:10}
\left| \frac{d \gamma_{\varepsilon,R}^{ac}}{dt}(t)\right|
\lesssim \sqrt{\varepsilon } 
\quad 
\mbox{for all $t \in [0, t_{\varepsilon}]$},
\end{equation}
where the implicit constant depends only on $d$ and $p$. 
\par 
Before proving (\ref{08/10/19/18:10}), we describe how it yields (\ref{09/09/26/18:21}). We easily verify by (\ref{08/10/19/18:10}) that  
\begin{equation}\label{10/04/11/15:36}
|\gamma_{\varepsilon,R}^{ac}(t_{\varepsilon})|-|\gamma_{\varepsilon,R}^{ac}(0)|
 \le 
\int_{0}^{t_{\varepsilon}}\left| \frac{d\gamma_{\varepsilon,R}^{ac}}{dt}(t)\right|\,dt
\lesssim \sqrt{\varepsilon }t_{\varepsilon},
\end{equation}
where the implicit constant depends only on $d$ and $p$. Then, it follows from (\ref{08/10/19/23:22}) and $\gamma_{\varepsilon}(0)=0$ that   
\begin{equation}\label{10/04/11/15:38}
\begin{split}
\sqrt{\varepsilon} t_{\varepsilon} 
&\gtrsim  
|\gamma_{\varepsilon,R}^{ac}(t_{\varepsilon})|-2R
\\[6pt]
& \ge 
|\gamma_{\varepsilon}(t_{\varepsilon})|-|\gamma_{\varepsilon,R}^{ac}(t_{\varepsilon})-\gamma_{\varepsilon}(t_{\varepsilon})|-2R
\\[6pt]
& \ge 10R -2R -2R=6R,
\end{split}
\end{equation}
which gives (\ref{09/09/26/18:21}).
\par 
Finally, we prove (\ref{08/10/19/18:10}). Using the formula (\ref{08/10/19/16:54}) and the property (\ref{08/08/18/17:42}), we obtain that 
\begin{equation}\label{10/04/11/15:41}
\begin{split}
\left| \frac{d \gamma_{\varepsilon,R}^{ac}}{dt}(t)\right|^{2}
&=\sum_{j=1}^{d}
\left| 2\Im (\nabla \vec{w}_{20R}^{j}\cdot\nabla \Psi(t),\Psi(t)) \right|^{2}
\\[6pt]
&\le 4\sum_{j=1}^{d}
\left\|\nabla \vec{w}_{20R}^{j} \right\|_{L^{\infty}}^{2}
\left\|\Psi(t)\right\|_{L^{2}}^{2}
 \int_{|x|\ge 20R} \left|\nabla \Psi(x,t)\right|^{2}\,dx
\quad 
\mbox{for all $t\ge 0$}.
\end{split}
\end{equation}
Applying (\ref{08/10/18/22:17})  and (\ref{08/10/23/18:35}) to the right-hand side above, we further obtain that  
\begin{equation}\label{09/09/26/18:55}
\left| \frac{d \gamma_{\varepsilon,R}^{ac}}{dt}(t)\right|^{2}
\lesssim  \int_{|x|\ge 20R}  
\left| \nabla \Psi(x,t)\right|^{2}\,dx
\quad 
\mbox{for all $t\ge 0$}, 
\end{equation}
where the implicit constant depends only on $d$ and $p$. Since the estimate (\ref{08/10/19/22:02}) shows that 
\begin{equation}\label{10/04/11/15:48}
B_{R}(\gamma_{\varepsilon}(t)) 
\subset B_{20R}(0)
\quad 
\mbox{for all $t \in [0,t_{\varepsilon}]$},
\end{equation}
the estimate (\ref{09/09/26/18:55}), together with the tightness (\ref{08/11/01/15:41}), leads to (\ref{08/10/19/18:10}). 
\end{proof} 

Lemma \ref{08/09/04/12:16} implies that $\Psi$ found in Proposition \ref{08/10/16/21:12} is in a bound motion, rather, a 
 standing wave. On the other hand, the generalized virial identity ((\ref{08/03/29/19:05}) in Lemma \ref{08/10/23/18:45}) suggests that $\Psi$ is in a scattering motion. As we already mentioned at the beginning of this Section \ref{09/05/06/9:13}, these two facts contradict each other; Thus, we see that $\widetilde{N}_{2}=\widetilde{N}_{c}$. Here, we show this point precisely: 
\par 
The generalized virial identity (\ref{08/03/29/19:05}), together with 
(\ref{08/12/29/14:16}) and (\ref{08/04/20/12:50}), yields that 
\begin{equation}\label{10/04/11/16:13}
\begin{split}
&(W_{R},|\Psi(t)|^{2}) 
\\[6pt]
&\ge 
(W_{R},|\Psi(0)|^{2})+2t\Im{(\vec{w}_{R}\cdot \nabla \Psi(0),\Psi(0))}
+2\int_{0}^{t}\int_{0}^{t'}\mathcal{K}(\Psi(t''))\,dt''dt'
\\[6pt]
&\qquad 
-2 \int_{0}^{t}\int_{0}^{t'} \int_{|x|\ge R}
\rho_{1}(x)|\nabla \Psi(x,t'')|^{2} + \rho_{2}(x) \left|\frac{x}{|x|}\cdot \nabla \Psi(x,t'')\right|^{2}dx dt''dt'
\\[6pt]
&\qquad 
-\frac{1}{2}\int_{0}^{t}\int_{0}^{t'}
\left\| \Delta ({\rm div}\, \vec{w}_{R})
\right\|_{L^{\infty}} \left\|\Psi(t'')\right\|_{L^{2}}^{2} dt''dt'
\qquad 
\mbox{for all $R>0$}.
\end{split}
\end{equation}
Applying the estimates (\ref{08/10/18/22:17}), (\ref{08/04/20/16:54}) and (\ref{08/12/29/12:53}) to the right-hand side above, we obtain that 
\begin{equation}\label{10/04/11/16:31}
\begin{split}
&(W_{R},|\Psi(t)|^{2}) 
\\[6pt]
&\ge 
(W_{R},|\Psi(0)|^{2})+2t\Im{(\vec{w}_{R}\cdot \nabla \Psi(0),\Psi(0))}
+2\int_{0}^{t}\int_{0}^{t'}\mathcal{K}(\Psi(t''))\,dt''dt'
\\[6pt]
& \quad 
-2 \left( K_{1}+K_{2} \right) 
\int_{0}^{t}\int_{0}^{t'} \int_{|x|\ge R}
|\nabla \Psi(x,t'')|^{2} \,dx dt''dt'
-\frac{10d^{2}K}{R^{2}}
t^{2}
\quad 
\mbox{for all $R>0$},
\end{split}
\end{equation}
where $K$ is the constant defined in (\ref{10/02/27/22:14}), and $K_{1}$ and $K_{2}$ are the constants found in Lemma \ref{08/04/20/0:27}. 
Moreover, it follows from the estimate (\ref{09/06/21/19:30}) in Proposition \ref{09/06/21/19:28} that 
\begin{equation}\label{08/09/04/12:09}
\begin{split}
&(W_{R},|\Psi(t)|^{2}) 
\\[6pt] 
&\ge (W_{R},|\Psi(0)|^{2})+2t\Im{(\vec{w}_{R}\cdot \nabla \Psi(0),\Psi(0))}
+t^{2}\omega_{0}\mathcal{H}(\Psi(0))
\\[6pt]
& \quad 
-2(K_{1}+K_{2}) \int_{0}^{t}\int_{0}^{t'} 
\int_{|x|\ge R}
|\nabla \Psi(x,t'')|^{2} \, dx dt''dt'
-\frac{10d^{2}K}{R^{2}}t^{2}
\quad 
\mbox{for all $R>0$},
\end{split}
\end{equation}
where we put
\begin{equation}\label{10/04/08/23:11}
\omega_{0} :=1-\frac{\widetilde{\mathcal{N}}_{2}(\Psi(0))}{\widetilde{N}_{2}}.
\end{equation}
Here, we have by Lemma \ref{08/09/04/12:16} that: For any $\varepsilon \in (0,\frac{1}{100})$, there exists $R_{\varepsilon}>0$ with the following property:
 for any $R\ge R_{\varepsilon}$, there exists $\gamma_{\varepsilon,R}^{ac} \in C^{1}([0,\infty);\mathbb{R}^{d})$ such that
\begin{align}\label{09/03/05/18:10}
&\left| \gamma_{\varepsilon,R}^{ac}(t) \right| \le 20R
\quad 
\mbox{for all $t \in \bigm[0,\, \alpha\frac{R}{\sqrt{\varepsilon }} \bigm]$},
\\[6pt]
\label{08/10/20/2:59}
&\int_{|x-\gamma_{\varepsilon,R}^{ac}(t)|\ge 4R} 
|\nabla \Psi(x,t)|^{2}dx
< \varepsilon 
\quad 
\mbox{for all $t \in \bigm[0,\, \alpha\frac{R}{\sqrt{\varepsilon }} \bigm]$},
\end{align}
where $\alpha$ is some constant depending only on $d$ and $p$. 
\par 
We employ (\ref{09/03/05/18:10}) to obtain that 
\begin{equation}\label{10/04/11/18:04}
|x-\gamma_{\varepsilon,R}^{ac}(t)|\ge 4R
\quad 
\mbox{for all $R\ge R_{\varepsilon}$, $t \in \bigm[0,\, \alpha\frac{R}{\sqrt{\varepsilon }}\bigm]$ and $x \in \mathbb{R}^{d}$ with $|x|\ge 24R$}.
\end{equation} 
Hence, (\ref{08/09/04/12:09}), together with the tightness (\ref{08/10/20/2:59}), leads to 
that 
\begin{equation}\label{10/04/11/21:21}
\begin{split}
(W_{50R},|\Psi(t)|^{2}) 
&\ge (W_{50R},|\Psi(0)|^{2})+2t\Im{(\vec{w}_{50R}\cdot \nabla \Psi(0),\Psi(0))}
+t^{2}\omega_{0}\mathcal{H}(\Psi(0))
\\[6pt]
& \quad 
-(K_{1}+K_{2})t^{2}\varepsilon 
-\frac{10d^{2}K}{(50R)^{2}}t^{2}
\quad 
\mbox{for all $R\ge R_{\varepsilon}$ and $t \in \bigm[0,\, \alpha\frac{R}{\sqrt{\varepsilon }} \bigm]$
}.
\end{split}
\end{equation}
We choose $\varepsilon$ so small that 
\begin{equation}
\label{10/04/11/21:29}
0<
\varepsilon < \min 
\left\{ \frac{1}{100},
\ 
\frac{\omega_{0}}{4(K_{1}+K_{2})}\mathcal{H}(\Psi(0))
\right\},
\end{equation}
and $R$ so large that 
\begin{equation}\label{08/12/30/14:42}
R 
\ge   
\max\left\{ R_{\varepsilon}, 
\ 
\frac{d\sqrt{K}}{\sqrt{\omega_{0}\mathcal{H}(\Psi(0))}}
\right\}.
\end{equation}
Then, it follows from (\ref{10/04/11/21:21}) that 
\begin{equation}\label{08/09/04/13:45}
\begin{split}
(W_{50R},|\Psi(t)|^{2})
\ge 
(W_{50R},|\Psi(0)|^{2})+2t\Im{(\vec{w}_{50R}\cdot \nabla \Psi(0),\Psi(0))}
+
\frac{t^{2}}{2}\omega_{0}\mathcal{H}(\Psi(0))&
\\[6pt]
\mbox{for all $t \in \bigm[0,\, \alpha\frac{R}{\sqrt{\varepsilon }} \bigm]$}&.
\end{split}
\end{equation}
Dividing the both sides of (\ref{08/09/04/13:45}) by $t^{2}$ and applying the estimates (\ref{08/10/18/22:17}) and (\ref{08/12/29/12:48}), we obtain that  
\begin{equation}\label{08/12/30/14:28}
\begin{split}
\frac{8(50R)^{2}}{t^{2}}
\ge 
\frac{1}{t^{2}}(W_{50R},|\Psi(0)|^{2})
+
\frac{2}{t}\Im{(\vec{w}_{50R}\cdot \nabla \Psi(0),\Psi(0))}
+
\frac{\omega_{0}}{2}\mathcal{H}(\Psi(0))&
\\[6pt]
\mbox{for all $t \in \bigm[0,\, \alpha\frac{R}{\sqrt{\varepsilon }} \bigm]$}&.
\end{split}
\end{equation}
In particular, when $t= \alpha \frac{R}{\sqrt{\varepsilon}}$, 
we have  
by (\ref{08/10/18/22:17}), (\ref{08/10/23/18:42}) and (\ref{08/12/29/12:48}) that 
\begin{equation}\label{08/09/04/17:06}
\begin{split}
\frac{8(50)^{2}\varepsilon}{\alpha^{2}}
&\ge   
\frac{\varepsilon}{\alpha^{2}R^{2}}
(W_{50R},|\Psi(0)|^{2})
+
\frac{2\sqrt{\varepsilon}}{\alpha R}
\Im{ (\vec{w}_{50R}\cdot \nabla \Psi(0),\Psi(0))}
+
\frac{\omega_{0}}{2}\mathcal{H}(\Psi(0))
\\[6pt]
&\ge 
-\frac{8(50)^{2}\varepsilon}{\alpha^{2}}
-\frac{200 \sqrt{\varepsilon}}{\alpha }
\left\| \nabla \Psi(0) \right\|_{L^{2}}
+\frac{\omega_{0}}{2} \mathcal{H}(\Psi(0)),
\end{split}
\end{equation}
so that  
\begin{equation}\label{10/04/11/23:32}
\frac{8(50)^{2}\varepsilon}{\alpha^{2}}
+
\frac{8(50)^{2}\varepsilon}{\alpha^{2}}
+\frac{200 \sqrt{\varepsilon}}{\alpha }
\left\| \nabla \Psi(0) \right\|_{L^{2}}
\ge 
\frac{\omega_{0}}{2} \mathcal{H}(\Psi(0)).
\end{equation}
However, taking $\varepsilon \to 0$ in (\ref{10/04/11/23:32}), we obtain a contradiction. This absurd conclusion comes from the existence of one-soliton-like solution $\Psi$ (see Proposition \ref{08/10/16/21:12}). Thus, it must hold that $\widetilde{N}_{c}=\widetilde{N}_{2}$, provided that Proposition \ref{08/09/04/12:16} is valid.

\subsection{Solving the variational problem for $\widetilde{N}_{c}$}
\label{09/05/05/10:03}
In this section, we construct a candidate for the one-soliton-like solution, considering the variational problem for $\widetilde{N}_{c}$. 
\par 
Supposing that $\widetilde{N}_{c}<\widetilde{N}_{2}$, we can take a minimizing sequence $\{\delta_{n}\}_{n \in \mathbb{N}}$ such that 
\begin{equation}\label{10/04/15/19:55}
\widetilde{N}_{c}<\delta_{n}<\widetilde{N}_{2}
\quad 
\mbox{for all $n \in \mathbb{N}$},
\qquad 
\lim_{n\to \infty}\delta_{n}=\widetilde{N}_{c}.
\end{equation}
Moreover, Lemma \ref{08/09/01/23:45} enables us to take a sequence 
$\{\psi_{0,n}\}_{n\in \mathbb{N}}$ in $PW_{+}$ such that  
\begin{align}\label{09/05/05/10:26}
&\left\|\psi_{0,n}\right\|_{L^{2}}=1
\quad 
\mbox{for all $n \in \mathbb{N}$},
\\[6pt]
\label{09/05/05/10:25}
&\widetilde{N}_{c} < \widetilde{\mathcal{N}}_{2}(\psi_{0,n}) <\delta_{n}
\quad
\mbox{for all $n\in \mathbb{N}$}.
\end{align}
Note that (\ref{09/05/05/10:25}), together with (\ref{10/04/15/19:55}), leads to that  
\begin{equation}\label{10/04/15/20:14}
\lim_{n\to \infty}\widetilde{\mathcal{N}}_{2}(\psi_{0,n})=\widetilde{N}_{c}.
\end{equation}
Let $\psi_{n}$ be the solution to (\ref{08/05/13/8:50}) with $\psi_{n}(0)=\psi_{0,n}$. Then, (\ref{09/05/05/10:25}), together with the definition of $\widetilde{N}_{c}$ (see (\ref{08/09/02/18:06})), implies that  
\begin{equation}\label{09/05/05/10:36}
\left\|\psi_{n} \right\|_{X(\mathbb{R})}=\infty 
\quad 
\mbox{for all $n \in \mathbb{N}$}.
\end{equation}
We also find that  
\begin{equation}\label{10/04/15/20:26}
\limsup_{n\to \infty}\left\| e^{\frac{i}{2}t\Delta}\psi_{0,n} \right\|_{L^{\infty}(\mathbb{R};L^{\frac{d}{2}(p-1)})}>0.
\end{equation}
Indeed, if (\ref{10/04/15/20:26}) fails, then the small data theory (Proposition \ref{08/08/22/20:59}) concludes that 
\[
\left\| \psi_{n} \right\|_{X(\mathbb{R})}< \infty
\quad 
\mbox{for sufficiently large $n\in \mathbb{N}$},
\]
which contradicts (\ref{09/05/05/10:36}). 
\par 
The following lemma gives us a candidate for the one-soliton-like solution in 
 Proposition \ref{08/10/16/21:12}. 
 
\begin{lemma}
\label{08/08/19/23:07}
Assume that $d\ge 1$ and $2+\frac{4}{d}< p+1 <2^{*}$. Suppose that $\widetilde{N}_{c}<\widetilde{N}_{2}$. Let $\{ \psi_{n} \}$ be a sequence of global solutions to the equation (\ref{08/05/13/8:50}) in $C(\mathbb{R};H^{1}(\mathbb{R}^{d}))$ such that   
\begin{align}
\label{09/05/03/12:29}
&\psi_{n}(t) \in PW_{+}
\quad 
\mbox{for all $t \in \mathbb{R}$ and $n\in \mathbb{N}$},
\\[6pt]
\label{10/05/13/10:06}
&\left\|\psi_{n}(t) \right\|_{L^{2}}=1
\quad 
\mbox{for all $n \in \mathbb{N}$ and $t \in \mathbb{R}$}, 
\\[6pt]
\label{09/02/02/23:40}
&\sup_{n \in \mathbb{N}}\left\| \psi_{n}(0) \right\|_{H^{1}}<\infty
, 
\\[6pt]
\label{08/08/26/15:56}
&\lim_{n\to \infty}\widetilde{\mathcal{N}}_{2}(\psi_{n}(t))= 
 \widetilde{N}_{c}
 \quad 
 \mbox{for all $t \in \mathbb{R}$}, 
\\[6pt]
\label{09/02/17/14:53}
&\left\|\psi_{n} \right\|_{X(\mathbb{R})}=\infty 
\quad 
\mbox{for all $n\in \mathbb{N}$}.
\end{align}
Furthermore, we suppose that 
\begin{equation}\label{08/08/21/17:39}
\limsup_{n\to \infty}\left\| e^{\frac{i}{2}t\Delta}\psi_{n}(0) \right\|_{L^{\infty}(\mathbb{R};L^{\frac{d}{2}(p-1)})}>0.
\end{equation}
Then, there exists a subsequence of $\{\psi_{n}\}$ (still denoted by the same symbol), which satisfies the following property: There exist 
\\
{\rm (i)} a nontrivial global solution $\Psi \in C(\mathbb{R};H^{1}(\mathbb{R}^{d}))$ to the equation (\ref{08/05/13/8:50}) with 
\begin{align}
\label{09/05/02/10:11}
&\left\|\Psi \right\|_{X(\mathbb{R})}=\infty ,
\\[6pt]
\label{09/05/02/10:12}
&\Psi(t) \in PW_{+} 
\quad 
\mbox{for all $t \in \mathbb{R}$},  
\\[6pt]
\label{10/05/13/10:09}
&\left\| \Psi(t) \right\|_{L^{2}}=1
\quad 
\mbox{for all $t \in \mathbb{R}$}, 
\\[6pt]
\label{09/05/02/10:13}
&\widetilde{\mathcal{N}}_{2}(\Psi(t))=\widetilde{N}_{c}
\quad 
\mbox{for all $t \in \mathbb{R}$}, 
\end{align}
and 
\\
{\rm (ii)}  a nontrivial  function $f \in PW_{+}$, a sequence $\{\tau_{n}\}$ in $\mathbb{R}$ with $\displaystyle{\lim_{n\to \infty}\tau_{n}= \tau_{\infty}}$ for some $\tau_{\infty} \in \mathbb{R}\cup \{\pm \infty\}$, and a sequence $\{\eta_{n}\}$ in $\mathbb{R}^{d}$ such that  
\begin{align}
\label{09/05/02/10:10}
& \lim_{n\to \infty}e^{\frac{i}{2}\tau_{n}\Delta}e^{\eta_{n}\cdot \nabla }\psi_{n}(0) =  f \quad \mbox{ weakly in $H^{1}(\mathbb{R}^{d})$, and a.e. in $\mathbb{R}^{d}$},
\\[6pt]
\label{09/07/27/1:49}
&\lim_{n\to \infty}\left\| e^{\frac{i}{2}\tau_{n}\Delta}
e^{\eta_{n}\cdot \nabla }\psi_{n}(0) -  f \right\|_{L^{2}(\mathbb{R}^{d})}=0,
\\[6pt]
\label{09/05/02/10:08}
&\lim_{n\to \infty}\left\| e^{\frac{i}{2}(t+\tau_{n})\Delta} 
e^{\eta_{n}\cdot \nabla } \psi_{n}(0)-e^{\frac{i}{2}t\Delta} f 
\right\|_{L^{\infty}(\mathbb{R};L^{\frac{d}{2}(p-1)})\cap X(\mathbb{R})}=0,
\\[6pt]
\label{09/05/02/10:09}
&\lim_{n\to \infty}\left\|\Psi(-\tau_{n})-e^{-\frac{i}{2}\tau_{n}\Delta}f \right\|_{H^{1}}=0 ,
\\[6pt]
\label{09/05/02/10:07}
& \lim_{n\to \infty} \left\|e^{\frac{i}{2}t\Delta}\psi_{n}(0)
 -e^{\frac{i}{2}t\Delta}\left(e^{-\tau_{n}\frac{\partial }{\partial t}-\eta_{n}\cdot \nabla}\Psi\right) (0) 
\right\|_{X(\mathbb{R})}=0.
\end{align}
Especially, we have
\begin{equation}\label{10/04/15/20:43}
\left\| f \right\|_{L^{2}}=\left\| \Psi(t) \right\|_{L^{2}}
\quad 
\mbox{for all $t \in \mathbb{R}$}, 
\qquad 
\left\|\nabla f  \right\|_{L^{2}}=\lim_{n\to \infty}\left\| \nabla \Psi(-\tau_{n})\right\|_{L^{2}}.
\end{equation}
\end{lemma}
\begin{remark}
In this Lemma \ref{08/08/19/23:07}, $\Psi$ is actually a solution to the integral equation (\ref{09/09/14/15:32}).
\end{remark}
For the proof of Lemma \ref{08/08/19/23:07},  we prepare the following Lemmata \ref{08/08/21/14:12} and \ref{08/08/23/16:32}. The former is a compactness lemma  and 
 the latter is a key ingredient to prove the superposition principle (\ref{09/09/14/17:16}) as mentioned in Section \ref{09/05/06/9:13}.

\begin{lemma}\label{08/08/21/14:12}
Assume that $d\ge 1$ and $2+\frac{4}{d}<p+1<2^{*}$. Let $\{ f_{n} \}$ be a uniformly bounded sequence in $H^{1}(\mathbb{R}^{d})$. Suppose that   
\begin{equation}\label{09/06/28/16:04}
\limsup_{n\to \infty}\left\| e^{\frac{i}{2}t\Delta}f_{n} \right\|_{L^{\infty}(\mathbb{R};L^{\frac{d}{2}(p-1)})}>0. 
\end{equation}
Then, there exists a subsequence of $\{f_{n}\}$ (still denoted by the same symbol), which satisfies the following property: There exist a nontrivial function $f \in H^{1}(\mathbb{R}^{d})$,  a sequence $\{t_{n}\}$ in $\mathbb{R}$ with $\displaystyle{\lim_{n\to \infty}t_{n}= t_{\infty}}$ for some $t_{\infty}\in \mathbb{R} \cup \{\pm \infty\}$, and a sequence $\{y_{n}\}$ in $\mathbb{R}^{d}$ such that, putting 
$
f_{n}^{*}(x):=
e^{\frac{i}{2}t_{n}\Delta}f_{n}(x+y_{n})
$, 
we have that 
\begin{align}
\label{08/08/21/14:23}
&\lim_{n\to \infty}f_{n}^{*}= f 
\quad  
\mbox{weakly in $H^{1}(\mathbb{R}^{d})$, and strongly in $L_{loc}^{q}(\mathbb{R}^{d})$ for all $q\in [2,2^{*})$},
\\[6pt]
\label{08/08/21/17:25}
&\lim_{n\to \infty} \left\{ \left\| |\nabla|^{s} f_{n} \right\|_{L^{2}}^{2}-
\left\| |\nabla |^{s} (f_{n}^{*}-f) \right\|_{L^{2}}^{2}\right\}
=\left\| |\nabla |^{s}f\right\|_{L^{2}}^{2}
\quad 
\mbox{for all $s \in [0,1]$},
\\[6pt]
\label{08/08/21/14:35}
&\lim_{n\to \infty} \left\{ \left\| f_{n} \right\|_{L^{q}}^{q}
-\left\|e^{-\frac{i}{2}t_{n}\Delta}(f_{n}^{*}-f) \right\|_{L^{q}}^{q}- \left\| e^{-\frac{i}{2}t_{n}\Delta}f \right\|_{L^{q}}^{q}\right\}=0
\quad 
\mbox{for all $q\in [2,2^{*})$},  
\\[6pt]
\label{08/08/21/14:36}
&\lim_{n\to \infty}\left\{ \mathcal{H}\left(f_{n} \right)-\mathcal{H}\left(e^{-\frac{i}{2}t_{n}\Delta}(f_{n}^{*}-f) \right)-\mathcal{H}\left(e^{-\frac{i}{2}t_{n}\Delta}f \right) \right\}=0,
\\[6pt]
\label{10/04/22/13:28}
&\sup_{n\in \mathbb{N}}\left\|f_{n}^{*}-f \right\|_{H^{1}}
\le 
\sup_{n\in \mathbb{N}}\left\|f_{n} \right\|_{H^{1}}<\infty.
\end{align}
\end{lemma}
\begin{proof}[Proof of Lemma \ref{08/08/21/14:12}]
Put 
\begin{equation}\label{10/04/29/10:58}
A=\frac{1}{2}\limsup_{n\to \infty}\left\| e^{\frac{i}{2}t_{n}\Delta} f_{n} \right\|_{L^{\frac{d}{2}(p-1)}}^{\frac{d}{2}(p-1)}.
\end{equation}
Note that $A>0$ by (\ref{09/06/28/16:04}). 
\par 
We take a subsequence of $\{f_{n}\}$ (still denoted by the same symbol) and a sequence $\{t_{n}\}$ in $\mathbb{R}$ such that 
\begin{equation}\label{08/08/21/14:55}
\inf_{n\in \mathbb{N}}\left\| e^{\frac{i}{2}t_{n}\Delta} f_{n} \right\|_{L^{\frac{d}{2}(p-1)}}^{\frac{d(p-1)}{2}}\ge A.
\end{equation}
Here, extracting some subsequence of $\{t_{n}\}$, we may assume that 
\begin{equation}\label{10/04/24/16:59}
\lim_{n\to \infty} t_{n}=t_{\infty} 
\quad 
\mbox{for some $t_{\infty}\in \mathbb{R}$ 
or 
$t_{\infty}\in \{\pm \infty\}$}.
\end{equation}
In addition to (\ref{08/08/21/14:55}), we have by the Sobolev embedding that   
\begin{equation}
\label{09/06/28/16:10}
\sup_{n\in \mathbb{N}}\left\| e^{\frac{i}{2}t_{n}\Delta}f_{n} \right\|_{L^{p+1}}^{p+1}
\le C_{p}B^{p+1} 
\end{equation}
for some constant $C_{p}$ depending only on $d$ and $p$, where we put 
\begin{equation}\label{10/04/29/10:59}
B=\sup_{n\in \mathbb{N}}\left\| f_{n} \right\|_{H^{1}}<\infty.
\end{equation}
Then, Lemma \ref{08/3/28/20:21} (with $\alpha=2$, $\beta=\frac{d(p-1)}{2}$ and $\gamma=p+1$) shows that 
\begin{equation}\label{10/04/28/21:53}
\begin{split}
&\mathcal{L}^{d}\left( \left[ \left| e^{\frac{i}{2}t_{n}\Delta}f_{n}
 \right| > \eta \right] \right)>
\frac{A}{2}\eta^{\frac{d}{2}(p-1)}
\\[6pt]
& \qquad 
\mbox{for all 
$n \in \mathbb{N}$ 
and 
$0<\eta<\min
\left\{1, \left( \frac{A}{4B^{2}}\right)^{\frac{1}{\frac{d(p-1)}{2}-2}},
\left( \frac{A}{4C_{p}B^{p+1}} \right)^{\frac{1}{p+1-\frac{d(p-1)}{2}}}
\right\}$
}.
\end{split}
\end{equation}
Moreover, Lemma \ref{08/03/28/21:00} implies that: There exists a sequence $\{y_{n}\}$ in $\mathbb{R}^{d}$ such that 
\begin{equation}\label{09/12/07/9:31}
\mathcal{L}^{d}\left( \left[ \left| e^{\frac{i}{2}t_{n}\Delta}f_{n}(\cdot +y_{n}) \right| > \frac{\eta}{2} \right]\cap B_{1}(0) \right)
\gtrsim 
\left( \frac{1+A\eta^{\frac{d}{2}(p-1)+2} }{1+B}\right)^{\frac{2}{2^{\dagger}-2}}
\  
\mbox{for all $n \in \mathbb{N}$}, 
\end{equation}
where $2^{\dagger}=4$ if $d=1,2$ and $2^{\dagger}=2^{*}$ if $d\ge 3$, and the implicit constant depends only on $d$ and $p$. For the sequence $\{y_{n}\}$ found above, we put 
\begin{equation}\label{10/04/29/12:30}
f_{n}^{*}(x)=e^{\frac{i}{2}t_{n}\Delta}f(x+y_{n}).
\end{equation}
Then, $\{f_{n}^{*}\}$ is uniformly bounded in $H^{1}(\mathbb{R}^{d})$ as well as $\{f_{n}\}$. Hence, there exist a subsequence of $\{f_{n}^{*}\}$ (still denoted by the same symbol), a function $f \in H^{1}(\mathbb{R}^{d})$ such that 
\begin{equation}\label{09/12/07/10:03}
\lim_{n\to \infty}f_{n}^{*} =  f 
\quad 
\mbox{weakly in $H^{1}(\mathbb{R}^{d})$}.
\end{equation} 
Here, (\ref{09/12/07/10:03}), with the help of the Sobolev embedding, 
 also yields that 
\begin{equation}\label{10/04/15/21:37}
\lim_{n\to \infty}f_{n}^{*} =  f 
\quad 
\mbox{strongly in $L_{loc}^{q}(\mathbb{R}^{d})$
\quad 
for all $q\in [2, 2^{*})$}.
\end{equation}
Therefore, we have by (\ref{09/12/07/9:31}) that 
\begin{equation}\label{10/04/29/12:38}
\begin{split}
\left\| f \right\|_{L^{q}}^{q} &\ge \left\| f \right\|_{L^{q}(B_{1}(0))}^{q}
=\lim_{n\to \infty}\left\| f_{n}^{*} \right\|_{L^{q}(B_{1}(0))}^{q}
\gtrsim \eta^{q}
\left( \frac{1+A\eta^{\frac{d}{2}(p-1)+2}}{1+B}\right)^{\frac{2}{2^{\dagger}-2}}\\[6pt]
&\quad  
\mbox{for all $q\in [2,2^{*})$
and 
$0<\eta<\min
\left\{1, \left( \frac{A}{4B^{2}}\right)^{\frac{1}{\frac{d(p-1)}{2}-2}},
\left( \frac{A}{4C_{p}B^{p+1}} \right)^{\frac{1}{p+1-\frac{d(p-1)}{2}}}
\right\}$}, 
\end{split}
\end{equation}
where the implicit constant depends only on $d$ and $q$. Thus, $f$ is nontrivial.
\par 
Now, we shall show that the sequence $\{f_{n}^{*}\}$ satisfies the property 
(\ref{08/08/21/17:25}). Indeed, it follows from the weak convergence (\ref{09/12/07/10:03}) that 
\begin{equation}\label{10/04/15/21:50}
\begin{split}
&\left\| |\nabla |^{s} \left(f_{n}^{*} - f \right)\right\|_{L^{2}}^{2}
\\[6pt]
&=
\left\| |\nabla |^{s} f_{n}^{*} \right\|_{L^{2}}^{2}
-
\left\| |\nabla |^{s} f \right\|_{L^{2}}^{2}
-2\Re{
\int_{\mathbb{R}^{d}}|\nabla|^{s}\left( f_{n}^{*}(x)-f(x)\right)
|\nabla |^{s}f(x)\,dx  
}
\\[6pt]
&=\left\| |\nabla |^{s} f_{n} \right\|_{L^{2}}^{2}
-
\left\| |\nabla |^{s} f \right\|_{L^{2}}^{2}+o_{n}(1)
\qquad 
\mbox{for all $s\in [0, 1]$},
\end{split}
\end{equation}
so that (\ref{08/08/21/17:25}) holds.
\par 
Next, we shall show (\ref{08/08/21/14:35}). We first consider the case 
 $\displaystyle{t_{\infty}=\lim_{n\to \infty}t_{n} \in \mathbb{R}}$. Then, we can easily verify that 
\begin{align}
\label{09/12/07/11:14}
& \lim_{n\to \infty}
e^{-\frac{i}{2}t_{n}\Delta}f_{n}^{*}
=
e^{-\frac{i}{2}t_{\infty}\Delta}f 
\quad 
\mbox{weakly in $H^{1}(\mathbb{R}^{d})$, and a.e. in $\mathbb{R}^{d}$},
\\[6pt]
\label{09/02/11/18:06}
&
\lim_{n\to \infty}
e^{-\frac{i}{2}t_{n}\Delta}f 
= 
e^{-\frac{i}{2}t_{\infty}\Delta}f 
\quad \mbox{strongly in $H^{1}(\mathbb{R}^{d})$}.
\end{align}   
Moreover, the triangle inequality gives us that    
\begin{equation}\label{09/02/11/18:08}
\begin{split}
&\left| \left\| f_{n} \right\|_{L^{q}}^{q} -
\left\| e^{-\frac{i}{2}t_{n}\Delta} \left( f_{n}^{*}- f\right)   \right\|_{L^{q}}^{q}
-\left\| e^{-\frac{i}{2}t_{n}\Delta}f \right\|_{L^{q}}^{q}
\right|
\\[6pt]
&\le 
\left|
\left\| e^{-\frac{i}{2}t_{n}\Delta}f_{n}^{*} \right\|_{L^{q}}^{q} -
\left\| e^{-\frac{i}{2}t_{n}\Delta}f_{n}^{*}-e^{\frac{i}{2}t_{\infty}\Delta}f \right\|_{L^{q}}^{q}
-\left\| e^{-\frac{i}{2}t_{\infty}\Delta}f \right\|_{L^{q}}^{q}
\right|
\\[6pt]
&\qquad 
+\left| 
\left\| e^{-\frac{i}{2}t_{n}\Delta} \left( f_{n}^{*}- f \right) \right\|_{L^{q}}^{q}
-
\left\|e^{-\frac{i}{2}t_{n}\Delta} f_{n}^{*}- e^{-\frac{i}{2}t_{\infty}\Delta} f  \right\|_{L^{q}}^{q}
\right|
\\[6pt]
&\qquad 
+\left| 
\left\| e^{-\frac{i}{2}t_{n}\Delta}f \right\|_{L^{q}}^{q}
-
\left\| e^{-\frac{i}{2}t_{\infty}\Delta}f \right\|_{L^{q}}^{q}
\right|
\qquad
\mbox{for all $q\in [2, 2^{*})$}.
\end{split}
\end{equation}
Here, we have by the Sobolev embedding that 
\begin{equation}\label{10/04/25/10:34}
\sup_{n\in \mathbb{N}}
\left\| e^{-\frac{i}{2}t_{n}\Delta} f_{n}^{*}\right\|_{L^{q}}
\lesssim 
\sup_{n\in \mathbb{N}}
\left\|e^{-\frac{i}{2}t_{n}\Delta} f_{n}^{*} \right\|_{H^{1}}
=
\sup_{n\in \mathbb{N}}
\left\| f_{n} \right\|_{H^{1}}<\infty
\quad
\mbox{for all $q\in [2, 2^{*})$}.
\end{equation} 
Therefore,  Lemma \ref{08/03/28/21:21}, together with (\ref{09/12/07/11:14}) and (\ref{10/04/25/10:34}), implies that the first term on the right-hand side 
 of (\ref{09/02/11/18:08}) vanishes as $n\to \infty$. The second term also vanishes as $n\to \infty$. Indeed, it follows from (\ref{09/09/27/21:09}) and (\ref{09/02/11/18:06}) that 
\begin{equation}\label{09/02/11/18:07}
\begin{split}
&\lim_{n\to \infty}
\left| 
\left\| e^{-\frac{i}{2}t_{n}\Delta} \left( f_{n}^{*}- f \right) \right\|_{L^{q}}^{q}
-\left\|e^{-\frac{i}{2}t_{n}\Delta} 
f_{n}^{*}- e^{-\frac{i}{2}t_{\infty}\Delta} f  \right\|_{L^{q}}^{q}
\right|
\\[6pt]
&\lesssim 
\lim_{n\to \infty}\left( \left\| f_{n} \right\|_{H^{1}}^{q-1} + \left\|f  \right\|_{H^{1}}^{q-1}\right)
\left\| e^{-\frac{i}{2}t_{n}\Delta}f -e^{-\frac{i}{2}t_{\infty}\Delta}f \right\|_{L^{q}}
= 0.
\end{split}
\end{equation}
Moreover, (\ref{09/02/11/18:06}) immediately yields that the third term vanishes as $n\to \infty$. Thus, we have proved the property (\ref{08/08/21/14:35}) when $t_{\infty} \in \mathbb{R}$. 
\par 
We next suppose that $t_{\infty}\in \{\pm \infty\}$. In this case, (\ref{09/09/27/21:09}) and Lemma \ref{10/04/08/9:20} yield 
that   
\begin{equation}\label{10/04/08/10:19}
\begin{split}
&\lim_{n\to \infty}\left| 
\left\| f_{n} \right\|_{L^{q}}^{q} - \left\| e^{-\frac{i}{2}t_{n}\Delta} \left(  f_{n}^{*}-f \right) \right\|_{L^{q}}^{q} 
-\left\| e^{-\frac{i}{2}t_{n}\Delta}f \right\|_{L^{q}}^{q}
\right|
\\[10pt]
&=
\lim_{n\to \infty}\left| 
\left\| e^{-\frac{i}{2}t_{n}\Delta} f_{n}^{*} \right\|_{L^{q}}^{q} - \left\| e^{-\frac{i}{2}t_{n}\Delta} \left( f_{n}^{*}-f \right) \right\|_{L^{q}}^{q} 
-\left\| e^{-\frac{i}{2}t_{n}\Delta}f \right\|_{L^{q}}^{q}
\right|
\\[10pt]
&\le 
\lim_{n\to \infty}
\left| 
\left\| e^{-\frac{i}{2}t_{n}\Delta} f_{n}^{*} \right\|_{L^{q}}^{q} - \left\| e^{-\frac{i}{2}t_{n}\Delta} \left( f_{n}^{*}-f \right) \right\|_{L^{q}}^{q} \right|+
\lim_{n\to \infty}
\left\| e^{-\frac{i}{2}t_{n}\Delta}f \right\|_{L^{q}}^{q}
\\[10pt]
&\lesssim 
\lim_{n\to \infty} \left( \left\| f_{n} \right\|_{H^{1}}^{q-1} +\left\| f \right\|_{H^{1}}^{q-1} \right)\left\| e^{-\frac{i}{2}t_{n}\Delta}f \right\|_{L^{q}}
+
\lim_{n\to \infty}
\left\| e^{-\frac{i}{2}t_{n}\Delta}f \right\|_{L^{q}}^{q}
\\[6pt]
&= 0 
\qquad
\mbox{for all $q\in (2,2^{*})$}.
\end{split}
\end{equation}
Since (\ref{08/08/21/14:35}) with $q=2$ follows from (\ref{08/08/21/17:25}), we have proved (\ref{08/08/21/14:35}).
\par 
The property (\ref{08/08/21/14:36}) immediately follows form (\ref{08/08/21/17:25}) and (\ref{08/08/21/14:35}). 
\par
Finally, we shall prove (\ref{10/04/22/13:28}). It follows from (\ref{08/08/21/17:25}) that 
\begin{equation}\label{10/04/24/17:18}
\left\|f_{n}^{*}-f \right\|_{H^{1}} \le 
\left\| f_{n} \right\|_{H^{1}}
\quad 
\mbox{for sufficiently large $n \in \mathbb{N}$}.
\end{equation} 
Hence, extracting further some subsequence of $\{f_{n}^{*}\}$, we obtain (\ref{10/04/22/13:28}).
\end{proof}

\begin{lemma}[Dichotomy]
\label{08/08/23/16:32}
Let $u$ be a function in $X(\mathbb{R})$ and let 
 $\{(\eta_{n}^{1},\tau_{n}^{1})\}$, \ldots, 
$\{ (\eta_{n}^{L},\tau_{n}^{L})\}$ be sequences in $\mathbb{R}^{d}\times \mathbb{R}$ with 
\begin{equation}\label{10/04/18/18:07}
\lim_{n\to \infty}\left(
| \tau_{n}^{k}-\tau_{n}^{l}|
+
|\eta_{n}^{k}-\eta_{n}^{l}|\right)=\infty
\quad 
\mbox{for all $1\le k < l \le L$}.
\end{equation}
Then, putting  
$
u_{n}^{l}(x,t):=u(x-\eta_{n}^{l},t-\tau_{n}^{l})
$,
we have 
\begin{equation}\label{08/08/23/16:35}
\lim_{n\to \infty}\left\| \biggm| 
\sum_{k=1}^{L}u_{n}^{k} \biggm|^{p-1}\sum_{l=1}^{L}u_{n}^{l} 
-\sum_{l=1}^{L}|u_{n}^{l}|^{p-1}u_{n}^{l}
\right\|_{L^{\widetilde{r}_{1}'}(\mathbb{R};L^{q_{1}'})}=0
\end{equation}
and 
\begin{equation}\label{09/12/07/11:54}
\lim_{n\to \infty}\left\| 
\left| u_{n}^{k} \right|^{q_{j}-1} u_{n}^{l}
\right\|_{L^{\frac{r_{j}}{q_{j}}}(\mathbb{R};L^{1})}=0 
\quad 
\mbox{for all $j=1,2$ and $k\neq l$},
\end{equation}
where $q_{j}$ and $r_{j}$ ($j=1,2$) are the exponents introduced in Section 
\ref{08/10/07/9:01}.
\end{lemma}
\begin{proof}[Proof of Lemma \ref{08/08/23/16:32}] 
Put   
\begin{equation}\label{10/05/11/15:20}
r_{n}=\frac{1}{2}\inf_{1\le k<l\le L}\left(
\left| \tau_{n}^{k}-\tau_{n}^{l} \right|+
\left| \eta_{n}^{k}-\eta_{n}^{l} \right|\right)
\end{equation}
and define functions $\chi_{n}$ and $\chi_{n}^{l}$ on $\mathbb{R}^{d} \times \mathbb{R}$ by 
\begin{equation}\label{10/05/11/15:21}
\chi_{n}(x,t):=\left\{ 
\begin{array}{ccl}
1 &\mbox{if}& |(x,t)|< r_{n},
\\[6pt]
0 &\mbox{if}& |(x,t)|\ge r_{n},
\end{array}
\right.
\qquad 
\chi_{n}^{l}(x,t):=\chi_{n}(x-\eta_{n}^{l},t-\tau_{n}^{l}).
\end{equation}

Now, we shall prove (\ref{08/08/23/16:35}). The triangle inequality and Lemma \ref{09/03/06/16:46} show that 
\begin{equation}\label{10/04/11/1:43}
\begin{split}
&\left\| \biggm| 
\sum_{k=1}^{L}u_{n}^{k} \biggm|^{p-1}\sum_{l=1}^{L}u_{n}^{l} 
-\sum_{l=1}^{L}|u_{n}^{l}|^{p-1}u_{n}^{l}
\right\|_{L^{\widetilde{r}_{1}'}(\mathbb{R};L^{q_{1}'})}
\\[6pt]
&
\le 
\sum_{l=1}^{L} \left\| \biggm| \sum_{k=1}^{L}u_{n}^{k} \biggm|^{p-1} u_{n}^{l} 
-|u_{n}^{l}|^{p-1}u_{n}^{l}
\right\|_{L^{\widetilde{r}_{1}'}(\mathbb{R};L^{q_{1}'})}
\\[6pt]
&\lesssim 
\sum_{l=1}^{L} \sum_{{k=1}\atop {k\neq l}}^{L}
\left\| \left| u_{n}^{k} \right|^{p-1} |u_{n}^{l}| 
\right\|_{L^{\widetilde{r}_{1}'}(\mathbb{R};L^{q_{1}'})},
\end{split}
\end{equation}
where the implicit constant depends only on $p$ and $L$. Hence, for (\ref{08/08/23/16:35}), it suffices to show the following estimate: 
\begin{equation}\label{10/04/12/1:51}
\lim_{n\to \infty}
\left\| \left| u_{n}^{k} \right|^{p-1} |u_{n}^{l}| 
\right\|_{L^{\widetilde{r}_{1}'}(\mathbb{R};L^{q_{1}'})}
=0
\quad 
\mbox{for all $k\neq l$}.
\end{equation}
We begin with the following estimate:  
\begin{equation}\label{09/01/11/23:40}
\begin{split}
\left\| \left| u_{n}^{k} \right|^{p-1} |u_{n}^{l}| 
\right\|_{L^{\widetilde{r}_{1}'}(\mathbb{R};L^{q_{1}'})}
&\le 
\left\| \left| u_{n}^{k} \right|^{p-1} |u_{n}^{l}|
-
\left| \chi_{n}^{k}u_{n}^{k} \right|^{p-1} |\chi_{n}^{l}u_{n}^{l}| 
\right\|_{L^{\widetilde{r}_{1}'}(\mathbb{R};L^{q_{1}'})}
\\[6pt]
&\quad +
\left\| \left| \chi_{n}^{k}u_{n}^{k} \right|^{p-1} |\chi_{n}^{l}u_{n}^{l}| 
\right\|_{L^{\widetilde{r}_{1}'}(\mathbb{R};L^{q_{1}'})}
.
\end{split}
\end{equation}
The H\"older inequality gives an estimate for the first term on the right-hand side of (\ref{09/01/11/23:40}):     
\begin{equation}\label{10/04/12/0:33}
\begin{split}
&\left\| \left| u_{n}^{k} \right|^{p-1} |u_{n}^{l}|
-
\left| \chi_{n}^{k}u_{n}^{k} \right|^{p-1} |\chi_{n}^{l}u_{n}^{l}| 
\right\|_{L^{\widetilde{r}_{1}'}(\mathbb{R};L^{q_{1}'})}
\\[6pt]
&\le 
\left\| 
\left| u_{n}^{k} \right|^{p-1} |u_{n}^{l}| 
-
\left| \chi_{n}^{k}u_{n}^{k} \right|^{p-1} |u_{n}^{l}|
\right\|_{L^{\widetilde{r}_{1}'}(\mathbb{R};L^{q_{1}'})}
\\[6pt]
&\qquad \qquad 
+
\left\| 
\left| \chi_{n}^{k}u_{n}^{k} \right|^{p-1} |u_{n}^{l}| 
-
\left| \chi_{n}^{k}u_{n}^{k} \right|^{p-1} |\chi_{n}^{l}u_{n}^{l}|
\right\|_{L^{\widetilde{r}_{1}'}(\mathbb{R};L^{q_{1}'})}
\\[6pt]
&\lesssim
\left\|(1-\chi_{n}^{k}) u_{n}^{k}
 \right\|_{X(\mathbb{R})}^{p-1}
\left\| u_{n}^{l} \right\|_{X(\mathbb{R})}
+
\left\| u_{n}^{k} \right\|_{X(\mathbb{R})}^{p-1}
\left\|(1-\chi_{n}^{l})u_{n}^{l}\right\|_{X(\mathbb{R})}
\\[6pt]
&=
\left\|(1-\chi_{n}) u
 \right\|_{X(\mathbb{R})}^{p-1}
\left\| u \right\|_{X(\mathbb{R})}
+
\left\| u \right\|_{X(\mathbb{R})}^{p-1}
\left\|(1-\chi_{n})u\right\|_{X(\mathbb{R})}
,
\end{split}
\end{equation}
where the implicit constant depends only on $d$, $p$ and $q_{1}$. 
Note here that 
\begin{align}\label{10/04/12/0:39}
&\lim_{n\to \infty}(1-\chi_{n}(x,t)) u(x,t) 
= 0  \quad \mbox{a.a. $(x,t) \in \mathbb{R}^{d}\times \mathbb{R}$ and all $l=1,\ldots, L$}, 
\\[6pt]
\label{10/04/12/0:40}
&\left| (1-\chi_{n}(x,t)) u(x,t) \right| \le |u(x,t)| \quad 
\mbox{a.a. $(x,t)\in \mathbb{R}^{d}\times \mathbb{R}$ and all $l=1,\ldots, L$},
\end{align}
so that the Lebesgue dominated convergence theorem gives us that 
\begin{equation}\label{09/12/07/18:34}
\lim_{n\to \infty}\left\| (1-\chi_{n})u\right\|_{X(\mathbb{R})}
=0
.
\end{equation}
Hence, the first term on the right-hand side of (\ref{09/01/11/23:40}) vanishes as $n\to \infty$.
\par 
On the other hand, it follows from 
\begin{equation}\label{10/04/12/0:50}
{\rm supp}\, \chi_{n}^{k}\cap {\rm supp}\,\chi_{n}^{l}=\emptyset
\quad 
\mbox{for $k\neq l$}
\end{equation}
that the second term on the right-hand side of (\ref{09/01/11/23:40}) is zero.
 Thus, we have proved (\ref{10/04/12/1:51}). 
\par 
In a way similar to the proof of (\ref{10/04/12/1:51}), we can obtain 
(\ref{09/12/07/11:54}).  
\end{proof}

Now, we are in a position to prove Lemma \ref{08/08/19/23:07}.

\begin{proof}[Proof of Lemma \ref{08/08/19/23:07}]
Put $f_{n}=\psi_{n}(0)$. Since $\psi_{n}(0) \in PW_{+}$, Proposition \ref{09/06/21/19:28} and (\ref{08/08/26/15:56}) give us that   
\begin{equation}\label{09/02/23/17:18}
\begin{split}
\limsup_{n\to \infty}\mathcal{N}_{2}(f_{n})
&\le \limsup_{n\to \infty}
\sqrt{\frac{d(p-1)}{d(p-1)-4}}^{\frac{d}{2}(p-1)-2}\widetilde{\mathcal{N}}_{2}(f_{n})
\\
&= \sqrt{\frac{d(p-1)}{d(p-1)-4}}^{\frac{d}{2}(p-1)-2} \widetilde{N}_{c}
 \\
 &<\sqrt{\frac{d(p-1)}{d(p-1)-4}}^{\frac{d}{2}(p-1)-2} \widetilde{N}_{2}
 =N_{2}.
\end{split}
\end{equation}

Now, we apply Lemma \ref{08/08/21/14:12} to the sequence $\{f_{n}\}$ and obtain that: There exist a subsequence of $\{f_{n}\}$ (still denoted by the same symbol), a nontrivial function $f^{1} \in H^{1}(\mathbb{R}^{d})$, a sequence $\{t_{n}^{1}\}$ in $\mathbb{R}$ with $t_{n}^{1}\to t_{\infty}^{1} \in \mathbb{R}\cup \{\pm \infty\}$, and a sequence $\{y_{n}^{1}\}$ in $\mathbb{R}^{d}$ such that, putting  
\begin{equation}\label{10/05/11/16:15}
f_{n}^{1}(x):=\left( e^{\frac{i}{2}t_{n}^{1}\Delta}e^{y_{n}^{1}\cdot \nabla }f_{n}\right)(x)=e^{\frac{i}{2}t_{n}^{1}\Delta}f_{n}(x+y_{n}^{1}),
\end{equation}
we  have:  
\begin{align}\label{08/08/21/18:10}
&\lim_{n\to \infty}f_{n}^{1}= f^{1} 
\quad 
\mbox{weakly in $H^{1}(\mathbb{R}^{d})$} 
\quad  
\mbox{for all $s  \in [0,1]$},
\\[6pt]
\label{08/08/26/16:39}
&\lim_{n\to \infty}\left\{ \left\| |\nabla|^{s} f_{n} \right\|_{L^{2}}^{2}-\left\| |\nabla|^{s} (f_{n}^{1} -f^{1}) \right\|_{L^{2}}^{2} \right\}
=\left\| |\nabla|^{s} f^{1}\right\|_{L^{2}}^{2}
\quad 
\mbox{for all $s\in [0,1]$},
\\[6pt]
\label{08/08/26/16:40}
&\lim_{n\to \infty}\left\{ \left\| f_{n} \right\|_{L^{q}}^{q}
\!-\left\| e^{-\frac{i}{2}t_{n}^{1}\Delta}(f_{n}^{1} -f^{1})\right\|_{L^{q}}^{q}\!\!
-\left\| e^{-\frac{i}{2}t_{n}^{1}\Delta}f^{1} \right\|_{L^{q}}^{q}\right\}=0
\quad 
\mbox{for all $q \in [2,2^{*})$}, 
\\[6pt]
\label{08/08/21/17:49}
&\lim_{n\to \infty}\left\{ 
\mathcal{H}(f_{n})-\mathcal{H}(e^{-\frac{i}{2}t_{n}^{1}\Delta}
(f_{n}^{1} -f^{1}))-\mathcal{H}(e^{-\frac{i}{2}t_{n}^{1}\Delta}f^{1}) \right\}=0,
\\[6pt]
\label{09/02/23/17:55}
&
\sup_{n\in \mathbb{N}}
\left\| f_{n}^{1} -f^{1} \right\|_{H^{1}} 
<\infty 
.
\end{align}
Besides, it follows from (\ref{09/02/23/17:18}) and (\ref{08/08/26/16:39})  
that
\begin{align}
\label{09/12/05/16:00}
&\mathcal{N}_{2}(f^{1})
\le 
\limsup_{n\to \infty}\mathcal{N}_{2}(f_{n})<N_{2},
\\[6pt]
\label{09/08/03/0:14}
&\limsup_{n\to \infty}\mathcal{N}_{2}(e^{-\frac{i}{2}t_{n}^{1}\Delta}(f_{n}^{1} -f^{1}))
=\limsup_{n\to \infty}\mathcal{N}_{2}(f_{n}^{1} -f^{1})
\le \limsup_{n\to \infty}\mathcal{N}_{2}(f_{n})
<N_{2},
\end{align}
so that we have by the definition of $N_{2}$ (see (\ref{08/07/02/23:24}))
that  
\begin{align}
\label{09/02/23/17:19}
&0< \mathcal{K}(f^{1}) < \mathcal{H}(f^{1}),
\\[6pt]
\label{09/02/22/21:22}
&0< \mathcal{K}(e^{-\frac{i}{2}t_{n}^{1}\Delta}(f_{n}^{1} -f^{1})) <
\mathcal{H}(e^{-\frac{i}{2}t_{n}^{1}\Delta}(f_{n}^{1} -f^{1}))
\quad 
\mbox{for sufficiently large $n \in \mathbb{N}$}.
\end{align}

We shall show that 
\begin{equation}\label{09/05/02/11:03}
\left\{ \begin{array}{ccl} 
e^{-\frac{i}{2}t_{\infty}^{1}\Delta}f^{1} \in PW_{+} &\mbox{if}& t_{\infty}^{1} \in \mathbb{R}, 
\\[10pt]
f^{1} \in \Omega &\mbox{if}& t_{\infty}^{1}=\pm \infty. 
 \end{array} \right.
\end{equation}
Suppose first that $t_{\infty}^{1} \in \mathbb{R}$. Then, the estimate (\ref{09/02/23/17:18}) gives us that  
\begin{equation}\label{10/04/20/18:16}
\mathcal{N}_{2}(e^{-\frac{i}{2}t_{\infty}^{1}\Delta}f^{1})=\mathcal{N}_{2}(f^{1})<N_{2}, 
\end{equation}
which, together with the definition of $N_{2}$, yields that   
\begin{equation}\label{09/05/02/11:39}
0< 
\mathcal{K}(e^{-\frac{i}{2}t_{\infty}^{1}\Delta}f^{1})
<\mathcal{H}(e^{-\frac{i}{2}t_{\infty}^{1}\Delta}f^{1}).
\end{equation}
Moreover, (\ref{08/08/26/16:39}) and (\ref{08/08/21/17:49}), with the help of (\ref{09/02/22/21:22}), show that 
\begin{equation}\label{10/04/20/18:23}
\widetilde{\mathcal{N}}_{2}(e^{-\frac{i}{2}t_{\infty}^{1}\Delta}f^{1})
= 
\lim_{n\to \infty}\widetilde{\mathcal{N}}_{2}(e^{-\frac{i}{2}t_{n}^{1}\Delta}f^{1})
\le 
\lim_{n\to \infty}\widetilde{\mathcal{N}}_{2}(f_{n}).
\end{equation}
Combining (\ref{10/04/20/18:23}) with (\ref{08/08/26/15:56}), we obtain that 
\begin{equation}\label{09/05/02/11:40}
\widetilde{\mathcal{N}}_{2}(e^{-\frac{i}{2}t_{\infty}^{1}\Delta}f^{1})
\le 
\widetilde{N}_{c}<\widetilde{N}_{2}.
\end{equation}
Hence, (\ref{09/05/02/11:39}) and (\ref{09/05/02/11:40}), with the help of the relation (\ref{08/06/15/14:38}), lead to that $e^{-t_{\infty}^{1}\Delta}f^{1} \in PW_{+}$. 
\par 
We next suppose that $t_{\infty}^{1}\in \{\pm \infty\}$. In this case , 
 the formula (\ref{08/08/21/17:49}), together with (\ref{09/02/22/21:22}), 
yields that    
\begin{equation}\label{10/05/14/11:19}
\begin{split}
\left\| \nabla f^{1} \right\|_{L^{2}}^{2}&=
\mathcal{H}(e^{-\frac{i}{2}t_{n}^{1}\Delta}f^{1}) 
+\frac{2}{p+1}\left\| e^{-\frac{i}{2}t_{n}^{1}\Delta}f^{1} \right\|_{L^{p+1}}^{p+1}
\\[6pt]
&\le  
\mathcal{H}(f_{n}) +o_{n}(1)
+\frac{2}{p+1}\left\| e^{-\frac{i}{2}t_{n}^{1}\Delta}f^{1} \right\|_{L^{p+1}}^{p+1},
\end{split}
\end{equation}
so that we have by  Lemma \ref{10/04/08/9:20} that 
\begin{equation}\label{09/02/18/18:02}
\left\| \nabla f^{1} \right\|_{L^{2}}^{2}\le 
\lim_{n\to \infty}\mathcal{H}(f_{n}).
\end{equation}
This estimate and (\ref{08/08/26/16:39}) with $s=0$, together with (\ref{08/08/26/15:56}),  show that  
\begin{equation}\label{10/05/11/16:39}
\mathcal{N}_{2}(f^{1})
\le 
\lim_{n\to \infty} 
\left\|f_{n} \right\|_{L^{2}}^{p+1-\frac{d}{2}(p-1)}
\sqrt{\mathcal{H}(f_{n})}^{\frac{d}{2}(p-1)-2}
= \widetilde{N}_{c} <\widetilde{N}_{2}.
\end{equation}
Hence, we have that $f^{1} \in \Omega$.  
\par 
Now, we suppose that 
\begin{equation}\label{10/05/11/16:49}
\limsup_{n\to \infty} \left\| e^{\frac{i}{2} t \Delta } \left( f_{n}^{1} -f^{1} \right) \right\|_{L^{\infty}(\mathbb{R};L^{\frac{d}{2}(p-1)}) }>0.
\end{equation}
Then, we can apply Lemma \ref{08/08/21/14:12} to the sequence $\{f_{n}^{1}-f^{1}\}$, so that we find that: There exist a subsequence of $\{f_{n}^{1}-f^{1}\}$ (still denoted by the same symbol), a nontrivial function $f^{2} \in H^{1}(\mathbb{R}^{d})$,  a sequence $\{t_{n}^{2}\}$ in $\mathbb{R}$ with $\displaystyle{\lim_{n\to \infty}t_{n}^{2}= t_{\infty}^{2}\in \mathbb{R}\cup \{\pm \infty\}}$ and a sequence $\{y_{n}^{2}\}$ in $\mathbb{R}^{d}$ such that, putting 
\begin{equation}\label{09/09/14/17:44}
f_{n}^{2}:=e^{\frac{i}{2}t_{n}^{2}\Delta}
e^{ y_{n}^{2}\cdot \nabla }
\left( f_{n}^{1}-f^{1} \right),
\end{equation}
we have: 
\begin{align}
\label{09/03/04/16:39}
&\lim_{n\to \infty}f_{n}^{2}= f^{2} 
\quad \mbox{weakly in $H^{1}(\mathbb{R}^{d})$},
\\[6pt]
\label{09/07/31/18:13}
&\lim_{n\to \infty}\left\{ 
\left\| |\nabla|^{s} \left( f_{n}^{1}-f^{1}\right) \right\|_{L^{2}}^{2}
-\left\| |\nabla |^{s}\left( f_{n}^{2}-f^{2} \right) \right\|_{L^{2}}^{2} 
\right\}=\left\||\nabla |^{s} f^{2} \right\|_{L^{2}}^{2}
\quad 
\mbox{for all $s\in [0,1]$},
\\[6pt]
\label{09/07/31/18:14}
&
\lim_{n\to \infty}\left\{ \left\| f_{n}^{1}-f^{1} \right\|_{L^{q}}^{q}
\!
-\left\| e^{-\frac{i}{2}t_{n}^{2}\Delta}(f_{n}^{2} -f^{2})\right\|_{L^{q}}^{q} 
\!\!
-\left\| e^{-\frac{i}{2}t_{n}^{2}\Delta}f^{2} \right\|_{L^{q}}^{q}\right\}=0
\quad
\mbox{for all $q\in [2,2^{*})$}, 
\\[6pt]
\label{09/07/31/18:15}
&
\lim_{n\to \infty}\left\{ 
\mathcal{H}( f_{n}^{1}-f^{1})-\mathcal{H}(e^{-\frac{i}{2}t_{n}^{2}\Delta}
(f_{n}^{2} -f^{2}))-\mathcal{H}(e^{-\frac{i}{2}t_{n}^{2}\Delta}f^{2}) \right\}=0,
\\[6pt]
\label{09/07/31/18:34}
&\sup_{n\in \mathbb{N}}
\left\| |\nabla|^{s} (f_{n}^{2} -f^{2}) \right\|_{L^{2}} < \infty.
\end{align}
Note here that $f_{n}$ is represented in the form
\begin{equation}\label{10/04/25/16:09}
f_{n}
=
e^{-\frac{i}{2}\tau_{n}^{1}\Delta}e^{-\eta_{n}^{1} \cdot \nabla }f^{1}
+e^{-\frac{i}{2}\tau_{n}^{2}\Delta}e^{-\eta_{n}^{2} \cdot \nabla }f^{2}
+ e^{-\frac{i}{2}\tau_{n}^{2}\Delta}e^{-\eta_{n}^{2} \cdot \nabla }\left( f_{n}^{2}-f^{2} \right)
\quad  
\mbox{for all $n\in \mathbb{N}$},   
\end{equation}
where 
\begin{equation}\label{09/09/14/17:41}
\tau_{n}^{1}:=t_{n}^{1},
\quad 
\eta_{n}^{1}:=y_{n}^{1},
\qquad 
\tau_{n}^{2}:=\tau_{n}^{1}+t_{n}^{2},
\quad  
\eta_{n}^{2}:=\eta_{n}^{1}+y_{n}^{2}.
\end{equation}
We also note that the following formulas hold: 
\begin{align}
\label{09/08/01/21:47}
&\lim_{n\to \infty}\left\{ 
\left\| |\nabla|^{s} f_{n} \right\|_{L^{2}}^{2}
-\left\| |\nabla |^{s}\left( f_{n}^{2}-f^{2} \right) \right\|_{L^{2}}^{2} 
\right\}=\sum_{k=1}^{2} \left\||\nabla |^{s} f^{k} \right\|_{L^{2}}^{2}
\quad 
\mbox{for all $s \in [0,1]$}, 
\\[6pt]
\label{09/08/01/21:48}
&\lim_{n\to \infty}\left\{ \left\| f_{n} \right\|_{L^{q}}^{q}
\!-
\left\| e^{-\frac{i}{2}\tau_{n}^{2}\Delta}(f_{n}^{2} -f^{2})\right\|_{L^{q}}^{q}\!\!
-\sum_{k=1}^{2}\left\| e^{-\frac{i}{2}\tau_{n}^{k}\Delta}f^{k} \right\|_{L^{q}}^{q}\right\}=0
\quad 
\mbox{for all $q \in [2,2^{*})$},
\\[6pt]
\label{09/08/01/21:49}
&\lim_{n\to \infty}\left\{ 
\mathcal{H}( f_{n})-\mathcal{H}(e^{-\frac{i}{2}\tau_{n}^{2}\Delta}
(f_{n}^{2} -f^{2}))-\sum_{k=1}^{2}
\mathcal{H}(e^{-\frac{i}{2}\tau_{n}^{k}\Delta}f^{k}) \right\}=0.
\end{align}
Moreover, in a similar way to the proofs of (\ref{09/02/22/21:22}) and (\ref{09/05/02/11:03}), we obtain 
\begin{equation}\label{10/05/11/17:03}
0< \mathcal{K}\left(e^{-\frac{i}{2}\tau_{n}^{2}\Delta}\left( f_{n}^{2}-f^{2} \right) \right) < \mathcal{H}\left(e^{-\frac{i}{2}\tau_{n}^{2}\Delta}\left( f_{n}^{2}-f^{2} \right) \right)
\end{equation}
and 
\begin{equation}\label{10/05/11/17:02}
\left\{ \begin{array}{ccl} 
e^{-\frac{i}{2}\tau_{\infty}^{2}}f^{2} \in PW_{+} &\mbox{if}& \tau_{\infty}^{2} \in \mathbb{R}, 
\\[10pt]
f^{2} \in \Omega &\mbox{if}& \tau_{\infty}^{2}=\pm \infty. 
 \end{array} \right.
\end{equation}

We shall show that 
\begin{equation}\label{09/02/23/16:16}
\lim_{n\to \infty}|\tau_{n}^{2}-\tau_{n}^{1}|+|\eta_{n}^{2}-\eta_{n}^{1}|= 
\infty.
\end{equation}
Supposing that (\ref{09/02/23/16:16}) fails, we can take convergent subsequences of $\{\tau_{n}^{2}-\tau_{n}^{1} \}$ and $\{ \eta_{n}^{2}-\eta_{n}^{1} \}$ . Then, (\ref{08/08/21/18:10}), together with the unitarity of the operators $e^{-\frac{i}{2}(\tau_{n}^{2}-\tau_{n}^{1})\Delta}$ and $e^{-(\eta_{n}^{2}-\eta_{n}^{1})\cdot \nabla}$ on $H^{1}(\mathbb{R}^{d})$, shows that
\begin{equation}\label{10/05/11/17:06}
\lim_{n\to \infty}f_{n}^{2} = 0 
\quad \mbox{weakly in $H^{1}(\mathbb{R}^{d})$},
\end{equation}
which contradicts the fact that $f^{2}$ is nontrivial. Thus, (\ref{09/02/23/16:16}) holds.
\par 
Now, we suppose that 
\begin{equation}\label{10/05/11/17:53}
\limsup_{n\to \infty} \left\| e^{\frac{i}{2} t \Delta } \left( f_{n}^{2} -f^{2} \right) \right\|_{L^{\infty}(\mathbb{R};L^{\frac{d}{2}(p-1)}) }
>0.
\end{equation} 
Then, we can repeat the above procedure; Iterative use of Lemma \ref{08/08/21/14:12} implies the following Lemma:

\begin{lemma}\label{09/09/15/15:12} For some subsequence of $\{f_{n}\}$ (still denoted by the same symbol), there exists
\\  
(i) a family of nontrivial functions in $H^{1}(\mathbb{R}^{d})$, $\{ f^{1}, f^{2}, f^{3}, \ldots \}$ , and \\
(ii) a family of sequences in $\mathbb{R}^{d}\times \mathbb{R}$, $\left\{ \{(\eta_{n}^{1}, \tau_{n}^{1}) \}, \{(\eta_{n}^{2}, \tau_{n}^{2}) \}, \{(\eta_{n}^{3}, \tau_{n}^{3}) \}, \ldots \right\}$  with  
\begin{equation}\label{10/05/11/18:05}
\displaystyle{\lim_{n\to \infty}\tau_{n}^{l}= \tau_{\infty}^{l} \in \mathbb{R} \cup \{\pm \infty \}}
\quad \mbox{for all $l\ge 1$},
\end{equation}      
\begin{equation}\label{09/02/03/0:22}
\left\{ \begin{array}{ccl} e^{-\frac{i}{2}\tau_{\infty}^{l}\Delta}f^{l} \in PW_{+}
&\mbox{if} & 
 \tau_{\infty}^{l}\in \mathbb{R}, 
 \\[6pt] 
 f^{l} \in \Omega &\mbox{if} & \tau_{\infty}^{l}=\pm \infty 
 \end{array} \right. 
 \mbox{for all $l\ge 1$}, 
\end{equation}
and 
\begin{equation}\label{09/02/03/0:50}
\lim_{n\to \infty}|\tau_{n}^{l}-\tau_{n}^{k}|+|\eta_{n}^{l}-\eta_{n}^{k}|= \infty
\quad
\mbox{for all $1\le k <l $}, 
\end{equation}
such that, putting 
\[
\begin{split}
f_{n}^{0}&:=f_{n}, \quad f^{0}:=0, \quad \tau_{n}^{0}:=0, \quad \eta_{n}^{0}:=0,\\[6pt]
f_{n}^{l}&:=e^{\frac{i}{2}(\tau_{n}^{l}-\tau_{n}^{l-1})\Delta}
e^{(\eta_{n}^{l}-\eta_{n}^{l-1})\cdot \nabla}
(f_{n}^{l-1}-f^{l-1}) 
\quad 
\mbox{for $l\ge 1$}, 
\end{split}
\]
we have, for all $l\ge 1$: 
\begin{align}
\label{08/08/21/18:17}
&\lim_{n\to \infty}f_{n}^{l}= f^{l} 
\quad 
\mbox{weakly in $H^{1}(\mathbb{R}^{d})$, and  
 strongly in $L^{q}_{loc}(\mathbb{R}^{d})$ 
 for all $q \in [2,2^{*})$},
\\[6pt]
\label{08/08/24/1:47}
&\lim_{n\to \infty}\left\{ 
\left\| |\nabla|^{s} f_{n}\right\|_{L^{2}}^{2}
-\left\| |\nabla |^{s}\left( f_{n}^{l}-f^{l}\right) \right\|_{L^{2}}^{2} 
\right\}=\sum_{k=1}^{l}\left\||\nabla |^{s} f^{k} \right\|_{L^{2}}^{2}
\quad 
\mbox{for all $s \in [0,1]$},
\\[6pt]
\label{08/08/22/15:03}
&\lim_{n\to \infty}
\left\{ 
\left\| f_{n} \right\|_{L^{q}}^{q}
-\left\| e^{-\frac{i}{2}\tau_{n}^{l}\Delta} \left( f_{n}^{l}-f^{l}\right) \right\|_{L^{q}}^{q} 
-\sum_{k=1}^{l}\left\|e^{-\frac{i}{2}\tau_{n}^{k}\Delta}f^{k} \right\|_{L^{q}}^{q}\right\}=0
\quad 
\mbox{for all $q \in [2,2^{*})$},
\\[6pt]
\label{08/08/21/18:22}
&\lim_{n\to \infty}
\left\{ 
\mathcal{H}(f_{n})-\mathcal{H}(e^{-\frac{i}{2}\tau_{n}^{l}\Delta} \left( f_{n}^{l}-f^{l}\right) )
-\sum_{k=1}^{l}\mathcal{H}(e^{-\frac{i}{2}\tau_{n}^{k}\Delta}f^{k})
\right\}=0,
\\[6pt]
\label{09/03/05/18:36}
&\mathcal{K}(e^{-\frac{i}{2}\tau_{n}^{l}\Delta}\left( f_{n}^{l}-f^{l}\right) )>0.
\end{align}
\noindent 
Furthermore, putting $N:=\#\{ f^{1}, f^{2},f^{3},\ldots \}$, we have the alternatives: if $N$ is finite, then 
\begin{equation}
\label{09/09/15/16:19}
\lim_{n\to \infty}\left\|e^{\frac{i}{2}t\Delta}\left( f_{n}^{N}-f^{N} \right) \right\|_{L^{\infty}(\mathbb{R};L^{\frac{d}{2}(p-1)})\cap X(\mathbb{R})}=0;
\end{equation}
if $N=\infty$, then 
\begin{equation}
\label{09/09/15/16:20}
\lim_{l \to \infty}\lim_{n\to \infty}\left\|e^{\frac{i}{2}t\Delta}\left( f_{n}^{l}-f^{l} \right) \right\|_{L^{\infty}(\mathbb{R};L^{\frac{d}{2}(p-1)})\cap X(\mathbb{R})}=0.
\end{equation}
\end{lemma}
\begin{proof}[Proof of Lemma \ref{09/09/15/15:12}] 

The last assertion (\ref{09/09/15/16:20}) is nontrivial. We prove this, provided that the other properties are proved. We first show that     
\begin{equation}
\label{09/09/15/16:40}
\lim_{l \to \infty}\lim_{n\to \infty}\left\|e^{\frac{i}{2}t\Delta}\left( f_{n}^{l}-f^{l} \right) \right\|_{L^{\infty}(\mathbb{R};L^{\frac{d}{2}(p-1)})}=0.
\end{equation}
Suppose the contrary that there exists a constant $\varepsilon_{0}>0$ such that \begin{equation}\label{10/01/02/11:25}
\limsup_{l \to \infty}\lim_{n\to \infty}\left\|e^{\frac{i}{2}t\Delta}\left( f_{n}^{l}-f^{l} \right) \right\|_{L^{\infty}(\mathbb{R};L^{\frac{d}{2}(p-1)})}\ge \varepsilon_{0}. 
\end{equation}
Then, we can take an increasing sequence of indices $\{j(L)\}_{L\in \mathbb{N}}$ such that 
\begin{equation}\label{10/01/02/11:29}
\lim_{n\to \infty}
\left\|e^{\frac{i}{2}t\Delta}\left( f_{n}^{j(L)}-f^{j(L)} \right) \right\|_{L^{\infty}(\mathbb{R};L^{\frac{d}{2}(p-1)})}\ge \frac{\varepsilon_{0}}{2}.
\end{equation}
Here, it follows from the construction of the family $\{f^{1},f^{2},f^{3}, \ldots \}$ (see also (\ref{10/04/29/12:38}) in the proof of Lemma \ref{08/08/21/14:12}) that there exists a constant $C(\varepsilon_{0})>0$ depending only on 
 $d$, $p$, $\varepsilon_{0}$ and $\displaystyle{\sup_{n\in \mathbb{N}}\left\| f_{n} \right\|_{L^{2}}}$ such that 
\begin{equation}\label{09/01/08/0:51}
\left\|f^{j(L)+1} \right\|_{L^{2}}\ge C(\varepsilon_{0})
\quad 
\mbox{for all $L \in \mathbb{N}$}.
\end{equation}
Then, the uniform bound (\ref{09/02/02/23:40}) (recall that $f_{n}=\psi_{n}(0)$), (\ref{08/08/24/1:47}) and (\ref{09/01/08/0:51}), yield that  
\begin{equation}
\label{08/08/22/15:28}
\begin{split}
\sup_{n\in \mathbb{N}}\left\| f_{n} \right\|_{L^{2}}^{2} 
&\ge \lim_{n\to \infty}\left\{ 
\left\| f_{n} \right\|_{L^{2}}^{2}
-\left\| f_{n}^{{j(M)+1}}-f^{{j(M)+1}}\right\|_{L^{2}}^{2} 
\right\}
=\sum_{k=1}^{j(M)+1}\left\|f^{k} \right\|_{L^{2}}^{2}
\\[6pt]
&\ge \sum_{L=1}^{M}\left\| f^{j(L)+1}\right\|_{L^{2}}^{2} 
\ge MC(\varepsilon_{0})
\quad
\mbox{for all $M\in \mathbb{N}$}.
\end{split}
\end{equation}
Since $\displaystyle{\sup_{n\in \mathbb{N}}\left\| f_{n} \right\|_{L^{2}}<\infty}$, taking $M \to \infty$ in (\ref{08/08/22/15:28}), we have a contradiction. Thus, (\ref{09/09/15/16:40}) holds. 
\par 
It remains to prove that 
\begin{equation}
\label{10/05/11/18:21}
\lim_{l \to \infty}\lim_{n\to \infty}\left\|e^{\frac{i}{2}t\Delta}\left( f_{n}^{l}-f^{l} \right) \right\|_{X(\mathbb{R})}=0.
\end{equation}
 This estimate follows from Lemma \ref{09/04/29/15:57} and (\ref{09/09/15/16:40}). 
\end{proof}
\noindent 
{\it Proof of Lemma \ref{08/08/19/23:07} (continued)} We are back to the proof of Lemma \ref{08/08/19/23:07}. 
\par 
We begin with proving that that $N=1$ in Lemma \ref{09/09/15/15:12}. To this end, we define an approximate solution $\psi_{n}^{app}$ of $\psi_{n}$: Putting $L=N$ if $N<\infty$;  $L$ sufficiently large number specified later exactly if $N=\infty$, we define \begin{equation}\label{09/07/26/23:03}
\psi_{n}^{app}(x,t) := \sum_{l=1}^{L} \psi^{l}(x-\eta_{n}^{l},t-\tau_{n}^{l} ).
\end{equation}
Here, each $\psi^{l}$ is the solution to (\ref{09/09/14/15:32}) (or (\ref{09/09/14/15:33})) with $f^{l}$ just found in Lemma \ref{09/09/15/15:12}, and 
 each $(\eta_{n}^{l},\tau_{n}^{l})$ is the sequence found in Lemma \ref{09/09/15/15:12}, so that we find that 
\begin{align}
\label{09/02/04/8:41}
&\lim_{n\to \infty}\left\| \psi^{l}(-\tau_{n}^{l})-e^{-\frac{i}{2}\tau_{n}^{l}\Delta}f^{l} \right\|_{H^{1}}=0
\quad 
\mbox{for all $1\le l \le L$}, 
\\[6pt]
\label{10/05/16/12:36}
&
\psi^{l}(t) \in PW_{+} 
\quad 
\mbox{for all $1\le l \le L$ and $t \in \mathbb{R}$}.
\end{align}
We note again that if $\tau_{\infty}^{l}=\pm \infty$, then $\psi^{l}$ is the solution to the final value problem (\ref{09/09/14/15:33}); since $f^{l} \in \Omega$ if $\tau_{\infty}^{l}=\pm\infty$ (see (\ref{09/02/03/0:22})),  we actually obtain the desired solution $\psi^{l}$ by Proposition \ref{09/01/12/16:36}. 
\par 
Now, we shall show that     
\begin{equation}\label{09/02/04/11:57}
\| \psi^{l} \|_{X(\mathbb{R})}<\infty
 \quad 
\mbox{for all $1\le l\le L$}, \quad \mbox{if $N\ge 2$}. 
\end{equation}
For (\ref{09/02/04/11:57}), it suffices to show that $\widetilde{\mathcal{N}}_{2}(\psi^{l}(0)) < \widetilde{N}_{c}$ by the definition of $\widetilde{N}_{c}$ (see (\ref{08/09/02/18:06})). Suppose $N\ge 2$, so that $L\ge 2$. Then, we employ (\ref{08/08/24/1:47}), (\ref{08/08/21/18:22}), (\ref{09/03/05/18:36}) and obtain that  
\begin{equation}\label{10/05/12/1:11}
\begin{split}
\widetilde{\mathcal{N}}_{2}(f_{n})
&=
\sqrt{ 
 \left\| f_{n}^{L}-f^{L} \right\|_{L^{2}}^{2}+
\sum_{k=1}^{L} \left\| f^{k} \right\|_{L^{2}}^{2}
+
o_{n}(1)
}^{p+1-\frac{d}{2}(p-1)}
\\[6pt]
& \qquad \times 
\sqrt{ 
\mathcal{H} \left( e^{-\frac{i}{2}\tau_{n}^{L} \Delta} 
 \left( f_{n}^{L}-f^{L} \right) \right)
+
\sum_{k=1}^{L} \mathcal{H} 
\left( e^{-\frac{i}{2}\tau_{n}^{k}\Delta}f^{k} \right) 
+o_{n}(1) 
}^{\frac{d}{2}(p-1)-2}
\\[6pt]
&\ge   
\sqrt{ 
\sum_{k=1}^{L} \left\| f^{k} \right\|_{L^{2}}^{2}
}^{p+1-\frac{d}{2}(p-1)} 
\sqrt{ 
\sum_{k=1}^{L} \mathcal{H} 
\left( e^{-\frac{i}{2}\tau_{n}^{k}\Delta}f^{k} \right) 
}^{\frac{d}{2}(p-1)-2}                             
+o_{n}(1).
\end{split}
\end{equation}
Since (\ref{09/02/03/0:22}) and Lemma \ref{10/04/08/9:20} imply that  
\begin{equation}\label{10/05/12/0:43}
\mathcal{H}\left(e^{-\frac{i}{2}\tau_{n}^{k}\Delta}f^{k}\right) >0 
\quad
\mbox{for all $k\ge 1$ and sufficiently large $n$},
\end{equation}
and since $\psi^{l}$ is the solution to (\ref{09/09/14/15:32}) or (\ref{09/09/14/15:33}), 
we have from (\ref{08/08/26/15:56}) and (\ref{10/05/12/1:11}) that 
\begin{equation}\label{10/04/25/16:27}
\begin{split}
\widetilde{N}_{c}&=
\lim_{n\to \infty}\widetilde{\mathcal{N}}_{2}(f_{n})
\\[6pt]
&>
\lim_{n\to \infty}\left\| f^{l} \right\|_{L^{2}}^{p+1-\frac{d}{2}(p-1)}
\sqrt{ 
\mathcal{H} 
\left( e^{-\frac{i}{2}\tau_{n}^{l}\Delta}f^{l} \right)  
}^{\frac{d}{2}(p-1)-2} 
\\[6pt]
&=
\lim_{n\to \infty}\left\| \psi^{l}(-\tau_{n}) \right\|_{L^{2}}^{p+1-\frac{d}{2}(p-1)}
\sqrt{ 
\mathcal{H} 
\left( \psi^{l}(-\tau_{n}) \right)  
}^{\frac{d}{2}(p-1)-2} 
\\[6pt]
&=
\lim_{n\to \infty}\left\| \psi^{l}(0) \right\|_{L^{2}}^{p+1-\frac{d}{2}(p-1)}
\sqrt{ 
\mathcal{H} 
\left( \psi^{l}(0) \right)  
}^{\frac{d}{2}(p-1)-2} 
\\[6pt]
&=\widetilde{\mathcal{N}}_{2}(\psi^{l}(0))
\qquad 
\mbox{for all $1\le l\le L$ and sufficiently large $n$}, 
\end{split}
\end{equation}
Thus, (\ref{09/02/04/11:57}) holds. 
\par 
We know by (\ref{09/02/04/11:57}) that 
\begin{equation}\label{10/05/12/1:51}
\sup_{n\in \mathbb{N}}\left\| \psi_{n}^{app} \right\|_{X(\mathbb{R})}<\infty.
\end{equation}
Furthermore, we will see that: When $N=\infty$, there exists $A>0$ with the property that for any $L\in \mathbb{N}$ (the number of components of $\psi_{n}^{app}$, see (\ref{09/07/26/23:03})), there exists $n_{L} \in \mathbb{N}$ such that \begin{equation}\label{09/09/18/14:46}
\sup_{n\ge n_{L}}\left\| \psi_{n}^{app} \right\|_{X(\mathbb{R})}\le A .
\end{equation} 
We shall prove this fact. Recall that $X(\mathbb{R})=L^{r_{1}}(\mathbb{R};L^{q_{1}})\cap L^{r_{2}}(I;L^{q_{2}})$. Lemma \ref{09/03/06/16:46} yields that  
\begin{equation}\label{09/09/21/18:19}
\begin{split}
&\left\| \psi_{n}^{app}(t) \right\|_{L^{q_{j}}}^{q_{j}}
= 
\left\| \sum_{l=1}^{L}\psi^{l}(\cdot -\eta_{n}^{l}, t-\tau_{n}^{l}) \right\|_{L^{q_{j}}}^{q_{j}}
\\[6pt]
&\le 
\sum_{l=1}^{L}\left\|\psi^{l}(t-\tau_{n}^{l}) \right\|_{L^{q_{j}}}^{q_{j}}
\\[6pt]
&\quad +
C_{L}\sum_{l=1}^{L}\sum_{{k=1}\atop {k\neq l}}^{L}
\int_{\mathbb{R}^{d}} \left| \psi^{k}(x-\eta_{n}^{k},t-\tau_{n}^{k}) \right|
\left| \psi^{l}(x-\eta_{n}^{l},t-\tau_{n}^{l})\right|^{q_{j}-1}
\hspace{-0.2cm}dx
\quad  
\mbox{for $j=1,2$}
\end{split}
\end{equation}
for some constant $C_{L}>0$ depending only on $d$, $p$, $q_{1}$ and $L$. 
Therefore, we have by the triangle inequality that 
\begin{equation}\label{09/09/21/18:17}
\begin{split}
&\left\| \psi_{n}^{app} \right\|_{
L^{r_{j}} (\mathbb{R};L^{q_{j}}) }^{q_{j}}
=
\left\| \left\| \psi_{n}^{app} \right\|_{L^{q_{j}}}^{q_{j}} 
\right\|_{L^{\frac{r_{j}}{q_{j}}}(\mathbb{R})}
\\[6pt]
&\le 
\sum_{l=1}^{L}
\left\| \psi^{l}
\right\|_{L^{r_{j}}(\mathbb{R};L^{q_{j}})}^{q_{j}}
\\[6pt]
&\quad  +
C_{L} \sum_{l=1}^{L}\sum_{{k=1}\atop {k\neq l}}^{L}
\left\|
\left| \psi^{k}(\cdot-\eta_{n}^{k},\cdot-\tau_{n}^{k}) \right|
\hspace{-2pt}
\left| \psi^{l}(\cdot-\eta_{n}^{l},\cdot-\tau_{n}^{l})\right|^{q_{j}-1}
\right\|_{L^{\frac{r_{j}}{q_{j}}}(\mathbb{R};L^{1})}
\\[6pt]
&=:I_{j}+II_{j}
\qquad 
\mbox{for $j=1,2$}.
\end{split}
\end{equation}
We first consider the term $I_{j}$. The formula (\ref{08/08/24/1:47}) and the uniform bound (\ref{09/02/02/23:40})  show that 
\begin{equation}\label{09/09/18/14:55}
\sum_{l=1}^{\infty}\left\| f^{l} \right\|_{H^{1}}^{2} < \infty,
\end{equation}
so that  
\begin{equation}\label{09/09/18/14:56}
\lim_{l\to \infty}\left\| f^{l} \right\|_{H^{1}}= 0.
\end{equation}
Hence, it follows from Proposition \ref{08/08/22/20:59} (when $\tau_{\infty}^{l}\in \mathbb{R}$) and  Proposition \ref{09/01/12/16:36} (when $\tau_{n}^{l}=\pm \infty$) that: There exists $l_{0} \in \mathbb{N}$, independent of $L$ (the number of components of $\psi_{n}^{app}$), such that   
\begin{equation}\label{09/09/23/13:43}
\| \psi^{l} \|_{X(\mathbb{R})} 
\lesssim  
\| f^{l} \|_{H^{1}}\le 1
\quad 
\mbox{for all $l \ge l_{0}$}, 
\end{equation}
where the implicit constant depends only on $d$, $p$ and $q_{1}$. 
 Then, (\ref{09/02/04/11:57}), (\ref{09/09/18/14:55}) and (\ref{09/09/23/13:43}), together with the fact $q_{j}> 2$ for $j=1,2$, lead to that      
\begin{equation}\label{09/09/23/14:07}
\begin{split}
I_{j}&= \sum_{l=1}^{l_{0}}
\left\| \psi^{l} \right\|_{L^{r_{j}}(\mathbb{R};L^{q_{j}})}^{q_{j}}
 + 
\sum_{l=l_{0}+1}^{\infty}\left\| \psi^{l} \right\|_{L^{r_{j}}
(\mathbb{R};L^{q_{j}})}^{q_{j}}
\\[6pt]
&\lesssim  
\sum_{l=1}^{l_{0}}\left\| \psi^{l} \right\|_{X(\mathbb{R})}^{q_{j}}
+
\sum_{l=l_{0}+1}^{\infty} 
\left\| f^{l} \right\|_{H^{1}}^{2}
< \infty, 
\end{split}
\end{equation}
where the implicit constant depends only on $d$, $p$ and $q_{1}$. 
\par 
Next, we consider the term $II_{j}$. Lemma \ref{08/08/23/16:32}, together with the condition (\ref{09/02/03/0:50}), implies that: There exists $n_{L} \in \mathbb{N}$ such that 
\begin{equation}\label{10/05/12/2:36}
\left\|\hspace{-2pt}
\left| \psi^{k}(\cdot-\eta_{n}^{k},\cdot-\tau_{n}^{k}) \right|
\hspace{-3pt}
\left| \psi^{l}(\cdot-\eta_{n}^{l},\cdot-\tau_{n}^{l})\right|^{q_{j}-1}
\right\|_{L^{\frac{r_{j}}{q_{j}}}(\mathbb{R};L^{1})}
\hspace{-6pt}
\le \frac{1}{C_{L}L^{2}}
\quad 
\mbox{for all $n \ge n_{L}$ and $k\neq l$}.
\end{equation} 
Hence, we see that  
\begin{equation}\label{09/09/23/14:09}
II_{j} \le 1
\quad 
\mbox{for all $n \ge n_{L}$}.
\end{equation}
Thus, we have proved (\ref{09/09/18/14:46}). 
\par 
We shall show that the case $N\ge 2$ can not occur. Note that $\psi_{n}^{app}$ solves the following equation:  
\begin{equation}\label{09/09/18/15:20}
2i\frac{\partial }{\partial t} \psi_{n}^{app}+\Delta \psi_{n}^{app} +|\psi_{n}^{app}|^{p-1}\psi_{n}^{app} =e_{n},
\end{equation}
where 
\begin{equation}\label{10/05/12/15:40}
e_{n}(x,t):= |\psi_{n}^{app}(x,t)|^{p-1}\psi_{n}^{app}(x,t)
-\sum_{l=1}^{L}\left|\psi^{l}(x-\eta_{n}^{l},t-\tau_{n}^{l})\right|^{p-1}
\!\! \psi^{l}(x-\eta_{n}^{l},t-\tau_{n}^{l}).
\end{equation}
Proposition \ref{08/08/05/14:30} (long time perturbation theory), with the help of (\ref{09/09/18/14:46}), tells us that  there exists $\varepsilon_{1}>0$, independent of $L$ when $N=\infty$, with the following property: If there exists $n\ge n_{L}$  ($n_{L}$ is the number found in (\ref{09/09/18/14:46}) if $N=\infty$, $n_{L}=1$ if $N<\infty$) such that     
\begin{equation}\label{09/09/16/2:10}
\left\| e^{\frac{i}{2}t\Delta}\left( \psi_{n}(0)-\psi_{n}^{app}(0)
\right)\right\|_{X(\mathbb{R})}
\le \varepsilon_{1}
\end{equation}
and 
\begin{equation}\label{09/09/16/2:11}
\left\| e_{n} \right\|_{L^{\widetilde{r}'}(\mathbb{R};L^{q_{1}'})}\le  \varepsilon_{1}, 
\end{equation}
then 
\begin{equation}\label{10/05/12/15:48}
\left\| \psi_{n} \right\|_{X(\mathbb{R})}<\infty.
\end{equation}
In the sequel, we show that if $N\ge 2$, then (\ref{09/09/16/2:10}) and (\ref{09/09/16/2:11}) hold valid for some $L$, which shows $N=1$ since (\ref{10/05/12/15:48}) contradicts (\ref{09/02/17/14:53}). It is worth while noting here that 
\begin{equation}\label{08/08/23/22:06}
f_{n}
=\sum_{l=1}^{L}e^{-\frac{i}{2}\tau_{n}^{l}\Delta}e^{-\eta_{n}^{l}\cdot \nabla }f^{l}+e^{-\frac{i}{2}\tau_{n}^{L}\Delta}e^{-\eta_{n}^{L}\cdot \nabla }\left( f_{n}^{L}-f^{L}\right),
\end{equation}
in other words,  
\begin{equation}\label{09/09/15/17:42}
e^{\frac{i}{2}t\Delta}f_{n}
=\sum_{l=1}^{L}e^{\frac{i}{2}(t-\tau_{n}^{l})\Delta}e^{-\eta_{n}^{l}\cdot \nabla }f^{l}+e^{\frac{i}{2}(t-\tau_{n}^{L})\Delta}e^{-\eta_{n}^{L}\cdot \nabla }\left( f_{n}^{L}-f^{L}\right).
\end{equation}

Invoking Lemma \ref{08/08/23/16:32}, we find that (\ref{09/09/16/2:11}) holds for all sufficiently large $n$.  We next consider (\ref{09/09/16/2:10}). 
 The formula (\ref{09/09/15/17:42}), with the help of (\ref{08/10/25/23:16}), shows that  
\begin{equation}\label{09/09/16/12:26}
\begin{split}
&\left\| e^{\frac{i}{2}t\Delta}
\left( \psi_{n}(0) -\psi_{n}^{app}(0) 
\right)\right\|_{X(\mathbb{R})}
= 
\left\| 
e^{\frac{i}{2}t\Delta}
\left( f_{n}-\sum_{l=1}^{L}e^{-\eta_{n}^{l}\cdot \nabla}\psi^{l}(-\tau_{n}^{l})\right)
\right\|_{X(\mathbb{R})}
\\[6pt]
&\le 
\left\| 
e^{\frac{i}{2}(t-\tau_{n}^{L})\Delta}e^{-\eta_{n}^{L}\cdot \nabla }\left( f_{n}^{L}-f^{L}\right) 
\right\|_{X(\mathbb{R})}
\\[6pt]
&\qquad +
\sum_{l=1}^{L}\left\| e^{\frac{i}{2}t\Delta} 
\left( e^{-\frac{i}{2}\tau_{n}^{l}\Delta}e^{-\eta_{n}^{l}\cdot \nabla }f^{l}
-e^{-\eta_{n}^{l}\cdot \nabla}\psi^{l}(-\tau_{n}^{l}) \right)\right\|_{X(\mathbb{R})}
\\[6pt]
&\le 
\left\| 
e^{\frac{i}{2}t \Delta}\left( f_{n}^{L}-f^{L}\right) 
\right\|_{X(\mathbb{R})}
+ C
\sum_{l=1}^{L}\left\| e^{-\frac{i}{2}\tau_{n}^{l}\Delta}f^{l}
-\psi^{l}(-\tau_{n}^{l}) \right\|_{H^{1}},
\end{split}
\end{equation}
where $C$ is some constant depending only on $d$, $p$ and $q_{1}$.
Here, we have by (\ref{09/09/15/16:19}) and (\ref{09/09/15/16:20}) that 
\begin{equation}\label{09/09/27/22:08}
\begin{split}
&\lim_{n\to \infty}\left\|e^{\frac{i}{2}t\Delta}\left( f_{n}^{L}-f^{L}\right) \right\|_{X(\mathbb{R})}
\le \frac{\varepsilon_{1}}{4}
\\[6pt] 
&\hspace{3cm}
\mbox{for $L=N$ if $N<\infty$, and sufficiently large $L$ if $N=\infty$}.
\end{split}
\end{equation}
Hence, for all $L \in \mathbb{N}$ satisfying (\ref{09/09/27/22:08}), there exists $n_{L,1} \in \mathbb{N}$ such that 
\begin{equation}\label{09/09/27/22:25}
\left\| 
e^{\frac{i}{2}t \Delta}\left( f_{n}^{L}-f^{L}\right) 
\right\|_{X(\mathbb{R})}\le \frac{\varepsilon_{1}}{2}
\quad 
\mbox{for all $n\ge n_{L,1}$}.
\end{equation} 
Moreover, (\ref{09/02/04/8:41}) shows that for all $L \le N$, there exists $n_{L,2}\in \mathbb{N}$  such that  
\begin{equation}\label{09/09/16/2:37}
\left\| 
e^{-\frac{i}{2}\tau_{n}^{l}\Delta}f^{l}
-\psi^{l}(-\tau_{n}^{l})
\right\|_{H^{1}}\le \frac{\varepsilon_{1}}{2CL}
\quad 
\mbox{for all $n\ge n_{L,2}$ and $1\le l \le L$},
\end{equation}
where $C$ is the constant found in (\ref{09/09/16/12:26}). Combining (\ref{09/09/16/12:26}) with (\ref{09/09/27/22:25}) and (\ref{09/09/16/2:37}), we see that 
 for all $L \in \mathbb{N}$ satisfying (\ref{09/09/27/22:08}), there exists $n_{L,3} \in \mathbb{N}$ such that     
\begin{equation}\label{09/02/04/9:03}
\left\| e^{\frac{i}{2}t\Delta}
\left( \psi_{n}(0) -\psi_{n}^{app}(0) 
\right)\right\|_{X(\mathbb{R})} \le  \varepsilon_{1}
\quad 
\mbox{for all $n\ge n_{L,3}$},
\end{equation}
which gives (\ref{09/09/16/2:10}). 
\par 
We have just proved $N=1$, and therefore $L$ should be one; 
\[
\psi_{n}^{app}(x,t)=\psi^{1}(x-\eta_{n}^{1},t-\tau_{n}^{1})=
\left( e^{-\tau_{n}^{1}\frac{\partial}{\partial t}-\eta_{n}^{1}
\cdot \nabla } \psi^{1} \right)(x,t).
\]
Put $\Psi =\psi^{1}$, $f=f^{1}$, $(\gamma_{n},\tau_{n})=(\gamma_{n}^{1},\tau_{n}^{1})$ and $\tau_{\infty}=\tau_{\infty}^{1}$. Then, these are what we want. Indeed, we have already shown that these satisfy the properties (\ref{09/05/02/10:12}), (\ref{09/05/02/10:10}), (\ref{09/05/02/10:08}), (\ref{09/05/02/10:09}); see (\ref{10/05/16/12:36}) for (\ref{09/05/02/10:12}), (\ref{08/08/21/18:17}) for (\ref{09/05/02/10:10}), (\ref{09/09/15/16:19}) for (\ref{09/05/02/10:08}), and (\ref{09/02/04/8:41}) for (\ref{09/05/02/10:09}). Moreover, the property (\ref{10/04/15/20:43}) immediately follows from (\ref{09/05/02/10:09}).
\par 
It remains to prove (\ref{09/05/02/10:11}), (\ref{10/05/13/10:09}), (\ref{09/05/02/10:13}), (\ref{09/07/27/1:49}) and (\ref{09/05/02/10:07}).
\par 
We first prove that $\Psi$ satisfies the property (\ref{09/05/02/10:11}):
$\left\|\Psi \right\|_{X(\mathbb{R})}=\infty$.
Suppose the contrary that $\left\|\Psi \right\|_{X(\mathbb{R})}<\infty$. Then, a quite similar argument above works well, so that we obtain an absurd conclusion 
\begin{equation}\label{10/05/12/16:34}
\left\| \psi_{n} \right\|_{X(\mathbb{R})}<\infty 
\quad 
\mbox{for sufficiently large $n \in \mathbb{N}$}.
\end{equation}
Hence, we have proved (\ref{09/05/02/10:11}). 
\par 
In order to prove (\ref{09/05/02/10:13}) and (\ref{09/07/27/1:49}), we shall show that there exists a subsequence of $\{f_{n}\}$ (still denoted by the same symbol) such that    
\begin{align}
&\label{10/05/12/20:54}
\left\| f \right\|_{L^{2}}
=
\lim_{n\to \infty}\left\| f_{n} \right\|_{L^{2}},
\\[6pt]
&\label{10/05/12/20:53}
\lim_{n\to \infty}
\mathcal{H}\left(e^{-\frac{i}{2}\tau_{n}\Delta}f\right)
=
\lim_{n\to \infty} \mathcal{H}\left(f_{n}\right).
\end{align}
Since $f_{n}^{1}
=
e^{\frac{i}{2}\tau_{n}^{1}\Delta}
e^{\eta_{n}^{1}\cdot \nabla}f_{n}(0)$ and $f=f^{1}$, the weak convergence result (\ref{08/08/21/18:17}), with the help of extraction of some subsequence, leads to that  
\begin{equation}\label{09/10/08/14:09}
\left\| f \right\|_{L^{2}}
\le 
\lim_{n\to \infty}\left\| f_{n} \right\|_{L^{2}}.
\end{equation}
Extracting some subsequence further, we also have by  (\ref{08/08/21/18:22}) and (\ref{09/03/05/18:36}) that 
\begin{equation}\label{09/10/08/14:10}
\lim_{n\to \infty}
\mathcal{H}\left(e^{-\frac{i}{2}\tau_{n}\Delta}f\right)
\le
\lim_{n\to \infty} \mathcal{H}\left(f_{n}\right).
\end{equation}
Here, we employ the mass conservation law (\ref{08/05/13/8:59}) and (\ref{09/05/02/10:09}) (or the formula (\ref{09/02/04/8:41}) with $l=1$) to obtain that 
\begin{equation}\label{09/02/18/18:15}
\left\| \Psi(0) \right\|_{L^{2}}=\lim_{n\to \infty}
\left\| \Psi(-\tau_{n})\right\|_{L^{2}}=\left\| f \right\|_{L^{2}}.
\end{equation}
Moreover, the energy conservation law (\ref{08/05/13/9:03}) and (\ref{09/02/04/8:41}) with $l=1$ give us that 
\begin{equation}\label{09/02/18/18:18}
\mathcal{H}(\Psi(0))=\lim_{n\to \infty}\mathcal{H}(\Psi(-\tau_{n}))
=\lim_{n\to \infty}\mathcal{H}(e^{-\frac{i}{2}\tau_{n}\Delta}f).
\end{equation}
Hence, supposing the contrary that (\ref{10/05/12/20:54}) or (\ref{10/05/12/20:53}) fails, we have by the minimizing property (\ref{08/08/26/15:56})  that  
\begin{equation}\label{10/05/12/18:23}
\begin{split}
\widetilde{\mathcal{N}}_{2}(\Psi(0)) 
&=
\left\| f \right\|_{L^{2}}^{p+1-\frac{d}{2}(p-1)}
\lim_{n\to \infty} \sqrt{\mathcal{H}\left(e^{-\frac{i}{2}\tau_{n}\Delta}f\right)}^{\frac{d}{2}(p-1)-2}
\\[6pt]
&< 
\lim_{n\to \infty}\left\| f_{n} \right\|_{L^{2}}^{p+1-\frac{d}{2}(p-1)}
\lim_{n\to \infty} \sqrt{\mathcal{H}\left(f_{n}\right)}^{\frac{d}{2}(p-1)-2}
\\[6pt]
&= \lim_{n\to \infty}\widetilde{\mathcal{N}}_{2}(f_{n})
=
\widetilde{N}_{c}.
\end{split}
\end{equation}
This estimate, together with the definition of $\widetilde{N}_{c}$ (see (\ref{08/09/02/18:06})), leads to that $\left\|\Psi \right\|_{X(\mathbb{R})}<\infty$, which contradicts $\left\|\Psi \right\|_{X(\mathbb{R})}=\infty$. 
Thus, (\ref{10/05/12/20:54}) and (\ref{10/05/12/20:53}) holds. 
Then, an estimate similar to (\ref{10/05/12/18:23}) 
gives (\ref{09/05/02/10:13}). Moreover, (\ref{09/07/27/1:49}) follows from (\ref{08/08/24/1:47}) with $l=1$, $s=0$ and (\ref{10/05/12/20:54}).  
\par 
The properties (\ref{09/07/27/1:49}) and (\ref{09/05/02/10:09}) (or (\ref{10/04/15/20:43})), together with (\ref{10/05/13/10:06}) and the mass conservation law (\ref{08/05/13/8:59}), yield (\ref{10/05/13/10:09}). Moreover, the formula (\ref{09/05/02/10:07}) follows from the estimate (\ref{09/09/16/12:26}) with $N=1$, (\ref{09/09/15/16:19}) and (\ref{09/02/04/8:41}). 
\end{proof}

\subsection{Proof of Proposition \ref{08/10/16/21:12}}\label{08/09/25/17:01}
We shall prove Proposition \ref{08/10/16/21:12}, showing that the candidate $\Psi$ found in Lemma \ref{08/08/19/23:07} is actually one-soliton-like solution as $t\to +\infty$. 

\begin{proof}[Proof of Proposition \ref{08/10/16/21:12}]
The properties (\ref{08/10/20/2:27}) and (\ref{08/10/18/22:17}) have been obtained in Lemma \ref{08/08/19/23:07}. Moreover, (\ref{09/06/21/21:58}) in Proposition \ref{09/06/21/19:28}, together with (\ref{08/10/20/2:27}) and (\ref{08/10/18/22:17}), yields (\ref{10/04/05/18:07}).
\par 
We shall prove (\ref{08/08/18/17:42}): the momentum of $\Psi$ is zero. Using the Galilei transformation, we set
\begin{equation}\label{10/05/16/21:33}
\Psi_{\xi}(x,t):=e^{\frac{i}{2}(x\xi-t|\xi|^{2})}\Psi(x-\xi t, t), \quad \xi \in \mathbb{R}^{d}.
\end{equation}
It is easy to verify that 
\begin{align}
\label{08/08/31/15:33}
&\left\| \Psi_{\xi}(t)\right\|_{L^{q}}=\left\| \Psi(t)\right\|_{L^{q}} 
\quad 
\mbox{for all $\xi \in \mathbb{R}^{d}$, $q\in [2,2^{*})$ and $t\in \mathbb{R}$},\\[6pt]
\label{08/08/31/15:03}
&\left\| \Psi_{\xi}\right\|_{X(\mathbb{R})}=\left\| \Psi \right\|_{X(\mathbb{R})}=\infty 
\quad 
\mbox{for all $\xi \in \mathbb{R}^{d}$}.
\end{align}
Moreover, a simple calculation, together with the mass and momentum conservation laws (\ref{08/05/13/8:59}) and (\ref{08/10/20/4:37}), shows that 
\begin{equation}\label{10/05/16/21:55}
\begin{split}
&\left\|\nabla \Psi_{\xi}(t) \right\|_{L^{2}}^{2}
=
\left\| i\xi \Psi_{\xi}(t) +e^{\frac{i}{2}(x \xi -t|\xi|^{2})}(\nabla \Psi)(x-\xi t, t)\right\|_{L_{x}^{2}}^{2}
\\[6pt]
&=
\int_{\mathbb{R}^{d}}
|\xi|^{2}|\Psi_{\xi}(x,t)|^{2}\,dx 
+ \int_{\mathbb{R}^{d}}|\nabla \Psi|^{2}(x,t)\,dt 
\\[6pt]
& \qquad \quad 
+ 2\Re \int_{\mathbb{R}^{d}} \overline{i\xi \Psi_{\xi} }(x,t) e^{\frac{i}{2}(x \xi -t|\xi|^{2})}(\nabla \Psi)(x-\xi t, t)\,dx
\\[6pt]
&=|\xi|^{2}\left\| \Psi(t)\right\|_{L^{2}}^{2}+\left\| \nabla \Psi(t) \right\|_{L^{2}}^{2}
+
2\xi\cdot \Im \int_{\mathbb{R}^{d}} \overline{\Psi}(x,t) \nabla \Psi(x,t)\,dx 
\\[6pt]
&=
|\xi|^{2}\left\| \Psi(0)\right\|_{L^{2}}^{2}+\left\| \nabla \Psi(t) \right\|_{L^{2}}^{2}
+
2\xi\cdot \Im \int_{\mathbb{R}^{d}} \overline{\Psi}(x,0) \nabla \Psi(x,0)\,dx
\quad 
\mbox{for all $\xi \in \mathbb{R}^{d}$}.
\end{split}
\end{equation}
This equality (\ref{10/05/16/21:55}), together with the energy conservation law  (\ref{08/05/13/9:03}), yields that 
\begin{equation}\label{10/05/16/22:22}
\begin{split}
&\mathcal{H}(\Psi_{\xi}(t))
=
\left\| \nabla \Psi_{\xi}(t)\right\|_{L^{2}}^{2}
-\frac{2}{p+1}\left\| \Psi_{\xi}(t) \right\|_{L^{p+1}}^{p+1}
\\[6pt]
&=\left\| \nabla \Psi_{\xi}(t)\right\|_{L^{2}}^{2}
-\frac{2}{p+1}\left\| \Psi(t) \right\|_{L^{p+1}}^{p+1}
\\[6pt]
&= 
\mathcal{H}(\Psi(0))
+|\xi|^{2}\left\|\Psi(0) \right\|_{L^{2}}^{2}
+
2\xi\cdot \Im \int_{\mathbb{R}^{d}} \overline{\Psi}(x,0)\nabla \Psi(x,0)\,dx
\quad
\mbox{for all $\xi \in \mathbb{R}^{d}$}.
\end{split}
\end{equation}
Put  
\begin{equation}\label{08/08/31/15:28}
\xi_{0}=-\frac{\Im \int_{\mathbb{R}^{d}} \overline{\Psi}(x,0) \nabla \Psi(x,0)\,dx}{\left\| \Psi(0) \right\|_{L^{2}}}.
\end{equation} 
Then, we have by (\ref{10/05/16/21:55}) and (\ref{10/05/16/22:22}) that 
\begin{align}
\label{08/08/31/15:43}
&\left\| \nabla \Psi_{\xi_{0}} (t)\right\|_{L^{2}}^{2} = \left\|\nabla \Psi(t) \right\|_{L^{2}}^{2}- 
\left( \Im \int_{\mathbb{R}^{d}} \overline{\Psi}(x,0) \nabla \Psi(x,0)\,dx \right)^{2}
\ 
\mbox{for all $t \in \mathbb{R}$},
\\[6pt]
\label{08/08/31/15:31}
&\mathcal{H}(\Psi_{\xi_{0}}(t))
=
\mathcal{H}(\Psi(0))- 
\left( \Im \int_{\mathbb{R}^{d}} \overline{\Psi}(x,0) \nabla \Psi(x,0)\,dx \right)^{2} 
\  
\mbox{for all $t \in \mathbb{R}$}.
\end{align}
Now, we suppose that 
\begin{equation}\label{10/05/16/22:43}
\Im \int_{\mathbb{R}^{d}} \overline{\Psi}(x,t) \nabla \Psi(x,t)\,dx 
=
\Im \int_{\mathbb{R}^{d}} \overline{\Psi}(x,0) \nabla \Psi(x,0)\,dx
\neq 0.
\end{equation}
Then, by (\ref{08/08/31/15:33}) with $q=2$, (\ref{08/08/31/15:43}), (\ref{09/06/21/21:58}) in Proposition \ref{09/06/21/19:28} and (\ref{08/10/20/2:27}), we have that 
\begin{equation}\label{10/08/21/18:04}
\begin{split}
\mathcal{N}_{2}(\Psi_{\xi_{0}}(t))
&< 
\mathcal{N}_{2}(\Psi(t))
\\[6pt]
&\le \sqrt{\frac{d(p-1)}{d(p-1)-4}}^{\frac{d}{2}(p-1)-2}
\widetilde{\mathcal{N}}_{2}(\Psi(t))
= 
\sqrt{\frac{d(p-1)}{d(p-1)-4}}^{\frac{d}{2}(p-1)-2}\widetilde{N}_{c}
\\[6pt]
&<  \sqrt{\frac{d(p-1)}{d(p-1)-4}}^{\frac{d}{2}(p-1)-2}\widetilde{N}_{2}
=N_{2}
\qquad \mbox{for all $t \in \mathbb{R}$}.
\end{split}
\end{equation}
This inequality, together with the definition of $N_{2}$ (see \ref{08/07/02/23:24}), implies that  

\begin{equation}\label{10/05/16/22:47}
\mathcal{K}(\Psi_{\xi_{0}}(t))
>0 
\quad 
\mbox{for all $t \in \mathbb{R}$}.
\end{equation}
Moreover, using (\ref{08/08/31/15:33}) with $q=2$, (\ref{08/08/31/15:31}), and (\ref{08/10/20/2:27}) again, we obtain that  
\begin{equation}\label{10/05/16/22:46}
\widetilde{\mathcal{N}}_{2}(\Psi_{\xi_{0}}(t)) < \widetilde{N}_{c}.
\end{equation}
Hence, it follows from  the definition of $\widetilde{N}_{c}$ that 
$\left\| \Psi_{\xi_{0}}\right\|_{X(\mathbb{R})}< \infty$, which contradicts (\ref{08/08/31/15:03}). Thus, the momentum of $\Psi$ must be zero.\footnote{The proof of zero momentum is quite similar to the one of Proposition 4.1 in \cite{Holmer-Roudenko}. We can find an analogous trick in Appendix D of \cite{Nawa8}.}
\par
Next, we shall prove the tightness of either family $\{\Psi(t)\}_{t\in [0,\infty)}$ or $\{\Psi(t)\}_{t\in (-\infty,0]}$ in $H^{1}(\mathbb{R}^{d})$. Since $\left\|\Psi \right\|_{X(\mathbb{R}))}=\infty$, it holds that   
$\left\|\Psi \right\|_{X([0,\infty))}=\infty$ or $\left\|\Psi \right\|_{X((-\infty,0]))}=\infty$. The time reversibility of our equation (\ref{08/05/13/8:50}) allows us to assume that $\left\|\Psi \right\|_{X([0,\infty))}=\infty$; if not, we consider $\overline{\Psi(x,-t)}$ instead of $\Psi(x,t)$. 
\par 
We introduce the following quantity to employ Proposition \ref{08/10/03/9:46}: 
\begin{equation}
\label{09/07/28/19:58}
A:=\sup_{R>0}\liminf_{t\to \infty}\sup_{y\in \mathbb{R}^{d}} \int_{|x-y|\le R}
\left| \Psi(x,t) \right|^{2} \,dx.
\end{equation}
By the definition of $A$, we can take a sequence $\{t_{n}\}$ in $[0,\infty)$ such that 
\begin{align}
\label{10/05/17/1:11}
&\lim_{n\to \infty}t_{n}=+\infty,
\\[6pt]
\label{08/10/18/15:36}
&\sup_{y\in \mathbb{R}^{d}}\int_{|x-y|\le R} \left| \Psi(x,t_{n}) 
\right|^{2} \,dx
\le 
\left( 1+\frac{1}{n} \right)A
\quad 
\mbox{for all $R>0$ and $n \in \mathbb{N}$}.
\end{align}
Put  
\begin{equation}\label{09/09/16/19:19}
\Psi_{n}(x,t):=\Psi(x,t+t_{n}).
\end{equation}
\noindent 
We investigate the behavior of this $\Psi_{n}$ in order to show $A=1$. Note here that since $\|\Psi(t)\|_{L^{2}}=1$ for all $t \in \mathbb{R}$, we have $A\le 1$, so that $A\ge 1$ implies $A=1$. 
\par 
The sequence $\{\Psi_{n}(0)\}$ satisfies the conditions of Lemma \ref{08/08/19/23:07}:
\begin{align}
\label{09/02/17/16:33}
&\Psi_{n}(t) \in PW_{+}
\quad 
\mbox{for all $t \in \mathbb{R}$}, 
\\[6pt]
\label{10/05/17/22:50}
&\left\| \Psi_{n}(t)\right\|_{L^{2}}=1
\quad 
\mbox{for all $n\in \mathbb{N}$ and $t \in \mathbb{R}$}, 
\\[6pt]
\label{09/02/17/16:34}
&\sup_{n\in \mathbb{N}}\left\| \Psi_{n}(0) \right\|_{H^{1}}<\infty, 
\\[6pt]
\label{09/02/17/16:35}
&\widetilde{\mathcal{N}}_{2}(\Psi_{n}(t))=\widetilde{N}_{c}
\quad 
\mbox{for all $n \in \mathbb{N}$ and $t \in \mathbb{R}$},
\\[6pt]
\label{09/02/17/16:36}
&\left\| \Psi_{n} \right\|_{X(\mathbb{R})}=\infty  
\quad 
\mbox{for all $n \in \mathbb{N}$}, 
\\[6pt]
\label{09/02/17/16:21}
&\limsup_{n\to \infty}\left\| e^{\frac{i}{2}t\Delta} \Psi_{n}(0)\right\|_{ L^{\infty}(\mathbb{R};L^{\frac{d}{2}(p-1)})}>0.
\end{align}
The assertion of (\ref{09/02/17/16:21}) may need to be verified. If (\ref{09/02/17/16:21}) fails, then we have 
\begin{equation}\label{09/09/16/19:24}
\lim_{n\to \infty}\left\| e^{\frac{i}{2}t\Delta}
\Psi(t_{n})\right\|_{L^{\infty}(\mathbb{R};L^{\frac{d}{2}(p-1)})}=0,
\end{equation}
which leads us to the contradictory conclusion $\left\|\Psi \right\|_{X(\mathbb{R})}<\infty$ by the small data theory (Proposition \ref{08/08/22/20:59}). 
\par 
We apply Lemma \ref{08/08/19/23:07} to $\{\Psi_{n}\}$, so that there exists a subsequence of $\{\Psi_{n}\}$ (still denoted by the same symbol) with the following properties: There exists  
\\
(i) a nontrivial global solution $\Psi_{\infty} \in C(\mathbb{R};H^{1}(\mathbb{R}^{d}))$ to the equation (\ref{08/05/13/8:50}) with 
\begin{align}
\label{09/09/17/14:55}
&\left\| \Psi_{\infty} \right\|_{X(\mathbb{R})}=\infty,
\\[6pt]
\label{10/05/31/9:00} 
&\Psi_{\infty}(t) \in PW_{+} 
\quad
\mbox{for all $t \in \mathbb{R}$},
\\[6pt]  
\label{10/05/17/1:29} 
&\widetilde{\mathcal{N}}_{2}(\Psi_{\infty}(t))=\widetilde{N}_{c}
\quad 
\mbox{for all $t \in \mathbb{R}$},
\end{align}
and 
\\
(ii) a nontrivial function $\Phi \in PW_{+}$,  a sequence $\{\tau_{n}\}$ in $\mathbb{R}$ with $\displaystyle{\lim_{n\to \infty}\tau_{n}= 
\tau_{\infty} \in \mathbb{R}\cup \{\pm \infty\}}$,  and a sequence $\{\gamma_{n}\}$ in $\mathbb{R}^{d}$, such that, putting 
\begin{equation}\label{10/05/17/1:31}
\widetilde{\Psi}_{n}:=e^{\frac{i}{2}\tau_{n}\Delta}e^{\gamma_{n}\cdot \nabla}
\Psi_{n},
\quad 
\widetilde{\epsilon}_{n}:=
\Psi_{n}-e^{-\tau_{n}\frac{\partial}{\partial t}-\gamma_{n}\cdot \nabla}
\Psi_{\infty},
\end{equation}
we have:
\begin{align}
\label{09/02/16/15:35}
& \lim_{n\to \infty}\widetilde{\Psi}_{n}(0)=\Phi
\quad 
\mbox{weakly in $H^{1}(\mathbb{R}^{d})$, and a.e. in $\mathbb{R}^{d}$}, 
\\[6pt]
\label{10/05/17/1:47}
&\lim_{n\to \infty}
\left\| \widetilde{\Psi}_{n}(0) - \Phi 
\right\|_{L^{2}}=0,
\\[6pt]
\label{09/02/17/16:44}
&\lim_{n\to \infty}\left\| e^{\frac{i}{2}t\Delta}\left( \widetilde{\Psi}_{n}(0)-\Phi \right)\right\|_{L^{\infty}(\mathbb{R};L^{\frac{d}{2}(p-1)})\cap X(\mathbb{R})}=0,
\\[6pt]
\label{09/02/24/14:28}
&\lim_{n\to \infty}\left\|\Psi_{\infty}(-\tau_{n})-e^{-\frac{i}{2} \tau_{n}\Delta}\Phi \right\|_{H^{1}}=0,
\\[6pt]
\label{08/08/27/13:16}
&\lim_{n\to \infty} \left\|e^{\frac{i}{2}t\Delta}\widetilde{\epsilon}_{n}(0) 
\right\|_{ X(\mathbb{R})}=0.
\end{align}
Note here that the estimate (\ref{08/10/25/23:16}) gives us that 
\begin{equation}\label{09/03/05/17:39}
\lim_{T\to \infty}\left\|e^{\frac{i}{2}t\Delta} \Phi \right\|_{X((-\infty,-T])}
=   
\lim_{T\to \infty}
\left\| e^{\frac{i}{2}t\Delta} \Phi \right\|_{X([T,+\infty))}
=0.
\end{equation}
We claim that $\tau_{\infty} \in \mathbb{R}$. 
We suppose the contrary that $\tau_{\infty}=+ \infty$ or $\tau_{\infty}=-\infty$. If $\tau_{\infty}=+ \infty$, then (\ref{09/02/17/16:44}) and (\ref{09/03/05/17:39}) show that  
\begin{equation}\label{10/05/17/2:00}
\begin{split}
\left\|e^{\frac{i}{2}t\Delta}\Psi_{n}(0) \right\|_{X((-\infty,0])}
&=
\left\| e^{\frac{i}{2}(t+\tau_{n})\Delta}\Psi_{n}(0) \right\|_{X((-\infty,-\tau_{n}])}
\\[6pt]
&\le 
\left\|e^{\frac{i}{2}t\Delta} \Phi \right\|_{X((-\infty,-\tau_{n}])}
+
\left\| e^{\frac{i}{2}t\Delta}\left( \widetilde{\Psi}_{n}(0)-\Phi \right) \right\|_{X((-\infty,-\tau_{n}])}
\\[6pt]
&\to 0
\quad 
\mbox{as $n\to \infty$}.
\end{split}
\end{equation}
Then, the small data theory (Proposition \ref{08/08/22/20:59}) leads us to that
\begin{equation}\label{10/05/17/2:04}
\left\| \Psi \right\|_{X((-\infty,t_{n}])}
=
\left\| \Psi_{n} \right\|_{X((-\infty,0])} \lesssim 1
\quad 
\mbox{for all sufficiently large $n\in \mathbb{N}$},
\end{equation}
where the implicit constant is independent of $n$. Hence, taking $n\to \infty$ in (\ref{10/05/17/2:04}), we obtain that 
\begin{equation}\label{10/05/17/2:07}
\left\| \Psi \right\|_{X(\mathbb{R})}<\infty,
\end{equation}
which contradicts $\left\| \Psi \right\|_{X(\mathbb{R})}=\infty$. Thus, the case $\tau_{\infty}= +\infty$ never happens. 
On the other hand, if $\tau_{\infty}=-\infty$, then we have by (\ref{09/02/17/16:44}) and (\ref{09/03/05/17:39})  that 
\begin{equation}\label{10/05/17/2:11}
\begin{split}
\left\|e^{\frac{i}{2}t\Delta}\Psi_{n}(0) \right\|_{X([0,+\infty))}
&=
\left\| e^{\frac{i}{2}(t+\tau_{n})\Delta}\Psi_{n}(0) \right\|_{X([-\tau_{n},+\infty))}
\\[6pt]
&\le 
\left\|e^{\frac{i}{2}t\Delta} \Phi \right\|_{X([-\tau_{n}, +\infty))}
+
\left\| e^{\frac{i}{2}t\Delta}\left( \widetilde{\Psi}_{n}(0)-\Phi \right) \right\|_{X((-\tau_{n},+\infty))}
\\[6pt]
&\to 0
\quad
\mbox{as $n\to \infty$}.
\end{split}
\end{equation}
Hence, the small data theory (Proposition \ref{08/08/22/20:59}) shows that
\begin{equation}\label{10/05/17/2:15}
\left\| \Psi \right\|_{X([t_{n},+\infty))}
=
\left\| \Psi_{n} \right\|_{X([0,+\infty)}<\infty 
\quad 
\mbox{for all sufficiently large $n$},
\end{equation} 
so that    
\begin{equation}\label{10/05/17/2:16}
\left\| \Psi \right\|_{X([0,+\infty))}<\infty.
\end{equation}
However, (\ref{10/05/17/2:16}) contradicts the our working hypothesis $\left\| \Psi \right\|_{X([0,+\infty))}=\infty$. Thus, we always have $\tau_{\infty} \in \mathbb{R}$. 
\par
We establish the strong convergence in $L^{p+1}(\mathbb{R}^{d})$: 
\begin{equation}\label{09/02/24/13:58}
\lim_{n\to \infty}e^{\gamma_{n}\cdot \nabla }\Psi(t_{n})
=
\lim_{n\to \infty}
e^{-\frac{i}{2}\tau_{n}\Delta}\widetilde{\Psi}_{n}(0)
=
e^{-\frac{i}{2}\tau_{\infty}\Delta}\Phi
\quad 
\mbox{strongly in $L^{p+1}(\mathbb{R}^{d})$}.
\end{equation}
The Gagliardo-Nirenberg inequality shows that 
\begin{equation}\label{10/05/17/22:16}
\begin{split}
&\left\| e^{-\frac{i}{2}\tau_{n}\Delta}\widetilde{\Psi}_{n}(0) - e^{-\frac{i}{2}\tau_{\infty}\Delta}\Phi \right\|_{L^{p+1}}^{p+1}
\\[6pt]
&\lesssim  
\left\| e^{-\frac{i}{2}\tau_{n}\Delta}\widetilde{\Psi}_{n}(0) - e^{-\frac{i}{2}\tau_{\infty}\Delta}\Phi \right\|_{L^{\frac{d}{2}(p-1)}}^{p-1}
\left\| \nabla \left( e^{-\frac{i}{2}\tau_{n}\Delta}\widetilde{\Psi}_{n}(0) - e^{-\frac{i}{2}\tau_{\infty}\Delta}\Phi \right)
\right\|_{L^{2}}^{2}
\\[6pt]
&\le 
\left\| e^{-\frac{i}{2}\tau_{n}\Delta}\widetilde{\Psi}_{n}(0) - e^{-\frac{i}{2}\tau_{\infty}\Delta}\Phi \right\|_{L^{\frac{d}{2}(p-1)}}^{p-1}
\left( 
\left\| \nabla \Psi_{n}(0) \right\|_{L^{2}}^{2}
+
\left\| \nabla \Phi \right\|_{L^{2}}^{2}
\right),
\end{split}
\end{equation}
where the implicit constant depends only on $d$ and $p$. 
This estimate, together with (\ref{09/02/17/16:34}) and (\ref{09/02/17/16:44}), immediately yields (\ref{09/02/24/13:58}). 
\par 
Next, we shall show that 
\begin{equation}\label{09/02/24/14:22}
\lim_{n\to \infty}e^{\gamma_{n}\cdot \nabla}\Psi(t_{n})
=
\lim_{n\to \infty}e^{-\frac{i}{2}\tau_{n}\Delta}\widetilde{\Psi}_{n}(0)
=
e^{-\frac{i}{2}\tau_{\infty}\Delta}\Phi
\quad 
\mbox{strongly in $H^{1}(\mathbb{R}^{d})$}.
\end{equation}
Since  (\ref{10/05/17/1:47}) implies the strong convergence in $L^{2}(\mathbb{R}^{d})$,  if (\ref{09/02/24/14:22}) fails, then  we have by the weak convergence (\ref{09/02/16/15:35})
 that 
\begin{equation}\label{09/02/24/14:24}
\liminf_{n\to \infty}\left\| \nabla \Psi_{n}(0) \right\|_{L^{2}}
=
\liminf_{n\to \infty}\left\| \nabla \widetilde{\Psi}_{n}(0) \right\|_{L^{2}}
> \left\| \nabla \Phi \right\|_{L^{2}},
\end{equation}
so that we have by (\ref{09/02/24/13:58}) that 
\begin{equation}\label{10/05/17/20:08}
\lim_{n\to \infty}\mathcal{H}(\Psi_{n}(0))
=
\lim_{n\to \infty}\mathcal{H}(e^{-\frac{i}{2}\tau_{n}\Delta}\widetilde{\Psi}_{n}(0)) 
> \mathcal{H}(e^{-\frac{i}{2}\tau_{\infty}\Delta}\Phi).
\end{equation}
Moreover, it follows from (\ref{10/05/17/1:47}), (\ref{10/05/17/20:08}) and (\ref{09/02/17/16:35}) that  
\begin{equation}\label{09/02/24/14:30}
\widetilde{\mathcal{N}}_{2}(e^{-\frac{i}{2}\tau_{\infty}\Delta}\Phi)<
\lim_{n\to \infty}
\widetilde{\mathcal{N}}_{2}(\Psi_{n}(0))
=\widetilde{N}_{c}.
\end{equation}
Since $\Psi_{\infty}$ is the solution to (\ref{08/05/13/8:50}) with $\Psi_{\infty}(-\tau_{\infty})=e^{-\frac{i}{2}\tau_{\infty}\Delta}\Phi$ (see (\ref{09/02/24/14:28})), the estimate (\ref{09/02/24/14:30}), together with the definition of $\widetilde{N}_{c}$ (see (\ref{08/09/02/18:06})), leads to that     
\begin{equation}\label{10/05/17/22:41}
\left\|\Psi_{\infty} \right\|_{X(\mathbb{R})}<\infty ,
\end{equation}
which contradicts (\ref{09/09/17/14:55}). Thus, (\ref{09/02/24/14:22}) must hold. 
\par 
Now, we are in a position to prove $A=1$. We take an arbitrarily small number $\varepsilon>0$, and fix it in the following argument; we assume that $\varepsilon<\frac{1}{4}$ at least.
\par 
Let $R_{\varepsilon}>0$ be a number such that 
\begin{equation}\label{10/05/18/0:14}
\int_{|x|\le R_{\varepsilon}}|\Phi(x)|^{2}\,dx 
\ge 
\left\| \Phi \right\|_{L^{2}}^{2}
-\frac{\varepsilon}{2}
=
1-\frac{\varepsilon}{2},
\end{equation}
where we have used the fact that $\|\Phi \|_{L^{2}}=1$ (see (\ref{10/05/17/22:50}) and (\ref{10/05/17/1:47})). Then, it follows from (\ref{09/02/24/14:22}) and (\ref{10/05/18/0:14}) that  
\begin{equation}\label{08/11/01/15:20}
\int_{|x|\le R} | \Psi(x+\gamma_{n},t_{n})|^{2}\,dx > 
1-\varepsilon
\quad 
\mbox{for all sufficiently large $n\in \mathbb{N}$ and $R\ge R_{\varepsilon}$}.
\end{equation}
Combining (\ref{08/10/18/15:36}) and (\ref{08/11/01/15:20}), we obtain that 
\begin{equation}\label{10/05/17/23:00}
1-\varepsilon \le \left( 1+\frac{1}{n} \right)A
\quad 
\mbox{for all sufficiently large $n \in \mathbb{N}$ and $R\ge R_{\varepsilon}$}.\end{equation}
Taking $\varepsilon \to 0$ and $n\to \infty$ in (\ref{10/05/17/23:00}), we see that $1\le A$. Hence, $A=1$ as stated above. 
\par 
We apply Proposition \ref{08/10/03/9:46} to $|\Psi|^{2}$ and find that: There exist $R_{\varepsilon}>0$, $T_{\varepsilon}>0$ and a continuous path $y_{\varepsilon} \in C([T_{\varepsilon},\infty); \mathbb{R}^{d})$ such that 
\begin{equation}\label{09/10/08/16:24}
\int_{|x-y_{\varepsilon}(t)|<R}\left| \Psi(x,t)\right|^{2}
 \,dx > 1-\varepsilon 
\quad 
\mbox{for all $t \in [T_{\varepsilon}, \infty)$ 
and $R>R_{\varepsilon}$}.
\end{equation} 
We claim that, if necessary, taking $R_{\varepsilon}$ much larger ,  the following inequality holds for the same path $y_{\varepsilon}$ just found above:
\begin{equation}\label{09/08/10/16:50}
\int_{|x-y_{\varepsilon}(t)|<R}\left| \nabla \Psi(x,t)\right|^{2}
 \,dx > \left\| \nabla \Psi(t) \right\|_{L^{2}}^{2}-\varepsilon 
\quad 
\mbox{for all $t \in [T_{\varepsilon}, \infty)$ and $R\ge R_{\varepsilon}$}.
\end{equation}
We prove this by contradiction: Suppose the contrary that for any $k \in \mathbb{N}$, there exists $t_{k}^{0} \in [T_{\varepsilon},\infty)$ such that
\begin{equation}\label{08/11/04/8:36}
\int_{|x|< k} |\nabla \Psi(x+y_{\varepsilon}(t_{k}^{0}),t_{k}^{0})|^{2} \le \left\| \nabla \Psi(t_{k}^{0}) \right\|_{L^{2}}^{2}-\varepsilon .
\end{equation}
Put 
\begin{equation}\label{10/05/18/0:41}
\Psi_{k}^{0}(x,t):=
\Psi(x+y_{\varepsilon}(t_{k}^{0}),t+t_{k}^{0}).
\end{equation}
Then, in a similar argument for $\Psi_{n}$, we obtain    
 a subsequence of $\{ \Psi_{k}^{0} \}$ (still denoted by the same symbol), a sequence $\{\tau_{k}^{0}\}$ in $\mathbb{R}$ with $\displaystyle{\lim_{k\to \infty}
\tau_{k}^{0}=\tau_{\infty}^{0} \in \mathbb{R}}$, a sequence $\{y_{k}^{0}\}$ in $\mathbb{R}^{d}$, and a nontrivial function $\Phi^{0} \in H^{1}(\mathbb{R}^{d})$ such that  
\begin{equation}\label{08/11/02/1:25}
\lim_{k\to \infty}
\Psi_{k}^{0}(\cdot+y_{k}^{0},0)= e^{-{\frac{i}{2}\tau_{\infty}^{0}\Delta}}\Phi^{0} 
\quad 
\mbox{strongly in $H^{1}(\mathbb{R}^{d})$}.
\end{equation}
Furthermore, (\ref{08/11/02/1:25}) implies that there exists $\widetilde{R}_{\varepsilon}>0$ such that 
\begin{align}
\label{08/11/04/8:18}
&\int_{|x-y_{k}^{0}|>\widetilde{R}_{\varepsilon}}\left| \Psi(x+y_{\varepsilon}(t_{k}^{0}), t_{k}^{0})\right|^{2}\,dx \le \varepsilon 
\quad 
\mbox{for all sufficiently large $k \in \mathbb{N}$},
\\[6pt] 
\label{09/09/27/10:52}
&\int_{|x-y_{k}^{0}|\le \widetilde{R}_{\varepsilon}}\left| \nabla \Psi(x+y_{\varepsilon}(t_{k}^{0}), t_{k}^{0})\right|^{2}\,dx 
> \left\| \nabla \Psi(t_{k}^{0})\right\|_{L^{2}}^{2} -\varepsilon
\quad 
\mbox{for all sufficiently large $k \in \mathbb{N}$}.
\end{align}
On the other hand, we have by (\ref{09/10/08/16:24}) that  
\begin{equation}\label{08/11/04/8:19}
\int_{|x|\le R_{\varepsilon}}\left|\Psi (x+y_{\varepsilon}(t_{k}^{0}),t_{k}^{0})\right|^{2}\,dx > 1-\varepsilon.
\end{equation}
If $\displaystyle{\sup_{k\in \mathbb{N}}|y_{k}^{0}|= \infty}$, then we can take a subsequence of $\{y_{k}^{0}\}$ (still denoted by the same symbol) such that 
$\displaystyle{\lim_{k\to \infty}|y_{k}^{0}|= \infty}$, so that   
\begin{equation}\label{08/11/04/8:17}
\left\{x \in \mathbb{R}^{d} \left| |x-y_{k}^{0}|>\widetilde{R}_{\varepsilon}  \right. \right\}
\supset \left\{ x \in \mathbb{R}^{d} \left| |x|\le R_{\varepsilon} \right. \right\}
\quad 
\mbox{for all sufficiently large $k \in \mathbb{N}$}.
\end{equation}
However, the result derived from (\ref{08/11/04/8:18}) and (\ref{08/11/04/8:17}) contradicts (\ref{08/11/04/8:19}). Hence, it must hold that $\displaystyle{\sup_{k\in \mathbb{N}}|y_{k}^{0}|<\infty}$. Put $\displaystyle{R_{0}:= \sup_{k \in \mathbb{N}}|y_{k}^{0}|}$. Then, it follows from (\ref{09/09/27/10:52}) that  
\begin{equation}\label{10/05/18/10:30}
\begin{split}
\int_{|x|\le \widetilde{R}_{\varepsilon}+R_{0}} \left| \nabla \Psi(x+y_{\varepsilon}(t_{k}^{0}), t_{k}^{0})\right|^{2}\,dx
&\ge 
\int_{|x-y_{k}^{0}|\le \widetilde{R}_{\varepsilon}} \left| \nabla \Psi(x+y_{\varepsilon}(t_{k}^{0}), t_{k}^{0})\right|^{2}\,dx 
\\[6pt]
&> \left\| \nabla \Psi(t_{k}^{0})\right\|_{L^{2}}^{2}-\varepsilon,
\end{split}
\end{equation}
which contradicts (\ref{08/11/04/8:36}) for $k> \widetilde{R}_{\varepsilon}+R_{0}$. Thus, (\ref{09/08/10/16:50}) holds. 
\par 
Finally, we obtain (\ref{08/10/18/22:46}) and (\ref{08/11/01/15:41}) from (\ref{09/10/08/16:24}) and (\ref{09/08/10/16:50}), respectively,  by a space-time translation: regarding $\Psi(x+y_{\varepsilon}(T_{\varepsilon}), t+T_{\varepsilon})$ as our $\Psi(x,t)$ with $\gamma_{\varepsilon}(t)=y_{\varepsilon}(t+T_{\varepsilon})-y_{\varepsilon}(T_{\varepsilon})$. 
\end{proof}

\subsection{Proofs of Theorem \ref{08/05/26/11:53} and Corollary \ref{09/12/23/22:22}}
\label{09/03/04/19:08}

We begin with the proof of Theorem \ref{08/05/26/11:53}.

\begin{proof}[Proof of Theorem \ref{08/05/26/11:53}] The claims (i) and (ii) are direct consequences of  Proposition \ref{09/06/21/19:28}. 
\par 
We shall prove (iii). Proposition \ref{09/01/12/16:36} shows 
 that the wave operators $W_{\pm}$ exist on $\Omega$ and continuous. 
 It remains to prove the bijectivity of $W_{\pm}$ and the continuity of $W_{\pm}^{-1}$.
\par 
We first show that $W_{\pm}$ is surjective from $\Omega$ to $PW_{+}$. Let $\psi_{0} \in PW_{+}$ and let $\psi$ be the solution to (\ref{08/05/13/8:50}) with $\psi(0)=\psi_{0}$. Since we have shown that $\widetilde{N}_{c}=\widetilde{N}_{2}$ in Section \ref{09/05/06/9:13}, $\|\psi\|_{X(\mathbb{R})}<\infty$. Therefore, it follows from Proposition \ref{08/08/18/16:51} that  $\psi$ has asymptotic states $\phi_{\pm}$ in $H^{1}(\mathbb{R}^{d})$. It remains to prove that $\phi_{\pm} \in \Omega$.
\par 
The energy conservation law (\ref{08/05/13/9:03}) and Theorem \ref{09/05/18/10:43} show that 
\begin{equation}\label{10/05/09/12:15}
0<\mathcal{H}(\psi_{0})=\lim_{t\to \pm \infty}
\left( 
\left\| \nabla \psi(t) \right\|_{L^{2}}^{2}
-
\frac{2}{p+1}
\left\| \psi(t) \right\|_{L^{p+1}}^{p+1}
\right)
=
\left\| \nabla \phi_{\pm } \right\|_{L^{2}}^{2}.
\end{equation}
Moreover, the mass conservation law (\ref{08/05/13/8:59}) gives us that 
\begin{equation}\label{10/05/08/14:29}
\left\| \phi_{\pm} \right\|_{L^{2}}^{2}
=
\lim_{t\to \pm \infty}\left\|\psi(t)  \right\|_{L^{2}}^{2}
=\left\| \psi_{0} \right\|_{L^{2}}^{2}.
\end{equation}
Since $\psi_{0}\in PW_{+}$, we obtain from (\ref{10/05/09/12:15}), (\ref{10/05/08/14:29})  and (\ref{09/12/16/17:07}) that 
\begin{equation}\label{10/06/13/13:30}
\mathcal{N}_{2}(\phi_{\pm})= \widetilde{\mathcal{N}}_{2}(\psi_{0})<\widetilde{N}_{2},
\end{equation}
so that $\phi_{\pm}\in \Omega$.   
\par
Next, we shall show  the injectivity of $W_{\pm}$. Take any $\phi_{+,1}, \phi_{+,2} \in \Omega$ and suppose that  
\begin{equation}\label{10/05/08/14:34}
W_{+}\phi_{+,1}=W_{+}\phi_{+,2}.
\end{equation}
Put $\psi_{0}=W_{+}\phi_{+,1}=W_{+}\phi_{+,2}$ and let $\psi$ be the solution to the equation (\ref{08/05/13/8:50}) with $\psi(0)=\psi_{0}$. 
Then, we have 
\begin{equation}\label{10/05/08/14:46}
\begin{split}
\left\|\phi_{+,1}-\phi_{+,2} \right\|_{H^{1}}
&\le 
\lim_{t\to +\infty}
\left\| \psi(t)-e^{\frac{i}{2}t\Delta}\phi_{+,1} \right\|_{H^{1}}
+
\lim_{t\to +\infty}
\left\| \psi(t)-e^{\frac{i}{2}t\Delta}\phi_{+,2} \right\|_{H^{1}}
\\[6pt]
&=0,
\end{split}
\end{equation}
so that $\phi_{+,1}=\phi_{+,2}$. Similarly, we see that $W_{-}$ is injective from $\Omega$ to $PW_{+}$.
\par 
Finally, we shall prove that $W_{+}^{-1}$ is continuous. Take a function $\psi_{0}\in PW_{+}$ and a sequence $\{\psi_{0,n}\}$ in $PW_{+}$ such that 
\begin{equation}\label{10/05/07/18:17}
\lim_{n\to \infty}\psi_{0,n} = \psi_{0} 
\quad 
\mbox{strongly in $H^{1}(\mathbb{R}^{d})$}.
\end{equation}
Let $\psi$ and $\psi_{n}$ be the solutions to (\ref{08/05/13/8:50}) with $\psi(0)=\psi_{0}$ and $\psi_{n}(0)=\psi_{n,0}$. Since $\psi_{0}, \psi_{0,n} \in PW_{+}$ and $\widetilde{N}_{c}=\widetilde{N}_{2}$, we have by (\ref{08/09/02/18:06}) that 
\begin{equation}\label{10/05/08/14:56}
\left\| \psi \right\|_{X(\mathbb{R})}<\infty,\quad  \left\| \psi_{n} \right\|_{X(\mathbb{R})}< \infty.
\end{equation}
Note that since $\left\| \psi \right\|_{X(\mathbb{R})}<\infty$, for any $\delta>0$, there exists  $T_{\delta}>0$ such that 
\begin{equation}\label{10/05/07/19:02}
\left\| \psi \right\|_{X([T_{\delta},\infty))}<\delta.
\end{equation}
Now, we put 
\begin{equation}\label{10/05/07/18:27}
w_{n}=\psi_{n}-\psi.
\end{equation}
Then, it follows from the inhomogeneous Strichartz estimate (\ref{10/03/29/11:06}), an elementary inequality (\ref{09/09/27/21:10}) and the H\"older inequality that
\begin{equation}
\label{10/05/07/18:29}
\begin{split}
\left\| w_{n} \right\|_{X([t_{0},t_{1}))} 
 &\le 
\left\|e^{\frac{i}{2}(t-t_{0})\Delta} w(t_{0}) \right\|_{X([t_{0},t_{1}))}
\\[6pt]
&
\quad 
+
C
\left( \left\| \psi \right\|_{X([t_{0},t_{1}))}^{p-1} 
+ 
\left\|w_{n} \right\|_{X([t_{0}, t_{1}))}^{p-1}
\right)
\left\|w_{n} \right\|_{X([t_{0}, t_{1}))}
\\[6pt]
&\hspace{5cm}
\mbox{for all $0\le t_{0}<t_{1}\le \infty$},
\end{split}
\end{equation}
where $C$ is some constant depending only on $d$, $p$ and $q_{1}$. 
We claim that 
\begin{equation}
\label{10/05/07/19:04}
\lim_{n\to \infty}
\left\| w_{n} \right\|_{X([T_{\delta},\infty))}
=0
\quad 
\mbox{for all 
$\delta \in \Big(0,\  \left(\frac{1}{1+4C)} \right)^{\frac{1}{p-1}}\Big)$},
\end{equation}
where $C$ is the constant found in (\ref{10/05/07/18:29}) and $T_{\delta}$ is the time satisfying (\ref{10/05/07/19:02}). Let $\varepsilon$ be a number such that 
\begin{equation}\label{10/05/07/19:12}
0<
\varepsilon  
< 
\left( \frac{1}{(1+4C)^{p}} \right)^{\frac{1}{p-1}}.
\end{equation}
Then, the estimate (\ref{08/10/25/23:16}) and the continuous dependence of solutions on initial data show that there exists $N_{\varepsilon}\in \mathbb{N}$ such that  \begin{equation}\label{10/05/07/19:08}
\left\| e^{\frac{i}{2}(t-T_{\delta})}w_{n}(T_{\delta}) \right\|_{X(\mathbb{R})}<\varepsilon
\quad 
\mbox{for all $n\ge N_{\varepsilon}$}.
\end{equation}
We define a time $T_{\delta,n}$ by 
\begin{equation}\label{10/05/07/22:37}
T_{\delta, n}
=
\sup \left\{ t\ge T_{\delta} \bigm| 
\left\|w_{n} \right\|_{X([T_{\delta},t))}\le (1+4C) \varepsilon \right\}
.
\end{equation}
For (\ref{10/05/07/19:04}), it suffices to show that $T_{\delta,n}=\infty$ for all $n\ge N_{\varepsilon}$. We shall prove this. Note that (\ref{10/05/07/18:29}) with $t_{0}=T_{\delta}$, together with (\ref{10/05/07/19:08}), implies that $T_{\delta,n}>T_{\delta}$ for all $n\ge N_{\varepsilon}$. Supposing the contrary that $T_{\delta,n}<\infty$ for some $n\ge N_{\varepsilon}$, we have from the continuity of $w_{n}$ that 
\begin{equation}\label{10/05/07/22:57} 
\left\| w_{n} \right\|_{X([T_{\delta},T_{\delta,n}])} 
=(1+4C)\varepsilon
.
\end{equation} 
However, (\ref{10/05/07/18:29}), together with (\ref{10/05/07/19:02}), 
(\ref{10/05/07/19:12}), (\ref{10/05/07/19:08}) and (\ref{10/05/07/22:57}), 
 shows that 
\begin{equation}\label{10/05/07/22:59}
\begin{split}
\left\| w_{n} \right\|_{X([T_{\delta},T_{\delta,n}])} 
&\le 
\varepsilon
+ C 
\left(  
\delta^{p-1}
+
(1+4C)^{p-1}\varepsilon^{p-1}
\right)
(1+4C)\varepsilon 
\\[6pt]
&
\le \varepsilon+2C \varepsilon
\qquad 
\mbox{for all 
$\delta \in \Big(0,\  \left(\frac{1}{1+4C)} \right)^{\frac{1}{p-1}}\Big)$ and 
$n \ge N_{\varepsilon}$}.
\end{split}
\end{equation}
This is a contradiction. Thus, we see that $T_{\delta,n}=\infty$ for all $n\ge N_{\varepsilon}$. 
\par 
Besides (\ref{10/05/07/19:04}), we see that there exists $\delta_{0}>0$ depending only on $d$, $p$ and $q_{1}$ such that 
\begin{equation}\label{10/05/10/9:57}
\lim_{n\to \infty}
\left\|w_{n} \right\|_{S([T_{\delta},\infty))}
=0
\quad 
\mbox{for all $\delta \in (0,\delta_{0})$}.
\end{equation}
Indeed,  the Strichartz estimate, together with (\ref{09/09/27/21:10}), yields that     
\begin{equation}\label{10/05/07/23:10}
\left\| w_{n} \right\|_{S([T_{\delta},\infty))}
\lesssim 
\left\| w_{n}(T_{\delta})\right\|_{L^{2}}
+
\left( \left\| \psi \right\|_{X([T_{\delta},\infty))}^{p-1}
+
\left\| w_{n} \right\|_{X([T_{\delta},\infty))}^{p-1}\right)
\left\| w_{n} \right\|_{S([T_{\delta},\infty))},
\end{equation}
where the implicit constant depends only on $d$, $p$ and $q_{1}$. This estimate, together with (\ref{10/05/07/19:04}), gives (\ref{10/05/10/9:57}).  
\par 
We claim that 
\begin{equation}\label{10/05/10/10:28}
\lim_{n\to \infty}
\left\| \psi_{n}-\psi \right\|_{L^{\infty}([T,\infty);H^{1})}
=
\lim_{n\to \infty}
\left\| w_{n} \right\|_{L^{\infty}([T,\infty);H^{1})}=0
\quad 
\mbox{for sufficiently large $T>0$}.
\end{equation}
Considering the integral equations of $\partial_{j}\psi$ and $\partial_{j}\psi_{n}$ for $1\le j\le n$, we obtain that 
\begin{equation}\label{10/05/10/10:09}
\begin{split}
\left\| \partial_{j}\psi_{n}-\partial_{j} \psi \right\|_{S([T_{\delta}, \infty))}
&\lesssim 
\left\| \psi_{n}(T_{\delta})-\psi(T_{\delta})\right\|_{H^{1}}
+
\left\| 
|\psi_{n}|^{p-1}(\partial_{j}\psi_{n}-\partial_{j}\psi)
\right\|_{L^{r_{0}'}([T_{\delta},\infty);L^{q_{1}'})}
\\[6pt]
&\quad + 
\left\| 
\left( |\psi_{n}|^{p-1}-|\psi|^{p-1} \right) \partial_{j}\psi 
\right\|_{L^{r_{0}'}([T_{\delta},\infty);L^{q_{1}'})}
\\[6pt]
&\quad +
\left\| 
|\psi_{n}|^{p-3}\psi_{n}^{2}\left(\partial_{j}\overline{\psi_{n}}
-\partial_{j}\overline{\psi}\right)
\right\|_{L^{r_{0}'}([T_{\delta},\infty);L^{q_{1}'})}
\\[6pt]
&\quad +
\left\| 
\left( 
|\psi_{n}|^{p-3}\psi_{n}^{2}
-|\psi|^{p-3}\psi^{2}
\right) \partial_{j}\overline{\psi}
\right\|_{L^{r_{0}'}([T_{\delta},\infty);L^{q_{1}'})}
\\[6pt]
&\lesssim 
\left\| \psi_{n}(T_{\delta})-\psi(T_{\delta})\right\|_{H^{1}}
+
\left\| \psi_{n} \right\|_{X([T_{\delta},\infty)}^{p-1}
\left\| \partial_{j}\psi_{n}-\partial_{j}\psi \right\|
_{S([T_{\delta},\infty))}
\\[6pt]
&\quad + 
\left\| |\psi_{n}|^{p-1}-|\psi|^{p-1} 
\right\|_{L^{\frac{r_{2}}{p-1}}([T_{\delta},\infty);L^{\frac{q_{2}}{p-1}})} 
\left\| \partial_{j}\psi \right\|_{S([T_{\delta},\infty))}
\\[6pt]
&\quad 
+
\left\| \psi_{n} \right\|_{X([T_{\delta},\infty))}^{p-1}
\left\| \partial_{j}\psi_{n}
-\partial_{j}\psi\right\|_{S([T_{\delta},\infty))}
\\[6pt]
&\quad +
\left\| 
|\psi_{n}|^{p-3}\psi_{n}^{2}
-|\psi|^{p-3}\psi^{2}
\right\|_{L^{\frac{r_{2}}{p-1}}([T_{\delta},\infty);L^{\frac{q_{2}}{p-1}})}
\left\| \partial_{j}\psi 
\right\|_{S([T_{\delta},\infty))}, 
\end{split}
\end{equation}
where the implicit constant depends only on $d$, $p$ and $q_{1}$. 
This estimate, with the help of (\ref{10/05/07/19:04}) and (\ref{10/05/10/9:57}),  leads to (\ref{10/05/10/10:28}).
\par 
Now, put $\phi_{+,n}=W_{+}^{-1}\psi_{0,n}$. Then, we have by (\ref{10/05/10/10:28})  that  
\begin{equation}\label{10/05/07/23:15}
\begin{split}
\left\|\phi_{+,n}-\phi_{+} \right\|_{H^{1}}
&\le 
\left\|\psi_{n}(t)-e^{\frac{i}{2}t\Delta}\phi_{+,n} \right\|_{H^{1}}
+
\left\|\psi(t)-e^{\frac{i}{2}t\Delta}\phi_{+} \right\|_{H^{1}}
\\[6pt]
&+
\left\| \psi_{n}(t) -\psi(t) \right\|_{H^{1}}
\\[6pt]
&\to 0 
\qquad 
\mbox{as $t \to \infty$ and $n\to \infty$}. 
\end{split}
\end{equation}
Thus, we have proved that $W_{+}$ is a homeomorphism from $\Omega$ to $PW_{+}$. Similarly, we can prove that $W_{-}$ is a homeomorphism from $\Omega$ to $PW_{+}$.
\end{proof}

We shall give the proof of Corollary \ref{09/12/23/22:22}.
\begin{proof}[Proof of Corollary \ref{09/12/23/22:22}]
Let $f_{1}, f_{2}\in PW_{+}\cup \{0\}$. Then, Theorem \ref{08/05/26/11:53} shows that the corresponding solutions $\psi_{1}$ and $\psi_{2}$ have asymptotic states at $+\infty$. Hence, it follows from Theorem \ref{09/05/18/10:43} that 
\begin{equation}\label{10/05/10/0:41}
\lim_{t\to \infty}\left\| \psi_{1}(t)\right\|_{L^{p+1}}
=
\lim_{t\to \infty}\left\| \psi_{2}(t)\right\|_{L^{p+1}}
=0,
\end{equation} 
so that we can obtain  
\begin{equation}\label{08/11/20/18:12}
\lim_{t\to \infty}\left\| \psi_{j}(t) \right\|_{L^{q}}
=0
\quad 
\mbox{for all $q \in (2,2^{*})$ and $j=1,2$}.
\end{equation}
Since the solutions are continuous in $H^{1}(\mathbb{R}^{d})$, we find by (\ref{08/11/20/18:12}) that $f_{1}$ and $f_{2}$ are connected by the path $\{ \psi_{1}(t) \bigm| t\ge 0\} \cup \{0\} \cup \{ \psi_{2}(t) \bigm| t\ge 0\}$ in $L^{q}(\mathbb{R}^{d})$ with $q \in (2,2^{*})$.  
\end{proof}

\section{Analysis on $\boldsymbol{PW_{-}}$}
\label{08/07/03/0:40}
We shall give the proofs of Theorem \ref{08/06/12/9:48}, Theorem \ref{08/04/21/9:28}, Proposition \ref{08/11/02/22:52} and Proposition \ref{10/01/26/14:49}.
\par 
We begin with the proof of Theorem \ref{08/06/12/9:48}. We employ the idea of Nawa \cite{Nawa8}.  Here, the generalized virial identity below (see (\ref{08/03/29/19:05})) plays an important role:
\[
\begin{split}
&(W_{R}, |\psi(t)|^{2})
\\
&=(W_{R},|\psi_{0}|^{2})
+2t \Im{(\vec{w}_{R} \cdot \nabla \psi_{0},\psi_{0})}+2\int_{0}^{t}\int_{0}^{t'}\mathcal{K}(\psi(t''))\,dt''dt'
\\
& \qquad 
-2\int_{0}^{t}\int_{0}^{t'}\mathcal{K}^{R}(\psi(t''))\,dt''dt'
-\frac{1}{2}\int_{0}^{t}\int_{0}^{t'}(\Delta ({\rm div}\,{\vec{w}_{R}}), |\psi(t'')|^{2})\,dt''dt'
.
\end{split}
\]
 
\begin{proof}[Proof of Theorem \ref{08/06/12/9:48}] 
Since we have already proved (\ref{09/12/23/22:48}) (see Proposition \ref{08/05/26/10:57}), proofs of (\ref{09/05/14/13:32}) and (\ref{09/05/13/15:58}) remain.  For simplicity, we consider the forward time only. The problem for the backward time can be proved in a similar way. 
\par 
Take any $\psi_{0} \in PW_{-}$ and let $\psi$ be the corresponding solution to our equation (\ref{08/05/13/8:50}) with $\psi(0)=\psi_{0}$. When the maximal existence time $T_{\max}^{+}<\infty$, we have (\ref{09/05/14/13:32}) as mentioned in (\ref{10/01/27/11:26}). Therefore, it suffices to prove (\ref{09/05/13/15:58}).\par 
We suppose the contrary that (\ref{09/05/13/15:58}) fails when $T_{\max}^{+}=\infty$, so that there exists $R_{0}>0$ such that  
\begin{equation}\label{08/06/12/9:57}
M_{0}:=\sup_{t \in [0,\infty)} \int_{|x|\ge R_{0}}|\nabla \psi(x,t)|^{2}\,dx < \infty.
\end{equation}
Then, we shall derive a contradiction in three steps.
\\
{\it Step1}. We claim that: there exists a constant $m_{0}>0$ such that 
\begin{equation}\label{08/06/13/04:35}
m_{0}< \inf\left\{  
\int_{|x|\ge R}|v(x)|^{2}\,dx \left| 
\begin{array}{l}
v \in H^{1}(\mathbb{R}^{d}), 
\quad 
\mathcal{K}^{R}(v) \le 
-\frac{1}{4} \varepsilon_{0},  
\\[6pt]
\|\nabla v \|_{L^{2}(|x|\ge R)}^{2} \le M_{0}, 
\quad  \|v\|_{L^{2}}\le \|\psi_{0}\|_{L^{2}}  
\end{array}
\right.
\right\}
\  
\mbox{for all $R>0$},
\end{equation}
where $\varepsilon_{0}=\mathcal{B}(\psi_{0})-\mathcal{H}(\psi_{0})>0$. Let us prove this claim. Take any $v \in H^{1}(\mathbb{R}^{d})$ with the following properties:
\begin{equation}\label{10/03/01/15:45}
\mathcal{K}^{R}(v) \le 
-\frac{1}{4} \varepsilon_{0},  
\quad 
\|\nabla v \|_{L^{2}(|x|\ge R)}^{2} \le M_{0}, 
\quad  \|v\|_{L^{2}}\le \|\psi_{0}\|_{L^{2}}.
\end{equation}
Note here that the second and third properties in (\ref{10/03/01/15:45}) shows that 
\begin{equation}\label{10/03/01/16:22}
\|v \|_{H^{1}(|x|\ge R)}^{2}
\le 
\|\psi_{0}\|_{L^{2}}^{2}+M_{0}.
\end{equation}
We also have by the first property in (\ref{10/03/01/15:45}) and (\ref{08/04/20/12:50}) that  
\begin{equation}\label{08/06/14/22:02}
\frac{1}{4}\varepsilon_{0}
\le 
-\mathcal{K}^{R}(v)
\le 
\int_{|x|\ge R}\rho_{3}(x)|v(x)|^{p+1}\, dx.
\end{equation}

Now, we define $p_{*}$ by 
\begin{equation}\label{10/05/20/17:06}
p_{*}=\left\{ 
\begin{array}{ccc}
2p-1 &\mbox{if}& d=1,2,
\\
2^{*}-1 &\mbox{if }& d\ge 3.
\end{array}
\right.
\end{equation}
Then, the H\"older inequality and the Sobolev embedding show that 
\begin{equation}\label{08/06/14/22:18}
\begin{split}
\int_{|x|\ge R}\rho_{3}(x)|v(x)|^{p+1}\, dx
&\le 
\|\rho_{3}\|_{L^{\infty}}
\|v \|_{L^{2}(|x|\ge R)}^{p+1-\frac{(p-1)(p_{*}+1)}{p_{*}-1}}
\|v \|_{L^{p_{*}+1}(|x|\ge R)}^{\frac{(p-1)(p_{*}+1)}{p_{*}-1}}
\\
&\lesssim 
 \|\rho_{3}\|_{L^{\infty}}
\|v \|_{L^{2}(|x|\ge R)}^{p+1-\frac{(p-1)(p_{*}+1)}{p_{*}-1}}
\|v\|_{H^{1}(|x|\ge R)}^{\frac{(p-1)(p_{*}+1)}{p_{*}-1}},
\end{split}
\end{equation}
where the implicit constant depends only on $d$ and $p$. 
This estimate, together with (\ref{10/03/01/16:22}) and (\ref{08/06/14/22:02}), yields that  
\begin{equation}\label{10/03/01/16:35}
\frac{\varepsilon_{0}}
{\|\rho_{3}\|_{L^{\infty}} 
\sqrt{ \left\|\psi_{0}\right\|_{L^{2}}^{2}+M_{0} }^{\frac{(p-1)(p_{*}+1)}{p_{*}-1}}
}
\lesssim 
\|v\|_{L^{2}(|x|\ge R)}^{p+1-\frac{(p-1)(p_{*}+1)}{p_{*}-1}},
\end{equation}
where the implicit constant depends only on $d$ and $p$. Since $\|\rho_{3}\|_{L^{\infty}}\lesssim 1$ (see (\ref{08/12/29/12:53})), the estimate (\ref{10/03/01/16:35}) gives us the desired result (\ref{08/06/13/04:35}).  
\\[12pt]
{\it Step2}. Let $m_{0}$ be a constant found in (\ref{08/06/13/04:35}). Then, we prove that  
\begin{equation}\label{10/03/01/17:26}
\sup_{0\le t <\infty}\int_{|x|\ge R}|\psi(x,t)|^{2}\,dx \le m_{0}
\end{equation}
for all $R$ satisfying the following properties: 
\begin{align}
\label{10/03/01/17:21}
&R>R_{0}, 
\\[6pt]
\label{08/03/29/20:04}
&\frac{10d^{2}K}{R^{2}}\|\psi_{0}\|_{L^{2}}^{2}< \varepsilon_{0},
\\[6pt]
\label{08/03/29/20:05}
&\int_{|x|\ge R}|\psi_{0}(x)|^{2}\,dx < m_{0},
\\[6pt]
\label{08/03/29/20:06}
&\frac{1}{R^{2}}\left( 1+ \frac{2}{\varepsilon_{0}} \left\|\nabla \psi_{0}\right\|_{L^{2}}^{2} \right)(W_{R},|\psi_{0}|^{2})<m_{0},
\end{align}
where $K$ is the constant given in (\ref{09/09/26/11:16}). 
We remark that Lemma \ref{10/03/02/17:47} shows that we can take $R$ satisfying (\ref{08/03/29/20:06}). 
\par 
Now, for $R>0$ satisfying (\ref{10/03/01/17:21})--(\ref{08/03/29/20:06}), we put\begin{equation}
\label{10/03/03/16:49}
T_{R}=\sup
\left\{ T>0 \biggm| \sup_{0\le t <T }\int_{|x|\ge R}|\psi(x,t)|^{2}\,dx \le m_{0} \right\}.
\end{equation}
Note here that since $\psi \in C(\mathbb{R};L^{2}(\mathbb{R}^{d}))$ and $\psi(0)=\psi_{0}$, we have by (\ref{08/03/29/20:05}) that $T_{R}>0$. 
It is clear that $T_{R}=\infty$ shows (\ref{10/03/01/17:26}). 
\par 
We suppose the contrary that  $T_{R}<\infty$. Then, it follows from  $\psi \in C(\mathbb{R};L^{2}(\mathbb{R}^{d}))$ that 
\begin{equation}\label{10/03/03/17:19}
\int_{|x|\ge R}|\psi(x,T_{R})|^{2}\,dx =m_{0}.
\end{equation} 
Hence, the definition of $m_{0}$ (see (\ref{08/06/13/04:35})), together with (\ref{08/06/12/9:57}), (\ref{10/03/01/17:21}) and the mass conservation law (\ref{08/05/13/8:59}), 
 leads us to that 
\begin{equation}\label{10/03/01/19:19}
-\frac{1}{4}\varepsilon_{0}\le \mathcal{K}^{R}(\psi(T_{R})).
\end{equation}
Moreover, applying this inequality (\ref{10/03/01/19:19}) and (\ref{09/12/23/22:48}) (or (\ref{09/06/21/18:52})) to the generalized virial identity (\ref{08/03/29/19:05}), we obtain that  
\begin{equation}\label{10/03/02/15:40}
\begin{split}
(W_{R}, |\psi(T_{R})|^{2})
&< (W_{R},|\psi_{0}|^{2})
+2T_{R} \Im{(\vec{w}_{R} \cdot \nabla \psi_{0},\psi_{0})}
-\varepsilon_{0}^{2}T_{R}^{2}
\\[6pt]
&\qquad 
+\frac{1}{4}\varepsilon_{0}T_{R}^{2} 
-\frac{1}{2}\int_{0}^{T_{R}}\int_{0}^{t'}(\Delta ({\rm div}\,{\vec{w}_{R}}), |\psi(t'')|^{2})\,dt''dt'
.
\end{split}
\end{equation}
Here, the estimate (\ref{08/04/20/16:54}), the mass conservation law (\ref{08/05/13/8:59}) and (\ref{08/03/29/20:04}) show that the last term on the right-hand side above is estimated as follows:
\begin{equation}\label{10/03/02/16:08}
\begin{split}
-\frac{1}{2}\int_{0}^{T_{R}}\int_{0}^{t'}(\Delta ({\rm div}\,{\vec{w}_{R}}), |\psi(t'')|^{2})\,dt''dt'
&\le \frac{1}{2}\int_{0}^{T_{R}}\int_{0}^{t'}\frac{10d^{2}K}{R^{2}} \left\|\psi(t'') \right\|_{L^{2}}^{2}\,dt''dt'
\\[6pt]
&\le 
\frac{1}{4}\varepsilon_{0}T_{R}^{2}.
\end{split}
\end{equation}
Therefore, we have 
\begin{equation}\label{09/09/24/13:05}
\begin{split}
(W_{R}, |\psi(T_{R})|^{2})
&< (W_{R},|\psi_{0}|^{2})
+2T_{R} \Im{(\vec{w}_{R} \cdot \nabla \psi_{0},\psi_{0})}-\frac{1}{2}T_{R}^{2}\varepsilon_{0}.
\\[6pt]
&=(W_{R},|\psi_{0}|^{2}) -\frac{1}{2}\varepsilon_{0}
\left\{ T_{R} -\frac{2}{\varepsilon_{0}}\Im{(\vec{w}_{R} \cdot \nabla \psi_{0},\psi_{0})} \right\}^{2}
+ \frac{2}{\varepsilon_{0}}\left| (\vec{w}_{R} \cdot \nabla \psi_{0},\psi_{0}) \right|^{2}
\\[6pt]
&\le (W_{R},|\psi_{0}|^{2}) + \frac{2}{\varepsilon_{0}}\left| (\vec{w}_{R} \cdot \nabla \psi_{0},\psi_{0}) \right|^{2}
\\[6pt]
&\le 
(W_{R},|\psi_{0}|^{2}) + \frac{2}{\varepsilon_{0}}
\left\| \nabla \psi_{0} \right\|_{L^{2}}^{2}
\int_{\mathbb{R}^{d}} W_{R}(x)|\psi_{0}(x)|^{2}\,dx,
\end{split}
\end{equation}
where we have used the Schwarz inequality and (\ref{09/09/25/13:24}) 
to derive the last inequality. Combining this inequality (\ref{09/09/24/13:05}) with (\ref{08/03/29/20:06}), we see that 
\begin{equation}\label{10/03/02/16:28}
(W_{R}, |\psi(T_{R})|^{2}) \le 
\left( 1+ \frac{2}{\varepsilon_{0}}
 \left\|\nabla \psi_{0}\right\|_{L^{2}}^{2} \right)
 (W_{R},|\psi_{0}|^{2})<R^{2}m_{0}.
\end{equation}
On the other hand, since $W_{R}(x) \ge R^{2}$ for $|x|\ge R$ (see (\ref{10/02/27/22:19})), we have   
\begin{equation}\label{09/09/24/16:30}
\int_{|x|\ge R } |\psi(x,T_{R})|^{2} \,dx 
= 
\frac{1}{R^{2}}\int_{|x|\ge R} R^{2} |\psi(x,T_{R})|^{2} \,dx 
\le \frac{1}{R^{2}} (W_{R}, |\psi(T_{R})|^{2}).
\end{equation}
Thus, it follows form (\ref{10/03/02/16:28}) and (\ref{09/09/24/16:30}) that 
\begin{equation}\label{10/03/06/23:56}
\int_{|x|\ge R } |\psi(x,T_{R})|^{2} \,dx  <m_{0},
\end{equation}
which contradicts (\ref{10/03/03/17:19}), so that  $T_{R}=\infty$ and (\ref{10/03/01/17:26}) hold. 
\\
{\it Step3}. We complete the proof of Theorem \ref{08/06/12/9:48}. The definition of $m_{0}$, together with the mass conservation law (\ref{08/05/13/8:59}), (\ref{08/06/12/9:57}) and (\ref{10/03/01/17:26}), shows that 
\begin{equation}\label{10/03/03/17:48}
-\frac{1}{4}\varepsilon_{0}\le \mathcal{K}^{R}(\psi(t))
\qquad 
\mbox{for all $R>0$ satisfying (\ref{10/03/01/17:21})--(\ref{08/03/29/20:06}), and all $t\ge 0$}.
\end{equation}
Applying this estimate (\ref{10/03/03/17:48}) to the generalized virial identity (\ref{08/03/29/19:05}), we obtain the following estimate as well as {\it Step2} (see (\ref{10/03/02/15:40}) and (\ref{10/03/02/16:08})):
\begin{equation}\label{10/03/03/17:56}
(W_{R}, |\psi(t)|^{2})
\le (W_{R},|\psi_{0}|^{2})
+2t \Im{(\vec{w}_{R} \cdot \nabla \psi_{0},\psi_{0})}-\frac{1}{2}t^{2}\varepsilon_{0}
\quad 
\mbox{for all $t\ge 0$}.
\end{equation}
This inequality means that $(W_{R}, |\psi(t)|^{2})$ becomes negative in a finite time, so that $T_{\max}^{+}$ must be finite. However, this contradicts $T_{\max}^{+}=\infty$. Hence, (\ref{08/06/12/9:57}) derives an absurd conclusion: Thus,  (\ref{09/05/13/15:58}) holds. 
\end{proof}

Next, we give the proof of Theorem \ref{08/04/21/9:28}. 
\begin{proof}[Proof of Theorem \ref{08/04/21/9:28}] 
Let $\psi_{0}$ be a radially symmetric function in $PW_{-}$, and let $\psi$ be the solution to (\ref{08/05/13/8:50}) with $\psi(0)=\psi_{0}$.
\par 
To handle the term containing $\mathcal{K}^{R}$ in the generalized virial identity (\ref{08/03/29/19:05}) for $\psi$,  we consider the following variational problem, as well as Lemma 3.3 in \cite{Nawa8}: 
\begin{lemma}
\label{08/04/20/18:30}
Assume that $d\ge 3$, $2+\frac{4}{d}< p+1 < 2^{*}$, and $p\le 5$ if $d=2$. Then, there exists $R_{*}>0$ and $m_{*}>0$ such that  
\[
m_{*} < \inf\left\{  \int_{|x|\ge R}|v(x)|^{2}\,dx \left|
\begin{array}{l}
v \in H^{1}_{rad}(\mathbb{R}^{d}), 
\ \mathcal{K}^{R}(v) \le -\frac{1}{4}\varepsilon_{0}, 
\\[6pt] 
\|v\|_{L^{2}}\le \|\psi_{0}\|_{L^{2}}  
\end{array}
\right. 
\right\}
\quad 
\mbox{for all $R\ge R_{*}$},
\]
where $H^{1}_{rad}(\mathbb{R}^{d})$ is the set of radially symmetric functions in $H^{1}(\mathbb{R}^{d})$, and $\varepsilon_{0}=\mathcal{H}(\psi_{0})-\mathcal{B}(\psi_{0})>0$.
\end{lemma} 

\begin{proof}[Proof of Lemma \ref{08/04/20/18:30}] 
We take a function $v \in H_{rad}^{1}(\mathbb{R}^{d})$ with the following properties:  
\begin{align}
\label{10/03/05/13:28}
&\mathcal{K}^{R}(v) \le -\frac{1}{4}\varepsilon_{0},
\\[6pt]
\label{10/03/05/13:38}
&\|v\|_{L^{2}}\le \|\psi_{0}\|_{L^{2}},
\end{align}
where $R$ is a sufficiently large constant to be chosen later (see (\ref{10/03/06/18:09})). Since $v \in H_{rad}^{1}(\mathbb{R}^{d})$, 
 $\mathcal{K}^{R}(v)$ is written as follows (see Remark \ref{09/09/25/16:47}):
\begin{equation}\label{10/03/05/13:56}
\mathcal{K}^{R}(v)
=
\int_{\mathbb{R}^{d}}\rho_{0}(x)|\nabla v(x)|^{2}\,dx
-\int_{\mathbb{R}^{d}}\rho_{3}(x)|v(x)|^{p+1}\,dx.
\end{equation}
Hence, it follows from (\ref{10/03/05/13:28}) that 
\begin{equation}\label{08/04/20/19:49}
\begin{split}
\frac{1}{4}\varepsilon_{0} +\int_{\mathbb{R}^{d}}\rho_{0}(x)|\nabla v(x)|^{2}\,dx 
&\le -\mathcal{K}^{R}(v)+ \int_{\mathbb{R}^{d}}\rho_{0}(x)|\nabla v(x)|^{2}\,dx
\\[6pt]
&=\int_{\mathbb{R}^{d}}\rho_{3}(x)|v(x)|^{p+1}\,dx.
\end{split}
\end{equation}
To estimate the right-hand side of (\ref{08/04/20/19:49}), we employ the following inequality (see 
 Lemma 6.5.11 in \cite{Cazenave}): Assume that $d\ge 1$. Let $\kappa$ be a non-negative and radially symmetric function in $C^{1}(\mathbb{R}^{d})$ with $|x|^{-(d-1)}\max\{-\frac{x\nabla \kappa }{|x|},\, 0\} \in L^{\infty}(\mathbb{R}^{d})$. Then, we have that  
\begin{equation}\label{08/06/12/7:28}
\begin{split}
&
\|\kappa^{\frac{1}{2}} f\|_{L^{\infty}}
\\[6pt]
&\le  
K_{5}
\|f\|_{L^{2}}^{\frac{1}{2}}
\left\{ 
\left\| |x|^{-(d-1)}\kappa \, \frac{x \nabla f}{|x|} \right\|_{L^{2}}^{\frac{1}{2}}
+
\left\| |x|^{-(d-1)}
\max \left\{-\frac{x \nabla \kappa}{|x|}, \ 0 \right\}
\right\|_{L^{\infty}}^{\frac{1}{2}}
\left\| f \right\|_{L^{2}}^{\frac{1}{2}}
\right\}
\\[6pt]
&\hspace{9cm}
\mbox{for all $f \in H_{rad}^{1}(\mathbb{R}^{d})$}, 
\end{split}
\end{equation}
where $K_{5}>0$ is some constant depending only on $d$. 
\\
This inequality (\ref{08/06/12/7:28}), together with (\ref{08/12/29/12:53}), (\ref{08/04/20/13:06}) and  ${\rm supp}\, \rho_{3}=\{|x|\ge R\}$ (see (\ref{08/12/29/14:16})), yields the following estimate for the right-hand side of (\ref{08/04/20/19:49}):  
\begin{equation}\label{08/07/02/14:50}
\begin{split}
&\int_{\mathbb{R}^{d}}\rho_{3}(x)|v(x)|^{p+1}\,dx 
\\[6pt]
&\le 
\|\rho_{3}^{\frac{1}{4}}v\|_{L^{\infty}}^{p-1}
\int_{\mathbb{R}^{d}}\rho_{3}^{\frac{5-p}{4}}(x)|v(x)|^{2}\,dx
\\[6pt]
&\le \frac{K_{5}^{p-1}}{R^{\frac{p-1}{2}(d-1)}}
\left\| v \right\|_{L^{2}(|x|\ge R)}^{\frac{p-1}{2}}
\left\{ 
\left\|\sqrt{\rho_{3}} \frac{x\nabla v}{|x|}\right\|_{L^{2}}^{\frac{1}{2}}
\!\!
+
\left\|\max\left\{ -\frac{x \nabla \sqrt{\rho_{3}}}{|x|},\ 0 \right\} 
\right\|_{L^{\infty}}^{\frac{1}{2}}
\!\!
\left\| v \right\|_{L^{2}}^{\frac{1}{2}}
\right\}^{p-1}\hspace{-12pt}K_{3}^{\frac{5-p}{4}}\|v\|_{L^{2}}^{2}
\\[6pt]
&\le 
\frac{K_{5}^{p-1}K_{3}^{\frac{5-p}{4}}}{R^{\frac{p-1}{2}(d-1)}}
\left\| v \right\|_{L^{2}(|x|\ge R)}^{\frac{p-1}{2}}
\left\{ 
\left\|\sqrt{\rho_{3}}\nabla v \right\|_{L^{2}}^{\frac{1}{2}}
+
\frac{(K_{3}')^{\frac{1}{2}}}{R^{\frac{1}{2}}}
\left\| v \right\|_{L^{2}}^{\frac{1}{2}}
\right\}^{p-1}\hspace{-12pt}\|v\|_{L^{2}}^{2}
\\[6pt]
&\le 
\frac{K_{5}^{p-1}K_{3}^{\frac{5-p}{4}}}{R^{\frac{p-1}{2}(d-1)}}
\left\| v \right\|_{L^{2}(|x|\ge R)}^{\frac{p-1}{2}}
C_{p}\left\{ 
\left\|\sqrt{\rho_{3}}\nabla v \right\|_{L^{2}}^{\frac{p-1}{2}}
+
\frac{(K_{3}')^{\frac{p-1}{2}}}{R^{\frac{p-1}{2}}}
\left\| v \right\|_{L^{2}}^{\frac{p-1}{2}}
\right\} \|v\|_{L^{2}}^{2}
\end{split}
\end{equation}
for some constant $C_{p}>0$ depending only on $p$.
Moreover, using $p\le 5$ (so that $\frac{p-1}{2}\le 2$), $\|v\|_{L^{2}}\le \|\psi_{0}\|_{L^{2}}$, and the Young inequality ($ab \le \frac{a^{r}}{r}+\frac{b^{r'}}{r'}$ with $\frac{1}{r}+\frac{1}{r'}=1$), we obtain that 
\begin{equation}\label{10/03/06/17:54}
\begin{split}
\mbox{R.H.S. of (\ref{08/07/02/14:50})}
&\le 
\frac{p-1}{4}\left\| v \right\|_{L^{2}(|x|\ge R)}^{2}
\left\|\sqrt{\rho_{3}}\nabla v \right\|_{L^{2}}^{2}
+\frac{5-p}{4}
\left(
\frac{C_{p}\left\|\psi_{0}\right\|_{L^{2}}^{2}K_{5}^{p-1}K_{3}^{\frac{5-p}{4}}}{R^{\frac{p-1}{2}(d-1)}}\right)^{\frac{4}{5-p}}
\\[6pt]
&\qquad 
+C_{p}\frac{\left\|\psi_{0}\right\|_{L^{2}}^{p+1}K_{5}^{p-1}K_{3}^{\frac{5-p}{4}}(K_{3}')^{\frac{p-1}{2}}}{R^{\frac{d(p-1)}{2}}}.
\end{split}
\end{equation}
If $R$ is so large that 
\begin{equation}\label{10/03/06/18:09}
\frac{5-p}{4}
\left(
\frac{C_{p}\left\|\psi_{0}\right\|_{L^{2}}^{2}K_{5}^{p-1}K_{3}^{\frac{5-p}{4}}}{R^{\frac{p-1}{2}(d-1)}}\right)^{\frac{4}{5-p}}
+
\frac{C_{p}\left\|\psi_{0}\right\|_{L^{2}}^{p+1}K_{5}^{p-1}K_{3}^{\frac{5-p}{4}}(K_{3}')^{\frac{p-1}{2}}}{R^{\frac{d(p-1)}{2}}}
 \le \frac{1}{8}\varepsilon_{0},
\end{equation} 
then (\ref{08/04/20/19:49}), together with  (\ref{08/07/02/14:50}) and (\ref{10/03/06/17:54}),  shows that 
\begin{equation}\label{10/03/06/18:12}
\int_{\mathbb{R}^{d}}\left\{ \rho_{0}(x)- \frac{p-1}{4}\|v\|_{L^{2}(|x|\ge R)}^{2}\rho_{3}(x)\right\}|\nabla v(x)|^{2} \le -\frac{\varepsilon_{0}}{8}.
\end{equation}
Hence, it must hold that 
\begin{equation}\label{10/03/06/21:04}
\inf_{|x|\ge R}{ \frac{\rho_{0}(x)}{\rho_{3}(x)}}< 
\frac{p-1}{4}\|v\|_{L^{2}(|x| \ge R)}^{2},
\end{equation}
which, together with (\ref{08/04/20/16:07}), gives us the desired result.
\end{proof}

The following lemma tells us that the solution $\psi$ is well localized:
\begin{lemma}[cf. Lemma 3.4 in \cite{Nawa8}]
\label{08/03/30/15:12}
Assume that $d\ge 1$, $2+\frac{4}{d}<p+1<2^{*}$, and $p\le 5$ if $d=2$. 
Let $R_{*}$ and $m_{*}$ be constants found in Lemma \ref{08/04/20/18:30}.
Then, we have 
\[ 
T_{\max}^{+}=\sup{\left\{ T>0 \biggm|  \sup_{0\le t <T} \int_{|x|\ge  R}|\psi(x,t)|^{2}\,dx < m_{*}  \right\}}
\]
for all $R$ with the following properties:
\begin{align}
\label{10/03/07/0:10}
&R>R_{*}, 
\\[6pt]
\label{10/03/07/0:11}
&\frac{10d^{2}K}{R^{2}}\|\psi_{0}\|_{L^{2}}^{2}< \varepsilon_{0},
\\[6pt]
\label{10/03/07/0:12}
&\int_{|x|\ge R}|\psi_{0}(x)|^{2}\,dx < m_{*},
\\[6pt]
\label{10/03/07/0:13}
&\frac{1}{R^{2}}\left( 1+ \frac{2}{\varepsilon_{0}} \left\|\nabla \psi_{0}\right\|_{L^{2}}^{2} \right)(W_{R},|\psi_{0}|^{2})<m_{*}.
\end{align}
\end{lemma}
\begin{proof}[Proof of Lemma \ref{08/03/30/15:12}] For $R>0$ satisfying (\ref{10/03/07/0:10})--(\ref{10/03/07/0:13}), we put 
\[
T_{R}:=\sup{\left\{ T>0 \biggm|  \sup_{0\le t \le T}\int_{|x|\ge  R}|\psi(x,t)|^{2}\,dx < m_{*} \right\}}.
\]
We suppose the contrary that $T_{R}<T_{\max}^{+}$, so that 
it follows from $\psi \in C([0,T_{\max}^{+});L^{2}(\mathbb{R}^{d}))$ that 
\begin{equation}\label{08/03/30/16:09}
\int_{|x|\ge R}|\psi(x,T_{R})|^{2}\,dx = m_{*}.
\end{equation}
Hence, we obtain by Lemma \ref{08/04/20/18:30} that 
\begin{equation}\label{10/03/06/23:50}
-\frac{1}{4}\varepsilon_{0}< \mathcal{K}^{R}(\psi(T_{R})).
\end{equation}
Then, the same computation as (\ref{10/03/02/15:40})--(\ref{10/03/06/23:56}) 
 yields that  
\begin{equation}\label{10/03/06/23:58}
\int_{|x|\ge R } |\psi(T_{R},x)|^{2} \,dx <m_{*}, 
\end{equation}
which contradicts (\ref{08/03/30/16:09}): Thus, $T_{R}=T_{\max}^{+}$. 
\end{proof}
Now, we are in a position to complete the proof of Theorem \ref{08/04/21/9:28}.
\\
{\it Proof of Theorem \ref{08/04/21/9:28} (continued)} We take $R>0$ satisfying  (\ref{10/03/07/0:10})--(\ref{10/03/07/0:13}). Then, Lemma \ref{08/03/30/15:12} gives us that   
\begin{equation}\label{10/03/07/0:19}
\int_{|x|\ge R}|\psi(t,x)|^{2}\,dx < m_{*} 
\quad 
\mbox{for all $t \in [0,T_{\max}^{+})$}.
\end{equation}
Moreover, Lemma \ref{08/04/20/18:30}, together with (\ref{10/03/07/0:19}), shows that 
\begin{equation}\label{10/03/07/0:22}
-\frac{\varepsilon_{0}}{4}< \mathcal{K}^{R}(\psi(t))
\quad 
\mbox{for all $t \in [0,T_{\max}^{+})$}.
\end{equation} 
Applying (\ref{10/03/07/0:22}) and (\ref{09/12/23/22:48}) to the generalized virial identity (\ref{08/03/29/19:05}), we obtain that  
\begin{equation}\label{10/03/07/0:29}
\begin{split}
(W_{R}, |\psi(t)|^{2})
&< (W_{R},|\psi_{0}|^{2})
+2t \Im{(\vec{w}_{R} \cdot \nabla \psi_{0},\psi_{0})}
-\varepsilon_{0}^{2}t^{2}
\\[6pt]
&\qquad 
+\frac{1}{4}\varepsilon_{0}t^{2} 
-\frac{1}{2}\int_{0}^{t}\int_{0}^{t'}(\Delta ({\rm div}\,{\vec{w}_{R}}), |\psi(t'')|^{2})\,dt''dt'
.
\end{split}
\end{equation}
Here, it follows from (\ref{08/04/20/16:54}), the mass conservation law (\ref{08/05/13/8:59}) and (\ref{10/03/07/0:11}) that the last term on the right-hand side above is estimated as follows :
\begin{equation}\label{10/03/07/0:31}
\begin{split}
-\frac{1}{2}\int_{0}^{t}\int_{0}^{t'}(\Delta ({\rm div}\,{\vec{w}_{R}}), |\psi(t'')|^{2})\,dt''dt'
&\le \frac{1}{2}\int_{0}^{t}\int_{0}^{t'}\frac{10d^{2}K}{R^{2}} \left\|\psi(t'') \right\|_{L^{2}}^{2}\,dt''dt'
\\[6pt]
&\le 
\frac{1}{4}\varepsilon_{0}t^{2}.
\end{split}
\end{equation}
Combining (\ref{10/03/07/0:29}) with (\ref{10/03/07/0:31}), we obtain that 
\begin{equation}\label{10/03/07/0:30}
(W_{R}, |\psi(t)|^{2})
< (W_{R},|\psi_{0}|^{2})
+2t \Im{(\vec{w}_{R} \cdot \nabla \psi_{0},\psi_{0})}-\frac{1}{2}t^{2}\varepsilon_{0}.
\end{equation}
This inequality means that $(W_{R}, |\psi(t)|^{2})$ becomes negative in a finite time, so that $T_{\max}^{+}<\infty$: This fact, together with (\ref{10/01/27/11:26}), shows  (\ref{10/02/22/20:33}). 
\par
Now, we shall show the claims (i)--(iii) in Theorem \ref{08/04/21/9:28}. 
\par 
We can find that Lemma \ref{08/03/30/15:12} is still valid if we replace $m_{*}$ with any $m \in (0,m_{*})$. Thus, we see that (i) holds. 
\par 
We next consider (ii). Since we have (\ref{08/04/20/12:51}), in order to prove (\ref{08/06/10/10:21}), it suffices to show that 
\begin{equation}\label{10/03/07/16:52}
\int_{0}^{T_{\max}^{+}}(T_{\max}^{+}-t)
\left(
 \int_{\mathbb{R}^{d}}\rho_{0}(x)|\nabla \psi(x,t)|^{2}\,dx 
\right) dt < \infty
\quad 
\mbox{for all sufficiently large $R>0$}.
\end{equation} 
Let us prove this. Integrating by parts, we find that  
\begin{equation}\label{10/03/07/17:54}
\begin{split}
&\int_{0}^{t} (t-t')\left( \int_{\mathbb {R}^{d}}\rho_{0}(x)|\nabla \psi(x,t')|^{2}\,dx
\right)dt'
\\[6pt]
&=
\int_{0}^{t}\int_{0}^{t'} 
\left( 
\int_{\mathbb{R}^{d}}\rho_{0}(x)|\nabla \psi(x,t'')|^{2}\,dx \right) dt''\,dt'
\quad 
\mbox{for all $t\ge 0$}. 
\end{split}
\end{equation}
This formula  (\ref{10/03/07/17:54}) and the generalized virial identity (\ref{08/03/29/19:05}) lead us to the following estimate: 
\begin{equation}\label{10/03/07/17:48}
\begin{split}
&2 \int_{0}^{t} (t-t')\left( \int_{\mathbb {R}^{d}}\rho_{0}(x)|\nabla \psi(x,t')|^{2}\,dx
\right)dt'
\\[6pt]
&\le 
(W_{R},|\psi_{0}|^{2})
+2t \Im{(\vec{w}_{R} \cdot \nabla \psi_{0},\psi_{0})}+2\int_{0}^{t}\int_{0}^{t'}\mathcal{K}(\psi(t''))\,dt''dt'
\\[6pt]
& \quad 
+2\int_{0}^{t}(t-t') \left( 
\int_{\mathbb{R}^{d}} \rho_{3}(x)|\psi(x,t')|^{p+1}\,dx
\right)dt'
-\frac{1}{2}\int_{0}^{t}\int_{0}^{t'}(\Delta ({\rm div}\,{\vec{w}_{R}}), |\psi(t'')|^{2})\,dt''dt' .
\end{split}
\end{equation}
Here, applying the estimate (\ref{10/03/07/0:31}) to the right-hand side of (\ref{10/03/07/17:48}) (abbreviated to R.H.S. of (\ref{10/03/07/17:48})), we have that 
\begin{equation}\label{10/03/07/18:21}
\begin{split}
\mbox{R.H.S. of (\ref{10/03/07/17:48})}
&\le 
(W_{R},|\psi_{0}|^{2})
+2t \Im{(\vec{w}_{R} \cdot \nabla \psi_{0},\psi_{0})}-\varepsilon_{0}t^{2}
\\[6pt]
&\quad +2\int_{0}^{t}(t-t') \left( 
\int_{\mathbb{R}^{d}} \rho_{3}(x)|\psi(x,t')|^{p+1}\,dx
\right)dt'
+\frac{\varepsilon_{0}}{4}t^{2}.
\end{split}
\end{equation} 
Moreover, taking $R$ so large that (\ref{10/03/06/18:09}) holds, 
 we find by the same computation as (\ref{08/07/02/14:50})--(\ref{10/03/06/17:54}) that 
\begin{equation}\label{10/03/07/19:03}
\int_{\mathbb{R}^{d}}\rho_{3}(x)|\psi(x,t)|^{p+1}\,dx 
\le 
\frac{p-1}{4}\left\|\psi(t) \right\|_{L^{2}(|x|\ge R)}^{2}
\int_{\mathbb{R}^{d}}\rho_{3}(x)|\nabla \psi(x,t)|^{2}\,dx + \frac{\varepsilon_{0}}{8}.
\end{equation}
Combining (\ref{10/03/07/17:48}), (\ref{10/03/07/18:21}) and (\ref{10/03/07/19:03}), we obtain that 
\begin{equation}\label{10/03/07/19:11}
\begin{split}
&2\int_{0}^{t}
(t-t')
\left( \int_{\mathbb{R}^{d}} \left\{ \rho_{0}(x)-
\frac{p-1}{4}\left\|\psi(t) \right\|_{L^{2}(|x|\ge R)}^{2}
\rho_{3}(x)
\right\}
|\nabla \psi(x,t)|^{2}\,dx
\right)\,dt'
\\[6pt]
&\le 
(W_{R},|\psi_{0}|^{2})
+2t \Im{(\vec{w}_{R} \cdot \nabla \psi_{0},\psi_{0})}
-\frac{\varepsilon_{0}}{2}t^{2}
\\[6pt]
&\hspace{5cm} 
\mbox{for all $R>0$ satisfying (\ref{10/03/06/18:09}) and (\ref{10/03/07/0:10})--(\ref{10/03/07/0:13})}.
\end{split}
\end{equation}
Here, using the result of (i) (the formula (\ref{08/06/12/8:54})), we can take 
 a constant $R_{1}>0$ such that 
\begin{equation}\label{10/03/07/16:47}
\frac{p-1}{4}\left\| \psi(t) \right\|_{L^{2}(|x|\ge R_{1})}^{2}\frac{1}{K_{4}}
\le 
\frac{1}{2}
\quad 
\mbox{for all $t \in I_{\max}$},
\end{equation}
where $K_{4}$ is some positive constant found in (\ref{08/04/20/16:07}). Therefore, (\ref{10/03/07/19:11}), together with (\ref{10/03/07/16:47}) and (\ref{08/04/20/16:07}), shows that 
\begin{equation}\label{10/03/07/22:23}
\begin{split}
&
\int_{0}^{t}(t-t')\left( 
\int_{\mathbb{R}^{d}}\rho_{0}(x)|\nabla \psi(x,t)|^{2}\,dx
\right)dt'
\\[6pt]
&\le (W_{R},|\psi_{0}|^{2})
+2t \Im{(\vec{w}_{R} \cdot \nabla \psi_{0},\psi_{0})}
-\frac{\varepsilon_{0}}{2}t^{2}
\quad  
\mbox{for all sufficiently large $R$ and $t \in [0,T_{\max}^{+})$}.
\end{split}
\end{equation}
This, together with (\ref{08/04/20/12:50}) and (\ref{08/04/20/12:51}), gives (\ref{08/06/10/10:21}). 
\par 
Once (\ref{08/06/10/10:21}) is proved, we can easily obtain (\ref{08/06/10/10:22}). Indeed, (\ref{10/03/07/19:03}), together with (\ref{08/06/10/10:21}) and (\ref{10/03/01/15:06}), immediately yields (\ref{08/06/10/10:22}). Thus, we have proved (ii). 
\par 
Finally, we prove (iii). We take $R>0$ so large that 
 (\ref{08/06/10/10:21}) holds valid. 
\par 
To prove (\ref{08/06/12/6:56}), we employ the radial interpolation inequality by Strauss \cite{Strauss}: 
\begin{equation}\label{10/05/30/14:01}
\left\|f \right\|_{L^{\infty}}^{2}
\lesssim 
\frac{1}{R^{d-1}}
\left\|f \right\|_{L^{2}}
\left\| \nabla f \right\|_{L^{2}}
\quad 
\mbox{for all $f \in H_{rad}^{1}(\mathbb{R}^{d})$}, 
\end{equation}
where the implicit constant depends only on $d$.
\\
This inequality (\ref{10/05/30/14:01}), together with (\ref{08/06/10/10:21})
 and the mass conservation law (\ref{08/05/13/8:59}), gives us (\ref{08/06/12/6:56}). 
\par 
The estimate (\ref{08/06/12/6:57}) follows from (\ref{08/06/12/6:56}) and the inequality 
\[
\|\psi(t)\|_{L^{p+1}(|x|>R)}^{p+1}\le \|\psi_{0}\|_{L^{2}}^{2}\|\psi(t)\|_{L^{\infty}(|x|>R)}^{p-1}.
\] 
We obtain (\ref{08/06/12/6:58}) and (\ref{08/06/12/6:59}) from (\ref{08/06/12/6:56}) and (\ref{08/06/12/6:57}), respectively. 
\end{proof}

We also have variants of (\ref{08/06/12/6:56})--(\ref{08/06/12/6:59}). For example: in their proofs, replacing $R$ by $R(T_{\max}^{+}-t)^{-\frac{1}{2(d-1)}}$, we obtain 
\[
\int_{0}^{T_{\max}^{+}}\|\psi(t)\|_{L^{\infty}(|x|>R(T_{\max}^{+}-t)^{-\frac{1}{2(d-1)}})}^{4}\,dt < \infty 
\] 
and 
\[
\int_{0}^{T_{\max}^{+}}\|\psi(t)\|_{L^{p+1}(|x|>R(T_{\max}^{+}-t)^{-\frac{1}{2(d-1)}})}^{\frac{4(p+1)}{p-1}}\,dt < \infty .
\] 
Moreover, replacing $R$ by $R(T_{\max}^{+}-t)^{-\frac{1}{d-1}}$ in (\ref{08/06/12/6:56}) and (\ref{08/06/12/6:57}), we  obtain 
\[
\liminf_{t \to T_{\max}^{+}} \|\psi(t)\|_{L^{\infty}\left(|x|>R(T_{\max}^{+}-t)^{-\frac{1}{d-1}}\right)}^{2}=0 
\]
and 
\[
\liminf_{t \to T_{\max}^{+}} \|\psi(t)\|_{L^{p+1}\left(|x|>R(T_{\max}^{+}-t)^{-\frac{1}{d-1}}\right)}^{p+1} =0 ,
\]
respectively.
\\
\par 
Now, we shall give the proof of Propositions \ref{08/11/02/22:52}. 

\begin{proof}[Proof of Proposition \ref{08/11/02/22:52}] 
 Suppose that $\psi$ is a solution to the equation (\ref{08/05/13/8:50}) satisfying that 
\begin{equation}\label{10/01/26/12:50}
\limsup_{t\to T_{\max}^{+}}\left\|\nabla \psi(t)  \right\|_{L^{2}}=
\limsup_{t\to T_{\max}^{+}}\left\| \psi(t) \right\|_{L^{p+1}}=\infty.
\end{equation}
Then, there exists a sequence $\{t_{n}\}_{n\in \mathbb{N}}$ in 
 $[0,T_{\max}^{+})$ such that 
\begin{align}
\label{10/03/08/17:57}
&\lim_{n\to \infty}t_{n} = T_{\max}^{+},  
\\[6pt]  
\label{10/01/26/12:53} 
&\left\| \psi(t_{n})\right\|_{L^{p+1}}=\sup_{t\in [0,t_{n})}
\left\| \psi(t)\right\|_{L^{p+1}}.
\end{align}
Using such a sequence $\{t_{n}\}_{n\in \mathbb{N}}$, we define a number $\lambda_{n}$ by 
\begin{equation}\label{10/01/26/12:58}
\lambda_{n}=\left\| \psi(t_{n})\right\|_{L^{p+1}}^{-\frac{(p-1)(p+1)}{d+2-(d-2)p}},
\quad 
n \in \mathbb{N}.
\end{equation}
It is easy to see that     
\begin{equation}\label{10/03/08/17:59}
\lim_{n\to \infty}\lambda_{n} = 0.
\end{equation}
We consider the scaled functions $\psi_{n}$ defined by 
\begin{equation}\label{10/01/26/13:14}
\psi_{n}(x,t):=\lambda_{n}^{\frac{2}{p-1}}\overline{\psi(\lambda_{n}x,t_{n}-\lambda_{n}^{2}t)}, 
\quad 
(x,t)\in \mathbb{R}^{d}\times 
\left(-\frac{T_{\max}^{+}-t_{n}}{\lambda_{n}^{2}}, \frac{t_{n}}{\lambda_{n}^{2}}\right],
\quad n \in \mathbb{N}.
\end{equation}
We can easily verify that   
\begin{align}
\label{10/03/08/15:29}
&2i\displaystyle{\frac{\partial \psi_{n}}{\partial t}} +\Delta \psi_{n}+|\psi_{n}|^{p-1}\psi_{n}=0  \quad \mbox{in} \ \mathbb{R}^{d}\times 
\left(-\frac{T_{\max}^{+}-t_{n}}{\lambda_{n}^{2}}, \frac{t_{n}}{\lambda_{n}^{2}}\right], 
\\[6pt]
\label{10/03/08/15:33}
&\left\| \psi_{n}(t)\right\|_{L^{2}}
=
\lambda_{n}^{-\frac{d(p-1)-4}{2(p-1)}}\left\| \psi_{0} \right\|_{L^{2}}
\quad 
\mbox{for all $t\in \left(-\frac{T_{\max}^{+}-t_{n}}{\lambda_{n}^{2}}, \frac{t_{n}}{\lambda_{n}^{2}}\right]$},
\\[6pt]
\label{10/03/08/15:34}
&\mathcal{H}(\psi_{n}(t))=\lambda_{n}^{\frac{2(p+1)}{p-1}-d}\mathcal{H}(\psi_{0})
\quad 
\mbox{for all $t\in \left(-\frac{T_{\max}^{+}-t_{n}}{\lambda_{n}^{2}}, \frac{t_{n}}{\lambda_{n}^{2}}\right]$},
\\[6pt]
\label{08/11/16/16:22}
&\sup_{t\in \big[0,\ \frac{t_{n}}{\lambda_{n}^{2}}\big]}\left\| \psi_{n}(t)\right\|_{L^{p+1}}=1,
\end{align}
where we put $\psi_{0}=\psi(0)$. Besides, we have that    
\begin{equation}\label{08/11/16/16:23}
\sup_{t\in \big[0,\frac{t_{n}}{\lambda_{n}^{2}}\big]}\left\| \nabla \psi_{n}(t)\right\|_{L^{2}}^{2} \le 1
\quad 
\mbox{for all sufficiently large $n \in \mathbb{N}$}.
\end{equation}
Indeed, (\ref{10/03/08/15:34}) and (\ref{08/11/16/16:22}) lead us to that   
\begin{equation}\label{10/03/08/15:52}
\begin{split}
\left\| \nabla \psi_{n}(t)\right\|_{L^{2}}^{2}
&= 
\mathcal{H}(\psi_{n}(t))+\
\frac{2}{p+1}\left\| \psi_{n}(t) \right\|_{L^{p+1}}^{p+1}
\\
&\le \lambda_{n}^{\frac{2(p+1)}{p-1}-d}\mathcal{H}(\psi_{0})+\frac{2}{p+1}
\quad 
\mbox{for all $t\in \big[0, \frac{t_{n}}{\lambda_{n}^{2}}\big]$},
\end{split}
\end{equation}
which, together with (\ref{10/03/08/17:59}), immediately yields (\ref{08/11/16/16:23}).
\par 
Now, we suppose that 
\begin{equation}\label{10/01/27/12:38}
\left\| \psi \right\|_{L^{\infty}([0,T_{\max}^{+});L^{\frac{d}{2}(p-1)})}<\infty,
\end{equation} 
so that $\{\psi_{n}\}$ satisfies that  
\begin{equation}\label{10/03/10/18:08}
\sup_{n\in \mathbb{N}}\left\| \psi_{n} \right\|_{L^{\infty}([0,\frac{t_{n}}{\lambda_{n}^{2}}];L^{\frac{d}{2}(p-1)})}
\le 
\left\| \psi \right\|_{L^{\infty}([0,T_{\max}^{+});L^{\frac{d}{2}(p-1)})}<\infty
.
\end{equation} 
Take any $T>0$. Then, extraction of some subsequence of $\{\psi_{n}\}$ allows us to assume that:  
\begin{equation}\label{10/04/12/18:23}
\sup_{t \in [0,T]}
\left\| \psi_{n}(t) \right\|_{L^{\frac{d}{2}(p-1)}}\lesssim 1,
\
 \sup_{t \in [0,T]}
\left\| \psi_{n}(t) \right\|_{L^{p+1}}=1,
\  
\sup_{t \in [0,T]}
\left\| \nabla \psi_{n}(t) \right\|_{L^{2}}\le 1 
\quad  \mbox{for all $n \in \mathbb{N}$},
\end{equation}
where the implicit constant is independent of $n$. The condition (\ref{10/04/12/18:23}) enable us to apply Lemmata \ref{08/10/25/23:47} and \ref{08/10/25/23:48}, so that: There exist a constant $\delta>0$ and a sequence $\{y_{n}\}_{n\in \mathbb{N}}$ in $\mathbb{R}^{d}$ such that, putting $\widetilde{\psi}_{n}(x,t)=\psi_{n}(x+y_{n},t)$,
 we have 
\begin{equation}\label{10/03/10/18:02}
\sup_{t \in [0,T]}\mathcal{L}^{d}\left( 
\left[ 
\left| \widetilde{\psi}_{n}(\cdot +y_{n}, t)\right| \ge \delta 
\right]\cap B_{1}(0)
\right)
\gtrsim 1
\quad 
\mbox{for all $n \in \mathbb{N}$},  
\end{equation}
where $B_{1}(0)=\{x \in \mathbb{R}^{d} | |x|< 1\}$ and the implicit constant is independent of $n$. Besides, we have by (\ref{10/03/08/15:29}) and 
(\ref{10/04/12/18:23}) that 
\begin{align}
\label{10/03/10/18:31}
&2i\displaystyle{\frac{\partial \widetilde{\psi}_{n}}{\partial t}} +\Delta \widetilde{\psi}_{n}+|\widetilde{\psi}_{n}|^{p-1}\widetilde{\psi}_{n}=0  \quad \mbox{in} \ \mathbb{R}^{d}\times [0,T],
\\[6pt]
\label{10/03/14/17:57}
&\widetilde{\psi}_{n} \in C([0,T];\dot{H}^{1}(\mathbb {R}^{d})\cap L^{p+1}(\mathbb{R}^{d}))
\quad 
\mbox{for all $n\in \mathbb{N}$},
\\[6pt]
\label{10/03/10/18:33}
&
\sup_{t \in [0,T]}\left\| \widetilde{\psi}_{n}(t) \right\|_{L^{p+1}}= 1
,
\quad 
\sup_{t \in [0,T]}\left\| \nabla \widetilde{\psi}_{n}(t) \right\|_{L^{2}}\le 1
\quad 
\mbox{for all $n \in \mathbb{N}$}.
\end{align}
We shall prove that $\{\widetilde{\psi}_{n} \}$ is an  equicontinuous sequence in $C([0,T];L^{2}(\Omega))$ for any compact set $\Omega \subset \mathbb{R}^{d}$.
 Let $\chi$ be a function in $C_{c}^{\infty}(\mathbb{R}^{d})$ such that $\chi\equiv 1$ on $\Omega$. Then, we can verify that 
\begin{equation}\label{10/03/09/14:13}
\begin{split}
&\left\| \chi \widetilde{\psi}_{n}(t)- \chi \widetilde{\psi}_{n}(s) 
\right\|_{L^{2}}^{2}
\\[6pt]
&=\int_{s}^{t}\frac{d}{dt'} \left\|\chi\widetilde{\psi}_{n}(t')
-\chi\widetilde{\psi}_{n}(s) \right\|_{L^{2}}^{2}dt'
\\[6pt]
&=2\Re{
\int_{s}^{t}
\!\!
\int_{\mathbb{R}^{d}}
\!
\overline{\chi(x)(\partial_{t}\widetilde{\psi}_{n})(x,t')}
\left\{ \chi(x)\widetilde{\psi}_{n}(x,t')-\chi(x)\widetilde{\psi}_{n}(x,s) \right\}dxdt'}
\quad
\mbox{for all $s,t \in [0,T]$}.
\end{split}
\end{equation}
This identity (\ref{10/03/09/14:13}), together with (\ref{10/03/10/18:31}), yields that
\begin{equation}\label{08/12/19/15:09}
\begin{split}
&\left\| \widetilde{\psi}_{n}(t)-\widetilde{\psi}_{n}(s)\right\|_{L^{2}(\Omega)}^{2}
\\[6pt]
&\le \left\| \chi \widetilde{\psi}_{n}(t)- \chi \widetilde{\psi}_{n}(s) \right\|_{L^{2}}^{2}
\\[6pt]
&\le 2  
|t-s| \left\| \chi \partial_{t} \widetilde{\psi}_{n} \right\|_{L^{\infty}([0,T];H^{-1})}
\left\| \chi \widetilde{\psi}_{n} \right\|_{L^{\infty}([0,T];H^{1})}
\\[6pt]
&\le 2  
|t-s| \left\{ \left\| \chi \Delta \widetilde{\psi}_{n} \right\|_{L^{\infty}([0,T];H^{-1})}+
\left\| \chi (|\widetilde{\psi}_{n}|^{p-1}\widetilde{\psi}_{n}) \right\|_{L^{\infty}([0,T];H^{-1})}
\right\}
\\[6pt]
&\qquad  \times 
 \left\{ \left\| \chi \widetilde{\psi}_{n}\right\|_{L^{\infty}([0,T];L^{2})}
+
\left\| (\nabla \chi) \widetilde{\psi}_{n }\right\|_{L^{\infty}([0,T];L^{2})}
+ 
\left\| \chi  \nabla \widetilde{\psi}_{n}\right\|_{L^{\infty}([0,T];L^{2})}
\right\}
\\[6pt]
&\hspace{264pt}\mbox{for all $s,t \in [0,T]$}.
\end{split}
\end{equation}
We consider the terms in the first parentheses on the right-hand side of (\ref{08/12/19/15:09}). They are estimated by the Sobolev embedding and the H\"older inequality as follows:  
\begin{equation}\label{10/03/09/14:30}
\begin{split}
&\left\| \chi \Delta \widetilde{\psi}_{n} \right\|_{L^{\infty}([0,T];H^{-1})}
+
\left\| 
\chi (|\widetilde{\psi}_{n}|^{p-1}\widetilde{\psi}_{n}) 
\right\|_{L^{\infty}([0,T];H^{-1})}
\\[6pt]
&=
\left\| \Delta (\chi  \widetilde{\psi}_{n})-(\Delta \chi)\widetilde{\psi}_{n} -2\nabla \chi \nabla \widetilde{\psi}_{n} \right\|_{L^{\infty}([0,T];H^{-1})}
+
\left\| 
\chi (|\widetilde{\psi}_{n}|^{p-1}\widetilde{\psi}_{n}) 
\right\|_{L^{\infty}([0,T];H^{-1})}
\\[6pt]
&\lesssim  \left\| \nabla (\chi \widetilde{\psi}_{n}) \right\|_{L^{\infty}([0,T];L^{2})}
+ \left\| (\Delta \chi)\widetilde{\psi}_{n} \right\|_{L^{\infty}([0,T];L^{\frac{p+1}{p}})}
+  \left\| \nabla \chi \nabla \widetilde{\psi}_{n} \right\|_{L^{\infty}([0,T];L^{2})}
\\[6pt]
&\qquad  \qquad 
+\left\| 
\chi (|\psi_{n}|^{p-1}\widetilde{\psi}_{n}) 
\right\|_{L^{\infty}([0,T];L^{\frac{p+1}{p}})}
\\[6pt]
&\le 
\left\| \nabla \chi \right\|_{L^{\frac{2(p+1)}{p-1}}}
\left\| \widetilde{\psi}_{n} \right\|_{L^{\infty}([0,T];L^{p+1})}
+
\left\| \chi \right\|_{L^{\infty}}
\left\| \nabla \widetilde{\psi}_{n} \right\|_{L^{\infty}([0,T];L^{2})}
\\[6pt]
&\qquad + 
\left\| \Delta \chi \right\|_{L^{\frac{p+1}{p-1}}}
\left\| \widetilde{\psi}_{n} \right\|_{L^{\infty}([0,T];L^{p+1})}
+ \left\| \nabla \chi \right\|_{L^{\infty}} 
\left\| \nabla \widetilde{\psi}_{n} \right\|_{L^{\infty}([0,T];L^{2})}
\\[6pt]
& \qquad \qquad +
\left\| 
\chi \right\|_{L^{\infty}}\left\|\widetilde{\psi}_{n} 
\right\|_{L^{\infty}([0,T];L^{p+1})}^{p+1}, 
\end{split}
\end{equation}
where the implicit constant depends only on $d$ and $p$. 
 On the other hand, the terms in the second parentheses on the right-hand side of (\ref{08/12/19/15:09}) are estimated by the H\"older inequality as follows: 
\begin{equation}\label{10/03/09/16:08}
\begin{split}
&
\left\| \chi \widetilde{\psi}_{n}\right\|_{L^{\infty}([0,T];L^{2})}
+
\left\| (\nabla \chi) \widetilde{\psi}_{n }\right\|_{L^{\infty}([0,T];L^{2})}
+ 
\left\| \chi  \nabla \widetilde{\psi}_{n}\right\|_{L^{\infty}([0,T];L^{2})}
\\[6pt]
&\le 2
\left\| \chi \right\|_{W^{1,\frac{2(p+1)}{p-1}}}
\left\| \widetilde{\psi}_{n} \right\|_{L^{\infty}([0,T];L^{p+1})}
+
\left\|\chi \right\|_{L^{\infty}}
\left\| \nabla \widetilde{\psi}_{n} \right\|_{L^{\infty}([0,T];L^{2})}.
\end{split}
\end{equation}
Hence, combining these estimates (\ref{08/12/19/15:09})--(\ref{10/03/09/16:08})  with (\ref{10/03/10/18:33}), we obtain that  
\begin{equation}\label{10/01/26/18:23}
\left\| \widetilde{\psi}_{n}(t)-\widetilde{\psi}_{n}(s)\right\|_{L^{2}(\Omega)}^{2}
\lesssim |t-s|,
\end{equation}
where the implicit constant depends only on $d$, $p$ and $\Omega$ ($\chi$ is determined by $\Omega$). Thus, we see that $\{\widetilde{\psi}_{n}\}$ is equicontinuous in $C([0,T];L^{2}(\Omega))$ for all compact set $\Omega \subset \mathbb{R}^{d}$.
\par 
The estimate (\ref{10/01/26/18:23}), with the help of the  Gagliardo-Nirenberg inequality and  (\ref{10/03/10/18:33}), also shows that $\{\widetilde{\psi}_{n} \}$ is an equicontinuous sequence in $C([0,T];L^{q}(\Omega))$ for all compact set $\Omega \subset \mathbb{R}^{d}$ and $q \in [2,2^{*})$. Hence, (\ref{10/03/10/18:33}) and the Ascoli-Arz\'era theorem, together with (\ref{10/03/10/18:02}) and (\ref{10/03/10/18:31}), yield that: There exist a subsequence of $\{\widetilde{\psi}_{n}\}$ (still denoted by the same symbol) and a nontrivial function $\psi_{\infty} \in L^{\infty}([0,\infty);\dot{H}^{1}(\mathbb{R}^{d})\cap L^{p+1}(\mathbb{R}^{d}))$ 
such that 
\begin{align}
\label{10/01/27/12:25}
&\lim_{n\to \infty}\widetilde{\psi}_{n}= \psi_{\infty} \quad \mbox{ strongly in $L^{\infty}([0,T];L_{loc}^{q}(\mathbb{R}^{d}))$} \quad \mbox{for all $q \in [2, 2^{*})$},
\\[6pt]
\label{10/03/09/16:53}
&\lim_{n\to \infty}\nabla \widetilde{\psi}_{n} = \nabla \psi_{\infty} 
\quad 
\mbox{weakly* in $L^{\infty}([0,T];L^{2}(\mathbb{R}^{d}))$},
\\[6pt]
\label{10/03/09/18:11}
&2i\frac{\partial \psi_{\infty}}{\partial t}+\Delta \psi_{\infty}+|\psi_{\infty}|^{p-1}\psi_{\infty}=0
\quad 
\mbox{in $\mathcal{D}'([0,\infty);\dot{H}^{-1}(\mathbb{R}^{d})+ L^{\frac{p+1}{p}}(\mathbb{R}^{d}))$}.
\end{align}

Finally, we hall prove the formula (\ref{10/03/10/18:47}). The Lebesgue dominated convergence theorem leads to that for all $\varepsilon>0$, there exists $R>0$ such that  
\begin{equation}\label{10/04/16/18:52}
\int_{|x|\le R}|\psi_{\infty}(0)|^{\frac{d(p-1)}{2}}\,dx
\ge 
(1-\varepsilon) 
\int_{\mathbb{R}^{d}}|\psi_{\infty}(0)|^{\frac{d(p-1)}{2}}\,dx.
\end{equation}
Combining this estimate with the strong convergence (\ref{10/01/27/12:25}), we have that 
\begin{equation}\label{10/04/16/19:00}
\lim_{n\to \infty}
\int_{|x|\le R}
|\widetilde{\psi}_{n}(x,0)|^{\frac{d(p-1)}{2}}\,dx
\ge 
(1-\varepsilon) 
\int_{\mathbb{R}^{d}}|\psi_{\infty}(0)|^{\frac{d(p-1)}{2}}\,dx.
\end{equation}
Hence, it follows from  (\ref{10/04/16/19:00}) and the relation
\begin{equation}\label{10/04/16/16:33}
\int_{|x|\le R}|\widetilde{\psi}_{n}(x,0)|^{\frac{d(p-1)}{2}}\,dx
=
\int_{|x-\lambda_{n}y_{n}|\le \lambda_{n}R}|\psi(x,t_{n})|^{\frac{d(p-1)}{2}}\,dx
\end{equation}
that the desired result holds. 
\end{proof} 

Finally, we shall give the proof of Proposition \ref{10/01/26/14:49}.

\begin{proof}[Proof of Proposition \ref{10/01/26/14:49}] 
We use the same assumptions, definitions and notation in Proposition \ref{08/11/02/22:52}, except the condition (\ref{10/01/26/15:36}). 
\par 
We find from the proof of Proposition \ref{08/11/02/22:52} that: For all $T>0$, there exists a subsequence of $\{\psi_{n}\}$ in $C([0,T];H^{1}(\mathbb{R}^{d}))$ (still denoted by the same symbol) with the following properties: 
\begin{align}
\label{10/03/11/15:27}
&\sup_{t\in [0,T]}\left\|\psi_{n}(t) \right\|_{L^{p+1}}=1
\quad \mbox{for all $n \in \mathbb{N}$},
\\[6pt]
&\label{10/03/11/15:39} 
\sup_{t \in [0,T]}\left\| \nabla \psi_{n}(t) \right\|_{L^{2}}\le 1
\quad 
\mbox{for all $n \in \mathbb{N}$},
\\[6pt]
&\label{10/03/11/15:36}
2i\frac{\partial \psi_{n}}{\partial t}+\Delta \psi_{n}+|\psi_{n}|^{p-1}\psi_{n}=0
\quad 
\mbox{in $\mathbb{R}^{d}\times [0,T]$}.
\end{align}
For such a subsequence $\{\psi_{n}\}$, we define renormalized functions $\Phi_{n}^{RN}$ by \begin{equation}\label{10/01/26/13:15}
\Phi_{n}^{RN}(x,t)=\psi_{n}(x,t)-e^{\frac{i}{2}t\Delta}\psi_{n}(x,0),
\quad n \in \mathbb{N}.
\end{equation}
Here, it is worth while noting that 
\begin{align}
\label{10/03/15/18:18}
&\Phi^{RN}_{n} \in C([0,T];H^{1}(\mathbb{R}^{d}))
\quad 
\mbox{for all $n \in \mathbb{N}$},
\\[6pt]
\label{10/03/11/16:00}
&\Phi_{n}^{RN}(t)
=\frac{i}{2}\int_{0}^{t}e^{\frac{i}{2}(t-t')\Delta}|\psi_{n}(t')|^{p-1}\psi_{n}(t')\,dt'
\quad 
\mbox{for all $n \in \mathbb{N}$}.
\end{align}
We shall show that  
\begin{equation}
\label{08/11/21/17:55}
\sup_{n\in \mathbb{N}}\left\|\Phi^{RN}_{n}\right\|_{L^{\infty}([0,T];H^{1})}
\le C_{T}
\end{equation}
for some constant $C_{T}>0$ depending only on $d$, $p$ and $T$. 
Applying the Strichartz estimate to the formula (\ref{10/03/11/16:00}), and using (\ref{10/03/11/15:27}), we obtain the following two estimates:  
\begin{equation}\label{10/03/11/16:22}
\begin{split}
\left\|\Phi^{RN}_{n}\right\|_{L^{\infty}([0,T];L^{2})}
&\lesssim 
\left\||\psi_{n}|^{p-1}\psi_{n} \right\|_{L^{\frac{4(p+1)}{4(p+1)-d(p-1)}}([0,T];L^{\frac{p+1}{p}})}
\\[6pt]
& \le T^{1-\frac{d(p-1)}{4(p+1)}}\left\| \psi_{n} \right\|_{L^{\infty}([0,T];L^{p+1})}^{p}
\\[6pt]
&\le T^{1-\frac{d(p-1)}{4(p+1)}},
\end{split}
\end{equation}
\begin{equation}\label{10/03/11/17:29}
\begin{split}
\left\| \nabla \Phi^{RN}_{n} \right\|_{L^{\infty}([0,T];L^{2})}
&\lesssim 
\left\| \nabla \left( |\psi_{n}|^{p-1}\psi_{n}\right) \right\|_{L^{\frac{4(p+1)}{4(p+1)-d(p-1)}}([0,T];L^{\frac{p+1}{p}})}
\\
& \le T^{1-\frac{d(p-1)}{2(p+1)}}\left\| \psi_{n}\right\|_{L^{\infty}([0,T];L^{p+1})}^{p-1}\left\| \nabla \psi_{n}\right\|_{L^{\frac{4(p+1)}{d(p-1)}}([0,T];L^{p+1})}
\\
&\le T^{1-\frac{d(p-1)}{2(p+1)}}\left\| \nabla \psi_{n}\right\|_{L^{\frac{4(p+1)}{d(p-1)}}([0,T];L^{p+1})},
\end{split}
\end{equation}
where the implicit constants depend only on $d$ and $p$. Therefore, for the desired estimate (\ref{08/11/21/17:55}), it suffices to show that  
\begin{equation}\label{08/11/19/23:00}
\sup_{n\in \mathbb{N}}
\left\| \nabla \psi_{n} \right\|_{L^{\frac{4(p+1)}{d(p-1)}}([0,T];L^{p+1})}
\le D_{T}
\end{equation}
for some constant $D_{T}>0$ depending only on $d$, $p$ and $T$. Here, note that  the pair $(p+1,\ \frac{4(p+1)}{d(p-1)})$ is admissible. In order to prove 
 (\ref{08/11/19/23:00}), we introduce an admissible pair $(q,r)$ with $q=p+2$ if $d=1,2$ and $q=\frac{1}{2}(p+1+2^{*})$ if $d\ge 3$, so that $p+1<q<2^{*}$. 
 Then, it follows from the integral equation for $\psi_{n}$ and the Strichartz estimate that     
\begin{equation}\label{10/04/16/17:45}
\begin{split}
&\left\|\nabla \psi_{n} \right\|_{L^{r}\left([0,T];L^{q}\right)}
\\[6pt]
&\lesssim 
\left\| \nabla \psi_{n}(0)\right\|_{L^{2}}
+ \left\| \nabla \left( |\psi_{n}|^{p-1} \psi_{n}\right) \right\|_{L^{\frac{4(p+1)}{4(p+1)-d(p-1)}}([0,T]; L^{\frac{p+1}{p}})}
\\[6pt]
&\lesssim 
\left\| \nabla \psi_{n}(0)\right\|_{L^{2}}
+ \left\| \psi_{n} \right\|_{ L^{\frac{2(p+1)(p-1)}{d+2-(d-2)p}}([0,T]; L^{p+1})}^{p-1}
\left\| \nabla \psi_{n} \right\|_{L^{\frac{4(p+1)}{d(p-1)}}([0,T];L^{p+1})} 
\\[6pt]
&\le
\left\| \nabla \psi_{n}(0)\right\|_{L^{2}} +
T^{\frac{d+2-(d-2)p}{2(p+1)}}\left\| \psi_{n} \right\|_{L^{\infty}([0,T];L^{p+1})}^{p-1}
\left\|\nabla \psi_{n} \right\|_{L^{\infty}([0,T];L^{2})}^{1-\frac{q(p-1)}{(q-2)(p+1)}}
\left\| \nabla \psi_{n} \right\|_{L^{r}([0,T];L^{q})}^{\frac{q(p-1)}{(q-2)(p+1)}},
\end{split}
\end{equation}
where the implicit constants depend only on $d$ and $p$. 
 Combining (\ref{10/04/16/17:45}) with (\ref{10/03/11/15:27}) and (\ref{10/03/11/15:39}), we obtain that
\begin{equation}\label{10/03/11/17:56}
\left\|\nabla \psi_{n} \right\|_{L^{r}\left([0,T];L^{q}\right)}
\lesssim  
 1 +
T^{\frac{d+2-(d-2)p}{2(p+1)}}\left\| \nabla \psi_{n} \right\|_{L^{r}([0,T];L^{q})}^{\frac{q(p-1)}{(q-2)(p+1)}},
\end{equation}
where the implicit constant depends only on $d$ and $p$. Since $0<\frac{q(p-1)}{(q-2)(p+1)}<1$, this estimate, together with the Young inequality ($ab \le 
 \frac{1}{\gamma}a^{\gamma}+\frac{1}{\gamma'}b^{\gamma'}$ for $\gamma,\gamma'>1$ with $\frac{1}{\gamma}+\frac{1}{\gamma'}=1$), yields that 
\begin{equation}\label{08/11/19/22:29}
\left\| \nabla \psi_{n} \right\|_{L^{r}([0,T];L^{q})}
\lesssim  1+T^{\frac{(q-2)\{d+2-(d-2)p\}}{4\{ q-(p+1)\}}}
, 
\end{equation}
where the implicit constant depends only on $d$ and $p$. Hence, interpolating 
 (\ref{10/03/11/15:39}) and  (\ref{08/11/19/22:29}), we obtain (\ref{08/11/19/23:00}), so that (\ref{08/11/21/17:55}) holds.  
\par 
Next, we shall show that $\{\Phi^{RN}_{n}\}$ is an equicontinuous sequence in 
 $C([0,T];L^{q}(\mathbb{R}^{d}))$ for all $q \in [2,2^{*})$. 
Differentiating the both sides of (\ref{10/03/11/16:00}), we obtain that 
\begin{equation}\label{10/03/15/16:24}
\begin{split}
\partial_{t}\Phi^{RN}_{n}(t)
&= \frac{i}{2}|\psi_{n}(t)|^{p-1}\psi_{n}(t)
-\frac{1}{4}\Delta \int_{0}^{t} e^{\frac{i}{2}(t-t')\Delta}|\psi_{n}(t')|^{p-1}\psi_{n}(t')\,dt'
\\[6pt]
&=\frac{i}{2}|\psi_{n}(t)|^{p-1}\psi_{n}(t)+\frac{i}{2}\Delta \Phi^{RN}_{n}(t).
\end{split}
\end{equation}
This formula (\ref{10/03/15/16:24}) and the H\"older inequality show that  
\begin{equation}\label{10/03/15/16:36}
\begin{split}
&\left\|\Phi^{RN}_{n}(t)-\Phi^{RN}_{n}(s)\right\|_{L^{2}}^{2}
\\[6pt]
&
=\int_{s}^{t}\frac{d}{dt'} \left\|\Phi^{RN}_{n}(t')-\Phi^{RN}_{n}(s) \right\|_{L^{2}}^{2}dt'
\\[6pt]
&=
2\Re \int_{s}^{t}\int_{\mathbb{R}^{d}} \overline{\partial_{t}\Phi^{RN}_{n}(x,t')}\left\{ \Phi^{RN}_{n}(x,t')-\Phi^{RN}_{n}(x,s)\right\}dxdt'
\\[6pt]
&\lesssim 
|t-s|\left\|\psi_{n} \right\|_{L^{\infty}([0,T];L^{p+1})}^{p}
\left\| \Phi^{RN}_{n} \right\|_{L^{\infty}([0,T];L^{p+1})}
+ |t-s|\left\| \nabla \Phi^{RN}_{n}\right\|_{L^{\infty}([0,T];L^{2})}^{2}
,
\end{split}
\end{equation}
where the implicit constant depends only on $d$ and $q$. Combining this estimate with (\ref{10/03/11/15:27}) and (\ref{08/11/21/17:55}),  we obtain 
\begin{equation}\label{10/03/15/16:54}
\left\|\Phi^{RN}_{n}(t)-\Phi^{RN}_{n}(s)\right\|_{L^{2}}^{2}
\lesssim |t-s|
\quad 
\mbox{for all $s,t \in [0,T]$},
\end{equation} 
where the implicit constant depends only on $d$, $q$ and $T$. Moreover, the Gagliardo-Nirenberg inequality, together with (\ref{08/11/21/17:55}) and (\ref{10/03/15/16:54}), shows that $\{\Phi^{RN}_{n}\}$ is an equicontinuous sequence in 
 $C([0,T];L^{q}(\mathbb{R}^{d}))$ for all $q \in [2,2^{*})$. 
\par 
Now, we are in a position to prove the main property of $\{\Phi^{RN}_{n}\}$, the alternatives (i) and (ii). We first suppose that  
\begin{equation}\label{10/03/15/17:02}
\lim_{n\to \infty}\sup_{t\in [0,T]}\left\| \Phi_{n}^{RN}(t)\right\|_{L^{\frac{d}{2}(p-1)}}=0.
\end{equation}
Then, a direct calculation immediately yields (\ref{08/11/20/17:56}).
\par 
On the other hand, when 
\begin{equation}\label{10/03/15/17:04}
A:=\limsup_{n\to \infty}\sup_{t\in [0,T]}\left\| \Phi_{n}^{RN}(t) \right\|_{L^{p+1}}>0,
\end{equation}
we can take a subsequence $\{\Phi^{RN}_{n}\}$ (still denoted by the same symbol) such that   
\begin{equation}\label{08/11/21/18:29}
\sup_{t \in [0,T]}\left\|\Phi^{RN}_{n}(t) \right\|_{L^{p+1}}
 \ge \frac{A}{2}
\qquad 
\mbox{for all  $n \in \mathbb{N}$}.
\end{equation}
This inequality (\ref{08/11/21/18:29}) and the uniform bound (\ref{08/11/21/17:55}) enable us to apply Lemmata \ref{08/10/25/23:47} and \ref{08/10/25/23:48}. Thus, we find that there exist a constant $\delta>0$ and a sequence $\{y_{n}\}_{n\in \mathbb{N}}$ in $\mathbb{R}^{d}$ such that, putting $\widetilde{\Phi}^{RN}_{n}(x,t)=\Phi^{RN}_{n}(x+y_{n},t)$, we have 
\begin{equation}\label{10/03/15/17:57}
\sup_{t \in [0,T]}\mathcal{L}^{d}\left( 
\left[ \left|\widetilde{\Phi}^{RN}_{n}(t)\right| >
\frac{\delta}{2} \right]
\cap B_{1}(0)
\right)>C
\end{equation}
for some constant $C>0$ being independent of $n$. 
Besides, we find by (\ref{08/11/21/17:55}) and (\ref{10/03/15/16:54}) that:  
 $\{\widetilde{\Phi}^{RN}_{n}\}$ is 
\begin{align}
\label{10/04/16/18:13}
&\mbox{a uniformly bounded sequence in $C([0,T];H^{1}(\mathbb{R}^{d}))$, and}
\\[6pt]
&\mbox{an equicontinuous sequence in $C([0,T];L^{q}(\mathbb{R}^{d}))$ for all $q\in [2,2^{*})$}. 
\end{align}
Hence, the Ascoli-Arzel\'a theorem, together with (\ref{10/03/15/17:57}), gives us that: There exist a subsequence of $\{\widetilde{\Phi}^{RN}_{n}\}$ (still denoted by the same symbol) and a nontrivial function $\Phi \in L^{\infty}([0,\infty);H^{1}(\mathbb{R}^{d}))$ such that 
\begin{align}
\label{08/12/18/16:24}
&\lim_{n\to \infty}\widetilde{\Phi}^{RN}_{n} = \Phi 
\quad \mbox{weakly* in $L^{\infty}([0,T];H^{1}(\mathbb{R}^{d}))$},
\\[6pt]
\label{08/12/17/18:26}
&
\lim_{n\to \infty}\widetilde{\Phi}^{RN}_{n}= \Phi 
\quad \mbox{strongly in $C([0,T];L_{loc}^{q}(\mathbb{R}^{d}))$
\quad 
for all  $q\in[2,2^{*})$}.
\end{align}
It remains to prove (\ref{10/03/15/19:52})--(\ref{08/11/20/17:57}). We find by (\ref{10/03/11/15:27}) that: There exists a function  
$F \in L^{\infty}([0,\infty);L^{\frac{p+1}{p}}(\mathbb{R}^{d}))$ such that   
\begin{equation}\label{08/12/18/16:23}
\lim_{n\to \infty}|\psi_{n}|^{p-1}\psi_{n} = F 
\quad 
\mbox{weakly* in $L^{\infty}([0,T];L^{\frac{p+1}{p}}(\mathbb{R}^{d}))$}.
\end{equation}
Then, it follows from the equation (\ref{10/03/15/16:24}) and (\ref{08/12/18/16:24}) that 
\begin{equation}\label{08/12/18/16:25}
2i\frac{\partial \Phi}{\partial t} +\Delta \Phi + F=0.
\end{equation}
Here, if $F$ were trivial, then $\Phi$ is so since $\displaystyle{\Phi(0)=\lim_{n\to \infty}\Phi^{RN}_{n}(0)=0}$. Therefore, $F$ is nontrivial.
\par 
We can prove the formula (\ref{08/11/20/17:57}) in a way similar to the proof of (\ref{10/03/10/18:47}) (cf. the estimates (\ref{10/04/16/18:52})--
 (\ref{10/04/16/16:33})). 
\end{proof}
\appendix
\section{Generalized virial identity}\label{08/10/19/14:57}
The proofs of Theorems \ref{08/05/26/11:53}, \ref{08/06/12/9:48} and \ref{08/04/21/9:28} are based on a generalization of the virial identity. To state it, we first introduce a positive function $w$ in $W^{3,\infty}([0,\infty))$, which is a variant of the function in \cite{Nawa8, Ogawa-Tsutsumi, Ogawa-Tsutsumi2}: 
\begin{equation}\label{10/02/27/22:14}
w(r)=
\left\{ \begin{array}{lcl} r &\mbox{if}& 0\le r <1,
 \\[6pt] 
 r-(r-1)^{\frac{d}{2}(p-1)+1} &\mbox{if}& 1\le r < 1+\left( \frac{2}{d(p-1)+2} \right)^{\frac{2}{d(p-1)}}, 
 \\[6pt] 
 \mbox{smooth and $w'\le 0$} &\mbox{if}& 1+\left( \frac{2}{d(p-1)+2} \right)^{\frac{2}{d(p-1)}} \le r <2, 
 \\[6pt]
 0 &\mbox{if}& 2\le r  . \end{array} \right.
\end{equation}
Since $w$ is determined by $d$ and $p$ only, we may assume that 
\begin{equation}\label{09/09/26/11:16}
K:=\left\| w \right\|_{W^{3,\infty}} \lesssim 1,
\end{equation}
where the implicit constant depends only on $d$ and $p$. 
Using this $w$, we define   
\begin{equation}\label{10/02/27/22:18}
\vec{w}_{R}(x)=(\vec{w}_{R}^{1}(x),\ldots , \vec{w}_{R}^{d}(x))
:=\frac{x}{|x|}Rw\left(\frac{|x|}{R}\right),
\quad 
R>0, \  x \in \mathbb{R}^{d}
\end{equation}
and 
\begin{equation}\label{10/02/27/22:19}
W_{R}(x):=2R \int_{0}^{|x|}w\left( \frac{r}{R}\right)\,dr,
\quad 
R>0, \  x \in \mathbb{R}^{d}.
\end{equation}
These functions have the following properties:  
\begin{lemma}\label{08/10/23/18:31}
Assume that $d\ge 1$ and $2+\frac{4}{d}\le p+1 <2^{*}$. 
Then, we have that  
\begin{align}
\label{09/09/25/13:24}
&\left| \vec{w}_{R}(x) \right|^{2} \le W_{R}(x) \quad 
\mbox{for all $R>0$ and $x \in \mathbb{R}^{d}$},
\\[6pt]
\label{08/10/23/18:42}
&\left\| \vec{w}_{R} \right\|_{L^{\infty}} \le 2R
\quad
\mbox{for all $R>0$}, 
\\[6pt]
\label{08/12/29/12:48}
&\left\| W_{R}\right\|_{L^{\infty}}
\le 8R^{2}
\quad
\mbox{for all $R>0$},
\\[6pt]
\label{08/10/23/18:35}
&\left\|\nabla \vec{w}_{R}^{j} \right\|  
\le 2dK
\quad  
\mbox{for all $R>0$ and $j=1, \ldots, d$},
\\[6pt]
\label{08/04/20/16:54}
&\|\Delta ({\rm div}\, \vec{w}_{R}) \|_{L^{\infty}} 
\le  
\frac{10d^{2}K}{R^{2}}
\quad 
\mbox{for all $R>0$},
\end{align}
where $K$ is the constant given in (\ref{09/09/26/11:16}). 
\end{lemma}

\begin{proof}[Proof of Lemma \ref{08/10/23/18:31}]
Since $w(0)=0$ and $w'\le 1$, we have that  
\[
\left| \vec{w}_{R}(x) \right|^{2}
=R^{2}w^{2} \left( \frac{|x|}{R}\right)
\\[6pt]
\le 2R \int_{0}^{|x|} w\left(\frac{r}{R}\right)w'\left(\frac{r}{R}\right)\,dr 
\le W_{R}(x),
\]
which shows (\ref{09/09/25/13:24}). 
\par 
Now, we can easily verify that that 
\begin{equation}\label{09/09/25/16:01}
\left\| w \right\|_{L^{\infty}}\le 2.
\end{equation}
This estimate (\ref{09/09/25/16:01}) immediately yields (\ref{08/10/23/18:42}): 
\begin{equation}\label{10/05/16/10:34}
\left| \vec{w}_{R}(x)\right|=R\left| w\left( \frac{|x|}{R}\right)\right|\le 2R.
\end{equation}
Moreover, (\ref{09/09/25/16:01}), together with the fact ${\rm supp}\,w \subset [0,2]$, gives (\ref{08/12/29/12:48}). Indeed, we have 
\begin{equation}\label{10/05/16/10:38}
\left|W_{R}(x)\right|
\le 
 2R \int_{0}^{2R} w\left(\frac{r}{R}\right)\,dr 
 \le  4R^{2} \left\| w \right\|_{L^{\infty}}
\le 8R^{2}.
\end{equation}

We shall prove (\ref{08/10/23/18:35}).  A simple calculation shows that   
\begin{equation}\label{09/09/26/10:14}
\begin{split}
\left| \partial_{k} \vec{w}_{R}^{j} (x)  \right| 
&=
\left|  
\frac{\delta_{jk}|x|^{2}-x_{j}x_{k}}{|x|^{3}}Rw\left( \frac{|x|}{R}\right) +\frac{x_{j}x_{k}}{|x|^{2}}w'\left( \frac{|x|}{R}\right)
 \right| 
\\[6pt]
&\le \frac{R}{|x|} \left| w\left( \frac{|x|}{R}\right)\right|
+\left| w'\left( \frac{|x|}{R}\right) \right|.
\end{split}
\end{equation}
Since we have 
\begin{equation}\label{09/09/26/10:15}
\frac{R}{|x|} \left| w\left( \frac{|x|}{R}\right)\right| \le \left\| w \right\|_{L^{\infty}}
\quad 
\mbox{for all $x \in \mathbb{R}^{d}$}, 
\end{equation}
the estimate (\ref{09/09/26/10:14}), together with (\ref{09/09/26/11:16}),
 leads to (\ref{08/10/23/18:35}).
\par 
The estimate (\ref{08/04/20/16:54}) follows from (\ref{09/09/26/11:16}) and the identity 
\begin{equation}\label{10/05/16/10:46}
\begin{split}
\Delta ({\rm div}\, \vec{w}_{R})(x)&=
\frac{1}{R^{2}}w'''\left(\frac{|x|}{R}\right)
+\frac{2(d-1)}{R|x|}w''\left(\frac{|x|}{R}\right)
\\[6pt]
& \quad +\frac{(d-1)(d-3)}{|x|^{2}}w'\left(\frac{|x|}{R}\right)-\frac{(d-1)(d-3)}{|x|^{3}}Rw\left(\frac{|x|}{R}\right).
\end{split}
\end{equation}
\end{proof}

\begin{lemma}\label{10/03/02/17:47}
Assume that $d\ge 1$ and $2+\frac{4}{d}\le p+1 <2^{*}$. Then, for any $m>0$, $C>0$ and $f \in L^{2}(\mathbb{R}^{d})$, there exists $R_{0}>0$ such that 
\begin{equation}\label{10/05/16/10:47}
\frac{C}{R^{2}} (W_{R},|f|^{2})<m
\quad 
\mbox{for all $R\ge R_{0}$}.
\end{equation}
\end{lemma}
\begin{proof}[Proof of Lemma \ref{10/03/02/17:47}]
For any $m>0$, $C>0$ and $f \in L^{2}(\mathbb{R}^{d})$, we can take $R_{0}'>0$ such that 
\begin{equation}\label{10/05/16/10:48}
\int_{|x|\ge R_{0}'}|f(x)|^{2}\,dx <\frac{m}{16C}. 
\end{equation}
Hence, it follows from (\ref{08/12/29/12:48}) that   
\begin{equation}\label{10/03/02/18:32}
\frac{C}{R^{2}}\int_{|x|\ge R_{0}'}W_{R}(x)|f(x)|^{2}\,dx
\le 
8C\int_{|x|\ge R_{0}'}|f(x)|^{2}\,dx <\frac{m}{2}
\quad 
\mbox{for all $R>0$}.
\end{equation}
Moreover, we have by the definition of $W_{R}$ (see (\ref{10/02/27/22:19})) 
that
\begin{equation}\label{10/03/02/18:16}
\frac{C}{R^{2}}\int_{|x|< R_{0}'}W_{R}(x)|f(x)|^{2}\,dx
\le C\frac{2R_{0}'K}{R}\left\|f\right\|_{L^{2}}^{2}
<\frac{m}{2}
\quad 
\mbox{for all $R> \frac{4CKR_{0}'\|f\|_{L^{2}}^{2}}{m}$}.
\end{equation}
Combining (\ref{10/03/02/18:32}) and (\ref{10/03/02/18:16}), we obtain the desired result.
\end{proof}
Now, we introduce our generalized virial identity: 
\begin{lemma}[Generalized virial identity]
\label{08/10/23/18:45}
Assume that $d\ge 1$ and $2+\frac{4}{d}\le p+1 < 2^{*}$. Then, we have
\begin{equation}\label{08/03/29/19:05}
\begin{split}
(W_{R}, |\psi(t)|^{2})
&=(W_{R},|\psi_{0}|^{2})
+2t \Im{(\vec{w}_{R} \cdot \nabla \psi_{0},\psi_{0})}+2\int_{0}^{t}\int_{0}^{t'}\mathcal{K}(\psi(t''))\,dt''dt'
\\[6pt]
&\qquad 
-2\int_{0}^{t}\int_{0}^{t'}\mathcal{K}^{R}(\psi(t''))\,dt''dt'
\\[6pt]
&\qquad -\frac{1}{2}\int_{0}^{t}\int_{0}^{t'}(\Delta ({\rm div}\,{\vec{w}_{R}}), |\psi(t'')|^{2})\,dt''dt'
\qquad 
\mbox{for all $R>0$}.
\end{split}
\end{equation}
Here,  $\mathcal{K}^{R}$ is defined by 
\begin{equation}\label{09/05/13/16:38}
\mathcal{K}^{R}(f)
=
\int_{\mathbb{R}^{d}}
\rho_{1}(x)|\nabla f(x)|^{2}
+
\rho_{2}(x) \left| \frac{x}{|x|} \cdot \nabla f(x)\right|^{2}
-
\rho_{3}(x)|f(x)|^{p+1} dx,
\quad 
f \in H^{1}(\mathbb{R}^{d}), 
\end{equation}
where 
\begin{align}
\label{10/04/11/16:09}
\rho_{1}(x)&:=1 -\frac{R}{|x|}w \left( \frac{|x|}{R}\right),
\\[6pt]
\label{10/04/11/16:10}
\rho_{2}(x)&:=\frac{R}{|x|}w \left(\frac{|x|}{R}\right)-w'\left( \frac{|x|}{R}\right),
\\[6pt]
\label{10/04/11/16:11}
\rho_{3}(x)&:=\frac{p-1}{2(p+1)}\left\{ d-w'\left( \frac{|x|}{R}\right)-\frac{d-1}{|x|}Rw\left( \frac{|x|}{R}\right) \right\}
.
\end{align}
\end{lemma}
\begin{remark}\label{09/09/25/16:47}
If $d=1$ or $\psi$ is radially symmetric, then we have 
\begin{equation}\label{10/05/16/11:02}
\mathcal{K}^{R}(\psi)=\int_{\mathbb{R}^{d}}\rho_{0}(x)|\nabla \psi(x)|^{2}-\rho_{3}(x)|\psi(x)|^{p+1}\, dx,
\end{equation}
where 
\begin{equation}\label{10/05/16/11:03}
\rho_{0}(x):=1-w'\left( \frac{|x|}{R}\right)
=\rho_{1}(x)+\rho_{2}(x).
\end{equation}
\end{remark}
\begin{proof}[Proof of Lemma \ref{08/10/23/18:45}]
By the formula (B.17) in \cite{Nawa8}, we immediately obtain (\ref{08/03/29/19:05}).
\end{proof}
In the next lemma, we give several properties of the weight functions $\rho_{1}$, $\rho_{2}$, $\rho_{3}$ and  $\rho_{0}$:
\begin{lemma}\label{08/04/20/0:27}
Assume that $d\ge 1$ and $2+\frac{4}{d}\le p+1 < 2^{*}$. Then, for all $R>0$, we have the followings:  
\begin{align}
\label{08/12/29/14:16}
& {\rm supp}\,{\rho_{j}}=\{x\in \mathbb{R}^{d} \bigm| |x|\ge R \}
\quad 
\mbox{for all $R>0$ and $j=0,1,2,3$}, 
\\[6pt]
\label{08/04/20/12:50}
&\inf_{x\in \mathbb{R}^{d}}\rho_{j}(x) \ge 0  
\quad \mbox{for all $R>0$ and $j=0,1,2,3$},
\\[6pt]
\label{08/04/20/12:51}
&\rho_{0}(x)=1
\quad   \mbox{if  $|x| \ge 2R$},
\\[6pt]
\label{10/03/01/15:06}
&\rho_{3}(x)=\frac{d(p-1)}{2(p+1)} 
\quad \mbox{if  $|x| \ge 2R$},
\\[6pt]
\label{08/12/29/12:53}
&\left\| \rho_{j} \right\|_{L^{\infty}}\le K_{j} 
\quad \mbox{for all $j=0,1,2,3$},
\\[6pt]
\label{08/04/20/13:06}
&\sup_{x\in \mathbb{R}^{d}}
\max\left\{ -\frac{x}{|x|}\nabla \sqrt{\rho_{3}(x)},\ 0 \right\}
\le  \frac{K_{3}'}{R},
\\[6pt] 
\label{08/04/20/16:07}
&\inf_{|x|\ge R}\frac{\rho_{0}(x)}{\rho_{3}(x)} \ge K_{4},
\end{align}
where $K_{j}$ ($j=1,2,3$), $K'_{3}$ and $K_{4}$ are some constants independent of $R$. 
\end{lemma}

\begin{proof}[Proof of Lemma \ref{08/04/20/0:27}]
We give proofs of (\ref{08/04/20/13:06}) and (\ref{08/04/20/16:07}) only. \par 
\par 
We first prove (\ref{08/04/20/13:06}). A direct calculation shows that   
\begin{equation}\label{09/09/26/14:14}
\frac{x}{|x|}\nabla \sqrt{\rho_{3}(x)}
=
\frac{ 
-\frac{1}{R}w''\left( \frac{|x|}{R}\right)
-\frac{d-1}{|x|}w'\left( \frac{|x|}{R}\right)
+\frac{(d-1)R}{|x|^{2}}w\left( \frac{|x|}{R}\right)
}
{\sqrt{\frac{2(p+1)}{p-1}}\sqrt{d-w'\left( \frac{|x|}{R}\right)-\frac{(d-1)R}{|x|}w\left( \frac{|x|}{R}\right)}}
.
\end{equation}
Since ${\rm supp}\,\rho_{3}= \{x \in \mathbb{R}^{d} \bigm| |x|\ge R\}$ (see (\ref{08/12/29/14:16})), it suffices to consider the case $|x|\ge R$. 
When $R\le |x|<R\left\{ 1+\left( \frac{2}{d(p-1)+2}\right)^{\frac{2}{d(p-1)}}\right\}$, we have from (\ref{09/09/26/14:14}) that 
\begin{equation}\label{09/09/26/15:04}
\begin{split}
&
\max\left\{ -\frac{x}{|x|}\nabla \sqrt{\rho_{3}(x)},\ 0 \right\}
\\[6pt]
&=\frac{
\max\left\{\left( \frac{|x|}{R}-1\right)^{\frac{d}{2}(p-1)-1}
\left( 
\frac{d^{2}(p-1)^{2}+4d(p-1)+4(d-1)}{4R}
+\frac{(d-1)R}{|x|^{2}}
-\frac{d(d-1)(p-1)+2}{2|x|}
\right)
 ,\ 0\right\}
}
{
\sqrt{\frac{2(p+1)}{p-1}}\left(\frac{|x|}{R}-1 \right)^{\frac{d}{4}(p-1)}
\sqrt{\frac{d}{2}(p-1)+d-\frac{(d-1)R}{|x|}}
}
\\[6pt]
&\le 
\left( \frac{|x|}{R}-1\right)^{\frac{d}{4}(p-1)-1}\frac{
\left\{ 
\frac{d^{2}(p-1)^{2}+4d(p-1)+4(d-1)}{4R}
+\frac{(d-1)R}{|x|^{2}}
+\frac{d(d-1)(p-1)+2}{2|x|}
\right\}
}
{
\sqrt{\frac{2(p+1)}{p-1}}
\sqrt{\frac{d}{2}(p-1)+1}
}
\\[6pt]
&\lesssim \frac{1}{R}
,
\end{split}
\end{equation}
where the implicit constant depends only on $d$ and $p$. 
\\
On the other hand, when $R\left\{ 1+\left( \frac{2}{d(p-1)+2}\right)^{\frac{2}{d(p-1)}}\right\}\le |x|$, we have 
\begin{equation}
w'\left(\frac{|x|}{R}\right)\le 0, \qquad  \frac{R}{|x|}w\left(\frac{|x|}{R}\right)\le 1
\end{equation}  
and therefore we obtain from (\ref{09/09/26/14:14}) that 
\begin{equation}\label{10/05/28/15:47}
\begin{split}
\max\left\{ -\frac{x}{|x|}\nabla \sqrt{\rho_{3}(x)},\ 0 \right\}
&\le 
\frac{ 
 \left| \frac{1}{R}w''\left( \frac{|x|}{R}\right) \right|
+\left| \frac{d-1}{|x|}w'\left( \frac{|x|}{R}\right)\right|
+ \left| \frac{(d-1)R}{|x|^{2}}w\left( \frac{|x|}{R}\right)\right|
}
{\sqrt{\frac{2(p+1)}{p-1}}\sqrt{d-\frac{(d-1)R}{|x|}w\left( \frac{|x|}{R}\right)}}
\\[6pt]
&\lesssim \frac{1}{R}\left\| w \right\|_{W^{2,\infty}} 
,
\end{split}
\end{equation}
where the implicit constant depends only on $d$ and $p$. 
Thus, we see that (\ref{08/04/20/13:06}) holds.  
\par 
Next, we  prove (\ref{08/04/20/16:07}). The starting point of the proof is the identity: 
\begin{equation} \label{08/04/20/16:16}
\frac{p-1}{2(p+1)}\frac{\rho_{0}(x)}{\rho_{3}(x)}=\frac{\rho_{0}(x)}{\rho_{0}(x)+(d-1)\left\{ 1-\frac{R}{|x|}w\left( \frac{|x|}{R}\right)\right\}}
\quad 
\mbox{for all $x \in \mathbb{R}^{d}$ with $|x|\ge R$}. 
\end{equation}
When $R \le |x| < R\left\{ 1+\left(\frac{2}{d(p-1)+2}\right)^{\frac{2}{d(p-1)}}\right\}$, we have from (\ref{08/04/20/16:16}) that  
\begin{equation}\label{10/03/06/16:18}\begin{split}
\frac{p-1}{2(p+1)} \frac{\rho_{0}(x)}{\rho_{3}(x)}&=
\frac{\rho_{0}(x)}{\rho_{0}(x)+(d-1)\frac{R}{|x|}\left(\frac{|x|}{R}-1 \right)^{\frac{d}{2}(p-1)+1}}
\\[6pt]
&=
\frac{ \left(\frac{|x|}{R}-1 \right)^{\frac{d}{2}(p-1)}}{\left(\frac{|x|}{R}-1 \right)^{\frac{d}{2}(p-1)}+\frac{2(d-1)}{d(p-1)+2}\frac{R}{|x|}
\left(\frac{|x|}{R}-1 \right)^{\frac{d}{2}(p-1)+1}}
\\[6pt]
&=
\frac{1}{1+\frac{2(d-1)}{d(p-1)+2}\left(1-\frac{R}{|x|} \right)}
\\[6pt]
&\ge 
\frac{1+\left(\frac{2}{d(p-1)+2} \right)^{\frac{2}{d(p-1)}}}
{1+\frac{d(p+1)}{d(p-1)+2}
\left( \frac{2}{d(p-1)+2} \right)^{\frac{2}{d(p-1)}}}.
\end{split}
\end{equation}
On the other hand, when $R\left\{ 1+\left( \frac{2}{d(p-1)+2} \right)^{\frac{2}{d(p-1)}}\right\} \le |x|$, we have 
\begin{equation}\label{10/05/28/16:47}
1\le \rho_{0}(x)\le 1+K,
\qquad 
0\le 1-\frac{R}{|x|}w\left( \frac{|x|}{R}\right)\le 1, 
\end{equation}
and therefore we obtain from (\ref{08/04/20/16:16}) that 
\begin{equation}\label{10/03/06/16:56}
\frac{p-1}{2(p+1)}\frac{\rho_{0}(x)}{\rho_{3}(x)} \ge \frac{1}{K+d}
\quad 
\mbox{for all $x \in \mathbb{R}^{d}$ with $R\left\{ 1+\left( \frac{2}{d(p-1)+2} \right)^{\frac{2}{d(p-1)}}\right\}\le |x|$}.
\end{equation}
Thus, we see that (\ref{08/04/20/16:07}) holds.
\end{proof}
\section{Compactness device I}
\label{08/10/03/15:12}
We recall the following sequence of lemmata.
\begin{lemma}[Fr\"olich, Lieb and Loss \cite{Frohlich-Leib-Loss}, Nawa \cite{Nawa1}]
\label{08/3/28/20:21}
Let $1< \alpha < \beta < \gamma < \infty$ and let $f$ be a measurable function on $\mathbb{R}^{d}$ with  
\[
\|f \|_{L^{\alpha}}^{\alpha} \le C_{\alpha},
\quad 
C_{\beta}\le \|f\|_{L^{\beta}}^{\beta},
\quad 
\|f\|_{L^{\gamma}}^{\gamma}\le C_{\gamma} 
\]
for some positive constants $C_{\alpha}, C_{\beta}, C_{\gamma}$. 
 Then, we have  
\[
 \mathcal{L}^{d}\left( \Big[ |f| > \eta \Big]\right) > \frac{C_{\beta}}{2}\eta^{\beta}
\quad 
\mbox{for all $0< \eta 
<
\min
\left\{1, \ \left( \frac{C_{\beta}}{4C_{\alpha}}\right)^{\frac{1}{\beta-\alpha}}, 
\ 
\left( \frac{C_{\beta}}{4C_{\gamma}}\right)^{\frac{1}{\gamma-\beta}}
\right\}$}.
\]
\end{lemma}
\begin{proof}[Proof of Lemma \ref{08/3/28/20:21}]
Let $\eta$ be a constant satisfying 
\begin{equation}\label{10/04/26/17:15}
0<\eta 
<
\min
\left\{1, \ \left( \frac{C_{\beta}}{4C_{\alpha}}\right)^{\frac{1}{\beta-\alpha}}, 
\ 
\left( \frac{C_{\beta}}{4C_{\gamma}}\right)^{\frac{1}{\gamma-\beta}}
\right\}.
\end{equation}
Then, we can easily verify that  
\begin{equation}\label{10/04/26/17:02}
\begin{split}
C_{\beta}
&=
\int_{\left[|f|\le \eta \right]}|f(x)|^{\beta}\,dx 
+
\int_{\left[\eta< |f|\le \frac{1}{\eta}\right]}|f(x)|^{\beta}\,dx
+
\int_{\left[\frac{1}{\eta}<|f|\right]}|f(x)|^{\beta}\,dx
\\[6pt]
&\le 
\eta^{\beta-\alpha} 
\int_{\left[|f|\le \eta \right]}|f(x)|^{\alpha}\,dx
+
\left(\frac{1}{\eta} \right)^{\beta}
\mathcal{L}\left( \Big[ |f|\ge \eta \Big] \right)
+
\eta^{\gamma-\beta} 
\int_{\left[|f|\le \eta \right]}|f(x)|^{\gamma}\,dx
\\[6pt]
&\le \frac{C_{\beta}}{2}+ \left(\frac{1}{\eta} \right)^{\beta}
\mathcal{L}\left( \Big[ |f|\ge \eta \Big] \right).
\end{split}
\end{equation}
This estimate (\ref{10/04/26/17:02}) immediately gives us the desired result.
\end{proof}

\begin{lemma}[Lieb \cite{Lieb}, Nawa \cite{Nawa1}]
\label{08/03/28/21:00}
Let $1< q < \infty$, and let $f$ be a function in $W^{1,q}(\mathbb{R}^{d})$ 
 with 
\begin{equation}\label{10/04/26/17:34}
\|\nabla f\|_{L^{q}}^{q}\le D_{1}, 
\quad 
\mathcal{L}^{d}\left( \Big[ |f| > \eta \Big] \right) \ge D_{2}
\end{equation}
for some positive constants $D_{1}$, $D_{2}$ and $\eta$. We put 
\begin{equation}\label{10/04/28/10:44}
q^{\dagger}=\left\{ \begin{array}{ccl}
\displaystyle{\frac{qd}{d-q}} &\mbox{if} &q<d,
\\[6pt]
2q &\mbox{if}& q \ge d .
\end{array}
\right.
\end{equation}
Then, there exists $y \in \mathbb{R}^{d}$ such that 
\begin{equation}\label{10/04/28/10:45}
\mathcal{L}^{d}\left( \left[ \left|f(\cdot + y)\right| \ge \frac{\eta}{2} \right] \cap B_{1}(0) \right)
\gtrsim 
\left( 
\frac{1+\eta^{q}D_{2}}{1+D_{1}}
\right)^{\frac{q}{q^{\dagger}-q}}
,
\end{equation}
where the implicit constant depends only on $d$ and $q$, and $B_{1}(0)$ is the ball in $\mathbb{R}^{d}$ with center $0$ and radius $1$.
\end{lemma}
\begin{proof}[Proof of Lemma \ref{08/03/28/21:00}]
Let $f$ be a function in $W^{1,q}(\mathbb{R}^{d})$ satisfying (\ref{10/04/26/17:34}). Put 
\begin{equation}\label{10/04/28/10:50}
g(x)=\max\left\{ |f(x)|-\frac{\eta}{2}, 0 \right\}.
\end{equation}
Then, we easily verify that 
\begin{align}
\label{10/04/28/10:55}
&g \in W^{1,q}(\mathbb{R}^{d}), 
\qquad 
\left\| \nabla g \right\|_{L^{q}(\mathbb{R}^{d})}
\le D_{1},
\\[6pt] 
\label{10/04/28/10:56}
&{\rm supp}\,{g}\subset \left[ |f|\ge \frac{\eta}{2}\right].
\end{align}
We first claim that   
\begin{equation}\label{10/04/26/17:35}
\int_{Q_{y}}|\nabla g(x)|^{q}\,dx 
< 
\left( 
1+ \frac{D_{1}}{\left\|g\right\|_{L^{q}}^{q}}
\right)
\int_{Q_{y}}|g(x)|^{q}\,dx
\quad 
\mbox{for some $y \in \mathbb{R}^{d}$},
\end{equation}
where $Q_{y}$ denotes the cube in $\mathbb{R}^{d}$ with the center 
$y$ and side length $\frac{2}{\sqrt{d}}$. We note that $Q_{y}$ is inscribed in $B_{1}(y)$ for all $y \in \mathbb{R}^{d}$. 
\par 
Let $\{y_{n}\}_{n\in \mathbb{N}}$ be a sequence in $\mathbb{R}^{d}$ such that 
\begin{equation}\label{10/04/26/17:43}
\bigcup_{n \in \mathbb{N}} Q_{y_{n}}=\mathbb{R}^{d},
\qquad 
\overset{\circ}{Q_{y_{m}}}\cap \overset{\circ}{Q_{y_{n}}}=\emptyset
 \quad \mbox{for $m\neq n$}
\end{equation}
Supposing the contrary that (\ref{10/04/26/17:35}) fails, we have 
\begin{equation}\label{10/04/28/15:48}
\int_{Q_{y_{n}}}|\nabla g(x)|^{q}\,dx 
\ge  
\left(
1+\frac{D_{1}}{\left\|g \right\|_{L^{q}}^{q}}
\right)
\int_{Q_{y_{n}}}|g(x)|^{q}\,dx
\quad 
\mbox{for all $n \in \mathbb{N}$}.
\end{equation}
Then, summing (\ref{10/04/28/15:48}) over all $n\in \mathbb{N}$ yields that 
\begin{equation}\label{10/04/28/16:00}
D_{1}\ge \int_{\mathbb{R}^{d}}|\nabla g(x)|^{q}\,dx 
\ge 
\left( 
1+\frac{D_{1}}{\left\|g \right\|_{L^{q}}^{q}}
\right)
\int_{\mathbb{R}^{d}}|g(x)|^{q}\,dx
>D_{1}, 
\end{equation} 
which is a contradiction. Hence, (\ref{10/04/26/17:35}) holds.
\par 
Now, it follows from (\ref{10/04/26/17:35}) that: There exists $y_{0} \in \mathbb{N}$ such that 
\begin{equation}\label{10/04/28/16:34}
\int_{Q_{y_{0}}}|\nabla g (x)|^{q}\,dx +|g(x)|^{q}\,dx 
< 
\left( 
2 +\frac{D_{1}}{1+\left\|g \right\|_{L^{q}}^{q}} 
\right)
\int_{Q_{y_{0}}}|g (x)|^{q}\,dx.
\end{equation}
On the other hand, the Sobolev embedding leads us to that 
\begin{equation}\label{10/04/28/16:41}
\left( \int_{Q_{y_{0}}}|g (x)|^{q^{\dagger}}\,dx,\right)^{\frac{q}{q^{\dagger}}}
\lesssim 
\int_{Q_{y_{0}}}|\nabla g (x)|^{q}\,dx +|g(x)|^{q}\,dx, 
\end{equation}
where $q^{\dagger}$ is the exponent defined in (\ref{10/04/28/10:44}), and the implicit constant depends only on $d$ and $q$. Combining these estimates (\ref{10/04/28/16:34}) and (\ref{10/04/28/16:41}), we obtain that 
\begin{equation}\label{10/04/28/16:47}
\begin{split}
\left( 
\int_{Q_{y_{0}}}|g (x)|^{q^{\dagger}}\,dx
\right)^{\frac{q}{q^{\dagger}}}
&\lesssim 
\left( 2 + 
\frac{D_{1}}{1+\left\|g\right\|_{L^{q}}^{q}} 
\right)
\int_{Q_{y_{0}}}|g (x)|^{q}\,dx
\\[6pt]
&\lesssim 
\left( 2 + 
\frac{D_{1}}{1+\left\|g\right\|_{L^{q}}^{q}} 
\right)
\mathcal{L}\left( Q_{y_{0}}\cap {\rm supp}\,g  \right)^{1-\frac{q}{q^{\dagger}}}
\left( \int_{Q_{y_{0}}}|g (x)|^{q^{\dagger}}\,dx
\right)^{\frac{q}{q^{\dagger}}},
\end{split}
\end{equation}
where the implicit constant depends only on $d$ and $q$.
Hence, we see that 
\begin{equation}\label{10/04/28/20:51}
\left(  
\frac{1+\left\|g\right\|_{L^{q}}^{q}}{
D_{1}+2+2\left\|g\right\|_{L^{q}}^{q}} 
\right)^{\frac{q^{\dagger}}{q^{\dagger}-q}}
\lesssim 
\mathcal{L}\left( Q_{y_{0}}\cap {\rm supp}\,g  \right),
\end{equation}
where the implicit constant depends only on $d$ and $q$. Here, it follows from  the definition of $g$ (see (\ref{10/04/28/10:50})) and the assumption (\ref{10/04/26/17:34}) that 
\begin{equation}\label{10/04/28/16:59}
\left\| g \right\|_{L^{q}}^{q}
\ge 
\int_{[|f|\ge \eta]}\left( |f(x)|-\frac{\eta}{2}\right)^{q}
\\[6pt]
\ge 
\left( \frac{\eta}{2} \right)^{q} 
D_{2} 
.
\end{equation}
Since ${\rm supp}\,g \subset \left[|f|\ge \frac{\eta}{2}\right]$ (see (\ref{10/04/28/10:56})), the estimate (\ref{10/04/28/20:51}), together with (\ref{10/04/28/16:59}), gives us the desired result.  
\end{proof}

\begin{lemma}[Lieb \cite{Lieb}, Nawa \cite{Nawa1}]
\label{08/09/27/22:43}
Let $1<q<\infty$, and let $\{f_{n}\}_{n \in \mathbb{N}}$ be a uniformly bounded sequence in $W^{1,q}(\mathbb{R}^{d})$ with  
\[
\inf_{n\in \mathbb{N}}\mathcal{L}^{d}\left( \Big[ |f_{n}| > \delta 
\Big]\right) \ge C
\]
for some constants $\delta>0$ and $C>0$. Then, there exists a sequence $\{y_{n}\}_{n\in \mathbb{R}^{d}}$ in $\mathbb{R}^{d}$ and a nontrivial function $f \in W^{1,q}(\mathbb{R}^{d})$ such that 
\begin{equation}
\label{10/04/28/21:17}
f_{n}(\cdot +y_{n}) \to f \quad \mbox{weakly in $W^{1,q}(\mathbb{R}^{d})$}.
\end{equation}
\end{lemma}

\begin{lemma}[Brezis and Lieb \cite{Brezis-Lieb}, Nawa \cite{Nawa1}] 
\label{08/03/28/21:21}
Let $0<q < \infty$ and let $\{f_{n}\}_{n \in \mathbb{N}}$ be a uniformly bounded sequence in $L^{q}(\mathbb{R}^{d})$ with $f_{n}\to f$ a.e. in $\mathbb{R}^{d}$ for some $f \in L^{q}(\mathbb{R}^{d})$. Then, we have 
\begin{equation}\label{08/09/28/12:26}
\lim_{n\to \infty} \int_{\mathbb{R}^{d}}\Big||f_{n}(x)|^{q}-|f_{n}(x)-f(x)|^{q}-|f(x)|^{q} \Big|\,dx = 0 .
\end{equation}
\end{lemma}
\section{Compactness device II}
\label{08/10/07/9:00}
We recall the following sequence of lemmas: All their proofs are found in Appendix C of \cite{Nawa8}.
\begin{lemma}[Nawa \cite{Nawa8}]
\label{08/10/25/23:47}
Let $1< \alpha < \beta < \gamma < \infty$, $I\subset \mathbb{R}$ and let $u \in C(I;L^{\alpha}(\mathbb{R}^{d}) \cap L^{\gamma}(\mathbb{R}^{d}))$ with  
\[
\sup_{t\in I}\|u(t) \|_{L^{\alpha}}^{\alpha} \le C_{\alpha},
\quad 
C_{\beta}\le \sup_{t\in I}\|u(t)\|_{L^{\beta}}^{\beta},
\quad 
\sup_{t\in I}\|u(t)\|_{L^{\gamma}}^{\gamma}\le C_{\gamma} 
\]
for some positive constants $C_{\alpha}$, $C_{\beta}$ and $C_{\gamma}$. Then, we have  \[
\sup_{t\in I} \mathcal{L}^{d}\left( \Big[ |u(t)| > \delta \Big]\right) >C
\]
for some constants $C>0$ and $\delta >0$ depending only on $\alpha$, $\beta$, $\gamma$, $C_{\alpha}$, $C_{\beta}$ and $C_{\gamma}$.  
\end{lemma}

\begin{lemma}[Nawa \cite{Nawa8}]
\label{08/10/25/23:48}
Let $1\le q < \infty$, $I \Subset \mathbb{R}$ and let $u$ be a function in $C(I;W^{1,q}(\mathbb{R}^{d}))$ such that 
\begin{align*}
&\sup_{t\in I}\|\nabla u(t)\|_{L^{q}}\le C_{q},
\\[6pt]
&\sup_{t\in I}\mathcal{L}^{d}\left( 
\Big[ |u(t)| > \delta \Big] \right)
\ge C
\end{align*} 
for some positive constants $C_{q}$, $\delta$ and  $C$. Then, there exists $y \in \mathbb{R}^{d}$ such that  
\[
\sup_{t\in I}\mathcal{L}^{d}\left( \left[ \left|u(\cdot + y,t)\right| > \frac{\delta}{2} \right] \cap B_{1}(0) \right)\ge C' 
\]
for some constant $C'>0$ depending only on $C_{q}$, $C$, $\delta$ and $d$, where $B_{1}(0)$ is the ball in $\mathbb{R}^{d}$ with center $0$ and radius $1$.
\end{lemma}

\begin{lemma}[Nawa \cite{Nawa8}] 
\label{08/10/28/08:14}
Let $0<q < \infty$, $I \Subset \mathbb{R}$ and let $\{u_{n}\}_{n \in \mathbb{N}}$ be a equicontinuous and uniformly bounded sequence in $C(I; L^{q}(\mathbb{R}^{d}))$ with $\displaystyle{\lim_{n\to \infty}u_{n}=u}$ a.e. in $\mathbb{R}^{d}\times I$ for some $u \in C(I; L^{q}(\mathbb{R}^{d}))$. Then, we have that:
\begin{align}
\label{09/12/05/16:05}
&\lim_{n\to \infty} \sup_{t \in I}\int_{\mathbb{R}^{d}}\Big||u_{n}(x,t)|^{q}-|u_{n}(x,t)-u(x,t)|^{q}-|u(x,t)|^{q} \Big|\,dx = 0,
\\[6pt]
\label{10/05/19/18:01}
&\lim_{n\to \infty}
\left\{ 
|u_{n}|^{q-1}u_{n}-|u_{n}-u|^{q-1}(u_{n}-u)-|u|^{q-1}u
\right\}= 0 
\quad 
\mbox{strongly in $L^{\infty}(I;L^{q'}(\mathbb{R}^{d}))$}.
\end{align}
\end{lemma}

\section{Elementary inequalities}\label{09/03/06/16:43}
In this section, we summarize elementary inequalities often used in this paper. We begin with the following obvious inequalities.  
\begin{lemma}\label{09/05/03/10:56}
Let $1<q<\infty$. Then, we have that 
\begin{align}
\label{09/09/27/21:09}
&\left| \left| a \right|^{q}-\left| b \right|^{q} \right| 
\lesssim  \left( |a|^{q-1}+|b|^{q-1}\right) |a-b|
\qquad 
\mbox{for all $a,b \in \mathbb {C}$},
\\[6pt]
\label{09/09/27/21:10}
&\left| \left| a \right|^{q-1}a -\left| b \right|^{q-1}b \right| 
\lesssim  \left( |a|^{q-1}+|b|^{q-1}\right) |a-b|
\qquad
\mbox{for all $a,b \in \mathbb {C}$},
\end{align}
where the implicit constants depend only on $q$. 
\end{lemma}
Besides the above, we also have the following inequality.

\begin{lemma}\label{09/03/06/16:46}
Let $0< q < \infty$, $L\in \mathbb{N}$ and $a_{1},\ldots, a_{L} \in \mathbb{C}$. Then, we have 
\[
\left| 
\left| \sum_{k=1}^{L}a_{k}  \right|^{q}\sum_{l=1}^{L}a_{l} 
-\sum_{l=1}^{L}\left| a_{l} \right|^{q}a_{l}
\right|
\le C \sum_{l=1}^{L}\sum_{{1\le k \le L} \atop { k\neq l}} \left| a_{l} \right| \left| a_{k} \right|^{q}
\]
for some constant $C>0$ depending only on $q$ and $L$.
\end{lemma}
\begin{proof}[Proof of Lemma \ref{09/03/06/16:46}] 
Since 
\[
\left| 
\left| \sum_{k=1}^{L}a_{k}  \right|^{q}\sum_{l=1}^{L}a_{l} 
-\sum_{l=1}^{L}\left| a_{l} \right|^{q}a_{l}
\right|
\le 
\sum_{l=1}^{L}\left| a_{l} \right|
\left| 
\left| \sum_{k=1}^{L}a_{k}  \right|^{q} 
-\left| a_{l} \right|^{q}
\right|,
\]
it suffices to prove that 
\begin{equation}\label{09/03/06/17:15}
\sum_{l=1}^{L}\left| a_{l} \right|
\left| 
\left| \sum_{k=1}^{L}a_{k}  \right|^{q}
-\left| a_{l} \right|^{q}
\right| 
\le C \sum_{l=1}^{L}\sum_{{1\le k \le L} \atop { k\neq l}} \left| a_{l} \right| \left| a_{k} \right|^{q}
\end{equation}
for some constant $C>0$ depending only on $q$ and $L$. Rearranging the sequence, we may assume that 
\[
|a_{1}|\le |a_{2}| \le \cdots \le |a_{L-1}|\le |a_{L}|.
\] 
Then, we have 
\[
\begin{split}
\sum_{l=1}^{L-1}\left| a_{l} \right|
\left| 
\left| \sum_{k=1}^{L}a_{k}  \right|^{q}
-\left| a_{l} \right|^{q}
\right| 
&\le 
\sum_{l=1}^{L-1}\left| a_{l} \right|
\left\{  \left( \sum_{k=1}^{L} \left| a_{k} \right| \right)^{q} + \left| a_{l} \right|^{q} \right\}
\\
&\le 
\sum_{l=1}^{L-1}\left| a_{l} \right| \left\{ \left(L|a_{L}|\right)^{q}+\left|a_{L}\right|^{q} \right\}
\\
&=(L^{q}+1)\sum_{l=1}^{L-1}\left| a_{l} \right| |a_{L}|^{q}
\\
&\le (L^{q}+1)\sum_{l=1}^{L}\sum_{{1\le k \le L}\atop {k\neq l}}|a_{l}|\left| a_{k} \right|^{q}.
\end{split}
\]
It remains an estimate for the term
\[
\left| a_{L} \right|
\left| 
\left| \sum_{k=1}^{L}a_{k}  \right|^{q}
-\left| a_{L} \right|^{q}
\right| .
\]
When $q>1$, we have by (\ref{09/09/27/21:09}) that 
\begin{equation*}\label{09/03/06/17:40}
\begin{split}
\left| a_{L} \right|
\left| 
\left| \sum_{k=1}^{L}a_{k}  \right|^{q}
-\left| a_{L} \right|^{q}
\right| 
&\lesssim \left| a_{L} \right| \left( \biggm| \sum_{k=1}^{L}  a_{k}  \biggm|^{q-1} + \left| a_{L} \right|^{q-1} \right) \left|\sum_{l=1}^{L-1} a_{l}  \right|
\\
&\le  \left(L^{q-1}+1 \right)\left| a_{L} \right|^{q}  \sum_{l=1}^{L-1}\left| a_{l}  \right|
\\
&\le  \left(L^{q-1}+1 \right) \sum_{l=1}^{L}\sum_{{1\le k \le L} \atop { k\neq l}} \left| a_{l} \right| \left| a_{k} \right|^{q},
\end{split}
\end{equation*}
where the implicit constant depends only on $q$. Hence, in the case $q>1$, we have obtained the result. We consider the case $0<q\le 1$. When $|a_{L-1}|\ge |a_{L}|/L$, we have 
\[
\begin{split}
\left| a_{L} \right|
\left| 
\left| \sum_{l=1}^{L}a_{l}  \right|^{q}
-\left| a_{L} \right|^{q}
\right| 
&\le 
\left| a_{L} \right|
\left( L^{q}+1 \right) \left| a_{L}\right|^{q}
\\
&\le L\left| a_{L-1} \right|
\left( L^{q}+1 \right) \left| a_{L}\right|^{q}
\\
&\le 
L\left(L^{q}+1 \right) \sum_{l=1}^{L}\sum_{{1\le k \le L} \atop { k\neq l}} \left| a_{l} \right| \left| a_{k} \right|^{q}.
\end{split}
\]
On the other hand, when $\left| a_{L-1}\right|\le \left| a_{L} \right|/L$, one can see that 
\[
\frac{\left| a_{L} \right|}{L} \le \left| a_{L} \right| - \left| \sum_{l=1}^{L-1}a_{l} \right| \le 
\left| \sum_{l=1}^{L}a_{l}\right|
\le 
\left| a_{L} \right|
+
\left| \sum_{l=1}^{L-1}a_{l} \right|.
\]
Moreover, we have by the convexity of the function $f(t)=t^{q}$ ($0<q\le 1$) that   
\[
\left(\left| a_{L} \right| +\left| \sum_{l=1}^{L-1}a_{l}\right|\right)^{q} 
-\left| a_{L} \right|^{q} \le \left| a_{L} \right|^{q} -\left( \left| a_{L} \right| -\left| \sum_{l=1}^{L-1}a_{l}\right| \right)^{q}.
\]
Hence, it follows from these estimates  that 
\[
\begin{split}
\left| \left| \sum_{l=1}^{L}a_{l}\right|^{q} -\left| a_{L} \right|^{q} \right|
&\le 
\left| a_{L}\right|^{q} - \left( \left| a_{L}\right| - \left| \sum_{l=1}^{L-1}a_{l} \right|\right)^{q}
\\
&\le 
q \left| \sum_{l=1}^{L-1}a_{l}\right|
\left(\left| a_{L}\right| -\left| \sum_{l=1}^{L-1}a_{l} \right| \right)^{q-1}
\\
&
\le 
q \left| \sum_{l=1}^{L-1}a_{l}\right| |a_{L}|^{q-1}
\end{split}
\]
and therefore  
\[
\left| a_{L} \right|\left| \left| \sum_{l=1}^{L}a_{l}\right|^{q} -\left| a_{L} \right|^{q} \right|
\le q \left| \sum_{l=1}^{L-1}a_{l}\right| |a_{L}|^{q}
\le q \sum_{l=1}^{L}\sum_{{1\le k \le L} \atop { k\neq l}} \left| a_{l} \right| \left| a_{k} \right|^{q}.
\]
Thus, we have completed the proof. 
\end{proof}

\section{Concentration function}
\label{08/10/07/9:02}
In this section, $B_{R}(a)$ denotes an open ball in $\mathbb{R}^{d}$ 
with the  center $a \in \mathbb{R}^{d}$ and the  radius $R$:
$$
  B_{R}(a)
   := 
  \bigl\{ 
                 x\in \mathbb{R}^d  
                                               \bigm | 
                                                              |x-a|<R
       \bigr\}.
$$

In order to trace the ``fake'' soliton, we prepare Proposition \ref{08/10/03/9:46} below.
We begin with:

\begin{lemma}[Nawa \cite{Nawa1-2}]\label{09/12/29/9:43}
Let  $0< T_{m} \le \infty$, $1<q<\infty$ and let $\rho$ be a non-negative function 
in $C([0,T_{m});L^{1}(\mathbb{R}^{d})\cap L^{q}(\mathbb{R}^{d}))$.
Suppose that 
\begin{equation}\label{10/05/19/19:45}
          \int_{|x-y_{0}|<R_{0}}\rho(x,t_{0})\,dx > C_{0}
\end{equation}
for some $C_{0}>0$, $R_{0}>0$ and $(y_{0},t_{0})\in \mathbb{R}^{d}\times [0,T_{m})$. 
Then, there exist $\theta >0$ and $\Gamma>0$, both depending only on 
 $d$, $q$, $\rho$, $C_{0}$, $R_{0}$ and $t_{0}$, such that 
\[
     \int_{|x-y|<R_{0}}\rho(x,t)\,dx  
         > C_{0},
         \quad
         \forall  (y,t)
              \in 
                B_{\Gamma}(y_0) \times (t_0-\theta, t_0 +\theta).
\] 
\end{lemma}
%
\begin{remark}
When $\rho \colon [0,T_{m}) \to L^{1}(\mathbb{R}^{d})$ is uniformly continuous, 
we can take $\theta$ independent of $t_{0}$.
Unfortunately, we are not in a case to assume the uniform continuity.
\end{remark}
%
\begin{proof}[Proof of Lemma \ref{09/12/29/9:43}]
Put 
\[
\varepsilon_{0}=\int_{|x-y_{0}|<R_{0}}\rho(x,t_{0})\,dx -C_{0}.
\]
Then, we have by (\ref{10/05/19/19:45}) that $\varepsilon_{0}>0$. 
Since we have by the H\"older inequality that 
\begin{equation}\label{10/05/19/19:49}
     \int_{B_{R_{0}}(y_{0}) \setminus B_{R_{0}}(y)}
       \rho(x,t_{0})
       \,dx
      \le 
         \mathcal{L}^{d}
         \left(
                 B_{R_{0}}(y_{0})\setminus B_{R_{0}}(y)
          \right)^{1-\frac{1}{q}}
                          \|\rho (t_{0})\|_{L^{q}(\mathbb{R}^{d})},
            \quad 
  \forall y \in \mathbb{R}^{d},
\end{equation}
we can take $\Gamma>0$, depending only on $d$, $q$, $\rho$, 
 $C_{0}$, $R_{0}$ and $t_{0}$, such that
\begin{equation}\label{09/12/29/10:59}
       \int_{B_{R_{0}}(y_{0}) \setminus B_{R_{0}}(y)}
          \rho(x,t_{0})
          \,dx
         < 
             \frac{\varepsilon_{0}}{3},
              \quad 
              \forall y \in B_{\Gamma}(y_{0}).
\end{equation}
Moreover, it follows from  (\ref{09/12/29/10:59}) that 
\begin{equation}\label{10/05/19/19:53}
  \begin{split}
       &    \int_{B_{R_{0}}(y_{0})}\rho(x,t_{0})\,dx
        \\[6pt]
         &    =
              \int_{B_{R_{0}}(y)}\rho(x,t_{0})\,dx
               +
                  \int_{B_{R_{0}}(y_{0}))\setminus  B_{R_{0}}(y)}
                        \hspace{-12pt}\rho(x,t_{0})
                     \,dx 
                   -
                        \int_{B_{R_{0}}(y)\setminus B_{R_{0}}(y_{0})}
                            \hspace{-12pt} \rho(x,t_{0})
                            \,dx
                            \\[6pt]
             &    <
                      \int_{B_{R_{0}}(y)}\rho(x,t_{0})\,dx
                       +
                         \frac{\varepsilon_{0}}{3},
                  \qquad 
                \forall  y \in B_{\Gamma}(y_{0}),
   \end{split}
\end{equation}
so that  we have 
\begin{equation}\label{08/10/05/15:10}
   \int_{B_{R_{0}}(y)}\rho(x,t_{0})\,dx
    >
      \int_{B_{R_{0}}(y_{0})}\rho(x,t_{0})\,dx
         - 
           \frac{\varepsilon_{0}}{3},
     \quad 
     \forall y \in B_{\Gamma}(y_{0}).
\end{equation}
Now, since $\rho$ is non-negative, we have that, 
for  any measurable set $\Omega \subset \mathbb{R}^{d}$,  
\begin{equation}\label{09/12/29/11:12}
   \int_{\Omega} \rho(x,t_{0})\,dx
        - 
        \int_{\Omega}\rho(x,t)\,dx 
    \le 
    \left\|\rho(t)  -    \rho(t_{0})\right\|_{L^{1}(\mathbb{R}^{d})},
    \quad
    \forall  t \in [0,T_{m}).
\end{equation}
This estimate (\ref{09/12/29/11:12}),
 together with the continuity of $\rho \colon [0,T_{m})\to L^{1}(\mathbb{R}^{d})$, gives us  that  
 there exists $\theta >0$, depending only on $d$, $\rho$, $C_{0}$, $R_{0}$ and $t_{0}$, such that 
\begin{equation}\label{10/05/19/20:05}
    \begin{split}
     \int_{B_{R_{0}}(y)\cap B_{R_{0}}(y_{0})}\rho(x,t_{0})\,dx 
     & -
        \int_{B_{R_{0}}(y)\cap B_{R_{0}}(y_{0})}\rho(x,t)\,dx 
       < 
         \frac{\varepsilon_{0}}{3},\\[8pt]
       &  \qquad\qquad
         \forall (y,t)\in \mathbb{R}^d \times (t_0-\theta, t_0 + \theta),
    \end{split}
\end{equation}
and 
\begin{equation}\label{10/05/19/20:06}
  \begin{split}
        \int_{B_{R_{0}}(y)\setminus B_{R_{0}}(y_{0})}
         \rho(x,t_{0})\,dx 
        &  - 
            \int_{B_{R_{0}}(y)\setminus B_{R_{0}}(y_{0})}
              \rho(x,t)\,dx 
             <
                 \frac{\varepsilon_{0}}{3},\\[8pt]
        &    \qquad\qquad
                    \forall (y,t)\in \mathbb{R}^d \times (t_0-\theta, t_0 + \theta),
   \end{split}
\end{equation}
Combining these estimates, we see that
there exists $\theta >0$, depending only on $d$, $\rho$, $C_{0}$, $R_{0}$ and $t_{0}$, such that   
\begin{equation}\label{08/10/05/15:31}
 \int_{B_{R_{0}}(y)}\rho(x,t_{0})\,dx 
 - 
          \int_{B_{R_{0}}(y)}
            \rho(x,t)\,dx 
          < 
              \frac{2\varepsilon_{0}}{3},\\[8pt]
\qquad
                         \forall (y,t)\in \mathbb{R}^d \times (t_0-\theta, t_0 + \theta).
\end{equation}
Moreover, it follows from the estimates (\ref{08/10/05/15:10}) and (\ref{08/10/05/15:31}) that
there exist $\theta >0$ and $\Gamma>0$, 
both depending only on $d$, $q$, $\rho$, $C_{0}$, $R_{0}$ and $t_{0}$, such that 
\[
   \begin{split}
         \int_{|x-y|<R_{0}}\rho(x,t)\,dx 
       &  = 
            \int_{B_{R_{0}}(y)}\rho(x,t_{0})\,dx
             -
              \left( 
                       \int_{B_{R_{0}}(y)}\rho(x,t_{0})\,dx
                         -
                           \int_{B_{R_{0}}(y)}\rho(x,t)\,dx
                 \right)\\[6pt]
        &  > 
                \int_{B_{R_{0}}(y)}\rho(x,t_{0})\,dx
                  -
                  \frac{2\varepsilon_{0}}{3}\\[6pt]
        &   \ge 
                    \int_{B_{R_{0}}(y_{0})}\rho(x,t_{0})\,dx
                   - 
                    \varepsilon_{0} = C_{0}
     \end{split}
\]
{for all $y \in \mathbb{R}^{d}$ with $|y-y_{0}|<\Gamma$, and all $t \in [0,T_{m})$ with $|t-t_{0}|<\theta$}.
Thus, we have proved the lemma.
\end{proof}

\begin{proposition}[Nawa \cite{Nawa1-2, Nawa1-3}]\label{08/10/03/9:46}
Assume that $d\ge 1$. Let $0< T_{m} \le \infty$, $1<q<\infty$, and let $\rho$ be a non-negative function 
in $C([0,T_{m});L^{1}(\mathbb{R}^{d})\cap L^{q}(\mathbb{R}^{d}))$
with  
\begin{equation}\label{09/12/27/11:28}
       \|\rho(t)\|_{L^{1}(\mathbb{R}^{d})}=1 \quad 
        \mbox{ for all $t \in [0,T_{m})$}.
\end{equation}
We put
\[
    A_{\rho}
     =
        \sup_{R>0}\liminf_{t \to T_{m}}\sup_{y\in \mathbb{R}^{d}}
          \int_{|x-y|\le R}\rho(x,t)\,dx.
\]
If $A_{\rho}>\frac{1}{2}$, then for all $\varepsilon \in (0,1)$, 
there exist $R_{\varepsilon}>0$, $T_{\varepsilon}>0$ 
and a continuous path $\gamma_{\varepsilon} \in C([T_{\varepsilon},T_{m}); \mathbb{R}^{d})$ such that 
\[
        \int_{|x-\gamma_{\varepsilon}(t)|\le R}
           \rho(x,t)
           \,dx 
        \ge 
               (1-\varepsilon)A_{\rho} 
               \quad 
         \mbox{ for all $R\ge 3R_{\varepsilon}$ and $t \in [T_{\varepsilon}, T_{m})$}.
\]
\end{proposition}
%

\vspace{0.3cm}
\begin{proof}[Proof of Proposition \ref{08/10/03/9:46}] 
We put $\eta:=2A_{\rho}-1$, so that $A_{\rho}=\frac{1}{2}+\frac{\eta}{2}$. 
Then, the assumptions (\ref{09/12/27/11:28}) and $A_{\rho}>\frac{1}{2}$ imply that $0< \eta \le 1$. 
We choose arbitrarily $\varepsilon \in (0,\frac{\eta}{1+\eta})$  and fix it.

It follows from the definition of $A_{\rho}$ that 
for our $\varepsilon>0$, there exist $R_{\varepsilon}>0$ and $T_{\varepsilon}>0$ with the following property: 
for all $t \in [T_{\varepsilon},T_{m})$, there exists a point $y_{\varepsilon}(t) \in \mathbb{R}^{d}$ such that
\begin{equation}\label{08/10/05/22:09}
      \int_{|x-y_{\varepsilon}(t)| \le R_{\varepsilon}} 
         \rho(x,t)
           \,dx 
            > 
            \left(
                     1-\varepsilon 
            \right)
                       A_{\rho}.
\end{equation}
For these $R_{\varepsilon}$ and $T_{\varepsilon}$, 
we shall construct a continuous path $\gamma_{\varepsilon} \colon [T_{\varepsilon}, T_{m}) \to \mathbb{R}^{d}$ 
satisfying that 
\begin{equation}\label{10/01/31/16:23}
      \int_{|x-\gamma_{\varepsilon}(t)| \le 3R_{\varepsilon}} 
        \rho(x,t)
         \,dx 
       \ge 
              \left(
                       1-\varepsilon
               \right)
                           A_{\rho} 
         \quad 
             \mbox{for all $t \in [T_{\varepsilon},T_{m})$},
\end{equation}
which will lead us to  the desired conclusion. 
To this end, we define:
\begin{equation}\label{09/12/27/16:58}
       T_{\varepsilon}^{*}
        =
            \sup
                       \left\{
                                t \in [T_{\varepsilon}, T_{m}) 
                                    \biggm| 
                                     \int_{|x-y_{\varepsilon}(T_{\varepsilon})| \le R_{\varepsilon}} 
                                       \rho(x,t)
                                        \,dx 
                                              >
                                              \left(
                                                         1- \varepsilon
                                                  \right) 
                                                              A_{\rho}
                               \right\},
\end{equation}
where $y(T_{\varepsilon})$ is a point found in (\ref{08/10/05/22:09}). 
By Lemma \ref{09/12/29/9:43}, it follows from (\ref{08/10/05/22:09}) that 
 $T_{\varepsilon}^{*}> T_{\varepsilon}$.
 If $T_{\varepsilon}^{*}=T_{m}$, 
then  there is nothing to prove: 
$\gamma_{\varepsilon}(t) \equiv y_{\varepsilon}(T_{\varepsilon})$ $(t\in  [T_{\varepsilon}, T_{m}) )$
is the desired continuous path. 

 Hence, we consider the case of $T_{\varepsilon}^{*}<T_{m}$. 
In this case, by Lemma \ref{09/12/29/9:43} again, we find that: 
 there exists a constant $\theta_{\varepsilon}(T_{\varepsilon}^{*})>0$ such that
\begin{equation}\label{08/10/16/0:41}
         \int_{|x-y_{\varepsilon}(T_{\varepsilon}^{*})| \le R_{\varepsilon}} 
          \rho(x,t)
            \,dx 
            > 
                \left(
                        1-\varepsilon 
                 \right)
                            A_{\rho}
           \quad 
                 \mbox{
                              for all 
                                        $
                                             t \in (T_{\varepsilon}^{*}-\theta_{\varepsilon}(T_{\varepsilon}^{*}), 
                                             T_{\varepsilon}^{*}+\theta_{\varepsilon}(T_{\varepsilon}^{*}))
                                          $
                              }.
\end{equation}
Here, we claim that:   
\begin{equation}\label{09/12/27/19:35}
      \left(
                   B_{R_{\varepsilon}}(y_{\varepsilon}(T_{\varepsilon}^{*})) 
                        \times 
                          \{t\} 
            \right)
             \cap 
               \left( 
                          B_{R_{\varepsilon}}(y_{\varepsilon}(T_{\varepsilon}))
                             \times 
                                \{t\} 
                    \right) 
              \neq 
                         \emptyset 
              \quad\text{for}\quad
                         t \in (T_{\varepsilon}^{*}-\theta_{\varepsilon}(T_{\varepsilon}^{*}) ,T_{\varepsilon}^{*}].
\end{equation}
For: we suppose the contrary that 
$t_{1} \in (T_{\varepsilon}^{*}-\theta_{\varepsilon}(T_{\varepsilon}^{*}), T_{\varepsilon}^{*}]$ 
such that  
\begin{equation}\label{09/12/28/11:50}
   \left( 
            B_{R_{\varepsilon}}(y_{\varepsilon}(T_{\varepsilon}^{*}))
                \times 
                \{t_{1}\} 
      \right) 
               \cap 
                      \left( 
                                B_{R_{\varepsilon}}(y_{\varepsilon}(T_{\varepsilon})) 
                                  \times 
                                    \{t_{1}\}
                         \right)
             = 
                  \emptyset .
\end{equation}
Then, it follows from the assumption (\ref{09/12/27/11:28}) and (\ref{09/12/28/11:50}) that 
\begin{equation}\label{09/12/27/20:02}
     \begin{split}
     1
        &   \ge 
                \int_{
                             B_{R_{\varepsilon}}(y_{\varepsilon}(T_{\varepsilon}^{*}))
                          \cup 
                               B_{R_{\varepsilon}}(y_{\varepsilon}(T_{\varepsilon}))
                          } 
                              \rho(x,t_{1})
                                 \,dx\\[6pt]
         &     = 
                 \int_{
                           B_{R_{\varepsilon}}(y_{\varepsilon}(T_{\varepsilon}^{*}))
                          } 
                             \rho(x,t_{1})
                      \,dx
                      +
                        \int_{
                                   B_{R_{\varepsilon}}(y_{\varepsilon}(T_{\varepsilon}))
                                 } 
                                     \rho(x,t_{1})\,dx.
          \end{split}
\end{equation}
Moreover, (\ref{08/10/16/0:41}) and the definition of $T_{\varepsilon}^{*}$  yeild that 
the right-hand side of (\ref{09/12/27/20:02}) is greater than $2(1-\varepsilon)A_{\rho}$. 
Hence, we obtain  
\begin{equation*}\label{09/12/27/20:06}
     1 
        > 
             2(1-\varepsilon)A_{\rho}
               =
                   2(1-\varepsilon)
                       \left(
                                 \frac{1}{2}+\frac{\eta} {2}
                           \right) 
                = 
                    1
                        +
                             \eta
                                  -
                                       \varepsilon
                                          (1+\eta),
\end{equation*}
so that  we have
\begin{equation}\label{10/01/31/22:37}
        \varepsilon
          > 
            \frac{\eta}{1+\eta},
\end{equation}
which (\ref{10/01/31/22:37}) contradicts our choice of  
$\varepsilon < \dfrac{\eta}{1+\eta}$. 
Therefore (\ref{09/12/27/19:35}) holds valid. 

\vspace{0.3cm}
Now, for our $\varepsilon \in  (0,\frac{\eta}{1+\eta})$, 
we define a path 
$\gamma_{\varepsilon}^{*} \colon [T_{\varepsilon},T_{\varepsilon}^{*}] \to \mathbb{R}^{d}$ by 
\begin{equation}\label{09/12/29/18:05}
    \gamma_{\varepsilon}^{*}(t)
      =
          \begin{cases}
                        \quad  y_{\varepsilon}(T_{\varepsilon})  &   \mbox{if} \qquad 
                                                         T_{\varepsilon} 
                                                             \le  
                                                                 t
                                                                   <
                                                                      T_{\varepsilon}^{*}    -  \theta_{\varepsilon}(T_{\varepsilon}^{*}),
                                                         \\[8pt]
                        \quad   y_{\varepsilon}(T_{\varepsilon}^{*})
                                         +
                                      \displaystyle\frac{T_{\varepsilon}^{*} - t}{\theta_{\varepsilon} (T_{\varepsilon}^{*})}
                                                   \left( 
                                                           y_{\varepsilon}(T_{\varepsilon})  -     y_{\varepsilon}(T_{\varepsilon}^{*})
                                                       \right) 
                                                                                 &\mbox{if}  \qquad
                                                            T_{\varepsilon}^{*}      -     \theta_{\varepsilon}(T_{\varepsilon}^{*}) 
                                                              \le 
                                                                  t 
                                                                \le
                                                                    T_{\varepsilon}^{*},\\[8pt]
                       \quad  y_{\varepsilon}(T_{\varepsilon}^{*})  &   \mbox{if} \qquad 
                                                         T_{\varepsilon}^{*} 
                                                             \le  
                                                                 t
                                                                   <
                                                                      T_{\varepsilon}^{*}+\theta_{\varepsilon}(T_{\varepsilon}^{*}).
              \end{cases}
\end{equation} 
Clearly we have 
$
\gamma_{\varepsilon}^{*}
  \in 
  C([T_{\varepsilon},T_{\varepsilon}^{*}+\theta_{\varepsilon}(T_{\varepsilon}^{*})]; \mathbb{R}^{d})
$. 
Moreover, it follows from (\ref{09/12/27/19:35})  that  
\begin{equation}\label{10/05/19/21:10}
    B_{3R_{\varepsilon}}(\gamma_{\varepsilon}^{*}(t))
           \supset 
             B_{R_{\varepsilon}}(y_{\varepsilon}(T_{\varepsilon}^{*}))
             \cup
             B_{R_{\varepsilon}}(y_{\varepsilon}(T_{\varepsilon})) 
        \quad 
                  \mbox{for all   $t 
                    \in 
                       [T_{\varepsilon}, T_{\varepsilon}^{*}+\theta_{\varepsilon}(T_{\varepsilon}^{*})]$}
\end{equation}
which, together with the definition of $T_{\varepsilon}^{*}$, 
yields that   
\begin{equation}\label{08/10/16/1:12}
    \begin{split}
            \int_{|x-\gamma_{\varepsilon}^{*}(t)|\le 3R_{\varepsilon}} 
                      \rho(x,t)
                    \,dx
             & \ge 
                     \int_{
                                            B_{R_{\varepsilon}}(y_{\varepsilon}(T_{\varepsilon}^{*}))
                                              \cup
                                                        B_{R_{\varepsilon}}(y_{\varepsilon}(T_{\varepsilon})) 
                                               }
                                         \rho(x,t)\,dx\\[8pt]
             &
                            \ge 
                      \left(
                                  1-\varepsilon 
                         \right)
                                                                A_{\rho}
                \qquad 
                    \mbox{ for all $t \in [T_{\varepsilon},T_{\varepsilon}^{*}+\theta_{\varepsilon}(T_{\varepsilon}^{*})]$}.
     \end{split}
\end{equation}
Here, we have enlarged the radius of the ball centered at 
the $\gamma_{\varepsilon}^{*}(t)$ for 
 $t \in [T_{\varepsilon},T_{\varepsilon}^{*}+\theta_{\varepsilon}(T_{\varepsilon}^{*})]$.

We shall extend the path $\gamma_{\varepsilon}^{*}$ to the whole interval $[T_{\varepsilon}, T_{m})$. 
We recall  (\ref{08/10/05/22:09}), so that we have from Lemma \ref{09/12/29/9:43} that,
 for all 
                            $\tau \in (T_{\varepsilon}^*,T_{m})$, 
there exists $y_{\varepsilon}(\tau)$, and exists a constant $\theta_{\varepsilon}(\tau)>0$ 
depending only on $d$, $q$, $\rho$, $\varepsilon$ and $\tau$ such that  
\begin{equation}\label{09/12/30/14:56}
    \int_{|x-y_{\varepsilon}(\tau)|\le R_{\varepsilon}} 
       \hspace{-6pt}
        \rho(x,t)
        \,dx 
        >
            \left( 
                      1-\varepsilon 
             \right)
                         A_{\rho} 
                \quad 
                          \mbox{for all 
                                     $t \in (\tau-\theta_{\varepsilon}(\tau), \tau+ \theta_{\varepsilon}(\tau))
                                            \subset 
                                                 (T_{\varepsilon}^*  ,  T_{m})$}
\end{equation}
Now we put 
    $
         I_{\tau}
           : =
                 (\tau-\theta_{\varepsilon}(\tau), \tau+\theta_{\varepsilon}(\tau))
     $. 
Then, we have    
\[
             \bigcup_{\tau \in (T_{\varepsilon}^*,T_{m}) }
                 I_{\tau}
                  =  
                       ( T_{\varepsilon}^*  ,   T_{m}),
\]
so that $\{I_{\tau} \}_{\tau \in (T_{\varepsilon}^*, T_{m})}$ 
is an open covering of $(T_{\varepsilon}^*, T_{m})$. 
Since $(T_{\varepsilon}^*, T_{m})$ is a Linder\"of space, 
we can take a countable subcovering $\{ I_{\tau_{k}}\}_{k \in \mathbb{N}}$, 
where $\{\tau_{k}\}_{k \in \mathbb{N}}$ is some increasing sequence in $(T_{\varepsilon}^*,T_{m})$ 
such that $\tau_k < \tau_{k+1}$, 
arranging that $I_{\tau_k}\cap I_{\tau_{k+2}}=\emptyset$ for $k=1,2, \cdots$.
We note that one can take $y_{\varepsilon}(\tau_1) = y_{\varepsilon}(T_{\varepsilon}^*)$.

We define $\gamma_{\varepsilon}:=\gamma_{\varepsilon}^*$ on $[T_{\varepsilon}, T_{\varepsilon}^*]$.
If necessary, we make an analogous procedure as in constructing $\gamma_{\varepsilon}^*$ in  \eqref{09/12/29/18:05} 
to define $\gamma_{\varepsilon}$ on $[T_{\varepsilon}^*, T_m)$:
by writing $I_{\tau_k}\cap I_{\tau_{k+1}} = (a_k, b_k)$ and $y_k:= y_{\varepsilon}(\tau_k)$,
\begin{equation*}
    \gamma_{\varepsilon}(t)
      =
          \begin{cases}
                        \quad  y_k  &   \mbox{if} \qquad  t\in I_{\tau_k}\setminus (a_k, b_k),
                                                         \\[6pt]
                        \quad   y_{k+1}
                                         +
                                      \displaystyle\frac{b_k - t}{\, b_k - a_k}
                                                   \left( 
                                                           y_{k}   -     y_{k+1}
                                                       \right) 
                                                                                 &\mbox{if}  \qquad t\in  (a_k, b_k),
                                                            \\[6pt]
                         \quad  y_{k+1} &   \mbox{if} \qquad t\in I_{\tau_{k+1}}\setminus (a_k, b_k).
              \end{cases}
\end{equation*} 
Then, $\gamma_{\varepsilon} \colon [T_{\varepsilon},T_m)\to \mathbb{R}^{d}$  
is continuous and satisfies that  
\[
      \int_{|x-\gamma_{\varepsilon}(t)|\le 3R_{\varepsilon}}
         \rho(x,t)
          \,dx
           \ge
                   (1-\varepsilon)
                                           A_{\rho}
                         \quad 
                         \mbox{for all $t \in [T_{\varepsilon}, T_{m})$},
\]
since we have
\begin{equation*}
      \left(
                   B_{R_{\varepsilon}}(y_{k}) 
                        \times 
                          \{t\} 
            \right)
             \cap 
               \left( 
                          B_{R_{\varepsilon}}(y_{k+1})
                             \times 
                                \{t\} 
                    \right) 
              \neq 
                         \emptyset 
              \quad \mbox{for all $t \in (a_k, b_k)$}.
\end{equation*}
\end{proof}

\begin{remark}
Here, we remark that:
 if $A_{\rho}=1$, the proof becomes easier.
Then the following estimate holds: 
 for all $\varepsilon>0$, 
\begin{equation}\label{09/12/27/20:12}
    \int_{
             B_{R_{\varepsilon}}(y_{\varepsilon}(T_{\varepsilon}^{*}))
               \cap 
                 B_{R_{\varepsilon}}(y_{\varepsilon}(T_{\varepsilon}))
                } 
                  \rho(x,t_{1})
            \,dx
        > 
               1-2\varepsilon
    \quad 
         \mbox{
                      for all 
                           $
                                 t \in (T_{\varepsilon}^{*}-\theta_{\varepsilon}(T_{\varepsilon}^{*}), T_{\varepsilon}^{*}]
                                 $
                           }.
\end{equation} 
Indeed, we have from (\ref{08/10/16/0:41}) and the definition of $T_{\varepsilon}^{*}$ that 
\[
     \begin{split}
   1
     &       \ge  
                 \int_{
                            B_{R_{\varepsilon}}(y_{\varepsilon}(T_{\varepsilon}^{*}))
                             \cup 
                                B_{R_{\varepsilon}}(y_{\varepsilon}(T_{\varepsilon}))
                             } 
                                \rho(x,t_{1})
                          \,dx\\
     &        =
                    \int_{
                              B_{R_{\varepsilon}}(y_{\varepsilon}(T_{\varepsilon}^{*}))
                               } 
                                  \rho(x,t_{1})
                                   \,dx
                +
                        \int_{
                                    B_{R_{\varepsilon}}(y_{\varepsilon}(T_{\varepsilon}))
                                  } 
                                    \rho(x,t_{1})
                              \,dx\\[6pt]
       & \hspace{60pt} -
                       \int_{
                                  B_{R_{\varepsilon}}(y_{\varepsilon}(T_{\varepsilon}^{*}))
                                    \cap 
                                        B_{R_{\varepsilon}}(y_{\varepsilon}(T_{\varepsilon}))
                                    } 
                                       \rho(x,t_{1})\,dx\\[6pt]
         &    > 
                 2\left(
                            1-\varepsilon
                    \right)
                                 -
                                    \int_{
                                                 B_{R_{\varepsilon}}(y_{\varepsilon}(T_{\varepsilon}^{*}))
                                                 \cap 
                                                     B_{R_{\varepsilon}}(y_{\varepsilon}(T_{\varepsilon}))
                                                 } 
                                                    \rho(x,t_{1})
                                                       \,dx
    \end{split}
\]
{
                                          for all 
                                                         $t \in (T_{\varepsilon}^{*}-\theta_{\varepsilon}(T_{\varepsilon}), T_{\varepsilon}^{*}]$
                                        }, 
which gives (\ref{09/12/27/20:12}). 
Thus we do not need to enlarge the radius of the ball centered at the $\gamma(t)$ for $t\in [T_{\varepsilon}, T_m)$.
\end{remark}
%

\section{Variational problems}\label{08/05/13/15:57}
We give the proofs of the variational problems,  Proposition \ref{08/06/16/15:24} and Proposition \ref{08/05/13/15:13}.
\par  
We begin with the proof of Proposition \ref{08/06/16/15:24}. 
\begin{proof}[Proof of Proposition \ref{08/06/16/15:24}]
We first prove the relation (\ref{08/05/13/11:33}):
\[
N_{1}^{\frac{p-1}{2}}=\left( \frac{2}{d}\right)^{\frac{p-1}{2}}
\left\{ \frac{d(p-1)}{(d+2)-(d-2)p}\right\}^{\frac{1}{4}\left\{ (d+2)-(d-2)p \right\}}N_{2}.
\]
 Take any  $f \in H^{1}(\mathbb{R}^{d})\setminus \{0\}$ with $\mathcal{K}(f)\le 0$ and put $f_{\lambda}(x)=\lambda^{\frac{2}{p-1}}f(\lambda x)$ for $\lambda>0$. Then, one can easily verify that:
\begin{align}
\|f_{\lambda}\|_{\widetilde{H}^{1}}^{2}
&=\lambda^{\frac{4}{p-1}+2-d}
\left\{ \frac{p-\left(1+\frac{4}{d}\right)}{p-1}\right\}
\|\nabla f\|_{L^{2}}^{2}
+\lambda^{\frac{4}{p-1}-d}\|f\|_{L^{2}}^{2},
\\[6pt]
\mathcal{K}(f_{\lambda})
&=\lambda^{\frac{4}{p-1}+2-d}\mathcal{K}(f)\le 0,
\\[6pt] 
\mathcal{N}_{2}(f_{\lambda})
&=
\mathcal{N}_{2}(f).
\end{align}
Moreover, an elementary calculus shows that $\|f_{\lambda}\|_{\widetilde{H}^{1}}^{2}$ takes the minimum at  
\begin{equation}\label{10/01/05/22:26}
\lambda =\left( \frac{d(p-1)}{d+2-(d-2)p}\right)^{\frac{1}{2}} \frac{\|f\|_{L^{2}}}{\|\nabla f\|_{L^{2}}},
\end{equation}
so that   
\begin{equation}\label{08/06/17/13:36}
\min_{\lambda>0}\|f_{\lambda}\|_{ \widetilde{H}^{1} }^{p-1}
=
\left( \frac{d}{2}\right)^{\frac{p-1}{2}} \left( \frac{d(p-1)}{d+2-(d-2)p} \right)^{\frac{1}{4}\left\{ d+2-(d-2)p+4 \right\} }\mathcal{N}_{2}(f). 
\end{equation}
Now, we consider a minimizing sequence $\{f_{n} \}_{n\in \mathbb{N}}$ of the variational problem for $N_{2}$ (see (\ref{08/07/02/23:24})). Then, it follows from (\ref{08/06/17/13:36}) and the definition of the variational value $N_{1}$ (see (\ref{08/07/02/23:23})) that 
\begin{equation}\label{10/01/03/11:06}
 N_{1}^{\frac{p-1}{2}}\le 
\left( \frac{d}{2}\right)^{\frac{p-1}{2}} \left( \frac{d(p-1)}{d+2-(d-2)p} \right)^{\frac{1}{4}\left\{ d+2-(d-2)p+4 \right\} }N_{2}.
\end{equation}
On the other hand, considering a minimizing sequence $\{f_{n}\}_{n\in \mathbb{N}}$ of the variational problem for $N_{1}$, we obtain
\begin{equation}\label{10/01/03/11:11}
N_{1}^{\frac{p-1}{2}}\ge 
\left( \frac{d}{2}\right)^{\frac{p-1}{2}} \left( \frac{d(p-1)}{d+2-(d-2)p} \right)^{\frac{1}{4}\left\{ d+2-(d-2)p+4 \right\} }N_{2}.
\end{equation}
This inequality (\ref{10/01/03/11:11}) together with (\ref{10/01/03/11:06}) gives (\ref{08/05/13/11:33}).
\par 
Next, we prove (\ref{08/05/13/11:25}):
\[
N_{3}=\frac{d(p-1)}{2(p+1)} N_{2}.
\] 
Let $\{f_{n}\}_{n\in \mathbb{N}}$ be a minimizing sequence of the variational problem for $N_{3}$ (see (\ref{08/07/02/23:25})). For each $f_{n}$, there is a constant $s_{n}>0$ such that $\mathcal{K}(s_{n}f_{n})=0$. We put $g_{n}=s_{n}f_{n}$. Then, one can easily verify that 
\[
\mathcal{I}(f_{n})=\mathcal{I}(g_{n}),
\]
so that 
\begin{equation}\label{10/01/03/11:39}
\lim_{n\to \infty}\mathcal{I}(g_{n})= N_{3}.
\end{equation}
Moreover, using $\mathcal{K}(g_{n})=0$, we find that      
\begin{equation}\label{10/01/03/11:43}
\mathcal{I}(g_{n})=\frac{d(p-1)}{2(p+1)}\mathcal{N}_{2}(g_{n})
  \ge \frac{d(p-1)}{2(p+1)}N_{2}
\quad 
\mbox{for all $n \in \mathbb{N}$}.
\end{equation}
Hence, (\ref{10/01/03/11:39}) and (\ref{10/01/03/11:43}) give us that 
\begin{equation}\label{10/01/03/11:45}
\frac{d(p-1)}{2(p+1)}N_{2} \le N_{3}. 
\end{equation}
Now, we consider a minimizing sequence $\{f_{n}\}_{n\in \mathbb{N}}$ of the variational problem for $N_{2}$ (see (\ref{08/07/02/23:24})). Then, it follows from  $\mathcal{K}(f_{n})\le 0$ and the definition of $N_{3}$ that 
\begin{equation}\label{10/01/03/11:58}
N_{3}\frac{2(p+1)}{d(p-1)}\|\nabla f_{n}\|_{L^{2}}^{2}
\le N_{3} \|f_{n}\|_{L^{p+1}}^{p+1}
\le \|f_{n}\|_{L^{2}}^{p+1-\frac{d}{2}(p-1)}\|\nabla f_{n} \|_{L^{2}}^{\frac{d}{2}(p-1)}.
\end{equation}
Dividing the both sides of (\ref{10/01/03/11:58}) by $\frac{2(p+1)}{d(p-1)}\|\nabla f_{n}\|_{L^{2}}^{2}$, we have   
\begin{equation}\label{10/01/03/11:59}
N_{3} \le \frac{d(p-1)}{2(p+1)}\|f_{n}\|_{L^{2}}^{p+1-\frac{d}{2}(p-1)}\|\nabla f_{n} \|_{L^{2}}^{\frac{d}{2}(p-1)-2}
=
\frac{d(p-1)}{2(p+1)}
\mathcal{N}_{2}(f_{n})
.
\end{equation}
Since $\{f_{n}\}_{n\in \mathbb{N}}$ is a minimizing sequence of the variational problem for $N_{2}$, taking $n\to \infty$ in (\ref{10/01/03/11:59}), 
we obtain 
\begin{equation*}
N_{3}\le \frac{d(p-1)}{2(p+1)}N_{2},
\end{equation*}
which together with (\ref{10/01/03/11:45}) gives (\ref{08/05/13/11:25}). 
\end{proof}

Next, we give the proof of Proposition \ref{08/05/13/15:13}.
\begin{proof}[Proof of Proposition \ref{08/05/13/15:13}] 
We first solve the variational problem for  
\[
N_{1}=\inf\left\{ \left\| f \right\|_{\widetilde{H}^{1}}^{2}
\, | \, f \in H^{1}(\mathbb{R}^{d})\setminus \{0\}, \ \mathcal{K}(f)\le 0
\right\}.
\]
Considering a concrete function, we verify that  
\begin{equation}\label{10/01/05/10:19}
N_{1} \lesssim 1,
\end{equation}
where the implicit constant depends only on $d$ and $p$. Moreover, the functional $\|\cdot \|_{\widetilde{H}^{1}}$ satisfies that 
\begin{equation}\label{09/12/19/15:38}
\|f\|_{H^{1}}^{2}
\le 
\frac{d(p-1)}{d(p-1)-4}\|f\|_{\widetilde{H}^{1}}^{2}
\quad 
\mbox{for all $f \in H^{1}(\mathbb{R}^{d})$}.
\end{equation}
We take a minimizing sequence $\{ f_{n}\}_{n \in \mathbb{N}}$ for this problem, so that     
\begin{align}\label{09/12/19/15:33}
\lim_{n\to \infty}\|f_{n}\|_{\widetilde{H}^{1}}^{2} 
&=N_{1}, 
\\[6pt]
\label{09/12/19/15:32}
\mathcal{K}(f_{n})&\le 0 
\quad 
\mbox{for all $n \in \mathbb{N}$}.
\end{align} 
Here, the property (\ref{09/12/19/15:32}) is equivalent to the inequality        
\begin{equation}\label{09/12/19/15:54}
\frac{2(p+1)}{d(p-1)}\|\nabla f_{n} \|_{L^{2}}^{2} 
\le 
\|f_{n}\|_{L^{p+1}}^{p+1}
\quad 
\mbox{for all $n \in \mathbb{N}$}. 
\end{equation}
By (\ref{10/01/05/10:19}), (\ref{09/12/19/15:38}) and (\ref{09/12/19/15:33}), we obtain the uniform bound of $\{f_{n}\}_{n \in \mathbb{N}}$ in $H^{1}(\mathbb{R}^{d})$: there exists $C_{1}>0$ depending only on $d$ and $p$ such that 
\begin{equation}\label{08/05/13/16:01}
\sup_{n \in \mathbb{N}}\|f_{n}\|_{H^{1}}\le C_{1}.
\end{equation}
Moreover, by the Gagliardo-Nirenberg inequality (we can obtain this inequality without solving the variational problem for $N_{3}$) and (\ref{08/05/13/16:01}),  we have   
\begin{equation}\label{09/12/19/16:04}
\|f_{n}\|_{L^{p+1}}^{p+1} 
\lesssim  
C_{1}^{p+1-\frac{d}{2}(p-1)}\|\nabla f_{n}\|_{L^{2}}^{\frac{d}{2}(p-1)},
\end{equation}
where the implicit constant depends only on $d$ and $p$. This inequality (\ref{09/12/19/16:04}), together with (\ref{09/12/19/15:54}), yields  that  
\begin{equation}\label{08/06/28/22:39}
\frac{2(p+1)}{d(p-1)} \lesssim 
C_{1}^{p+1-\frac{d}{2}(p-1)}
\|\nabla f_{n}\|_{L^{2}}^{\frac{d}{2}(p-1)-2}.
\end{equation}
Hence, using (\ref{09/12/19/15:54}) again, we have that: there exists  $C_{2}>0$ depending only on $d$ and $p$ such that 
\begin{equation}\label{08/05/13/15:39}
\inf_{n\in \mathbb{N}}\|f_{n}\|_{L^{p+1}}\ge C_{2}.
\end{equation}
The properties (\ref{08/05/13/16:01}) and (\ref{08/05/13/15:39}) enable us to apply Lemma \ref{08/3/28/20:21}, so that we obtain   
\[
\mathcal{L}^{d}\left[ |f_{n}| \ge \delta\right]>C
\]
for some constants $C$ and $\delta >0$ independent of $n$. 
Moreover, applying Lemma \ref{08/03/28/21:00}, we find that: there exists $y_{n} \in \mathbb{R}^{d}$ such that, putting $\widetilde{f}_{n}(x)=f_{n}(x+y_{n})$,
 we have     
\begin{equation}\label{08/05/19/14:04}
\mathcal{L}^{d}\left( \left[ \left|\widetilde{f}_{n}\right|\ge \frac{\delta}{2}\right]
\cap B_{1}(0) \right)>C'
\end{equation}
for some constant $C'>0$ independent of $n$. Here, we can easily verify that this sequence $\{\widetilde{f}_{n}\}_{n\in \mathbb{N}}$ has same properties as the original one $\{f_{n}\}_{n\in \mathbb{N}}$:
\begin{align}
\label{10/01/05/11:16}
\lim_{n\to \infty}\|\widetilde{f}_{n}\|_{\widetilde{H}^{1}}^{2} 
&=N_{1},
\\[6pt]
\label{10/01/05/11:18}
\mathcal{K}(\widetilde{f}_{n})&\le 0 
\quad 
\mbox{for all $n \in \mathbb{N}$},
\\[6pt]
\label{10/01/05/11:22}
\sup_{n \in \mathbb{N}}\|\widetilde{f}_{n}\|_{H^{1}}&\le C_{1}.
\end{align}
We apply Lemma \ref{08/09/27/22:43} to $\{\widetilde{f}_{n}\}_{n\in \mathbb{N}}$ and obtain a subsequence $\{\widetilde{f}_{n}\}_{n \in \mathbb{N}}$ (still denoted by the same symbol) and a nontrivial function $Q \in H^{1}(\mathbb{R}^{d})$ such that 
\begin{equation}\label{10/01/03/15:38}
\lim_{n\to \infty}
\widetilde{f}_{n} = Q 
\quad 
\mbox{weakly in $H^{1}(\mathbb{R}^{d})$}.
\end{equation}
The property (\ref{10/01/03/15:38}) also gives us that 
\begin{equation}
\begin{split}
\label{08/05/14/12:07}
\|\nabla \widetilde{f}_{n}\|_{L^{2}}^{2}
-\|\nabla \widetilde{f}_{n}-\nabla Q\|_{L^{2}}^{2}
-\|\nabla Q\|_{L^{2}}^{2}
&=2\Re \int_{\mathbb{R}^{d}} (\nabla \widetilde{f}_{n}-\nabla Q) \overline{\nabla Q}\,dx 
\\[6pt]
&\to 0 
\qquad 
\mbox{as $n \to \infty$},
\end{split}
\end{equation}
and 
\begin{equation}
\label{10/01/05/15:06}
\|Q \|_{\widetilde{H}^{1}}^{2}\le N_{1}.
\end{equation}
Here, (\ref{10/01/05/15:06}) follows from the lower continuity in the weak topology and (\ref{10/01/05/11:16}). Moreover, Lemma \ref{08/03/28/21:21}, together with (\ref{10/01/03/15:38}), gives us that: for all $2\le q < 2^{*}$,  
\begin{equation}\label{08/05/14/12:08}
\|\widetilde{f}_{n}\|_{L^{q}}^{q}-\|\widetilde{f}_{n}-Q\|_{L^{q}}^{q}-\|Q\|_{L^{q}}^{q} \to 0 
\quad 
\mbox{as $n \to \infty$}.
\end{equation}
It follows from (\ref{08/05/14/12:07}) and (\ref{08/05/14/12:08}) that 
\begin{equation}\label{08/05/14/12:17}
\mathcal{K}(\widetilde{f}_{n})-\mathcal{K}(\widetilde{f}_{n}-Q)-\mathcal{K}(Q) \to 0 
\quad 
\mbox{as $n \to \infty$}
\end{equation}
and 
\begin{equation}\label{08/05/14/12:26}
\|\widetilde{f}_{n} \|_{\widetilde{H}^{1}}^{2}
-\| \widetilde{f}_{n}-Q\|_{\widetilde{H}^{1}}^{2}
-\| Q \|_{\widetilde{H}^{1}}^{2} \to 0 
\quad 
\mbox{as $n \to \infty$}.
\end{equation}
We show that the function $Q$ is a minimizer of the variational problem for $N_{1}$. To this end, it suffices to prove that 
\begin{equation}\label{10/01/04/22:59}
\mathcal{K}(Q)\le 0.
\end{equation} 
Indeed, by the definition of $N_{1}$, (\ref{10/01/04/22:59}) implies that $N_{1}\le \|Q\|_{\widetilde{H}^{1}}$, which, together with (\ref{10/01/05/15:06}), yields that 
\begin{equation}\label{10/01/05/15:01}
\|Q\|_{\widetilde{H}^{1}}^{2}=N_{1}.
\end{equation} 
We prove (\ref{10/01/04/22:59}) by contradiction: suppose that $\mathcal{K}(Q)>0$. Then, (\ref{10/01/05/11:18}) and (\ref{08/05/14/12:17}) imply that   
\[
\mathcal{K}(\widetilde{f}_{n}-Q) \le 0
\qquad 
\mbox{for all sufficiently large $n \in \mathbb{N}$}.
\] 
Therefore, by the definition of $N_{1}$, we have  
\begin{equation}\label{08/05/14/12:27}
\|\widetilde{f}_{n}-Q\|_{\widetilde{H}^{1}}^{2}\ge N_{1}
\qquad 
\mbox{for all sufficiently large $n \in \mathbb{N}$}. 
\end{equation}
Combining (\ref{08/05/14/12:26}), (\ref{10/01/05/11:16}) and (\ref{08/05/14/12:27}), we obtain that $\|Q\|_{\widetilde{H}^{1}}=0$, which is a contradiction. Hence, we have proved (\ref{10/01/04/22:59}). 
\par 
Now, it follows from (\ref{10/01/05/11:16}),  (\ref{10/01/03/15:38}) and (\ref{10/01/05/15:01}) that 
\[
\begin{split}
\|\widetilde{f}_{n}-Q\|_{\widetilde{H}_{1}}^{2}
&=\|\widetilde{f}_{n}\|_{\widetilde{H}^{1}}^{2}
-\frac{2\left\{ p-\left( 1+\frac{4}{d}\right)\right\} }{p-1}\Re{\int_{\mathbb{R}^{d}} \nabla \widetilde{f}_{n}(x) \overline{\nabla Q(x)}\,dx }
\\
& \qquad \qquad 
-2\Re{\int_{\mathbb{R}^{d}} \widetilde{f}_{n}(x) \overline{Q(x)}\,dx }+
\|Q\|_{\widetilde{H}^{1}}^{2}
\\
& \to 0
\qquad 
\mbox{as $n \to \infty$}.
\end{split}
\]
This, together with (\ref{09/12/19/15:38}), immediately yields the strong convergence:
\begin{equation}\label{10/01/04/23:10}
\lim_{n\to \infty}\widetilde{f}_{n} = Q \ \mbox{ strongly in $H^{1}$}.
\end{equation}  
We shall prove 
\begin{equation}\label{10/01/05/21:19}
\mathcal{K}(Q)=0.
\end{equation}
Putting $Q_{s}=sQ$ for $s \in \mathbb{R}$, we have that 
\[
\mathcal{K}(Q_{s})>0 
\quad 
\mbox{for all  $s \in  \Big( 0,\    \left\{ \frac{2(p+1)}{d(p-1)}\frac{\|\nabla Q\|_{L^{2}}^{2}}{\|Q\|_{L^{p+1}}^{p+1}}\right\}^{\frac{1}{p-1}}
\Big)$}. 
\]
Here, (\ref{10/01/04/22:59}) implies that 
\[
\left\{ \frac{2(p+1)}{d(p-1)}\frac{\|\nabla Q\|_{L^{2}}^{2}}{\|Q\|_{L^{p+1}}^{p+1}}\right\}^{\frac{1}{p-1}}\le 1. 
\]
Supposing the undesired situation $\mathcal{K}(Q)<0$ ($\mathcal{K}(Q)\le 0$ has been proved already), we have by the intermediate value theorem that there exists $s_{0} \in (0,1)$ such that $\mathcal{K}(Q_{s_{0}})=0$, so that $\|Q_{s_{0}}\|_{\widetilde{H}^{1}}^{2}\ge N_{1}$ by the definition of $N_{1}$. However, it follows from (\ref{10/01/05/15:01}) that  
\[
\|Q_{s_{0}}\|_{\widetilde{H}^{1}}^{2}=s_{0}^{2}\|Q\|_{\widetilde{H}^{1}}^{2}
< \|Q\|_{\widetilde{H}^{1}}^{2}=N_{1},
\]
which is a contradiction. Hence, (\ref{10/01/05/21:19}) holds valid. 
\par 
Next, we consider the variational problem for $N_{2}$. We show that the function $Q$ obtained above is a minimizer of this problem. Put $Q_{\lambda}(x)=\lambda^{\frac{2}{p-1}}Q(\lambda x)$ for $\lambda>0$. We easily verify that 
\[
\mathcal{K}(Q_{\lambda})=0
 \quad 
 \mbox{for all $\lambda >0$},
\]
which implies that  
\[
\|Q_{\lambda}\|_{\widetilde{H}^{1}}^{2}\ge N_{1} 
\quad 
\mbox{for all $\lambda>0$}.
\]
Then, the same argument as (\ref{10/01/05/22:26}), together with (\ref{10/01/05/15:01}), shows that $\|Q_{\lambda}\|_{\widetilde{H}^{1}}\colon (0,\infty)\to [0,\infty)$ takes the minimum at 
\begin{equation}\label{10/01/05/22:28}
\lambda=\left( \frac{d(p-1)}{d+2-(d-2)p}\right)^{\frac{1}{2}}\frac{\left\| Q\right\|_{L^{2}}}{\left\| \nabla Q\right\|_{L^{2}}}=1.
\end{equation}
Therefore, as well as (\ref{08/06/17/13:36}), we have that 
\begin{equation*}\label{09/12/19/17:25}
\begin{split}
N_{1}^{\frac{p-1}{2}}&=\|Q\|_{\widetilde{H}^{1}}^{p-1}
=\|Q_{\lambda}\|_{\widetilde{H}^{1}}^{p-1}|_{\lambda=1}
\\
&=
\left( \frac{d}{2}\right)^{\frac{p-1}{2}} \left( \frac{d(p-1)}{d+2-(d-2)p} \right)^{\frac{1}{4}\left\{ d+2-(d-2)p+4 \right\} }\mathcal{N}_{2}(Q).
\end{split}
\end{equation*}
This, together with (\ref{08/05/13/11:33}) in Proposition \ref{08/06/16/15:24},   leads us to the conclusion that 
\begin{equation}\label{10/01/05/21:17}
\mathcal{N}_{2}(Q)=N_{2}.
\end{equation} 
Here, we remark that (\ref{10/01/05/21:19}) and (\ref{10/01/05/22:28}) yield the relation (\ref{08/09/24/15:13}): 
\[
\left\| Q \right\|_{L^{2}}^{2}
=
\frac{d+2-(d-2)p}{d(p-1)}
\left\| \nabla Q\right\|_{L^{2}}^{2}
=
\frac{d+2-(d-2)p}{2(p+1)}
\left\| Q \right\|_{L^{p+1}}^{p+1}.
\]  

We shall show that $Q$ is also a minimizer of the variational problem for $N_{3}$. Indeed, (\ref{08/05/13/11:25}) in Proposition \ref{08/06/16/15:24}, together with (\ref{10/01/05/21:19}) and (\ref{10/01/05/21:17}), yields that     
\begin{equation}\label{10/01/05/21:48}
\begin{split}
\mathcal{I}(Q)&=
\frac{ \|Q\|_{L^{2}}^{p+1-\frac{d}{2}(p-1)} 
\|\nabla Q\|_{L^{2}}^{ \frac{d}{2}(p-1) }}
 {\|Q\|_{L^{p+1}}^{p+1}}
\\[6pt]
&=\frac{d(p-1)}{2(p+1)}\mathcal{N}_{2}(Q)
=\frac{d(p-1)}{2(p+1)}N_{2}=\mathcal{N}_{3}.
\end{split}
\end{equation}
\indent 
Finally, we prove that $Q$ satisfies the equation (\ref{08/05/13/11:22}) with $\omega =1$:  
\begin{equation}\label{08/06/14/13:28}
\Delta Q -Q +|Q|^{p-1}Q=0.
\end{equation}
Since $\mathcal{I}(Q)$ is the critical value of $\mathcal{I}$ (see (\ref{10/01/05/21:48})), we have  
\begin{equation}\label{10/01/06/12:25}
\begin{split}
&0=\frac{d}{d\varepsilon} \mathcal{I}(Q+\varepsilon \phi) 
\biggm|_{\varepsilon=0}
\\[6pt]
&=\frac{
(p+1-\frac{d}{2}(p-1))
\|Q\|_{L^{2}}^{p-1-\frac{d}{2}(p-1)}
\|\nabla Q\|_{L^{2}}^{\frac{d}{2}(p-1)}
\Re \int Q \overline{\phi} dx
}{\|Q\|_{L^{p+1}}^{p+1}}
\\[6pt]
&\qquad 
+ 
\frac{
\frac{d}{2}(p-1)
\|Q\|_{L^{2}}^{p+1-\frac{d}{2}(p-1)}
\|\nabla Q\|_{L^{2}}^{\frac{d}{2}(p-1)-2}
\Re \int \nabla Q \overline{\nabla \phi} dx
}{\|Q\|_{L^{p+1}}^{p+1}}
\\[6pt]
&\qquad 
- 
\frac{
(p+1)
\|Q\|_{L^{2}}^{p+1-\frac{d}{2}(p-1)}
\|\nabla Q\|_{L^{2}}^{\frac{d}{2}(p-1)}
\Re \int |Q|^{p-1}Q \overline{\phi} dx
}{\|Q\|_{L^{p+1}}^{2(p+1)}}
\quad 
\mbox{for all $\phi \in C_{c}^{\infty}(\mathbb{R}^{d})$}.
\end{split}
\end{equation}
Combining  (\ref{10/01/06/12:25})  with (\ref{10/01/05/21:19}) ($\| Q \|_{L^{p+1}}^{p+1}=\frac{2(p+1)}{d(p-1)}\|\nabla Q \|_{L^{2}}^{2}$), we obtain that 
\begin{equation}\label{10/01/06/12:37}
\begin{split}
0&=\frac{d^{2}(p-1)^{2}}{4(p+1)}\left(\frac{2(p+1)}{d(p-1)} -1 \right)
\|Q\|_{L^{2}}^{p-1-\frac{d}{2}(p-1)}
\|\nabla Q\|_{L^{2}}^{\frac{d}{2}(p-1)-2}
\Re \int_{\mathbb{R}^{d}} Q(x)\overline{\phi(x)}\,dx 
\\[6pt]
& \qquad 
+
\frac{d^{2}(p-1)^{2}}{4(p+1)}
\frac{\|Q\|_{L^{2}}^{p+1-\frac{d}{2}(p-1)}
\|\nabla Q\|_{L^{2}}^{\frac{d}{2}(p-1)-2}}
{\|\nabla Q\|_{L^{2}}^{2}}
\Re \int_{\mathbb{R}^{d}} \nabla Q(x)\overline{\nabla \phi(x)}\,dx 
\\[6pt]
&\qquad -
\frac{d^{2}(p-1)^{2}}{4(p+1)}\frac{\|Q\|_{L^{2}}^{p+1-\frac{d}{2}(p-1)}
\|\nabla Q\|_{L^{2}}^{\frac{d}{2}(p-1)-2}}{\|\nabla Q\|_{L^{2}}^{2}}
\Re \int_{\mathbb{R}^{d}}|Q(x)|^{p-1}Q(x) \overline{\phi(x)}\,dx.
\end{split}
\end{equation}
This, together with (\ref{10/01/05/22:28}), shows that $Q$ satisfies the equation (\ref{08/06/14/13:28}) in a weak sense. Moreover, it turns out that $Q$ has the following properties: $Q$ is positive (see \cite{Lieb-Loss}), radially symmetric (see \cite{Gidas-Ni-Nirenberg}), unique up to the translations and phase shift (see \cite{Kwong}), and satisfies the decay estimate (see \cite{Berestycki-Lions, Lieb-Loss}):   
\begin{equation}\label{10/01/21/11:02}
|\partial^{\alpha} Q(x)|\le C e^{-\delta |x|}
\quad 
\mbox{for all multi-index $\alpha$ with $|\alpha|\le 2$},  
\end{equation}
where $C$ and $\delta$ are some positive constants. 
\par 
Finally, we shall show that $Q$ belongs to the Schwartz space $\mathcal{S}(\mathbb{R}^{d})$. Since $Q$ is a smooth and radially symmetric solution to 
 (\ref{08/06/14/13:28}),  we have that  
\begin{equation}\label{10/01/21/10:56}
\frac{d^{2}}{dr^{2}}\partial^{\alpha} Q= \partial^{\alpha}Q+\partial^{\alpha}Q^{p}-\partial^{\alpha} \frac{d-1}{r}\frac{d}{dr} Q 
\quad 
\mbox{for all multi-index $\alpha$ and $r=|x|$}. 
\end{equation}
Then, an induction argument, together with (\ref{10/01/21/11:02}), leads to 
 that $Q \in \mathcal{S}(\mathbb{R}^{d})$. 
\end{proof}

\section*{Acknowledgments}
The authors would like to express their deep gratitude to Professor Kenji Nakanishi for pointing out a mistake in the earliest version of this paper. The authors also thank Doctor Takeshi Yamada for his valuable comments.  
\par 
H. Nawa is partially supported by the Grant-in-Aid for Scientific Research (Challenging Exploratory Research \# 19654026) of JSPS.

\bibliographystyle{plain}

\vspace{24pt}
\noindent
{
Takafumi Akahori,
\\
Graduate School of Science and Engineering
\\
Ehime University, 
\\ 
2-5 Bunkyo-cho, Matsuyama, 790-8577, Japan 
\\ 
E-mail: akahori@math.sci.ehime-u.ac.jp
}
\\
\\
{
Hayato Nawa,
\\
Division of Mathematical Science, Department of System Inovation
\\
Graduate School of Engineering Science
\\
Osaka University, 
\\
Toyonaka 560-8531 , JAPAN 
\\
E-mail: nawa@sigmath.es.osaka-u.ac.jp
}
\end{document}